\documentclass{article}

\usepackage[utf8]{inputenc}
\usepackage[margin=1in]{geometry}
\usepackage[titletoc,title]{appendix}

\usepackage{natbib}
\setcitestyle{authoryear,round,citesep={;},aysep={,},yysep={;}}

\usepackage[colorlinks, citecolor=blue, pagebackref=true]{hyperref}
\usepackage{amsmath,amsfonts,amssymb,amsthm,mathtools,thmtools}
\usepackage{cleveref}
\usepackage{bbold}
\usepackage{bm}

\usepackage{graphicx}
\usepackage{float}
\usepackage{subfig}

\usepackage{booktabs} 
\usepackage{multirow}

\usepackage{algorithm}
\usepackage{algorithmic}

\usepackage{tikz-cd}

\renewcommand*\backref[1]{\ifx#1\relax \else (Cited on p. #1) \fi}

\title{Variational Analysis in the Wasserstein Hierarchy}
\author{Christophe Vauthier\footnote{Université Paris-Saclay, Laboratoire de Mathématiques d'Orsay}}
\date{September 2025}

\newtheorem{definition}{Definition}[section]
\newtheorem{theorem}[definition]{Theorem}
\newtheorem{proposition}[definition]{Proposition}

\newtheorem{example}[definition]{Example}
\newtheorem{remark}[definition]{Remark}
\newtheorem{lemma}[definition]{Lemma}
\newtheorem{corollary}[definition]{Corollary}
\newtheorem{assumption}[definition]{Assumption}

\crefname{assumption}{assumption}{assumptions}

\numberwithin{equation}{subsection}

\usepackage{amsmath,amssymb,mathtools}
\usepackage{amsthm}
\usepackage{enumitem}
\usepackage{cleveref}
\usepackage{physics}

\usepackage[textwidth=2cm]{todonotes}  
\usepackage{xargs}

\newcommand{\cB}{\mathcal{B}}
\newcommand{\cC}{\mathcal{C}}
\newcommand{\cD}{\mathcal{D}}
\newcommand{\cE}{\mathcal{E}}
\newcommand{\cF}{\mathcal{F}}
\newcommand{\cG}{\mathcal{G}}
\newcommand{\cH}{\mathcal{H}}
\newcommand{\cK}{\mathcal{K}}
\newcommand{\cL}{\mathcal{L}}
\newcommand{\cM}{\mathcal{M}}

\newcommand{\cP}{\mathcal{P}}
\newcommand{\cU}{\mathcal{U}}
\newcommand{\cV}{\mathcal{V}}

\newcommand{\bA}{\mathbb{A}}
\newcommand{\bB}{\mathbb{B}}
\newcommand{\bE}{\mathbb{E}}

\newcommand{\N}{\mathbb{N}}

\newcommand{\bP}{\mathbb{P}}
\newcommand{\bQ}{\mathbb{Q}}
\newcommand{\R}{\mathbb{R}}
\newcommand{\bS}{\mathbb{S}}

\newcommand{\bGamma}{\mathbb{\Gamma}}

\newcommand{\bOne}{\mathbb{1}}

\newcommand{\W}{\mathrm{W}}

\newcommand{\Id}{\mathrm{Id}}
\newcommand{\id}{\mathrm{id}}

\newcommand{\spt}{\mathrm{spt}}

\newcommand{\cPP}[1]{\mathcal{P}_2\big(\cP_2(#1)\big)}

\newcommand{\cPzn}[2]{\cP^{(#1)}(#2)}
\newcommand{\cPzndet}[2]{\cP_{\det}^{(#1)}(#2)}
\newcommand{\cPn}[2]{\cP_2^{(#1)}(#2)}
\newcommand{\cPndet}[2]{\cP_{2,\det}^{(#1)}(#2)}
\newcommand{\Pn}[2]{P^{(#1)}(#2)}
\newcommand{\Pndet}[2]{P_{\det}^{(#1)}(#2)}

\renewcommand{\dd}{\;\mathrm{d}}

\newcommand{\setcond}{\,|\,}

\newcommand{\PT}{\mathrm{PT}}
\newcommand{\Pol}{\mathrm{Pol}}
\newcommand{\Pold}{\mathrm{Pol}_{\mathrm{d}}}

\newcommand{\gW}{\nabla_{\W_2}}

\DeclareMathOperator{\Hess}{Hess}

\newcommand{\sca}[2]{\left\langle#1\middle|#2\right\rangle}
\newcommand{\veps}{\varepsilon}

\let\oldqedhere\qedhere
\renewcommand{\qedhere}{\pushQED{\qed}\oldqedhere}

\DeclareMathOperator{\ob}{\mathrm{ob}}
\DeclareMathOperator{\Hom}{\mathrm{Hom}}

\allowdisplaybreaks

\begin{document}

\maketitle

\begin{abstract}
    Let $\cM$ be a complete connected Riemannian manifold. For $n \geq 0$, we endow the Wasserstein space $\cPn{n}{\cM} = \cP_2(\ldots\cP_2(\cM)\ldots)$, equipped with the Wasserstein distance $\W_2$, with a variational structure that generalizes the standard variational structure on $\cP_2(\cM)$ provided by optimal transport theory. Our approach makes use of tools from category theory to lift the geometric structure of the manifold $\cM$ to the spaces $\cPn{n}{\cM}$, in order to establish in a principled way a rigorous theoretical framework for variational analysis on the space $\cPn{n}{\cM}$. In particular, we obtain a precise characterization of the constant speed geodesics of the space $\cPn{n}{\cM}$ in terms of optimal velocity plans. Moreover, we introduce a notion of gradient for functionals defined on $\cPn{n}{\cM}$, which allows us to study the differentiability and the convexity of various types of such functionals.
\end{abstract}

\section{Introduction} \label{sec:1_introduction}

Let $(X,d)$ be a separable and complete metric space. Denote by $\cP_2(X)$ the space of Borel probability measures on $X$ with finite second order moment, then, for every pair $\mu,\nu \in \cP_2(X)$, their \emph{(second order) Wasserstein distance} $\W_2(\mu,\nu)$ is defined by 
\begin{equation} \label{eq:l_70}
    \W_2^2(\mu,\nu) := \inf_{\sigma \in \Pi(\mu,\nu)} \int d^2(x,y) \dd\sigma(x,y),
\end{equation} 
where $\Pi(\mu,\nu)$ is the set of probability measures $\sigma$ on $X \times X$ with respective first and second marginals $\mu$ and $\nu$. These $\sigma$ are called \emph{transport plans} between $\mu$ and $\nu$. It is a general result of optimal transportation theory (see \citep[Theorem 4.1]{villani2009optimal}) that there exists $\sigma \in \Pi(\mu,\nu)$ which are minimizers of \eqref{eq:l_70}. These $\sigma$ are called \emph{optimal transport plans}, and we denote their subset by $\Pi_o(\mu,\nu)$. Moreover, $\W_2$ is a distance on the set $\cP_2(X)$, and the metric space $(\cP_2(X),\W_2)$ is itself separable and complete \citep[Theorem 6.18]{villani2009optimal}.
\medbreak
The space $(\cP_2(X),\W_2)$ has come to play an important role in many areas of mathematics and computer science, including partial differential equations \citep{santambrogio2015optimal}, optimization and sampling \citep{santambrogio2017euclidean, wibisono2018sampling, salim2020wasserstein}, image processing \citep{bonneel2023survey}, and machine learning \citep{peyre2019computational, torres2021survey, montesuma2024multi}. Indeed, not only does the Wasserstein distance $\W_2$ offer a method to compare probability measures which takes into account the geometrical properties of their supports better than other methods (such as the Kullback-Leibler divergence, which is infinite on pairs of measures with disjoint support), but when $X = \R^d$, the space $\cP_2(\R^d)$ can be equipped with a rich variational structure which offers a natural framework to tackle problems consisting of minimizing a functional $\cF : \cP_2(\R^d) \to \R$ defined on the space of probability measures. We refer for further details to textbooks such as \citep{ambrosio2005gradient, villani2009optimal}.
\medbreak
In this article, we are interested in the case where $X$ is itself a probability space ; more precisely, we study the case where $X = \cPn{n}{\cM} := \cP_2(\ldots\cP_2(\cM)\ldots)$, with $n \geq 0$ applications of $\cP_2$ to $\cM$, with $\cM$ some fixed complete connected Riemannian manifold (without boundary), equipped with its geodesic distance $d$. Indeed, we can define a sequence of Wasserstein spaces $\cM$, $\cP_2(\cM)$, $\cP_2(\cP_2(\cM))$, $\cP_2(\cP_2(\cP_2(\cM)))$... where, for every $n \geq 1$, the Wasserstein distance $\W_2$ on $\cPn{n+1}{\cM}$ is defined inductively using the Wasserstein distance $\W_2$ on $\cPn{n}{\cM}$: that is, for every $\bP, \bQ \in \cPn{n+1}{\cM}$, we define
\begin{equation}
    \W_2^2(\bP,\bQ) := \inf_{\bGamma \in \Pi(\bP,\bQ)} \int \W_2^2(\mu,\nu) \dd\bGamma(\mu,\nu).
\end{equation}
We will use the same notation $\W_2$ for the Wasserstein distances of all the spaces $\cPn{n}{\cM}$, as there will always be no ambiguity on the ``level" $n$ of the probability measures we are working on. 
\medbreak
Our study of this sequence of spaces, which we dub the \emph{Wasserstein hierarchy}, is motivated by the recent finding in the machine learning literature that the space $\cP_2(\cP_2(\cM))$ of probability measures over probability measures can be used to represent datasets of many types of data, including words \citep{vilnis2015word}, graphs \citep{bojchevski2017deep}, point clouds \citep{geuter2025ddeqs}, images \citep{sodini2025approximation, geuter2025universal}, and neuroscience data \citep{bonet2023sliced}. Moreover, $\cP_2(\cP_2(\cM))$ provides a natural way to represent a labeled dataset of any data living on $\cM$ \citep{alvarez2020geometric}. This motivated the development and analysis of different distances on this space \citep{alvarez2020geometric, catalano2024hierarchical, nguyen2025lightspeed, piening2025slicing, bonet2025busemann} ; with applications to various tasks, such as generative modeling \citep{dukler2019wasserstein, haviv2024wasserstein}, flowing datasets \citep{alvarez2021dataset, hua2023dynamic, bonet2025flowing}, or domain adaptation \citep{el2022hierarchical}. This in turn resulted in a growing interest for the spaces $\cPn{2}{X}$ in the optimal transport literature, which have been the object of recent theoretical works such as \citep{pinzi2025nested, pinzi2025totally}.

The case $n>2$ has to the best of our knowledge only been considered in \citep{beiglböck2025brenier}, who study the spaces $\cPn{n}{\cH}$ when $\cH$ is a separable Hilbert space. \citet{beiglböck2025brenier} also linked this problem to adapted optimal transport \citep{backhoff2017causal}, which allows to compare stochastic processes with $n$ snapshots. 
\medbreak
It is a priori difficult to consider optimization problems on the Wasserstein spaces $\cP_2(\cPn{n}{\cM})$. Indeed, much of the existing work on the geometric properties and variational structure of the space $\cP_2(X)$ makes use of the structure of the underlying space $X$, which is often taken to be a Hilbert space or a Riemannian manifold \citep{ambrosio2005gradient, villani2009optimal}, a structure which $\cPn{n}{\cM}$ lacks whenever $n \geq 1$. Nevertheless, we show that it is still possible to equip $\cP_2(\cPn{n}{\cM})$ with a variational structure which yields a precise characterization of its geodesics and a rigorous framework for studying the differentiation and the convexity of functionals, which in turn may allow the application of the framework of gradient flows on metric spaces \citep{ambrosio2005gradient} to this space.
\medbreak
In order to explain our results in more detail, it will first be useful to recall some basic properties of the Wasserstein space $\cP_2(\cM)$ (and $\cP_2(\R^d)$), which has been extensively studied in the optimal transport literature. A major result is Brenier-McCann's theorem \citep[Theorem 10.41]{villani2009optimal}, which states that for every $\mu, \nu \in \cP_2(\cM)$ which are compactly supported, and such that $\mu$ is absolutely continuous (with respect to the volume measure of $\cM$), there exists a unique optimal plan $\sigma$ between them, which is of the form $\sigma = (\id,S)_\#\mu$ (where $\#$ denotes the pushforward operation - see the Notations section below), where $S : \cM \mapsto \cM$ is a map of the form $S = \exp(-\nabla \phi)$, with $\phi : \cM \to \R$ such that
\begin{equation}
    \phi(x) = \inf_{y \in \cM} \frac 12 d^2(x,y) - \psi(y), \quad \forall x \in \cM
\end{equation}
for some $\psi : \cM \mapsto \R \cup \{+\infty\}$ (we say that $\phi$ is $d^2/2$-concave). Moreover, it is often more convenient to work with \emph{velocity plans} instead of transport plans: given $\mu, \nu \in \cP_2(\cM)$, a velocity plan from $\mu$ to $\nu$ is a probability measure $\gamma \in \cP_2(T\cM)$  such that $\pi_\#\gamma = \mu$ and $\exp_\#\gamma = \nu$ (where $T\cM$ is the tangent bundle on $\cM$, endowed with the Sasaki metric (see \Cref{sec:appendix:sasaki_metric}), $\pi : T\cM \mapsto \cM$ is the projection on $\cM$, and $\exp : T\cM \mapsto \cM$ the exponential map). Denoting by $\Gamma(\mu,\nu)$ the set of velocity plans from $\mu$ to $\nu$, it is then clear that to every velocity plan $\gamma \in \Gamma(\mu,\nu)$ corresponds a transport plan $(\pi,\exp)_\#\gamma \in \Pi(\mu,\nu)$. Velocity plans carry more information than transport plans: indeed, transport plans encode the information about which points of the supports of $\mu$ and $\nu$ are coupled, while velocity plans also encode the geodesics along which the mass of $\mu$ is sent. For instance, when $\cM = \bS^2$ is the 2-dimensional sphere, and $\mu = \delta_N$ and $\nu = \delta_S$ are respectively the North and South pole, then $\delta_{(N,S)}$ is the only optimal transport plan between $\mu$ and $\nu$, while for every tangent vector $v \in T_N\bS^2$ of norm $\|v\|_x = r \pi$ with $r$ odd integer, $\delta_{(N,v)}$ is a velocity plan from $\mu$ to $\nu$. Given $\mu,\nu$, one may define a subset of privileged velocity plans between them: the set of \emph{optimal velocity plans} defined by
\begin{equation}
    \Gamma_o(\mu,\nu) := \left\{\gamma \in \Gamma(\mu,\nu) \setcond \int \|v\|^2_x \dd\gamma(x,v) = \W_2^2(\mu,\nu) \right\}.
\end{equation}
Optimal velocity plans are the velocity plans that couple optimally the measures $\mu$ and $\nu$, and which do so by sending the mass of $\mu$ through length-minimizing geodesics in $\cM$ (see \citep[Definition 1.4]{gigli2012second}). The optimal velocity plans are exactly the elements $\gamma \in \Gamma(\mu,\nu)$ such that $(\pi,\exp)_\#\gamma$ is an optimal transport plan between $\mu$ and $\nu$ and such that $d(x,\exp_x(v)) = \|v\|_x$ for $\gamma$-a.e. $(x,v)$ ; furthermore, the map $\gamma \mapsto (\pi,\exp)_\#\gamma$ is a surjection from $\Gamma_o(\mu,\nu)$ to $\Pi_o(\mu,\nu)$ \citep[Propostion 2.1]{bonet2025flowing}. For instance, in the example before, the plans $\delta_{(N,v)}$ with $\|v\|_x = \pi$ are all optimal velocity plans from $\delta_N$ to $\delta_S$ (on the other hand, $\delta_{(N,v)}$ with $\|v\|_x = 3\pi$ is a velocity plan from $\delta_N$ to $\delta_S$ which is not optimal). There is a tight relation between constant speed geodesics in $\cP_2(\cM)$ and optimal velocity plans: indeed, by \citep[Theorem 1.11]{gigli2011inverse}, a curve $(\mu_t)_{t \in [0,1]}$ in $\cP_2(\cM)$ is a constant speed geodesic from $\mu = \mu_0$ to $\nu = \mu_1$ if and only if there exists $\gamma \in \Gamma_o(\mu,\nu)$ such that $\mu_t = ((x,v) \mapsto \exp_x(tv))_\#\gamma$ for every $t \in [0,1]$. Moreover, the plan $\gamma$ is uniquely identified by the geodesic, and two distinct geodesics from $\mu$ to $\nu$ cannot intersect at intermediate times.
\medbreak
A strong motivation for working with velocity plans is that they offer a natural way to define a notion of tangent vectors on the space $\cP_2(\cM)$, and, from this, to recover a variational structure \citep{gigli2011inverse, lanzetti2024variational, aussedat2025structure}. More precisely, fixing $\mu \in \cP_2(\cM)$, and defining $\cP_2(T\cM)_\mu$ to be the set of velocity plans $\gamma$ with $\pi_\#\gamma = \mu$, the elements of $\cP_2(T\cM)_\mu$ can be considered to be tangent vectors with respect to the base point $\mu$\footnote{In fact, it is common to define a ``tangent space of $\mu$" $T_\mu \cP_2(\cM)$ as a subset of $\cP_2(T\cM)_\mu$, and to only call ``tangent" the velocity plans belonging to this set. The precise definition of the tangent space varies with the authors. For instance, \citep{gigli2011inverse} defines the ``geometric tangent space" to be the closure with respect to the distance $\W_\mu$ (defined in \Cref{sec:3_wasserstein_hierarchy}) of the set of plans $\gamma \in \cP_2(T\cM)_\mu$ such that the curve $(\mu_t)_t$ defined by $\mu_t = ((x,v) \mapsto \exp_x(tv))_\#\gamma$ is a geodesic of $\cP_2(\cM)$ for $t > 0$ small enough, while \citep{lanzetti2024variational} (whose analysis is restricted to the case $\cM = \R^d$) has the slightly more stringent requirement that $(\mu_t)_{t \in [-\veps,\veps]}$ be a geodesic for some $\veps > 0$.}. While $\cP_2(T\cM)_\mu$ is of course not a vector space, it is still possible to equip it with a structure resembling that of a vector space. Without going into the details, it is in particular possible to associate to every velocity plan $\gamma \in \cP_2(T\cM)_\mu$ a ``norm" $\|\gamma\|$ and to every pair of velocity plans $\gamma_1,\gamma_2 \in \cP_2(T\cM)_\mu$ a ``scalar product" $\sca{\gamma_1}{\gamma_2}_\mu$ defined by
\begin{equation}
    \|\gamma\|^2 := \int \|v\|^2_x \dd\gamma(x,v), \quad\quad \sca{\gamma_1}{\gamma_2}_\mu := \sup_{\alpha \in \Gamma_\mu(\gamma_1,\gamma_2)} \int \sca{v_1}{v_2}_x \dd\alpha(x,v_1,v_2)
\end{equation}
where $\Gamma_\mu(\gamma_1,\gamma_2)$ is the set of \emph{couplings} between $\gamma_1$ and $\gamma_2$, that is, of probability measures $\alpha \in \cP_2(T^2\cM)$, where $T^2\cM$ is the manifold $\{(x,v_1,v_2), x \in \cM, v_1,v_2 \in T_x\cM\}$, such that $\pi_{1\#}\alpha = \gamma_1$ and $\pi_{2\#}\alpha = \gamma_2$ (where $\pi_i : T^2\cM \mapsto T\cM$ is the projection $(x,v_1,v_2) \to (x,v_i)$ for $i=1,2$). This ``vector space"-like structure of $\cP_2(T\cM)_\mu$, in turn, can be used to study differentiation on the space $\cP_2(\cM)$. Letting $\cF : \cP_2(\cM) \mapsto \R$ be some functional on the space of probability measures and fixing $\mu \in \cP_2(\cM)$, we say that the velocity plan $\gamma \in \cP_2(T\cM)_\mu$ is a \emph{subgradient}\footnote{Again, the precise definition varies with the authors, and, while we have not done so as we only aim to provide a sketch of the theory, it is common to impose restrictions to the velocity plans $\gamma$ and $\xi$ in the definition of subgradients. For instance, in \citep{lanzetti2024variational}, $\gamma$ must be taken in the tangent space, and $\xi$ must be an optimal velocity plan. On the other hand, \citep{erbar2010heat} requires that $\gamma = (\Id,w)_\#\mu$ and $\xi = (\Id,\Psi)_\#\mu$, with $w,\Psi \in L^2(\mu,T\cM)$ such that $\exp(\Psi)$ is an optimal transport map from $\mu$ to $\exp(\Psi)_\#\mu$. In this case, $\sca{\xi}{\gamma}_\mu$ is simply the scalar product $\sca{\Psi}{w}$ in $L^2(\mu,T\cM)$.} of $\cF$ at $\mu$ if for every measure $\nu \in \cP_2(\cM)$ and velocity plan $\xi \in \Gamma(\mu,\nu)$, it holds
\begin{equation}
    \cF(\nu) \geq \cF(\mu) + \sca{\xi}{\gamma}_\mu + o(\|\xi\|),
\end{equation}
and we can similarly define supergradients and gradients. This approach has been notably used by \citep{lanzetti2024variational} in the special case $\cM = \R^d$ to derive first-order optimality conditions for constrained optimization problems on $\cP_2(\R^d)$.
\medbreak
Now, how can we obtain a similar framework for the hierarchical spaces $\cPn{n}{\cM}$ with $n \geq 2$? The way forward is indicated by \citep{bonet2025flowing}, which initiated a study on the space $\cP_2(\cP_2(\cM))$. The concept of velocity plan is generalized the following way: given measures $\bP, \bQ \in \cP_2(\cP_2(\cM))$, a velocity plan from $\bP$ to $\bQ$ is a measure $\bGamma \in \cP_2(\cP_2(T\cM))$ such that $[\pi]_\#\bGamma = \bP$ and $[\exp]_\#\bGamma = \bQ$, where we have denoted by $[f]$ the pushforward map $\mu \mapsto f_\#\mu$. The set of velocity plans from $\bP$ to $\bQ$ is similarly denoted $\Gamma(\bP,\bQ)$. Given a velocity plan $\bGamma \in \Gamma(\bP,\bQ)$, the measure $([\pi],[\exp])_\#\bGamma$ is then a transport plan between $\bP$ and $\bQ$. The set of optimal velocity plans can then be defined by
\begin{equation}
    \Gamma_o(\bP,\bQ) = \left\{\bGamma \in \Gamma(\bP,\bQ) \setcond \iint \|v\|^2_x \dd\gamma(x,v) \dd\bGamma(\gamma) = \W_2^2(\bP,\bQ) \right\}.
\end{equation}
It can then be shown \citep[Proposition 3.1]{bonet2025flowing} that optimal velocity plans are precisely the velocity plans $\bGamma$ such that $([\pi],[\exp])_\#\bGamma$ is an optimal transport plan (for the $\W_2$ distance on $\cP_2(\cP_2(\cM))$ with $\W_2^2$ ground cost), and such that $\bGamma$ is supported on optimal velocity plans $\gamma \in \cP_2(T\cM)$ - in other words, \emph{optimal velocity plans in $\cP_2(\cP_2(T\cM))$ couple optimally their marginals and do so through geodesics in $\cP_2(\cM)$}.
\medbreak
This generalization of the notion of velocity plans suggests a natural way to generalize the variational structure of $\cP_2(\cM)$ defined above to $\cP_2(\cP_2(\cM))$\footnote{The paper \citep{bonet2025flowing} does define a notion of gradient for functionals on $\cP_2(\cP_2(\cM))$. However, they define their gradients to be $L^2$ vector fields, and while this choice dispenses them from having to extend the vector space-like structure to $\cP_2(\cP_2(T\cM))$, it is much more restrictive than velocity plans in the perturbations they allow (in particular, discrete measures perturbed by $L^2$ fields always yield discrete measures). As we strive for full generality, we will always work with velocity plans, although an equivalent of $L^2$ vector fields (fully deterministic velocity plans) will appear in our framework. Another limitation of \citep{bonet2025flowing} is that they work on $\cM$ compact, while our framework will hold for $\cM$ non compact.}. Letting $\bP \in \cP_2(\cP_2(\cM))$ be fixed, we denote $\cP_2(\cP_2(T\cM))_\bP$ the set of $\bGamma \in \cP_2(\cP_2(T\cM))$ such that $[\pi]_\#\bGamma = \bP$, and, given $\bGamma_1, \bGamma_2 \in \cP_2(\cP_2(T\cM))_\bP$, we may define the set of couplings $\Gamma_\bP(\bGamma_1,\bGamma_2)$ between them to be the set of measures $\bA \in \cP_2(\cP_2(T^2\cM))$ such that $[\pi_1]_\#\bA = \bGamma_1$ and $[\pi_2]_\#\bA = \bGamma_2$. Then, just as in the case of $\cP_2(\cM)$, we may use the couplings to define a ``norm" and a ``scalar product" on velocity plans:
\begin{equation}
    \|\bGamma\|^2 := \iint \|v\|^2_x \dd\gamma(x,v) \dd\bGamma(\gamma), \quad\quad \sca{\bGamma_1}{\bGamma_2}_\bP := \sup_{\bA \in \Gamma_\bP(\bGamma_1,\bGamma_2)} \iint \sca{v_1}{v_2}_x \dd\alpha(x,v_1,v_2) \dd\bA(\alpha)
\end{equation}
for every $\bGamma,\bGamma_1,\bGamma_2 \in \cP_2(\cP_2(T\cM))_\bP$. Then, given a functional $\cF : \cP_2(\cP_2(\cM)) \mapsto \R$ and $\bP \in \cP_2(\cP_2(\cM))$, we may define a subgradient of $\cF$ at $\bP$ to be a velocity plan $\bGamma \in \cP_2(\cP_2(T\cM))_\bP$ such that for every $\bQ \in \cP_2(\cP_2(\cM))$ and $\bar{\bGamma} \in \Gamma(\bP,\bQ)$, it holds
\begin{equation}
    \cF(\bQ) \geq \cF(\bP) + \sca{\bar{\bGamma}}{\bGamma}_\bP + o(\|\bar{\bGamma}\|).
\end{equation}
Now, comparing these tentative definitions to the ones corresponding to $\cP_2(\cM)$, we observe that they suggest a principled method to extend the variational structure of $\cP_2(\cM)$ to the space $\cPn{n}{\cM}$ for arbitrary $n \geq 0$. This method can be summarized as follows: to extend a given notion to the space $\cPn{n}{\cM}$, take its definition on $\cP_2(\cM)$, and apply the following recipe:
\begin{itemize}
    \item Replace the spaces $\cP_2(\cM)$, $\cP_2(T\cM)$,... by the spaces $\cPn{n}{\cM}$, $\cPn{n}{T\cM}$,... equipped with their associated $\W_2$ distances. 
    \item Replace all the pushforwards $f_\#\mu$, $g_\#\gamma$,... by $[f]^{(n)}(\mu)$, $[g]^{(n)}(\gamma)$..., where $[\cdot]^{(n)}$ denotes the pushforward operator $[\cdot]$ applied $n$ times.
    \item Replace the integrations $\int f(x) \dd\mu(x)$ by ``multi-level" integrations 
    \begin{equation}
        \iint \ldots \iint f(x) \dd\mu^1(x) \dd\mu^2(\mu^1) \ldots \dd\mu^{n-1}(\mu^{n-2}) \dd\mu(\mu^{n-1}).
    \end{equation}
\end{itemize}

In order to formalize this recipe into a framework, we will need a set of mathematical tools to reason rigorously with probability measures defined on spaces of probability measures, and in particular to tackle the inevitable measurability issues that come with the study of such spaces\footnote{An example of such an issue is the existence of coupling between velocity plans. Given $\mu \in \cP_2(\cM)$ and $\gamma_1,\gamma_2 \in \cP_2(T\cM)_\mu$, it is straightforward, using a result known as the ``gluing lemma" in optimal transport theory (\citep[Lemma 5.3.2]{ambrosio2005gradient}), to construct couplings $\alpha \in \cP_2(T^2\cM)$ between $\gamma_1$ and $\gamma_2$. On the other hand, given $\bP \in \cP_2(\cP_2(\cM))$ and $\bGamma_1,\bGamma_2 \in \cP_2(\cP_2(T\cM))_\bP$, the existence of couplings $\bA \in \cP_2(\cP_2(T^2\cM))$ between $\bGamma_1$ and $\bGamma_2$ is a priori less clear, and is something that we will have to prove.}. It turns out that \emph{category theory} \citep{maclane1978categories} gives such a set of tools, and can be used to derive elegantly the theory that we outlined above. To illustrate why, it is again useful to compare the variational structures of the first spaces of the Wasserstein hierarchy, $\cM$, $\cP_2(\cM)$ and $\cP_2(\cP_2(\cM))$.
\medbreak
Given a point $x \in \cM$, a tangent vector at $x$ is an element of $T_x\cM = \pi^{-1}(x)$. Letting $v \in T_x\cM$, we may consider its exponential $y = \exp_x(v)$. If we have two tangent vectors $v_1, v_2 \in T_x \cM$, we may define the set of couplings between $v_1$ and $v_2$ to be the set of $a \in T^2\cM$ such that $\pi_1(a) = (x,v_1)$ and $\pi_2(a) = (x,v_2)$, that is $\pi_1^{-1}(x,v_1) \cap \pi_2^{-1}(x,v_2)$ (this will not be a very interesting notion as the only coupling between $v_1$ and $v_2$ is the element $(x,v_1,v_2)$). Likewise, given a measure $\mu \in \cP_2(\cM)$, a tangent element at $\mu$ is an element of $\cP_2(T\cM)_\mu = [\pi]^{-1}(\mu)$. Letting $\gamma \in \cP_2(T\cM)_\mu$, we may consider its exponential $\nu = [\exp](\gamma)$ and, if we have $\gamma_1, \gamma_2 \in \cP_2(T\cM)_\mu$, the set of couplings between them is $\Gamma_\mu(\gamma_1,\gamma_2) = [\pi_1]^{-1}(\gamma_1) \cap [\pi_2]^{-1}(\gamma_2)$. In other words, the variational structure of $\cP_2(\cM)$ is obtained by considering the diagram of functions consisting of the maps $\pi_1, \pi_2 : T^2 \cM \mapsto T\cM$ and $\pi, \exp : T\cM \mapsto \cM$, and applying to it the ``operation" $\cP_2$ consising of replacing every space $X$ with its space of probability measures with finite second moment $\cP_2(X)$, and every function $f$ with its pushforward map $[f]$, as shown in the figure \eqref{diag:struct_of_M_to_struct_of_PM} below.
\begin{equation} \label{diag:struct_of_M_to_struct_of_PM}
    \begin{array}{ccc}
        \begin{tikzcd}
        & T^2\cM \arrow[dd, bend right, "\pi_1"'] \arrow[dd, bend left, "\pi_2"] & \\ 
        & & \\
        & T \cM \arrow[dl, "\pi"'] \arrow[dr, "\exp"] & \\
        \cM & & \cM   
        \end{tikzcd}
        & 
        \begin{array}{c}
            \cP_2 \\ \Longrightarrow
        \end{array}
        &
        \begin{tikzcd}
             & \cP_2(T^2 \cM) \arrow[dd, bend right, "{[\pi_1]}"'] \arrow[dd, bend left, "{[\pi_2]}"] & \\
             & & \\
             & \cP_2(T \cM) \arrow[dl, "{[\pi]}"'] \arrow[dr, "{[\exp]}"] & \\
             \cP_2(\cM) & & \cP_2(\cM)
        \end{tikzcd}
    \end{array}  
\end{equation}
Similarly, given $\bP \in \cP_2(\cP_2(\cM))$, a tangent element at $\bP$ is an element of $\cP_2(\cP_2(T\cM))_\bP = [[\pi]]^{-1}(\bP)$. Given $\bGamma \in \cP_2(\cP_2(T\cM))_\bP$, we may consider its exponential $\bQ = [[\exp]](\bGamma)$ and, given $\bGamma_1, \bGamma_2 \in \cP_2(\cP_2(T\cM))_\bP$, the set of couplings between them is $\Gamma_\bP(\bGamma_1,\bGamma_2) = [[\pi_1]]^{-1}(\bGamma_1) \cap [[\pi_2]]^{-1}(\bGamma_2)$. That is, the variational structure of $\cP_2(\cP_2(\cM))$ is obtained by applying one more time the ``operation" $\cP_2$ to the diagram on the right of \eqref{diag:struct_of_M_to_struct_of_PM}, as shown in \eqref{diag:struct_of_PM_to_struct_of_PPM}.
\begin{equation} \label{diag:struct_of_PM_to_struct_of_PPM}
    \begin{array}{ccc}
        \begin{tikzcd}
             & \cP_2(T^2 \cM) \arrow[dd, bend right, "{[\pi_1]}"'] \arrow[dd, bend left, "{[\pi_2]}"] & \\
             & & \\
             & \cP_2(T \cM) \arrow[dl, "{[\pi]}"'] \arrow[dr, "{[\exp]}"] & \\
             \cP_2(\cM) & & \cP_2(\cM)
        \end{tikzcd}
        & 
        \begin{array}{c}
            \cP_2 \\ \Longrightarrow
        \end{array}
        &
        \begin{tikzcd}
             & \cPP{T^2 \cM} \arrow[dd, bend right, "{[[\pi_1]]}"'] \arrow[dd, bend left, "{[[\pi_2]]}"] & \\
             & & \\
             & \cPP{T \cM} \arrow[dl, "{[[\pi]]}"'] \arrow[dr, "{[[\exp]]}"] & \\
             \cPP{\cM} & & \cPP{\cM}
        \end{tikzcd}
    \end{array}  
\end{equation}
In other words, \emph{the variational structure of $\cPn{n}{\cM}$ can be obtained by applying $n$ times the operation $\cP_2$ to the diagram consisting of the maps $\pi_1,\pi_2,\pi,\exp$}. But this operation $\cP_2$, which associates to a space $X$ its probability space $\cP_2(X)$ and to a map its pushforward $[f]$, is precisely what is called a \emph{functor} in category theory \citep{maclane1978categories}. The theory of variational analysis in the Wasserstein hierarchy can therefore be understood in category theoretic terms as the theory obtained by repeated applications of the $\cP_2$ functor on the variational structure of $\cM$. This will justify the introduction and use of concepts of category theory in our study, which will prove to be immensely useful to tackle many of the technical difficulties that come with the study of the hierarchical spaces.

\paragraph{Outline of the paper. } First, in \Cref{sec:2_preliminary}, we introduce a number of technical tools and preliminary results that will be used in the rest of the paper. These include the precise definition of several functors, including $\cP_2$, and some results on the properties of fiber products of Polish spaces and the surjectivity of pushforward maps. Then, in \Cref{sec:3_wasserstein_hierarchy}, we formally derive the main elements of the variational structure of the hierarchical Wasserstein spaces, including velocity plans, couplings, and the vector space-like structure on the velocity plans, and we prove their main properties. Next, in \Cref{sec:4_fully_det}, we identify a special class of velocity plans, which we call ``fully deterministic velocity plans", and we show that they enjoy remarkable properties, such as the fact that they form a Hilbert space. Following this, we establish in \Cref{sec:5_geodesics} a precise characterization of the geodesics of the hierarchical Wasserstein spaces, showing a one-to-one correspondence between them and the optimal velocity plans. Finally, in \Cref{sec:6_functionals}, we define a notion of subgradients and gradients in the hierarchical spaces, and we study the differentiability and convexity of various types of functionals defined on these spaces.

\paragraph{Related work. } The application of the spaces $\cP_2(\cP_2(X))$ of probabilities on probability measures to problems in machine learning has been studied in particular in \citep{alvarez2021dataset, hua2023dynamic, bonet2025flowing}. The problem of finding equivalents of the Brenier theorem to hierarchical spaces, that is of finding conditions under which optimal transport maps between a given pair of measures exist, has been the subject of works such as \citep{emami2024monge} for the case of $\cP_2(\cP_2(\cM))$ where $\cM$ is a compact connected manifold, \citep{pinzi2025totally} for $\cP_2(\cP_2(\cH))$ and \citep{beiglböck2025brenier} for the spaces $\cPn{n}{\cH}$, $n \geq 0$ where $\cH$ is a separable Hilbert space. A related problem consists in finding measures on the hierarchical spaces $\cPn{n}{X}$ that play a comparable role to the volume or Lebesgue measures, which is studied in publications such as \citep{schiavo2020rademacher} or \citep{von2009entropic, sturm2024wasserstein}. The article \citep{pinzi2025nested} addresses the characterization of the absolutely continuous curves and the geodesics of the space $\cP_2(\cP_2(X))$ when $X$ is a complete separable metric space.

\paragraph{Notations.} Given a measurable map $f : X \mapsto Y$ and a probability measure $\mu \in \cP(X)$, we denote by $f_\#\mu$ the pushforward of $\mu$ by $f$, which is the probability measure $\nu \in \cP(Y)$ defined by $\nu(\cdot) := \mu(f^{-1}(\cdot))$. We also denote $[f] : \cP(X) \mapsto \cP(Y)$ the pushforward map $[f](\mu) := f_\#\mu$, $\mu \in \cP(X)$. The support of a Borel probability measure $\mu$, denoted $\spt(\mu)$, is the complement of the largest open set with null $\mu$-measure. We will typically use the symbols $X$, $Y$, $Z$... to denote Polish spaces (i.e. separable completely metrizable spaces), and $\cM$ will always denote a complete connected Riemannian manifold without boundary. Given $\cM$ or a Polish space $X$, the measures in $\cPn{n}{X}$ or $\cPn{n}{\cM}$ will typically be denoted by $\mu$, $\nu$ ; we will often use the symbols $\gamma$, $\eta$, $\xi$ to denote velocity plans and transport plans and the symbols $\alpha$, $\beta$ to denote couplings between velocity plans. Often, when working simultaneously with two ``levels" $\cPn{n}{X}$ and $\cPn{n+1}{X}$, we will use the notations $\bP$, $\bQ$, $\bGamma$, $\bA$ to denote measures, velocity and transport plans, and couplings corresponding to the $n+1$ level. We will also sometimes denote the ``level" in which a measure lives at with a superscript: $\mu^n$, $\gamma^n$, $\alpha^n$, etc (but at other times, the superscript simply enumerates the elements of a sequence $(\mu^k)_k$). 

\paragraph{Acknowledgements.} The author wishes to thank Clément Bonet, Anna Korba and Quentin Mérigot for their insightful comments and discussions. The author acknowledges the support of Région Île-de-France through the DIM AI4IDF project, and the support of the Agence nationale
de la recherche, through the PEPR PDE-AI project (ANR-23-PEIA-0004).

\section{Preliminary results} \label{sec:2_preliminary}

\subsection{The \texorpdfstring{$\cP$}{P} and \texorpdfstring{$\cP_2$}{P2} functors}

In this subsection, we first define two functors that will play an important role in our analysis (see \Cref{sec:appendix:category_theory} for a brief review of the main concepts in category theory).
\medbreak
First, we define $\Pol$ to be the category whose objects are Polish spaces $X$, and whose morphisms are continuous map $f : X \mapsto Y$ between Polish spaces. We may define a functor $\cP : \Pol \mapsto \Pol$, which to every Polish space $X$ associates the space $\cP(X)$ of Borel probability measures on $X$ equipped with the topology of weak convergence, and which associates to every continuous map $f : X \mapsto Y$ the pushforward map $\cP(f) := [f] : \cP(X) \mapsto \cP(Y)$. This functor is well-defined, in that $\cP(X)$ is indeed a Polish space, $[f]$ is continuous (for the topology of weak convergence), and $\cP$ defined this way satisfies the conditions of a functor (we refer to \Cref{sec:appendix:weak_topology} for a review of the properties of the topology of weak convergence). To this functor $\cP$, we associate two natural transformations $\delta : \Id_{\Pol} \mapsto \cP$ and $\cE : \cP \circ \cP \mapsto \cP$, which associate to every Polish space $X$ two continuous maps $\delta_X : X \mapsto \cP(X)$ and $\cE_X : \cP(\cP(X)) \mapsto \cP(X)$, where $\delta_X$ maps every $x \in X$ to its corresponding Dirac mass $\delta_x$:
\begin{equation}
    \delta_X : \left\{\begin{array}{ccc} 
         X & \rightarrow & \cP(X)  \\
         x & \rightarrow & \delta_x 
    \end{array}\right.
\end{equation}
and where $\cE_X$ maps every $\bP \in \cP(\cP(X))$ to the unique measure $\cE_X(\bP) \in \cP(X)$ such that, for every nonnegative Borel measurable function $f : X \mapsto \R_+$,
\begin{equation}
    \int f(x) \dd\cE_X(\bP)(x) = \iint f(x) \dd\mu(x) \dd\bP(\mu).
\end{equation}
It may be verified that $\delta$ and $\cE$ are well defined as natural transformations, that is, for every continuous map $f : X \mapsto Y$ between Polish spaces, it holds $\delta_Y \circ f = [f] \circ \delta_X$ and $\cE_Y \circ [[f]] = [f] \circ \cE_X$. 
\begin{remark}
    It can furthermore be verified that the triple $(\cP,\delta,\cE)$ forms a \emph{monad} called the \emph{Giry monad} on $\Pol$ (see \citep{giry1982categorical}). We only state this for the sake of completeness ; the rest of this article will not use the notion of monads and a fortiori will not use the Giry monad.
\end{remark}
Recall that for every $n \geq 0$, we denote by $\cP^{(n)}$ the functor $\cP$ composed $n$ times with itself, with $\cP^{(0)} = \Id_\Pol$ (see \Cref{sec:appendix:category_theory}). Similarly, we denote $[\cdot]^{(n)}$ the operation consisting of taking $n$ times the pushforward map of a map $f : X \mapsto Y$. We also define the continuous map $\delta_X^{(n)} : X \mapsto \cPzn{n}{X}$ by $\delta_X^{(0)} = \id_X$ and $\delta_X^{(n)} = \delta_{\cPzn{n-1}{X}} \circ \ldots \delta_{\cP(X)} \circ \delta_X$ for $n \geq 1$. For every $x_0 \in X$, we then note $\delta^{(n)}_{x_0} := \delta_X^{(n)}(x_0) \in \cPzn{n}{X}$ : $\delta^{(n)}_{x_0}$ is the measure defined inductively by $\delta^{(0)}_{x_0} := x_0$ and $\delta^{(n)}_{x_0} := \delta_{\delta^{(n-1)}_{x_0}}$. By naturality of $\delta$, we have for every continuous map $f : X \mapsto Y$ that $[f]^{(n)} \circ \delta_X^{(n)} = \delta_Y^{(n)} \circ f$, so that $[f]^{(n)}(\delta_x^{(n)}) = \delta_{f(x)}^{(n)}$ for every $x \in X$.
\medbreak
We now consider another category: we let $\Pold$ be the category whose objects are pairs $(X,d)$ where $X$ is a Polish space and $d$ a distance metrizing it, and whose morphisms are continuous maps $f : (X,d_X) \mapsto (Y,d_Y)$ \emph{with linear growth}, defined as:

\begin{definition}
    A function $f : (X,d_X) \mapsto (Y,d_Y)$ between metric spaces is said to have \emph{linear growth} if there exists $(x_0,y_0) \in X \times Y$ and $C > 0$ such that for every $x \in X$, $d_Y(y_0,f(x)) \leq C(1 + d_X(x_0,x))$. One can check that if it holds for one $(x_0,y_0) \in X \times Y$, then it holds for any $(x_0,y_0) \in X \times Y$. 
\end{definition}

Note that if a function is Lipschitz continuous, then it has linear growth. The category $\Pol_d$ is well-defined as one can check that the identity maps of a metric space has linear growth, and that the composition of maps with linear growth also has linear growth. Note that since the morphisms are restricted to continuous maps with linear growth, being isomorphic in $\Pold$ is stronger than being isomorphic in $\Pol$. Notably, $(X,d)$ may fail to be isomorphic to $(X,d')$ even if $d$ and $d'$ metrize the same topology. On the other hand, if $d$ and $d'$ are bi-Lipschitz equivalent (that is, $\alpha d \leq d' \leq \beta d$ for some $\alpha, \beta > 0$), then the identity map of $X$ will be an isomorphism between $(X,d)$ and $(X,d')$ in the category $\Pold$. We will often designate an object of $\Pold$ simply by its underlying space $X$ instead of by the whole pair $(X,d)$ when the distance $d$ associated to $X$ is unambiguous. 
\medbreak
We can then define the functor $\cP_2 : \Pold \mapsto \Pold$ which associates to every Polish space $(X,d)$ the space $\cP_2(X,d) := (\cP_2(X),\W_2)$, where $\W_2$ is the corresponding Wasserstein distance of order $2$ on $\cP_2(X)$ (see \Cref{sec:appendix:w2_space}), and which associates to every morphism $f : X \mapsto Y$ its pushforward map $\cP_2(f) := [f]$. This functor will play a crucial role in our analysis.

\begin{proposition}
    The functor $\cP_2$ is well-defined. Moreover if $f : (X,d_X) \mapsto (Y,d_Y)$ is a Lipschitz function between Polish spaces, then $[f]$ is also a Lipschitz function with the same Lipschitz constant.
\end{proposition}

\begin{proof}
    We know from \Cref{th:prob_space_is_polish} that if $(X,d)$ is a Polish space, then the space $(\cP_2(X),\W_2)$ is also Polish. Thus, to prove that $\cP_2$ is well-defined as a functor $\Pold \mapsto \Pold$, we only need to prove that for every $f : (X,d_X) \mapsto (Y,d_Y)$ which is continuous with linear growth between Polish spaces, the pushforward map $[f] : \cP(X) \mapsto \cP(Y)$ is, when restricted to $\cP_2(X)$, valued in $\cP_2(Y)$, and continuous with linear growth with respect to the Wasserstein distances on both $\cP_2(X)$ and $\cP_2(Y)$. \newline
    Let $f : (X,d_X) \mapsto (Y,d_Y)$ be a continuous function between Polish spaces with linear growth. We let $C > 0$ and $x_0 \in X$ such that $d_Y(f(x_0),f(x)) \leq C(1 + d_X(x_0,x))$ for every $x \in X$. Let $\mu \in \cP_2(X)$, then
    \begin{align}
        \W_2^2([f](\delta_{x_0}), [f](\mu)) &= \W_2^2(\delta_{f(x_0)}, f_\#\mu) = \int d_Y^2(f(x_0), y) \dd f_\#\mu(y) \\
        &= \int d_Y^2(f(x_0), f(x)) \dd\mu(x) \\
        &\leq \int C^2(1 + d_X(x_0,x))^2 \dd\mu(x) \\
        &\leq \int C_1(1 + d^2_X(x_0,x)) \dd\mu(x) = C_1(1 + \W_2^2(\delta_{x_0},\mu)) < +\infty,
    \end{align}
    and from this we see that:
    \begin{itemize}
        \item $[f](\mu)$ has finite second order moment $\W_2^2(\delta_{f(x_0)},[f](\mu))$, so that $[f]$ restricts as a map $\cP_2(X) \mapsto \cP_2(Y)$,
        \item $[f]$ has linear growth (as a map $\cP_2(X) \mapsto \cP_2(Y)$).
    \end{itemize}
    We now show the continuity of $[f]$ for the $\W_2$ topology. Let $(\mu_n)_n$ be a sequence in $\cP_2(X)$ converging to $\mu \in \cP_2(X)$, and let $\nu_n := [f](\mu_n)$, $\nu := [f](\mu)$. Let $\varphi : Y \mapsto \R$ be a continuous function such that $\forall y \in Y, |\varphi(y)| \leq C'(1 + d_Y^2(y,y_0))$ for some $C' > 0$ and $y_0 \in Y$. We may take $y_0$ to be arbitrary, so we can assume that $y_0 = f(x_0)$ for some $x_0 \in X$. Then $\varphi \circ f$ is a continuous function, such that for every $x \in X$,
    \begin{equation}
        |\varphi \circ f|(x) \leq C'(1 + d^2_Y(f(x_0),f(x))) \leq C'(1 + C^2(1 + d_X(x_0,x))^2) \leq C_2(1 + d^2_X(x_0,x))
    \end{equation}
    so that, by point \ref{enum:w2_convergence:4_quadratic_growth} of \Cref{th:convergence_in_w2_space}, 
    \begin{equation}
        \lim_{n \to +\infty} \int \varphi \dd\nu_n = \lim_{n \to +\infty} \int \varphi \circ f \dd\mu_n = \int \varphi \circ f \dd\mu = \int \varphi \dd\nu
    \end{equation}
    and this implies, again by point \ref{enum:w2_convergence:4_quadratic_growth} of \Cref{th:convergence_in_w2_space}, that $\nu_n \to \nu$ in the $\W_2$ topology. Thus $[f]$ is continuous for the $\W_2$ topology, and we have shown that the functor $\cP_2$ is well-defined. \newline
    Finally, assume that $f$ is Lipschitz, with Lipschitz constant $L$. If $\mu, \nu \in \cP_2(X)$, and $\gamma \in \Pi_o(\mu,\nu)$ is an optimal transport plan between these two measures, then $(f,f)_\#\gamma$ is a (not necessarily optimal) transport plan between $[f](\mu) = f_\#\mu$ and $[f](\nu) = f_\#\nu$. We then have
    \begin{align}
        \W^2_2([f](\mu),[f](\nu)) &\leq \int d^2_Y(x,y) \dd(f,f)_\#\gamma(x,y) = \int d^2_Y(f(x),f(y)) \dd\gamma(x,y) \\
        &\leq L^2 \int d^2_X(x,y)\dd\gamma(x,y) = L^2 \W_2^2(\mu,\nu)
    \end{align}
    so $\W_2([f](\mu),[f](\nu)) \leq L \W_2(\mu,\nu)$, and $[f]$ is a Lipschitz continuous map from $\cP_2(X)$ to $\cP_2(Y)$ with the same Lipschitz constant as $f$ (its Lipschitz constant is no smaller than $L$ since $W_2([f](\delta_x),[f](\delta_y)) = \W_2(\delta_{f(x)},\delta_{f(y)}) = d(f(x),f(y))$ and $W_2(\delta_x,\delta_y) = d(x,y)$ for any $x,y \in X$).
\end{proof}

Let $(X,d)$ be a Polish space, then it is clear that the map $\delta_X : X \mapsto \cP(X)$ is valued in $\cP_2(X)$, and is continuous with linear growth with respect to the $\W_2$ topology - in fact, it is even an isometry, as $\W_2(\delta_x,\delta_y) = d(x,y)$ for every $x,y \in X$. Moreover, since the $\W_2$ topology is stronger than the topology of weak convergence, the inclusion $\cP_2(X) \subseteq \cP(X)$ is continuous. 
\medbreak
More generally, for every $n \geq 0$, $\cPn{n}{X}$ identifies as subset of $\cPzn{n}{X}$. To show this, we first note that we can define a ``Wasserstein distance" $\W_2$ on $\cPzn{n}{X}$ inductively by defining it to be the square root of the optimal transport cost associated to the cost function $\W_2^2$ on $\cPzn{n-1}{X}$. By \Cref{sec:appendix:ot_solutions}, it is then a lower semicontinous map $\cPzn{n}{X} \times \cPzn{n}{X} \mapsto [0,+\infty]$ with respect to the topology of weak convergence (on the other hand, unlike in the case of $\cPn{n}{X}$, it does not define a distance on $\cPzn{n}{X}$ as it may be infinite). We then construct a continuous map $i_n : \cPn{n}{X} \mapsto \cPzn{n}{X}$ in the following way: we let $i_0 = \id_X$, and for every $n > 0$, we set $i_n$ to be the composition of the inclusion $\cP_2(\cPn{n-1}{X}) \subseteq \cP(\cPn{n-1}{X})$ (which is continuous as the $\W_2$-topology is stronger than the topology of weak convergence) and of the pushforward map $[i_{n-1}] : \cP(\cPn{n-1}{X}) \mapsto \cPzn{n}{X}$ (which is continuous as $i_{n-1}$ is continuous). In other words, a measure $\mu \in \cPn{n}{X}$ is seen through $i_n$ as a measure in $\cPzn{n}{X}$ which is concentrated on $\cPn{n-1}{X}$ and which has finite second order moment. We now prove that $i_n$ is injective, and identifies $\cPn{n}{X}$ to a measurable subset of $\cPzn{n}{X}$: 

\begin{proposition} \label{prop:p2n_embeds_into_pn}
    For every $n \geq 0$, the map $i_n : \cPn{n}{X} \mapsto \cPzn{n}{X}$ is injective. Moreover, its image is a measurable subset of $\cPzn{n}{X}$, and it admits a measurable left inverse $s_n : \cPzn{n}{X} \mapsto \cPn{n}{X}$ (that is, $s_n \circ i_n = \id_{\cPn{n}{X}}$). Moreover, it satisfies $\W_2(i_n(\mu),i_n(\nu)) = \W_2(\mu,\nu)$ for every $\mu, \nu \in \cP(\cPn{n}{X})$, and its image is given by
    \begin{equation} \label{eq:image_of_p2n_into_pn_embedding}
        i_n(\cPn{n}{X}) = \{\mu \in \cPzn{n}{X} \setcond \W_2(\mu,\delta^{(n)}_{x_0}) < +\infty \}
    \end{equation}
    where $x_0 \in X$ is arbitrary. 
\end{proposition}

\begin{proof}
    We show this by induction. For $n = 0$, this is trivial since $i_0 = \id_X$. Fix $n > 0$ and assume the proposition holds for $n-1$. We first prove that $i_n$ is injective. For this, it suffices to show that $[i_{n-1}]$ is injective. Let $s_{n-1} : \cPzn{n}{X} \mapsto \cPn{n}{X}$ be the measurable left inverse of $i_{n-1}$. Then, for every $\bP,\bQ \in \cP(\cPn{n-1}{X})$ such that $[i_{n-1}](\bP) = [i_{n-1}](\bQ)$, since $s_{n-1} \circ i_{n-1}$ is the identity, we have $\bP = [s_{n-1}]([i_{n-1}](\bP)) = [s_{n-1}]([i_{n-1}](\bQ)) = \bQ$. Therefore $[i_{n-1}]$ is injective, and so is $i_n$. \newline
    We now want to prove that $\W_2(i_n(\bP),i_n(\bP)) = \W_2(\bP,\bQ)$ for every $\bP,\bQ \in \cPn{n}{X}$. In fact, we will prove something stronger: that $\W_2([i_{n-1}](\bP),[i_{n-1}](\bP)) = \W_2(\bP,\bQ)$ for every $\bP, \bQ \in \cP(\cPn{n-1}{X})$. Fix $\bP, \bQ \in \cP(\cPn{n-1}{X})$. Let $\bGamma \in \Pi_o(\bP,\bQ)$ be an optimal transport plan between them. Then $(i_{n-1},i_{n-1})_\#\bGamma$ is a (not necessarily optimal) transport plan between $[i_{n-1}](\bP)$ and $[i_{n-1}](\bQ)$ so that, using the induction hypothesis on $i_{n-1}$,
    \begin{equation}
        \W_2^2([i_{n-1}](\bP),[i_{n-1}](\bQ)) \leq \int \W_2^2(i_{n-1}(\mu),i_{n-1}(\nu)) \dd\bGamma(\mu,\nu) = \int \W_2^2(\mu,\nu) \dd\bGamma(\mu,\nu) = \W_2^2(\bP, \bQ). 
    \end{equation}
    Conversely, let $\tilde{\bGamma} \in \Pi([i_{n-1}](\bP), [i_{n-1}](\bQ))$ be an optimal transport plan (for the cost $\W_2^2$ on $\cPzn{n-1}{X}$) between $[i_{n-1}](\bP)$ and $[i_{n-1}](\bQ)$, and let $\bGamma := (s_{n-1},s_{n-1})_\#\tilde{\bGamma}$. Then $\bGamma$ is a (not necessarily optimal) transport plan between $\bP$ and $\bQ$. Indeed, we have $\pi_{1\#}\bGamma = s_{n-1\#}\pi_{1\#}\tilde{\bGamma} = s_{n-1\#}i_{n-1\#}\bP = \bP$ (as $s_{n-1} \circ i_{n-1}$ is the identity) and similarly $\pi_{2\#}\bGamma = \bQ$. Therefore, we have
    \begin{align}
        \W_2^2(\bP,\bQ) &\leq \int \W_2^2(\mu,\nu) \dd\bGamma(\mu,\nu) = \int \W_2^2(s_{n-1}(\mu),s_{n-1}(\nu)) \dd\tilde{\bGamma}(\mu,\nu) \\
        &= \int \W_2^2(\mu,\nu) \dd\tilde{\bGamma}(\mu,\nu) = \W_2^2([i_{n-1}](\bP), [i_{n-1}](\bQ)). 
    \end{align}
    where the second line comes from the fact that $\tilde{\bGamma}$ is concentrated on $i_{n-1}(\cPn{n-1}{X}) \times i_{n-1}(\cPn{n-1}{X})$, so that for $\tilde{\bGamma}$-a.e. $(\mu,\nu)$, it holds $\W_2(s_{n-1}(\mu),s_{n-1}(\nu)) = \W_2(i^{-1}_{n-1}(\mu),i^{-1}_{n-1}(\nu)) = \W_2(\mu,\nu)$. We have thus shown $\W_2(\bP,\bQ) = \W_2([i_{n-1}](\bP), [i_{n-1}](\bQ))$. \newline
    Now fix $x_0 \in X$, and set $F$ to be the set of $\bP \in \cPzn{n}{X}$ such that $\W_2(\bP,\delta_{x_0}^{(n)}) < +\infty$. Since $i_n(\delta^{(n)}_{x_0}) = \delta^{(n)}_{x_0}$ (this is straightforward to check by induction\footnote{Note that the notation $\delta_{x_0}^{(n)}$ refers here to a priori two different inductively defined Dirac measures, one in the space $\cPzn{n}{X}$ and the other in the space $\cPn{n}{X}$.}), we have for every $\bP \in \cPn{n}{X}$ that $\W_2(i_n(\bP),\delta_{x_0}^{(n)}) = \W_2(\bP,\delta_{x_0}^{(n)}) < +\infty$, so that $i_n(\bP) \in F$. Conversely, if $\bP \in F$, then
    \begin{equation}
        \W_2^2(\bP, \delta_{x_0}^{(n)}) = \int \W_2^2(\mu, \delta^{(n-1)}_{x_0}) \dd\bP(\mu) < +\infty,
    \end{equation}
    so that $\W_2(\mu,\delta^{(n-1)}_{x_0}) < +\infty$ for $\bP$-a.e. $\mu$. In particular, by the induction hypothesis, $\bP$ is concentrated on the image of $i_{n-1}$, so that $\tilde{\bP} := s_{n-1\#}\bP$ is an element of $\cP(\cPn{n-1}{X})$ such that $[i_{n-1}](\bP) = (i_{n-1} \circ s_{n-1})_\#\tilde{\bP} = \tilde{\bP}$ (as $i_{n-1} \circ s_{n-1}$ is the identity on the image of $i_{n-1}$, on which $\tilde{\bP}$ is concentrated). Moreover, we have $\W_2(\tilde{\bP},\delta^{(n)}_{x_0}) = \W_2([i_{n-1}](\tilde{\bP}), \delta^{(n)}_{x_0}) < +\infty$ as $[i_{n-1}]$ preserves the $\W_2$ ``distance", so that $\tilde{\bP}$ has finite second order moment and is thus an element of $\cPn{n}{X}$. Thus $\bP = i_n(\tilde{\bP})$, and $\bP$ is in the image of $i_n$. We have thus shown that $F$ is the image of $i_n$. Note finally that the set $F$ is measurable as it is the set where the lower semicontinuous function $\W_2(\cdot,\delta^{(n)}_{x_0})$ on $\cPzn{n}{X}$ is finite. \newline
    Now, notice that since $\W_2(i_n(\bP),i_n(\bP)) = \W_2(\bP,\bQ)$ for every $\bP,\bQ \in \cPn{n}{X}$, the image by $i_n$ of any closed ball $B(\bP,R)$ in $\cPn{n}{X}$ for the $\W_2$ distance is
    \begin{equation}
        i_n(B(\bP,R)) = \{\bQ \in \cPzn{n}{X} \setcond \W_2(\bQ,i_n(\bP)) \leq R, \W_2(\bQ, \delta_{x_0}^{(n)}) < +\infty \}
    \end{equation}
    which is measurable as $\W_2$ is lower semicontinuous on $\cPzn{n}{X}$. In particular, the inverse map $i_n^{-1} : i_n(\cPn{n}{X}) \mapsto \cPn{n}{X}$ is measurable, and we can extend it to a measurable left inverse $s_n : \cPzn{n}{X} \mapsto \cPn{n}{X}$ of $i_n$ by setting $s_n$ to be an arbitrary constant outside $i_n(\cPn{n}{X})$. This finishes the proof.
\end{proof}

Therefore, in the following, we will always identify $\cPn{n}{X}$ to a subset of $\cPzn{n}{X}$ through $i_n$. Note that \Cref{prop:p2n_embeds_into_pn} implies that the Borel $\sigma$-algebra on $\cPn{n}{X}$ induced by the $\W_2$ topology coincides with the restriction of the Borel $\sigma$-algebra on $\cPzn{n}{X}$ induced by the topology of weak convergence. Therefore, there will be no ambiguity when referring to ``the" Borel $\sigma$-algebra of $\cPn{n}{X}$.

\begin{remark}
    Another consequence of \Cref{prop:p2n_embeds_into_pn} is that $\cPn{n}{X}$ is a $F_\sigma$-set of $\cPzn{n}{X}$, that is, a countable union of closed sets: indeed, by \eqref{eq:image_of_p2n_into_pn_embedding}, we have for every $x_0 \in X$
    \begin{equation}
        \cPn{n}{X} = \bigcup_{M = 1}^{+\infty} \{\bP \in \cPzn{n}{X} \setcond \W_2(\bP,\delta_{x_0}^{(n)}) \leq M\},
    \end{equation}
    so that $\cPn{n}{X}$ is a countable union of sublevel sets of the lower semicontinuous function $\W_2(\cdot,\delta^{(n)}_{x_0})$ on $\cPzn{n}{X}$.
\end{remark}

\begin{proposition} \label{prop:collapse_restricts_to_p2}
    Let $(X,d)$ be a topological space. Then the map $\cE_X : \cPzn{2}{X} \mapsto \cP(X)$ restricts to a map $\cE_X : \cPn{2}{X} \mapsto \cP_2(X)$ which is continuous with linear growth with respect to the $\W_2$ topology. In particular, for every $\bP \in \cPn{2}{X}$ and $x_0 \in X$, it holds $\W_2(\delta_{x_0}, \cE_X(\bP)) = \W_2(\delta^{(2)}_{x_0}, \bP)$.
\end{proposition}

\begin{proof}
    First, if $\bP \in \cPn{2}{X}$, then for every $x_0 \in X$, we indeed have $\W_2(\delta_{x_0}, \cE_X(\bP)) = \W_2(\delta^{(2)}_{x_0}, \bP)$ since
    \begin{align}
        \W_2^2(\delta_{x_0}, \cE_X(\bP)) &= \int d^2(x_0,x) \dd(\cE_X(\bP))(x) = \iint d^2(x_0,x) \dd\mu(x) \dd\bP(\mu) \\
        &= \int \W_2^2(\delta_{x_0},\mu) \dd\bP(\mu) \\
        &= \W_2^2(\delta^{(2)}_{x_0},\bP) < +\infty. \label{eq:l_149}
    \end{align}
    From this we immediately deduce that $\cE_X(\bP)$ has finite second order moment, and that $\cE_X$ restricts to a map $\cP_2^{(2)}(X) \mapsto \cP_2(X)$ which has linear growth. All we have left to prove is that it is continuous with respect to the $\W_2$ topology. Let $(\bP_n)_n$ be a sequence in $\cPn{2}{X}$ converging to $\bP \in \cPn{2}{X}$ in the $\W_2$ topology. Then $\bP_n \rightharpoonup \bP$, so that, since $\cE_X$ is known to be continuous for the topology of weak convergence, $\cE_X(\bP_n) \rightharpoonup \cE_X(\bP)$. Moreover, we have using \eqref{eq:l_149} that for every $n$,
    \begin{equation}
        \W_2^2(\delta_{x_0}, \cE_X(\bP_n)) = \W_2^2(\delta^{(2)}_{x_0},\bP_n) \xrightarrow[n \to +\infty]{} \W_2^2(\delta^{(2)}_{x_0},\bP) = \W_2^2(\delta_{x_0}, \cE_X(\bP)),
    \end{equation}
    so using point \ref{enum:w2_convergence:1_2nd_moment_cvg} of \Cref{th:convergence_in_w2_space}, we conclude that $\cE_X(\bP_n) \mapsto \cE_X(\bP)$ in the $\W_2$ topology. This finishes the proof.
\end{proof}

These considerations allow us to also see $\delta$ and $\cE$ as natural transformations $\delta : \Id_{\Pold} \mapsto \cP_2$ and $\cE : \cP_2 \circ \cP_2 \mapsto \cP_2$, where for every Polish space $(X,d)$, $\delta_{(X,d)} := (\delta_X)_{|\cP_2(X)}$ and $\cE_{(X,d)} := (\cE_X)_{|\cP_2(X)}$ (when $d$ is not ambiguous, we will therefore note $\cE_{(X,d)} = \cE_X$ and $\delta_{(X,d)} = \delta_X$). 

\begin{remark}
    The triple $(\cP_2,\delta,\cE)$ can also easily be verified to be a monad. It may be called the \emph{Wasserstein monad} on $\Pold$.
\end{remark}

\subsection{Integration of hierarchical measures}

Let again $(X,d)$ be a Polish space. If $n \geq 1$, we define the continuous map $\cE_X^{(n)} : \cPzn{n}{X} \mapsto \cP(X)$ by $\cE_X^{(1)} = \id_{\cP(X)}$ and $\cE_X^{(n)} := \cE_X \circ \cE_{\cP(X)} \circ \ldots \circ \cE_{\cPzn{n-2}{X}}$ for $n \geq 2$. For every $\bP \in \cPzn{n}{X}$, $\tilde{\bP} := \cE_X^{(n)}(\bP)$ is then the unique probability measure such that
\begin{equation} \label{eq:total_collapse_expression}
    \int f(x) \dd\tilde{\bP}(x) = \iint \ldots \iint f(x) \dd\mu(x) \dd\mu^{2}(\mu) \ldots \dd\mu^{n-1}(\mu^{n-2}) \dd\bP(\mu^{n-1}).
\end{equation}
for every $f : X \mapsto \R_+$ nonnegative Borel measurable. We also define $\cE_X^{(0)} = \delta_X$. Applying inductively \Cref{prop:collapse_restricts_to_p2}, it is clear that for every $n \geq 0$, the map $\cE_X^{(n)}$ restricts to a map $\cPn{n}{X} \mapsto \cP_2(X)$ which is continuous in the $\W_2$ topology, and such that
\begin{equation} \label{eq:total_collapse_preserve_2nd_moment}
    \W_2(\delta^{(n)}_{x_0}, \mu) = \W_2^2(\delta_{x_0}, \cE_X^{(n)}(\mu)), \quad \forall n \geq 0, \forall \mu \in \cPn{n}{X}.
\end{equation}

\begin{definition}
    Let $n \geq 0$ and $\mu \in \cPzn{n}{X}$. Let $\tilde{\mu} := \cE^{(n)}_X(\mu)$. If $f : X \mapsto \R$ is Borel measurable and either nonnegative or $\tilde{\mu}$-integrable, we define its \emph{$n$-expectancy with respect to $\mu$} as
    \begin{equation}
        \bE^{(n)}_\mu[f] := \int f(x) \dd\tilde{\mu}(x).
    \end{equation}
\end{definition}

In particular, $\bE^{(1)}_\mu[f]$ is the usual expectancy $\bE_\mu[f]$ of $f$ over $\mu$, and $\bE^{(0)}_x[f] = f(x)$ is just the evaluation map at $x \in X$. Moreover, by naturality of $\cE$ (and $\delta$), for every morphism $g : X \mapsto Y$ in $\Pol$, we have $[g] \circ \cE_X^{(n)} = \cE_Y^{(n)} \circ [g]^{(n)}$, so that for every $\mu \in \cP(X)$ and $\nu := [g]^{(n)}(\mu) \in \cP(Y)$, we have $\bE^{(n)}_\nu[f] = \bE^{(n)}_\mu[f \circ g]$ whenever $f : Y \mapsto \R$ is measurable and either nonnegative or $\cE^{(n)}_Y(\nu)$-integrable. Indeed
\begin{equation} \label{eq:n_expectancy_of_pushforward}
    \bE^{(n)}_\nu[f] = \int f \dd(\cE_Y^{(n)}(\nu)) = \int f \dd([g](\cE_X^{(n)}(\mu))) = \int f \circ g \dd\cE_X^{(n)}(\mu) = \bE^{(n)}_\mu[f \circ g].
\end{equation}
Moreover, the $n$-expectancy satisfies a recurrence relation on $n$. Indeed, if $n \geq 1$ and $\bP \in \cPzn{n}{X}$, it is clear from the expression of $\tilde{\bP} = \cE_X^{(n)}(\bP)$ given by \eqref{eq:total_collapse_expression} that if $f : X \mapsto \R$ is a measurable function which is either nonnegative or $\tilde{\bP}$-integrable, then the function $\mu \in \cPzn{n-1}{X} \mapsto \bE^{(n-1)}_\mu[f]$ is well-defined, measurable, respectively nonnegative or $\bP$-integrable, and such that
\begin{equation}
    \bE^{(n)}_\bP[f] = \int \bE^{(n-1)}_\mu[f] \dd\bP(\mu).
\end{equation}
We have the following results on the regularity of the $n$-expectancy:
\begin{lemma} \label{lemma:n_expectancy_regularity}
    The following hold:
    \begin{enumerate}
        \item \label{enum:n_expect_reg:1} If $f \in C_b(X)$ is continuous and bounded, then for every $n \geq 0$, the map $\mu \in \cPzn{n}{X} \mapsto \bE^{(n)}_\mu[f]$ is continuous and bounded, with the same bound as $f$.
        \item \label{enum:n_expect_reg:2} If $f : X \mapsto \R$ is nonnegative and lower semicontinuous, then for every $n \geq 0$, the function $\mu \in \cPzn{n}{X} \mapsto \bE^{(n)}_\mu[f]$ is nonnegative and lower semicontinuous.
        \item \label{enum:n_expect_reg:3} If $f \in C(X)$ is a continuous function, such that $|f(x)| \leq C_1 + C_2 d^2(x,x_0)$ for every $x \in X$ for some $C_1, C_2 \geq 0$ and $x_0 \in X$, then for every $n \geq 0$, the function $\mu \in \cPn{n}{X} \mapsto \bE^{(n)}_\mu[f]$ is continuous (for the $\W_2$ topology), and such that $|\bE^{(n)}_\mu[f]| \leq C_1 + C_2 \W_2^2(\mu,\delta^{(n)}_{x_0})$ for every $\mu \in \cPn{n}{X}$.
    \end{enumerate}
\end{lemma}

\begin{proof}
    The case $n = 0$ is trivial since $\cPzn{0}{X} = X$ and $\bE_x^{(0)}[\cdot]$ is simply the evaluation at $x$. Moreover, the points \ref{enum:n_expect_reg:1} and \ref{enum:n_expect_reg:2} follow immediately from the discussion of the properties of weak convergence in \Cref{sec:appendix:weak_topology}. Now, let $f \in C(X)$ be such that there exists $x_0 \in X$ and $C_1, C_2 \geq 0$ for which $|f(x)| \leq C_1 + C_2 d^2(x,x_0)$ for every $x \in X$. Therefore, by point \ref{enum:w2_convergence:4_quadratic_growth} of \Cref{th:convergence_in_w2_space}, the map $\mu \in \cP_2(X) \mapsto \int f \dd\mu$ is continuous in the $\W_2$ topology. But since for every $\mu \in \cPn{n}{X}$, $\bE^{(n)}_\mu[f] = \int f \dd\cE_X^{(n)}(\mu)$, where the map $\cE_X^{(n)} : \cPn{n}{X} \mapsto \cP_2(X)$ is continuous in the $\W_2$ topology, the map $\mu \in \cPn{n}{X} \mapsto \bE^{(n)}_\mu[f]$ is also continuous in the $\W_2$ topology. Furthermore, for every $\mu \in \cPn{n}{X}$, we have
    \begin{align}
        |\bE^{(n)}_\mu[f]| &= \left|\int f \dd\cE_X^{(n)}(\mu) \right| \leq \int |f| \dd\cE_X^{(n)}(\mu) \leq C_1 + C_2 \int d^2(x,x_0) \dd\cE_X^{(n)}(\mu) \\
        &\leq C_1 + C_2 \W_2^2(\delta_{x_0}, \cE_X^{(n)}(\mu)) = C_1 + C_2 \W_2^2(\delta^{(n)}_{x_0}, \mu)
    \end{align}
    where we used \eqref{eq:total_collapse_preserve_2nd_moment} to obtain the last equality. This finishes proving point \ref{enum:n_expect_reg:3}.
\end{proof}

\begin{lemma} \label{lemma:n_holder_ineq} (Cauchy-Schwarz and Hölder's inequality) 
    Let $f, g : X \mapsto \R$ be two measurable functions and $1 < p,q < +\infty$ such that $1/p + 1/q = 1$, then for any $n \geq 0$ and $\mu \in \cPn{n}{X}$,
    \begin{equation}
        \bE^{(n)}_\mu[|fg|] \leq \bE^{(n)}_\mu[|f|^p]^{1/p} \bE^{(n)}_\mu[|g|^q]^{1/q}.
    \end{equation}
\end{lemma}

\begin{proof}
    This is simply a consequence of the definition of $\bE_\mu^{(n)}[\cdot]$ and of Hölder's inequality applied to $\cE_X^{(n)}(\mu)$.
\end{proof}

\begin{lemma} \label{lemma:n_minkowsky_ineq} (Minkowski inequality) 
    Let $f, g : X \mapsto \R$ be two measurable functions and $1 \leq p < +\infty$, then for any $n \geq 0$ and $\mu \in \cPn{n}{X}$,
    \begin{equation}
        \bE^{(n)}_\mu[|f+g|^p]^{1/p} \leq \bE^{(n)}_\mu[|f|]^{1/p} + \bE^{(n)}_\mu[|g|]^{1/p}.
    \end{equation}
\end{lemma}

\begin{proof}
    This is simply a consequence of the definition of $\bE_\mu^{(n)}[\cdot]$ and of the Minkowski inequality applied to $\cE_X^{(n)}(\mu)$.
\end{proof}

\begin{definition}
    Let $X$ be a Polish space, $n \geq 0$, and $\mu \in \cPzn{n}{X}$. The \emph{base support} of $\mu$, denoted $\spt_X(\mu)$, is the closed subset of $X$ defined to be the support of the measure $\cE_X^{(n)}(\mu)$: that is $\spt_X(\mu) := \spt(\cE_X^{(n)}(\mu))$.
\end{definition}

In particular, the base support of $\mu \in \cP(X)$ is its support, and for every $x \in X$, $\spt_X(x) = \spt(\delta_x) = \{x\}$. The base support satisfies the following recurrence relation:

\begin{lemma} \label{lemma:base_support_recurrence}
    For every $n > 0$ and $\bP \in \cPzn{n}{X}$, it holds
    \begin{equation} \label{eq:base_support_recurrence}
        \spt_X(\bP) = \overline{\bigcup_{\mu \in \spt(\bP)} \spt_X(\mu)}.
    \end{equation}
    If $\bP \in \cPn{n}{X}$, then the support $\spt(\bP)$ in \eqref{eq:base_support_recurrence} can also be taken with respect to the $\W_2$ topology.
\end{lemma}

\begin{proof}
    Recalling that the base support of a measure $\mu \in \cPzn{m}{X}$ is the support of $\cE_X^{(m)}(\mu)$, all we need to prove \eqref{eq:base_support_recurrence} is to show that for every open set $U \subseteq X$, $\cE_X^{(n)}(\bP)(U) = 0$ if and only if $\cE_X^{(n-1)}(\mu)(U) = 0$ for every $\mu \in \spt(\bP)$. Let $U \subseteq X$ be a fixed open set. Then we have
    \begin{equation}
        \cE_X^{(n)}(\bP)(U) = \bE^{(n)}_\bP[\bOne_U] = \int \bE^{(n-1)}_\mu[\bOne_U] \dd\bP(\mu) = \int \cE_X^{(n-1)}(\mu)(U) \dd\bP(\mu).
    \end{equation}
    and thus $\cE_X^{(n)}(\bP)(U) = 0$ if and only if $\cE_X^{(n-1)}(\mu)(U) = 0$ for $\bP$-a.e. $\mu$. In particular, if $\cE_X^{(n-1)}(\mu)(U) = 0$ for every $\mu \in \spt(\bP)$, we have $\cE_X^{(n)}(\bP)(U) = 0$. Conversely, if $\cE_X^{(n)}(\bP)(U) = 0$, then there exists a Borel set $B \subseteq \cPzn{n-1}{X}$ such that $\bP(B) = 1$ and $\cE_X^{(n-1)}(\mu)(U) = 0$ for every $\mu \in B$. Since the indicator map $\bOne_U$ is lower semicontinuous, the map $\mu \in \cPzn{n-1}{X} \mapsto \cE_X^{(n-1)}(\mu)(U) = \bE^{(n-1)}_\mu[\bOne_U]$ is also lower semicontinuous (by \Cref{lemma:n_expectancy_regularity}) and nonnegative, and it follows that $\cE_X^{(n-1)}(\mu)(U) = 0$ for every $\mu$ in the closure $\bar{B}$ of $B$. However, since $\bP(\bar{B}) = \bP(B) = 1$, we have $\spt(\bP) \subseteq \bar{\bP}$, and $\cE_X^{(n-1)}(\mu)(U) = 0$ for every $\mu \in \spt(\bP)$. This finishes proving \eqref{eq:base_support_recurrence}. Finally, note that if $\bP \in \cPn{n}{X}$, then the entire reasoning above also works with the $\W_2$ topology, as it is stronger than the topology of weak convergence.
\end{proof}

It is not difficult to see that if $F$ is a closed subset of $X$, then for every $n > 0$, $\cPzn{n}{F}$ identifies as the subset of measures in $\cPzn{n}{X}$ with base support contained in $F$, and likewise for $\cPn{n}{F}$ and $\cPn{n}{X}$. We will say that $\mu \in \cPzn{n}{X}$ has \emph{compact base support} if $\spt_X(\mu) \subseteq X$ is compact.

\begin{lemma} \label{lemma:when_maps_agree_on_base_support}
    Let $n \geq 0$ and $\mu \in \cPzn{n}{X}$ be fixed, then the following hold:
    \begin{enumerate}
        \item If $f : X \mapsto \R$ is a measurable function which vanishes on $\spt_X(\mu)$, then $\bE^{(n)}_\mu[f] = 0$.
        \item If $f, g : X \mapsto Y$ are two continuous maps between Polish spaces such that $f = g$ on $\spt_X(\mu)$, then $[f]^{(n)}(\mu) = [g]^{(n)}(\mu)$.
        \item If $f : X \mapsto Y$ is a continuous map between Polish spaces, and $\nu := [f]^{(n)}(\mu)$, then
        \begin{equation}
            \spt_Y(\nu) \subseteq \overline{f(\spt_X(\mu))}.
        \end{equation}
    \end{enumerate}
\end{lemma}

\begin{proof}
    The first point is immediate: since $\spt_X(\mu) = \spt(\tilde{\mu})$ with $\tilde{\mu} = \cE_X^{(n)}(\mu)$, then $f = 0$ $\tilde{\mu}$-a.e., so that $\bE^{(n)}_\mu[f]$ is well defined and equal to $\int f \dd\tilde{\mu} = 0$. \newline
    We prove the second point by induction. The case $n = 0$ is trivial: if $f$ and $g$ coincide on $\spt_X(x) = \{x\}$, then clearly $[f]^{(0)}(x) = f(x)$ and $[g]^{(0)}(x) = g(x)$ are equal. Now, let $n > 0$ and assume the result holds for $n-1$. Let $\bP \in \cPzn{n}{X}$ such that $f$ and $g$ coincide on $\spt_X(\bP)$. By \Cref{lemma:base_support_recurrence}, we have $\spt_X(\mu) \subseteq \spt_X(\bP)$ for every $\mu \in \spt(\bP)$, so that by the induction hypothesis, the continuous maps $[f]^{(n-1)}, [g]^{(n-1)}$ coincide on $\spt(\bP)$. Therefore,
    \begin{equation}
        [f]^{(n)}(\bP) = [f]^{(n-1)}_\#\bP = [g]^{(n-1)}_\#\bP = [g]^{(n)}(\bP)
    \end{equation}
    and this finishes proving the second point. \newline
    We now prove the third point. We have $\cE_Y^{(n)}(\nu) = [f](\cE_X^{(n)}(\mu))$ by naturality of $\cE$, so that
    \begin{align}
        \cE_Y^{(n)}(\nu)(\overline{f(\spt_X(\mu))}) &= [f](\cE_X^{(n)}(\mu))(\overline{f(\spt_X(\mu))}) \\
        &= \cE_X^{(n)}(\mu)(f^{-1}(\overline{f(\spt_X(\mu))})) \\
        &\geq \cE_X^{(n)}(\mu)(\spt_X(\mu)) = 1
    \end{align}
    so that the closed set $\overline{f(\spt_X(\mu))}$ has full $\cE_Y^{(n)}(\nu)$-measure, and thus 
    \begin{equation}
        \spt_X(\nu) = \spt(\cE_Y^{(n)}(\nu)) \subseteq \overline{f(\spt_X(\mu))}.
    \end{equation}
    This finishes the proof.
\end{proof}

\subsection{Fiber products of Polish spaces} 

The goal of this subsection is to investigate the properties of fiber products in the categories $\Pol$ and $\Pold$. Recall that, in a category $\cC$, the \emph{fiber product} (or \emph{pullback}) of a pair of morphisms $f_1 : X \mapsto Z$ and $f_2 : Y \mapsto Z$ with same codomain $Z$ is the given of an objet $X \times_Z Y$ of $\cC$ and of two morphisms $\pi_1 : X \times_Z Y \mapsto X$ and $\pi_2 : X \times_Z Y \mapsto Y$ such that $f_1 \circ \pi_1 = f_2 \circ \pi_2$, and satisfying the following universal property: for every object $W$ of $\cC$ and pair of morphisms $g_1 : W \mapsto X$ and $g_2 : W \mapsto Y$ such that $f_1 \circ g_1 = f_2 \circ g_2$, there exists a unique morphism $h : W \mapsto X \times_Z Y$ such that the following diagram commutes:
\begin{equation}
    \begin{tikzcd}
        & W \arrow[ldd, "g_1"', bend right] \arrow[rdd, "g_2", bend left] \arrow[d, dashed, "\exists !h"] & \\
        & X \times_Z Y \arrow[ld, "\pi_1"'] \arrow[rd, "\pi_2"] & \\
        X \arrow[rd, "f_1"'] & & Y \arrow[ld, "f_2"] \\
        & Z &
    \end{tikzcd}
\end{equation}
The fiber product $X \times_Z Y$, if it exists, is unique up to isomorphism (this is a consequence of the universal property). We sometimes denote it $X \times_{f_1,Z,f_2} Y$ when there is ambiguity on the morphisms $f_1$ and $f_2$.
\medbreak
\begin{proposition}
    Let $X$, $Y$ and $Z$ be three Polish spaces and $f_1 : X \mapsto Z$, $f_2 : Y \mapsto Z$ be continuous. Then they admit a fiber product in the category $\Pol$, which is the subset of $X \times Y$ given by
    \begin{equation}
        X \times_Z Y := \{(x,y) \in X \times Y \setcond f_1(x) = f_2(y) \} \subseteq X \times Y
    \end{equation}
    and where the maps $\pi_1 : X \times_Z Y \mapsto X$ and $\pi_2 : X \times_Z Y \mapsto Y$ are respectively the projections on the first and the second coordinate. Moreover, if $d_X$, $d_Y$ and $d_Z$ are distances metrizing the topologies of $X$, $Y$ and $Z$, and if $f_1$ and $f_2$ are assumed to have linear growth for these distances, then $f_1$ and $f_2$ admit a fiber product in the category $\Pold$, which is the fiber product $X \times_Z Y$ in $\Pol$ equipped with the distance $d := \sqrt{d_X^2+d_Y^2}$ defined by
    \begin{equation} \label{eq:dist_on_fiber_product}
        d^2((x,y),(x',y')) := d_X^2(x,x') + d_Y^2(y,y'), \quad (x,y),(x',y') \in X \times_Z Y.
    \end{equation}
\end{proposition}

\begin{proof}
    First, $X \times_Z Y$ with the induced topology from $X \times Y$ is indeed a Polish space as it is a closed subset of the product space $X \times Y$, which is Polish. Moreover, if $W$ is another Polish space and $g_1 : W \mapsto X$ and $g_2 : W \mapsto Y$ are two continuous functions such that $f_1 \circ g_1 = f_2 \circ g_2$, then it is easily verified that the map $h : W \mapsto X \times_Z Y$ defined by $h(w) = (g_1(w),g_2(w))$ is continuous, valued in $X \times_Z Y$ (as $f_1 \circ g_1 = f_2 \circ g_2$), and satisfies $\pi_1 \circ h = g_1$ and $\pi_2 \circ h = g_2$ (and it is the only map $W \mapsto X \times_Z Y$ that can satisfy these equalities as a point in $X \times_Z Y$ is determined by its two coordinates). This shows that $X \times_Z Y$ is indeed the fiber product of $f_1$ and $f_2$. \newline
    Now, assume that $d_X$, $d_Y$ and $d_Z$ are distances metrizing $X$, $Y$ and $Z$, and that $f_1$ and $f_2$ have linear growth for these distances. Clearly the distance $d$ on $X \times_Z Y$ as defined in \eqref{eq:dist_on_fiber_product} metrizes the topology of $X \times_Z Y$. Moreover, $\pi_1$ and $\pi_2$ have linear growth, and are even 1-Lipschitz, for this distance. Indeed, for every $(x,y),(x',y') \in X \times_Z Y$,
    \begin{equation}
        d_X(\pi_1(x,y),\pi_1(x',y')) = d_X(x,x') \leq \sqrt{d_X^2(x,x')+d_Y^2(y,y')} = d((x,y),(x',y'))
    \end{equation}
    and likewise for $\pi_2$. Thus $\pi_1$ and $\pi_2$ are morphisms $(X \times_Z Y,d) \mapsto (X,d_X)$ and $(X \times_Z Y,d) \mapsto (Y,d_Y)$ in $\Pold$. Now let $g_1 : (W,d_Y) \mapsto (X,d_X)$ and $g_2 : (W,d_W) \mapsto (Y,d_Y)$ be two continuous maps with linear growth between Polish spaces satisfying $f_1 \circ g_1 = f_2 \circ g_2$. Then the map $h : (W,d_W) \mapsto (X \times_Z Y, d)$ defined by $h(w) = (g_1(w),g_2(w))$ is again continuous and valued in $X \times_Z Y$. Moreover, it has linear growth: letting $w_0 \in W$ and $C > 0$ be such that for every $w \in W$, $d_X(g_1(w_0),g_1(w)) \leq C(1+d_W(w_0,w))$ and $d_Y(g_2(w_0),g_2(w)) \leq C(1+d_W(w_0,w))$, we have for every $w \in W$ that
    \begin{equation}
        d(h(w_0),h(w)) = \sqrt{d_X^2(g_1(w_0),g_1(w)) + d_Y^2(g_2(w_0),g_2(w))} \leq \sqrt{2}C(1+d_W(w_0,w))
    \end{equation}
    so that $h$ indeed has linear growth. Thus $h$ is a morphism $(W,d_W) \mapsto (X \times_Z Y, d)$ in $\Pold$, and it again is the only morphism satisfying the equalities $\pi_1 \circ h = g_1$ and $\pi_2 \circ h = g_2$. Therefore, $(X \times_Z Y, d)$ is the fiber product of $f_1$ and $f_2$ in the category $\Pold$. 
\end{proof}

\begin{remark}
    As observed in the proof, the map $h : W \mapsto X \times_Z Y$ induced by two morphisms $g_1 : W \mapsto X$ and $g_2 : W \mapsto Y$ (whether in $\Pol$ or in $\Pold$) is given by $h(w) = (g_1(w),g_2(w))$ for every $w \in W$.
\end{remark}

\begin{remark}
    If $Z$ is a singleton, then the fiber product $X \times_Z Y$ is simply the product space $X \times Y$.
\end{remark}

\begin{remark}
    For the fiber product in $\Pold$, it would have seemed more natural to equip $X \times_Z Y$ with the sum distance $d' := d_X + d_Y$ instead of the distance $d$ defined in \eqref{eq:dist_on_fiber_product}. Since these distances are bi-Lipschitz equivalent, the metric spaces $(X \times_Z Y, d)$ and $(X \times_Z Y, d')$ are isomorphic in $\Pold$, so they both define a fiber product of $X$, $Y$ and $Z$. However, we will always work with the distance $d$ and assume $X \times_Z Y$ to be equipped with it, as it interacts better with the second order Wasserstein distance than the distance $d'$, and makes the proofs of the subsequent results easier.
\end{remark}

Consider now the diagram in $\Pol$ formed by the morphisms $f_1 : X \mapsto Z$ and $f_1 : Y \mapsto Z$, and apply the functor $\cP$ to the commutative diagram formed by $f_1$, $f_2$ and the projections $\pi_1 : X \times_Z Y \mapsto X$ and $\pi_2 : X \times_Z Y \mapsto Y$. We then obtain another commutative diagram:
\begin{equation}
    \begin{array}{ccc}
        \begin{tikzcd}
        & X \times_Z Y \arrow[dl, "\pi_1"'] \arrow[dr, "\pi_2"] & \\
        X \arrow[dr, "f_1"'] & & Y \arrow[dl, "f_2"] \\
        & Z & 
        \end{tikzcd}
        & 
        \begin{array}{c}
            \cP \\ \Longrightarrow
        \end{array}
        &
        \begin{tikzcd}
        & \cP(X \times_Z Y) \arrow[dl, "{[\pi_1]}"'] \arrow[dr, "{[\pi_2]}"] & \\
        \cP(X) \arrow[dr, "{[f_1]}"'] & & \cP(Y) \arrow[dl, "{[f_2]}"] \\
        & \cP(Z) & 
        \end{tikzcd}
    \end{array}  
\end{equation}
Applying the universal property of the fiber product $\cP(X) \times_{\cP(Z)} \cP(Y)$, one obtains a morphism $h : \cP(X \times_Z Y) \mapsto \cP(X) \times_{\cP(Z)} \cP(Y)$ which makes the following diagram commute:
\begin{equation}
    \begin{tikzcd}
        & \cP(X \times_Z Y) \arrow[ddl, "{[\pi_1]}"', bend right] \arrow[ddr, "{[\pi_2]}", bend left] \arrow[d, "h", dashed] & \\
        & \cP(X) \times_{\cP(Z)} \cP(Y) \arrow[dl] \arrow[dr] & \\
        \cP(X) \arrow[dr, "{[f_1]}"'] & & \cP(Y) \arrow[dl, "{[f_2]}"] \\
        & \cP(Z) & 
    \end{tikzcd}
\end{equation}

We then have the following result, which acts as a sort of generalized ``gluing lemma" between probability measures:
\begin{proposition} \label{prop:p0_of_fiber_product_onto_fiber_product_of_p0}
    The morphism $h : \cP(X \times_Z Y) \mapsto \cP(X) \times_{\cP(Z)} \cP(Y)$ is surjective.
\end{proposition}

\begin{proof}
    What we want to prove is that for every pair $(\mu,\nu) \in \cP(X) \times \cP(Y)$ such that $f_{1\#}\mu = f_{2\#}\nu =: \eta \in \cP(Z)$, there exists some $\gamma \in \cP_2(X \times_Z Y)$ such that $\pi_{1\#}\gamma = \mu$ and $\pi_{2\#}\gamma = \nu$. We may apply the disintegration theorem (\Cref{th:disintegration}) to $\mu$ and $\nu$ to rewrite them as $\dd\mu(x) = \dd\mu_z(x) \dd\eta(z)$ and $\dd\nu(y) = \dd\nu_z(y) \dd\eta(z)$, where $\mu_z$ and $\nu_z$ are probability measures on respectively $X$ and $Y$ which are supported for $\eta$-a.e. $z$ on respectively $f_1^{-1}(z)$ and $f_2^{-1}(z)$. The family $(\mu_z \otimes \nu_z)_{z \in Z}$ is then a Borel family of measures\footnote{For a pair of Borel sets $B \subseteq X$ and $B' \subseteq Y$, the map $z \mapsto \mu_z(B)\nu_z(B')$ is Borel measurable as a product of Borel functions. Since the products of Borel sets generate the product $\sigma$-algebra on $X \times Y$, the family $(\mu_z \otimes \nu_z)_z$ is Borel by the discussion in \Cref{sec:appendix:disintegration}.}, so there exists a probability measure $\gamma \in \cP(X \times Y)$ such that, for every $f : X \times Y \mapsto \R$ nonnegative and measurable, we have
    \begin{equation}
        \int f(x,y) \dd\gamma(x,y) = \iiint f(x,y) \dd\mu_z(x) \dd\nu_z(y) \dd\eta(z).
    \end{equation}
    Then, one can easily check that $\pi_{1\#}\gamma = \mu$ and $\pi_{2\#}\gamma = \nu$. Moreover, for $\eta$-a.e. $z \in Z$, $\mu_z$ is supported in $f_1^{-1}(z)$ and $\mu_z$ is supported in $f_2^{-1}(z)$, so that $\mu_z \otimes \nu_z$ is supported in $f_1^{-1}(z) \times f_2^{-1}(z) \subseteq X \times_Z Y$. Therefore, $\gamma$ itself is supported in $X \times_Z Y$, so that we have found $\gamma \in \cP(X \times_Z Y)$ such that $h(\gamma) = ([\pi_1](\gamma), [\pi_2](\gamma)) = (\mu,\nu)$. Thus, we have proved that $h$ is surjective.
\end{proof}

Now, if $d_X$, $d_Y$ and $d_Z$ are distances metrizing $X$, $Y$ and $Z$, and if we assume that $f_1$ and $f_2$ have linear growth for these distances, then, by a similar procedure but this time using the functor $\cP_2$, we can construct a morphism $h_2 : \cP_2(X \times_Z Y) \mapsto \cP_2(X) \times_{\cP_2(Z)} \cP_2(Y)$ in the category $\Pold$.

\begin{proposition} \label{prop:p_of_fiber_product_onto_fiber_product_of_p}
    The morphism $h_2$ is surjective. In fact, for every $\gamma \in \cP(X \times_Z Y)$ and $(x_0,y_0) \in X \times_Z Y$, it holds
    \begin{equation} \label{eq:2nd_moment_of_transport_plan}
        \W_2^2(\delta_{x_0,y_0}, \gamma) = \W_2^2(\delta_{x_0},[\pi_1](\gamma)) + \W_2^2(\delta_{y_0}, [\pi_2](\gamma)).
    \end{equation}
\end{proposition}

\begin{proof}
    Let $\gamma \in \cP(X \times_Z Y)$ and $(x_0,y_0) \in X \times_Z Y$. Then, we have
    \begin{align}
        \W_2^2(\delta_{x_0,y_0}, \gamma) &= \int d^2((x_0,y_0),(x,y)) \dd\gamma(x,y) = \int d_X^2(x_0,x) + d^2_Y(y_0,y) \dd\gamma(x,y) \\
        &= \int d^2_X(x_0,x) \dd\mu(x) + \int d^2_Y(y_0,y) \dd\nu(y) = \W_2^2(\delta_{x_0}, \mu) + \W_2^2(\delta_{y_0}, \nu),
    \end{align}
    and this proves \eqref{eq:2nd_moment_of_transport_plan}. Now, let $\mu \in \cP_2(X)$ and $\nu \in \cP_2(Y)$ be such that $[f_1](\mu) = [f_2](\nu)$. By \Cref{prop:p0_of_fiber_product_onto_fiber_product_of_p0}, there exists $\gamma \in \cP(X \times_Z Y)$ such that $[\pi_1](\gamma) = \mu$ and $[\pi_2](\gamma) = \nu$. However, since $\mu$ and $\nu$ are assumed to have finite second order moment, by \eqref{eq:2nd_moment_of_transport_plan}, $\gamma$ also has finite second order moment. Therefore, we have found $\gamma \in \cP_2(X \times_Z Y)$ such that $h(\gamma) = ([\pi_1](\gamma),[\pi_2](\gamma)) = (\mu,\nu)$, and $h_2$ is surjective.
\end{proof}

We can also transpose the theory of optimal transport to fiber products. Let again $X$, $Y$ and $Z$ be three Polish spaces with a pair of continuous functions $f_1 : X \mapsto Z$ and $f_2 : Y \mapsto Z$, so that we can consider the corresponding fiber product $X \times_Z Y$ in $\Pol$. Let $c : X \times_Z Y \mapsto \R_+$ be a nonnegative lower semicontinuous cost function, and fix $\mu \in \cP(X)$ and $\nu \in \cP(Y)$ such that $f_{1\#}\mu = f_{2\#}\nu =: \eta \in \cP(Z)$. We can consider the following Kantorovich optimal transport problem on fiber products:
\begin{equation} \label{eq:ot_problem_in_fiber_product}
    \inf_{\gamma \in \Pi(\mu,\nu,Z)} \int c(x,y) \dd\gamma(x,y) \tag{KP-Fib}
\end{equation}
where
\begin{equation}
    \Pi(\mu,\nu,Z) := \{\gamma \in \cP(X \times_Z Y) \setcond \pi_{1\#}\gamma = \mu, \pi_{2\#}\gamma = \nu \},
\end{equation}
which is not empty by \Cref{prop:p0_of_fiber_product_onto_fiber_product_of_p0}. This is essentially an optimal transport problem between $\mu$ and $\nu$ where we are only allowed to move mass within the fibers $f_1^{-1}(z) \times f_2^{-1}(z)$, $z \in Z$ of $X \times_Z Y$. We may apply the disintegration theorem (\Cref{th:disintegration}) to $\mu$ and $\nu$ to rewrite them as
\begin{equation}
    \dd\mu(x) = \dd\mu_z(x) \dd\eta(z), \quad \dd\nu(y) = \dd\nu_z(y) \dd\eta(z)
\end{equation}
where $\mu_z$ and $\nu_z$ are probability measures on respectively $X$ and $Y$ which are supported for $\eta$-a.e. $z$ on respectively $f_1^{-1}(z)$ and $f_2^{-1}(z)$. Similarly, given $\gamma \in \Pi(\mu, \nu, Z)$, we may use the disintegration theorem to write it as
\begin{equation}
    \dd\gamma(x,y) = \dd\gamma_z(x,y) \dd\eta(z)
\end{equation}
where for $\eta$-a.e. $z$, $\gamma_z$ is a probability measure on $X \times_Z Y$ supported on the fiber $f_1^{-1}(z) \times f_2^{-1}(z)$. If $f : X \mapsto \R_+$ is a nonnegative measurable test function, we have
\begin{align}
    \int f(x) \dd\mu(x) &= \int f(\pi_1(x)) \dd\gamma(x,y) = \iint f(\pi_1(x)) \dd\gamma_z(x,y) \dd\eta(z) \\
    &= \iint f(x) \dd (\pi_{1\#}\gamma_z)(x) \dd\eta(z)
\end{align}
so by unicity of the disintegration, we have $\pi_{1\#}\gamma_z = \mu_z$ for $\eta$-a.e. $z$. Similarly, $\pi_{2\#}\gamma_z = \nu_z$ for $\eta$-a.e. $z$, so that $\gamma_z$ is a transport plan between $\mu_z$ and $\nu_z$ for $\eta$-a.e. $z$. In fact, up to replacing $\gamma_z$ by $\mu_z \otimes \nu_z$ for a $\eta$-negligible set of $z$, we may assume that $\gamma_z \in \Pi(\mu_z,\nu_z)$ for every $z \in Z$.

\begin{proposition} \label{prop:ot_problem_in_fiber_product_has_solutions}
    Assume that \eqref{eq:ot_problem_in_fiber_product} is finite. Then the set $\Pi_o(\mu,\nu,Z)$ of measures $\gamma \in \Pi(\mu,\nu,Z)$ that are minimizers of the problem \eqref{eq:ot_problem_in_fiber_product} is not empty. Moreover, a measure $\gamma \in \Pi(\mu, \nu, Z)$ is in $\Pi_o(\mu,\nu, Z)$ if and only for $\eta$-a.e. $z$, its disintegration $\gamma_z$ is an optimal transport plan between $\mu_z$ and $\nu_z$ for the cost $c$.
\end{proposition}

\begin{proof}
    Since \eqref{eq:ot_problem_in_fiber_product} is finite, this implies that there exists $\gamma \in \Pi(\mu,\nu,Z)$ such that $\int c \dd\gamma < +\infty$, that is, 
    \begin{equation}
        \iint c(x,y) \dd\gamma_z(x,y) \dd\eta(z) = \int c(x,y) \dd\gamma(x,y) < +\infty.
    \end{equation}
    Thus, there exists a set $Z' \subseteq Z$ of full measure for $\eta$ such that for every $z \in Z'$, $\int c(x,y) \dd\gamma_z(x,y) < +\infty$, that is, the optimal transportation cost between $\mu_z$ and $\nu_z$ for $c$ is finite. Let $h : Z' \times X \times Y \mapsto \R$ be defined by
    \begin{equation}
        h(z,x,y) := \begin{cases}
            c(x,y) & \hbox{ if } z = f_1(x) = f_2(y) \\
            +\infty & \hbox{ else,}
        \end{cases}
    \end{equation}
    then clearly $h_z = h(z,\cdot,\cdot)$ is nonnegative lower semicontinuous for every $z \in Z'$. Applying \Cref{th:fibered_opt_plan_selection} to $h$, $(\mu_z)_{z \in Z'}$ and $(\nu_z)_{z \in Z'}$, we deduce that there exists a Borel family $(\gamma_{0,z})_{z \in Z'}$ such that for every $z \in Z'$, $\gamma_{0,z}$ is an optimal transport plan between $\mu_z$ and $\nu_z$ for the cost $c$ (as $c = h_z$ on $\spt(\mu_z) \times \spt(\nu_z)$). We define then $\gamma_0 \in \Pi(\mu,\nu,Z)$ to be the unique measure such that
    \begin{equation}
        \int f(x,y) \dd\gamma_0(x,y) = \int_{Z'} \int f(x,y) \dd\gamma_{0,z}(x,y) \dd\eta(z)
    \end{equation}
    for every nonnegative measurable test function $f : X \times_Z Y \mapsto \R_+$. We want to prove that $\gamma_0 \in \Pi_o(\mu,\nu,Z)$: for this, fix $\gamma \in \Pi(\mu,\nu,Z)$, we then have
    \begin{align}
        \int c(x,y) \dd\gamma(x,y) &= \iint c(x,y) \dd\gamma_z(x,y) \dd\eta(z) \\
        &= \int_{Z'} \int c(x,y) \dd\gamma_z(x,y) \dd\eta(z) \\
        &\geq \int_{Z'} \int c(x,y) \dd\gamma_{0,z}(x,y) \dd\eta(z) \label{eq:l_337} \\
        &= \int c(x,y) \dd\gamma_0(x,y)
    \end{align}
    where the third line is obtained using the optimality of $\gamma_{0,z}$ for every $z \in Z'$. Therefore we have $\int c\dd\gamma \geq \int c \dd\gamma_0$, and $\gamma_0$ is indeed optimal. Moreover, the inequality \eqref{eq:l_337} is an equality if and only if $\int c \dd\gamma_z = \int c \dd\gamma_{0,z}$ for $\eta$-a.e. $z \in Z'$, and thus $\gamma$ is itself optimal if and only if $\gamma_z$ is optimal for $\eta$-a.e. $z$. This finishes the proof.
\end{proof}

\subsection{Surjectivity of pushforwards}

We close this subsection with a result on the surjectivity of pushforward maps, which will be useful later for proving the existence of certain measures:

\begin{definition}
    A function $f : (X,d_X) \mapsto (Y,d_Y)$ between metric spaces is said to have \emph{linearly growing antecedents in $B$}, where $B$ is a subset of $X$, if it is surjective and there exists $x_0 \in X$ and $C > 0$ such that for every $y \in Y$, there exists $x \in B$ such that $f(x) = y$ and $d_X(x_0,x) \leq C(1 + d_Y(f(x_0),y))$. In the case where $B = X$, we simply say that $f$ has \emph{linearly growing antecedents}.
\end{definition}

One can check that if a function $f$ satisfies this definition for some $x_0 \in X$, then it satisfies it for any $x_0 \in X$. Furthermore, the composition of two functions with linearly growing antecedents has linearly growing antecedents.

\begin{remark} \label{rk:linear_antecedents_p_of_fiber_product_onto_fiber_product_of_p}
    \Cref{prop:p_of_fiber_product_onto_fiber_product_of_p} directly implies that the natural morphism $h : \cP_2(X \times_Z Y) \mapsto \cP_2(X) \times_{\cP_2(Z)} \cP_2(Y)$ between the $\cP_2$ space of the fiber product and the fiber product of the $\cP_2$ spaces has linearly growing antecedents.
\end{remark}

\begin{lemma} \label{lemma:pushforward_surjective}
    Let $X$, $Y$ be Polish spaces, $B \subseteq X$ a Borel subset, and $f : B \mapsto Y$ be a continuous map. If $\mu \in \cP(Y)$ is concentrated on $f(B)$, then there exists $\nu \in \cP(X)$ concentrated on $B$ such that $f_\#\nu = \mu$.
\end{lemma}

\begin{proof}
    Note that this proof makes heavy use of various concepts and results from descriptive set theory, such as analytic sets, universal measurability, the completion of a measure, or the Jankov-von Neumann uniformization theorem. We refer to \citep{kechris_classical_1995} for their detailed definitions, statements, and proofs. \newline
    Consider the set $R \subseteq X \times Y$ defined by $R := \{(x,f(x)), x \in B\}$. This set is analytic, as it is the continuous image by $(\id,f)$ of $B$ which is a Borel set of a Polish space. Therefore, by the Jankov-von Neumann uniformization theorem, there exists a function $g : f(B) \mapsto X$ which is universally measurable and such that $(g(y),y) \in R$ for every $y \in f(B)$, that is $g(y) \in B$ and $f \circ g(y) = y$ for every $y \in f(B)$. Let $\mu \in \cP(Y)$ be concentrated on $f(B)$. Then, the probability measure $\nu := g_\#\mu \in \cP(X)$ is well-defined as, for every Borel set $B' \subseteq X$, $g^{-1}(B')$ is universally measurable and thus measurable with respect to $\mu$, so that $\nu(B') := \mu(g^{-1}(B'))$ is well-defined, and $\nu(X) = \mu(g^{-1}(X)) = \mu(f(B)) = 1$ so $\nu$ is indeed an element of $\cP(X)$. Moreover, since $g$ is valued in $B$, we have $g^{-1}(B) = f(B)$, so that $\nu(B) = \mu(g^{-1}(B)) = \mu(f(B)) = 1$, and $\nu$ is concentrated on $B$. Finally, we have $f_\#\nu = (f \circ g)_\#\mu = \mu$, and this finishes the proof.
\end{proof}

\begin{proposition} \label{prop:pushforward_by_surjective_is_surjective_in_P0}
    Let $f : X \mapsto Y$ be a continuous map between Polish spaces. If it is surjective, then $[f] : \cP(X) \mapsto \cP(Y)$ is surjective.
\end{proposition}

\begin{proof}
    This is a direct consequence of \Cref{lemma:pushforward_surjective} applied to $f$ and $B = X$.
\end{proof}

\begin{proposition} \label{prop:pushforward_by_surjective_is_surjective_in_P2}
    Let $f : (X,d_X) \mapsto (Y,d_Y)$ be a continuous map between Polish spaces, with linear growth and linearly growing antecedents in $B_0$, where $B_0$ is a Borel subset of $X$. Then the pushforward operator $[f] : \cP_2(X) \mapsto \cP_2(Y)$ is surjective, and with linearly growing antecedents in $\cP_2(B_0)$.
\end{proposition}

\begin{proof}
    Let $x_0 \in X$ and $C > 0$ be such that for every $y \in Y$, there exists $x \in B_0$ such that $f(x) = y$ and $d_X(x_0,x) \leq C(1 + d_Y(f(x_0),y))$. Up to intersecting $B_0$ with the closed set $\{(x,f(x)), x \in X, d_X(x,x_0) \leq C(1+d_Y(f(x_0),f(x)))\}$, we may further assume that $d_X(x,x_0) \leq C(1+d_Y(f(x_0),f(x)))$ for every $x \in B_0$. Fix $\mu \in \cP_2(Y)$. Applying \Cref{lemma:pushforward_surjective} to $f$ and $B_0$, there exists $\nu \in \cP(X)$ concentrated on $B_0$ such that $f_\#\nu = \mu$. Moreover, $\mu$ is in $\cP_2(X)$ as we have
    \begin{align}
        \W_2^2(\delta_{x_0}, \mu) &= \int d^2_X(x_0,x) \dd\nu(x) = \int_{B_0} d^2_X(x_0,x) \dd\nu(x) \\
        &\leq \int_{B_0} C^2(1 + d_Y(f(x_0),f(x)))^2 \dd\nu(x) \\
        &\leq C_2\left(1 + \int d_Y^2(f(x_0),f(x)) \dd\nu(x) \right) \\
        &\leq C_2\left(1 + \int d_Y^2(f(x_0),y) \dd\mu(x) \right) \\
        &\leq C_2(1 + \W_2^2(\delta_{f(x_0)}, \mu)) < +\infty. \label{eq:l_431}
    \end{align}
    Thus we have proved that $[f] : \cP_2(X) \mapsto \cP_2(Y)$ is surjective, and in fact inequality \eqref{eq:l_431} also implies that $[f]$ has linearly growing antecedents in $\cP_2(B_0)$.
\end{proof}

\subsection{Fiber products and hierarchical probability spaces}

We collect here various results on how the fiber products in the categories $\Pol$ and $\Pold$ interact with repeated applications of the functors $\cP$ and $\cP_2$. In order to express these results in full generality, we will work in a category $\cC$ and a functor $P : \cC \mapsto \cC$ which will be either $\Pol$ and $\cP$ or $\Pold$ and $\cP_2$. First, we have the following ``hierarchical gluing lemma", which will prove to be useful later for dealing with the existence of certain types of measures:

\begin{lemma} \label{lemma:hierarchical_gluing_lemma}
    Assume that we have a commutative diagram in $\cC$:
    \begin{equation}
        \begin{tikzcd}
            & W \arrow[ddl, bend right, "g_1"'] \arrow[ddr, bend left, "g_2"] \arrow[d, dashed, "h"] &  \\
            & X \times_Z Y \arrow[dl] \arrow[dr] & \\
            X \arrow[r, "f_1"'] & Z & Y \arrow[l, "f_2"] 
        \end{tikzcd}
    \end{equation}
    where $h$ is induced by the universal property of the fiber product. Then, for every $n \geq 0$, we also have a commutative diagram
    \begin{equation} \label{diag:l1912}
        \begin{tikzcd}
            & P^{(n)}(W) \arrow[ddl, bend right, "{P^{(n)}(g_1)}"'] \arrow[ddr, bend left, "{P^{(n)}(g_2)}"] \arrow[d, dashed, "h_n"] &  \\
            & P^{(n)}(X) \times_{P^{(n)}(Z)} P^{(n)}(Y) \arrow[dl] \arrow[dr] & \\
            P^{(n)}(X) \arrow[r, "{P^{(n)}(f_1)}"'] & P^{(n)}(Z) & P^{(n)}(Y) \arrow[l, "{P^{(n)}(f_2)}"] 
        \end{tikzcd}
    \end{equation}
    where $h_n$ is again induced by the universal property of the fiber product. Then:
    \begin{itemize}
        \item If $\cC = \Pol$, $P = \cP$ and $h$ is surjective, then $h_n$ is surjective for every $n \geq 0$.
        \item If $\cC = \Pold$, $P = \cP_2$ and $h$ is surjective with linearly growing antecedents, then $h_n$ is surjective with linearly growing antecedents for every $n \geq 0$.
    \end{itemize}
\end{lemma}

\begin{proof}
    We prove this by induction. For $n = 0$, we have $h_0 = h$ so this follows directly from the assumptions. Let $n > 0$, and assume the result holds for $n-1$. Applying the functor $P$ to the diagram \eqref{diag:l1912} with $n \leftarrow (n-1)$, we obtain a new commutative diagram
    \begin{equation}
        \begin{tikzcd}
            & P^{(n)}(W) \arrow[dddl, bend right=40, "{P^{(n)}(g_1)}"'] \arrow[dddr, bend left=40, "{P^{(n)}(g_2)}"] \arrow[d, twoheadrightarrow, "{P(h_{n-1})}"] &  \\
            & P(P^{(n-1)}(X) \times_{P^{(n-1)}(Z)} P^{(n-1)}(Y)) \arrow[ddl, bend right=20] \arrow[ddr, bend left=20] \arrow[d, dashed, twoheadrightarrow, "{\tilde{h}_n}"] & \\
            & P^{(n)}(X) \times_{P^{(n)}(Z)} P^{(n)}(Y) \arrow[dl] \arrow[dr] & \\
            P^{(n)}(X) \arrow[r, "{P^{(n)}(f_1)}"'] & P^{(n)}(Z) & P^{(n)}(Y) \arrow[l, "{P^{(n)}(f_2)}"] 
        \end{tikzcd}
    \end{equation}
    where $\tilde{h}_n$ is again the morphism obtained by applying the universal property of the fiber product. By unicity of the morphism $h_n$, one then has $h_n = \tilde{h}_n \circ P(h_{n-1})$. Then:
    \begin{itemize}
        \item In the case $\cC = \Pol$ and $P = \cP$: since $h_{n-1}$ is surjective by the induction hypothesis, so is $P(h_{n-1}) = [h_{n-1}]$ by \Cref{prop:pushforward_by_surjective_is_surjective_in_P0}. Moreover, $\tilde{h}_n$ is surjective by \Cref{prop:p0_of_fiber_product_onto_fiber_product_of_p0}. Therefore, $h_n$ is surjective as the composition of surjective maps.
        \item In the case $\cC = \Pold$ and $P = \cP_2$: since $h_{n-1}$ is surjective with linearly growing antecedents by the induction hypothesis, so is $P(h_{n-1}) = [h_{n-1}]$ by \Cref{prop:pushforward_by_surjective_is_surjective_in_P2}. Moreover, $\tilde{h}_n$ is surjective and with linearly growing antecedents by \Cref{prop:p_of_fiber_product_onto_fiber_product_of_p} and \Cref{rk:linear_antecedents_p_of_fiber_product_onto_fiber_product_of_p}. Therefore, $h_n$ is surjective and with linearly growing antecedents as the composition of $\tilde{h}_n$ and $[h_{n-1}]$.
    \end{itemize}
    This finishes the proof.
\end{proof}

\begin{remark}
    Let $X \times_Z Y$ be a fiber product in the category $\Pold$, with the projections maps $\pi_1 : X \times_Z Y \mapsto X$ and $\pi_2 : X \times_Z Y \mapsto Y$. Then, for every $n > 0$, $\alpha \in \cP(\cPn{n-1}{X \times_Z Y})$ and $(x_0,y_0) \in X \times_Z Y$, it holds
    \begin{equation} \label{eq:2nd_moment_of_coupling}
        \W_2^2(\delta^{(n)}_{(x_0,y_0)}, \alpha) = \W_2^2(\delta^{(n)}_{x_0},[\pi_1]^{(n)}(\alpha)) + \W_2^2(\delta^{(n)}_{y_0}, [\pi_2]^{(n)}(\alpha).
    \end{equation}
    This is straightforward to check by induction. Indeed, for $n = 1$, this is simply \eqref{eq:2nd_moment_of_transport_plan}, and, for $n > 1$, assuming the equality holds for $n-1$, we have for every $\bA \in \cP(\cPn{n-1}{X \times_Z Y})$
    \begin{align}
        \W_2^2(\delta^{(n)}_{(x_0,y_0)}, \bA) &= \int W_2^2(\delta^{(n-1)}_{(x_0,y_0)}, \alpha) \dd\bA(\alpha) \\
        &= \int \W_2^2(\delta^{(n-1)}_{x_0},[\pi_1]^{(n-1)}(\alpha)) + \W_2^2(\delta^{(n-1)}_{y_0}, [\pi_2]^{(n-1)}(\alpha) \dd\bA(\alpha) \\
        &= \int \W_2^2(\delta^{(n-1)}_{x_0},\gamma_1) \dd[\pi_1]^{(n)}(\bA)(\gamma_1) + \W_2^2(\delta^{(n-1)}_{y_0},\gamma_2) \dd[\pi_2]^{(n)}(\bA)(\gamma_2) \\
        &= \W_2^2(\delta^{(n)}_{x_0},[\pi_1]^{(n)}(\bA)) + \W_2^2(\delta^{(n)}_{y_0}, [\pi_2]^{(n)}(\bA)
    \end{align}
\end{remark}

\begin{proposition} \label{prop:p_of_fiber_to_fiber_of_p_is_proper}
    Let $X \times_Z Y$ be a fiber product in $\cC$, with associated projection maps $\pi_1$, $\pi_2$. Then for every $n \geq 0$, the map $([\pi_1]^{(n)},[\pi_2]^{(n)}) : P^{(n)}(X \times_Z Y) \mapsto P^{(n)}(X) \times_{P^{(n)}(Z)} P^{(n)}(Y)$ is proper (i.e. preimages of compact sets are compact).
\end{proposition}

\begin{proof}
    Note that this is clearly true in the case $n = 0$, as $(\pi_1,\pi_2)$ is the identity map of $X \times_Z Y$. Moreover, since (in both cases $\cC = \Pol$ and $\cC = \Pold$) $X \times_Z Y$ is a closed subset of $X \times Y$, we only need to prove that the map $([\pi_1]^{(n)},[\pi_2]^{(n)}) : P^{(n)}(X \times_Z Y) \mapsto P^{(n)}(X) \times P^{(n)}(Y)$ is proper for every $n \geq 0$. In the case $\cC = \Pol$, this is easily proved by induction by repeatedly applying \Cref{lemma:tightness_criterion} to the maps $[\pi_1]^{(n)}$ and $[\pi_2]^{(n)}$. Thus, we focus on the case $\cC = \Pold$, and we want to prove that the map $([\pi_1]^{(n)},[\pi_2]^{(n)}) : \cPn{n}{X \times_Z Y} \mapsto \cPn{n}{X} \times \cPn{n}{Y}$ is proper for every $n \geq 0$. We prove this by induction. We have seen that the case $n = 0$ is true, so we let $n > 0$ and we assume that $([\pi_1]^{(n-1)},[\pi_2]^{(n-1)})$ has been proved to be proper. \newline
    Let $\cK_1 \subseteq \cPn{n}{X}, \cK_2 \subseteq \cPn{n}{Y}$ be compact (in the $\W_2$ topology). By \Cref{prop:rel_compact_sets_in_w2_topology}, they are tight. Moreover, since $([\pi_1]^{(n-1)},[\pi_2]^{(n-1)})$ is proper, by \Cref{lemma:tightness_criterion}, the map $([\pi_1]^{(n)},[\pi_2]^{(n)}) : \cP(\cPn{n-1}{X \times_Z Y}) \mapsto \cP(\cPn{n-1}{X}) \times \cP(\cPn{n-1}{Y})$ is proper for the weak topology, and
    \begin{equation}
        \cL = ([\pi_1]^{(n)})^{-1}(\cK_1) \cap ([\pi_2]^{(n)})^{-1}(\cK_2) \cap \cPn{n}{X \times_Z Y}
    \end{equation}
    is tight (and $\W_2$-closed). Now, let $(\alpha^k)_k$ be a sequence in $\cL$, and let $\gamma_1^k := [\pi_1]^{(n)}(\alpha^k)$ and $\gamma_2^k := [\pi_2]^{(n)}(\alpha^k)$. Since $\cL$ is tight, and $\cK_1$ and $\cK_2$ are $\W_2$-compact, up to extracting subsequences, there exists $\alpha \in \cP(\cPn{n-1}{X \times_Z Y})$ and $\gamma_1 \in \cK_1$, $\gamma_2 \in \cK_2$ such that $\alpha^k \rightharpoonup \alpha$ and $\gamma_1^k \to \gamma_1$, $\gamma_2^k \to \gamma_2$ (in particular, $[\pi_1]^{(n)}(\alpha) = \gamma_1$ and $[\pi_2]^{(n)}(\alpha) = \gamma_2$). We want to prove that $\alpha^k \to \alpha$ in the $\W_2$ topology. By point \ref{enum:w2_convergence:1_2nd_moment_cvg} of \Cref{th:convergence_in_w2_space}, for this, we only need to prove that $\W_2^2(\delta^{(n)}_{(x_0,y_0)},\alpha^k) \xrightarrow[k \to +\infty]{} \W_2^2(\delta^{(n)}_{(x_0,y_0)},\alpha)$, where $(x_0,y_0) \in X \times_Z Y$ is arbitrary. But this is the case: indeed, we have, using \eqref{eq:2nd_moment_of_coupling}, that for every $k$,
    \begin{align}
        \W_2^2(\delta^{(n)}_{(x_0,y_0)},\alpha^k) &= \W_2^2(\delta^{(n)}_{x_0}, \gamma_1^k) + \W_2^2(\delta^{(n)}_{y_0}, \gamma_2^k) \\
        &\xrightarrow[k \to +\infty]{} \W_2^2(\delta^{(n)}_{x_0}, \gamma_1) + \W_2^2(\delta^{(n)}_{y_0}, \gamma_2) = \W_2^2(\delta^{(n)}_{(x_0,y_0)},\alpha)
    \end{align}
    by the $\W_2$-convergence of $\gamma_1^k$ and $\gamma_2^k$. Therefore, $\alpha^k$ converges to $\alpha \in \cL$ in the $\W_2$ topology. In particular, $\cL$ is $\W_2$-compact, and the map $([\pi_1]^{(n)},[\pi_2]^{(n)}) : \cPn{n}{T^2\cM} \mapsto \cPn{n}{T\cM} \times \cPn{n}{T\cM}$ is proper. This finishes the proof.
\end{proof}

\subsection{The bundle \texorpdfstring{$T^2\cM$}{T²M}} 

In the following, as in \citep{gigli2011inverse}, we mean by $T^2 \cM$ the metric space defined as
\begin{equation}
    T^2 \cM := \{(x,v_1,v_2), x \in \cM, v_1,v_2 \in T_x \cM \}
\end{equation}
equipped with the metric $d$ such that for every $(x,v_1,v_2), (x',v'_1,v'_2) \in T^2 \cM$,
\begin{equation}
    d^2((x,v_1,v_2),(x',v'_1,v'_2)) := d^2((x,v_1),(x',v'_1)) + d^2((x,v_2),(x',v_2)).
\end{equation}
In particular, this space must not be confused with the second-order tangent bundle $T(T\cM)$, which is also denoted $T^2 \cM$ is the differential geometry litterature.
\medbreak
The space $T^2 \cM$ with its metric identifies as the fiber product $T\cM \times_\cM T\cM$ of two copies of $T\cM$, with the projection $\pi$ as the morphisms $T\cM \mapsto \cM$\footnote{Note that we have previously defined the fiber product $X \times_Z Y$ of metric spaces as endowed with the metric $d = d_X + d_Y$, while here we are instead endowing it with the metric $d' = \sqrt{d_X^2 + d_Y^2}$. However, these two metrics are bi-Lipschitz equivalent, so they define the same object up to isomorphism in the category $\Pol$.}. In particular, it is a Polish space, and the projections $\pi_1^{T^2\cM}, \pi_2^{T^2\cM} : T^2 \cM \mapsto T\cM$ on the coordinates are continuous maps with linear growth. We may endow $T^2\cM$ with the structure of a differential manifold, given by the following atlas: for every local chart $\varphi : U \subseteq \cM \mapsto U' \subset \R^d$, we consider the local chart of $T^2\cM$ given by
\begin{equation}
    \tilde{\varphi} : \left\{\begin{array}{ccc}
        T^2U = \{(x,v_1,v_2), x \in U, v_1,v_2 \in T_x \cM \} & \rightarrow & U' \times \R^d \times \R^d  \\
        (x, v_1, v_2) & \rightarrow & (\varphi(x), D_x\varphi(v_1), D_x\varphi(v_2))
    \end{array}\right.
\end{equation}
One can check that these charts define an atlas for $T^2\cM$, for which the projections $\pi_1^{T^2\cM}$ and $\pi_2^{T^2\cM}$ are smooth. The topology given by the distance $d$ and the differential manifold structure coincide: indeed, in both topologies, $(x_n,v_n,w_n) \xrightarrow[n \to +\infty]{} (x,v,w)$ if and only if $(x_n,v_n) \xrightarrow[n \to +\infty]{} (x,v)$ and $(x_n,w_n) \xrightarrow[n \to +\infty]{} (x,w)$ in $T\cM$. Note, however, that there is a priori no Riemannian metric on $T^2\cM$ for which $d$ is the geodesic distance.
\medbreak
We may define the following \emph{parallel transport operator} on $T^2\cM$: let $(x,v,w) \in T^2\cM$ and $t \in \R$, then let $\PT_t(x,v,w)$ be the element of $T_{\exp_x(tv)} \cM$ obtained by applying to $w$ the parallel transport operator from $x$ to $\exp_x(tv)$ along the geodesic $s \to \exp_x(sv)$. Since $\cM$ is complete, any geodesic can be extended to a geodesic defined on all $\R$, so this operator is defined for any value of $t \in \R$ and $(x,v,w) \in T^2\cM$. Moreover, it is smooth, in the following sense:

\begin{proposition}
    The application $\PT : (t,x,v,w) \in \R \times T^2\cM \mapsto (\exp_x(tv),\PT_t(x,v,w)) \in T\cM$ is smooth, and has linear growth.
\end{proposition}

\begin{proof}
    Let $X$ be the vector field on $T^2\cM$ defined for every $(x,v,w) \in T^2\cM$ by $X(x,v,w) := (U,V,W) \in T(T^2\cM)$ where, in local coordinates, 
    \begin{align}
        U^k &= v^k, & k = 1,\ldots,d \\
        V^k &= -v^i v^j \Gamma_{i,j}^k(x), & k = 1,\ldots,d \\
        W^k &= -v^i w^j \Gamma_{i,j}^k(x), & k = 1,\ldots,d
    \end{align}
    (where $d$ is the dimension of $\cM$, and the $\Gamma_{i,j}^k$ are the Christoffel symbols corresponding to the Levi-Civita connection of $\cM$). Now, an integral curve of $X$ is a map $c(t) = (x(t),v(t),w(t))$ defined on an interval $I \subseteq \R$ such that $\dot{c(t)} = X(c(t))$, that is,
    \begin{align}
        \dot{x}(t) &= v^k(t) & k = 1,\ldots,d \label{eq:l_1601} \\
        \dot{v}^k(t) &= -v^i(t) v^j(t) \Gamma_{i,j}^k(x(t)) & k = 1,\ldots,d \label{eq:l_1602} \\
        \dot{w}^k(t) &= -v^i(t) w^j(t) \Gamma_{i,j}^k(x(t)) & k = 1,\ldots,d. \label{eq:l_1603}
    \end{align}
    Notice that the \eqref{eq:l_1601} and \eqref{eq:l_1602} mean that $x(t)$ is a geodesic of $\cM$ with $v(t) = \dot{x}(t)$, and that \eqref{eq:l_1603} means that $w(t)$ is parallel along the curve $x(t)$. In particular, $\cM$ is complete, the integral curves for all initial data are defined on the entire $\R$. Therefore, by the fundamental theorem on flows \citep[Theorem 9.12]{lee2012introductionsmooth}, there exists a smooth map $\Phi : \R \times T^2\cM \mapsto T^2\cM$ such that for every $(x_0,v_0,w_0) \in T^2\cM$, the map $c(t) := \Phi(t,x_0,v_0,w_0)$ is the unique integral curve of $X$ such that $c(t=0) = (x_0,v_0,w_0)$. Now, notice that $\PT = \pi_2^{T^2\cM} \circ \Phi$. This proves that $\PT$ is smooth.
\end{proof}

\medbreak
Similarly, for any $n > 0$ we define $T^n \cM$ as the fiber product $T\cM \times_\cM \ldots \times_\cM T\cM$ of $n$ copies of $T\cM$, which identifies as
\begin{equation}
    T^n \cM := \{(x,v_1,\ldots,v_n), x \in \cM, v_1,\ldots,v_n \in T_x \cM \}
\end{equation}
with the metric
\begin{equation}
    d^2((x,v_1,\ldots,v_n),(x',v'_1,\ldots,v'_n)) := \sum_{i=1}^n d^2((x,v_i),(x',v'_i)).
\end{equation}
For every subset of coordinates $\{i_1,\ldots,i_k\} \subseteq \{1,\ldots,n\}$, we have a map $\pi^{T^n \cM}_{i_1,\ldots,i_k} : T^n \cM \mapsto T^k \cM$ of projection onto these coordinates. As in the case of $T^2\cM$, we can endow $T^n \cM$ with a differential manifold structure using local charts of $\cM$, whose topology is the same as the one defined by the distance $d$, and for which the projection maps $T^n \cM \mapsto T^k \cM$ are smooth. For any $i \in \{1,\ldots,n\}$, we similarly have a parallel transport operator $\PT^n_i$ which is a smooth map $\R \times T^n \cM \mapsto T^n \cM$ defined for every $t \in \R$ and $(x,v_1,\ldots,v_n) \in T^n \cM$ by 
\begin{equation} \label{eq:pt_operator_on_n_uples}
    \PT^n_{i,t}(x,v_1,\ldots,v_n) := (\exp_x(tv_i), \PT_t(x,v_i,v_1), \ldots, \PT_t(x,v_i,v_n)).
\end{equation}

\section{The Wasserstein hierarchy} \label{sec:3_wasserstein_hierarchy}

\subsection{Hierarchical Wasserstein spaces, velocity plans, and couplings}

As explained in \Cref{sec:1_introduction}, using the functor $\cP_2 : \Pold \mapsto \Pold$, we can elegantly recover the variational structure of the space $\cP_2(\cM)$. Fix $n \geq 0$, we then introduce the following definitions:

\begin{itemize}
    \item We call $\cPn{n}{\cM}$ the \emph{$n$-th hierarchical Wasserstein space over $X$}.
    \item Given a measure $\mu \in \cPn{n}{\cM}$, the set of velocity plans with base $\mu$ is defined by
    \begin{equation}
        \cPn{n}{T\cM}_{\mu} := ([\pi]^{(n)})^{-1}(\mu) = \{\gamma \in \cPn{n}{T\cM} \setcond [\pi]^{(n)}(\gamma) = \mu \}
    \end{equation}
    \item Given two measures $\mu, \nu \in \cP_2^{(n)}(\cM)$, the set of velocity plans from $\mu$ to $\nu$ is defined by
    \begin{equation}
        \Gamma(\mu,\nu) = ([\pi]^{(n)})^{-1}(\mu) \cap ([\exp]^{(n)})^{-1}(\nu) = \{\gamma \in \cPn{n}{T\cM}_\mu \setcond [\exp]^{(n)}(\gamma) = \nu \}
    \end{equation}
    \item We define a ``norm" $\|\cdot\|$ on the space $\cPn{n}{T\cM}$ by
    \begin{equation}
        \|\gamma\|^2 := \bE^{(n)}_\gamma[\|v\|^2_x], \quad \gamma \in \cPn{n}{T\cM}
    \end{equation}
    By \Cref{lemma:n_expectancy_regularity}, $\gamma \to \|\gamma\|^2$ is continuous on $\cPn{n}{T\cM}$.
    \item Given two measures $\mu, \nu \in \cPn{n}{\cM}$, the set of \emph{optimal} velocity plans from $\mu$ to $\nu$ is defined by
    \begin{equation}
        \Gamma_o(\mu,\nu) := \{\gamma \in \Gamma(\mu,\nu) \setcond \|\gamma\| = \W_2(\mu,\nu)\}
    \end{equation}
    We further define $\Gamma_o^{(n)}(\cM)$ to be the set of velocity plans that are optimal between their marginals:
    \begin{equation}
        \Gamma_o^{(n)}(\cM) := \{\gamma \in \cPn{n}{T\cM} \setcond \gamma \in \Gamma_o([\pi]^{(n)}(\gamma), [\exp]^{(n)}(\gamma)) \}
    \end{equation}
    \item Given $\mu \in \cPn{n}{\cM}$ and $\gamma_1, \gamma_2 \in \cPn{n}{T\cM}_\mu$, we define the set of couplings between $\gamma_1$ and $\gamma_2$ by
    \begin{equation}
        \begin{split}
            \Gamma_\mu(\gamma_1,\gamma_2) &:= ([\pi_1]^{(n)})^{-1}(\gamma_1) \cap ([\pi_2]^{(n)})^{-1}(\gamma_2) \\
            &= \{\alpha \in \cPn{n}{T^2\cM} \setcond [\pi_1]^{(n)}(\alpha) = \gamma_1, [\pi_2]^{(n)}(\alpha) = \gamma_2 \}
        \end{split}
    \end{equation}
    \item The couplings give an ``almost linear" structure on $\cPn{n}{T\cM}_\mu$. Given $\mu \in \cPn{n}{\cM}$, $\gamma_1, \gamma_2 \in \cPn{n}{T\cM}_\mu$, and $\tau \in \R$, we define:
    \begin{enumerate}
        \item A ``local" zero: $\bm{0}_\mu := [\iota_0]^{(n)}(\mu)$, where $\iota_0 : \cM \mapsto T\cM$ is the map defined by $\iota_0(x) = (x,0)$.
        \item A scalar multiplication: $\tau \cdot \gamma_1 := [m_\tau]^{(n)}(\gamma_1)$, where $m_\tau : T\cM \mapsto T\cM$ is the map defined by $m_\tau(x,v) = (x,\tau v)$.
        \item An addition and a subtraction: for $\alpha \in \Gamma_\mu(\gamma_1,\gamma_2)$, $\gamma_1 +_\alpha \gamma_2 := [\pi_1 + \pi_2]^{(n)}(\alpha)$ and $\gamma_1 -_\alpha \gamma_2 := [\pi_1 - \pi_2]^{(n)}(\alpha)$.
    \end{enumerate}
    We can similarly define a sum of $p$ elements $\gamma_1 +_\alpha \ldots +_\alpha \gamma_p$ with $\alpha \in \Gamma_\mu(\gamma_1,\ldots,\gamma_p)$ (we have not defined this set but it this straightforward from what precedes).
    \item Given a measure $\alpha \in \cPn{n}{T^m \cM}$, we can also define a component-wise scalar multiplication: for every $(\tau_1,\ldots,\tau_m) \in \R^m$, we let $(\tau_1,\ldots,\tau_m) \cdot \alpha := [m_{\tau_1,\ldots,\tau_m}]^{(n)}(\alpha)$, where $m_{\tau_1,\ldots,\tau_m} : T^m \cM \mapsto T^m \cM$ is the map defined by $m_{\tau_1,\ldots,\tau_m}(x,v_1,\ldots,v_m) = (x,\tau_1 v_1,\ldots,\tau_m v_m)$.
    \item For $n > 0$, the \emph{tangent space} of $\cPn{n}{\cM}$ at $\mu \in \cPn{n}{\cM}$ is defined by
    \begin{equation}
        T_\mu \cPn{n}{\cM} := \{\gamma \in \cPn{n}{T\cM}_\mu \setcond \exists \veps_0 > 0, \forall \veps \in (-\veps_0,\veps_0), \veps \cdot \gamma \in \Gamma_o^{(n)}(\cM) \}
    \end{equation}
    We also define the tangent bundle $T\cPn{n}{\cM}$ to be the union of all the $T_\mu \cPn{n}{\cM}$, $\mu \in \cPn{n}{\cM}$.
\end{itemize}

\begin{remark} \label{rk:wasserstein_zero_definitions}
    These are also well-defined for $n = 0$. If $x \in \cPn{0}{\cM} = \cM$, then $\cPn{0}{T\cM} = T\cM$ and $\cPn{0}{T\cM}_x = T_x \cM$. If $x,y \in \cM$, then $\Gamma(x,y) = \{v \in T_x \cM \setcond \exp_x(v) = y\} = \log_x(y)$, and $\Gamma_o(x,y)$ is the set of $v \in \log_x(y)$ such that $d(x,y) = \|v\|_x$. The norm $\|(x,v)\|$ is just $\|v\|_x$. A coupling in $\Gamma_x((x,v_1),(x,v_2))$ is just the pair $(v_1,v_2) \in T_x^2 \cM$. The ``linear structure" on $T_x \cM$ is just its usual vector space structure.
\end{remark}

\subsection{Optimal velocity plans} 

We now investigate the properties of optimal velocity plans, most notably their existence and their relation to optimal transport plans.

\begin{lemma} \label{lemma:pseudo_norm_is_finite}
    Let $n \geq 0$. Then for every $\gamma \in \cPn{n}{\cM}$, $\|\gamma\|^2 = \W_2^2(\gamma,\delta^{(n)}_{(o,0)}) - \W_2^2([\pi]^{(n)}(\gamma), \delta^{(n)}_o)$, where $o \in \cM$ is any fixed base point.
\end{lemma}

\begin{proof}
    For every $\gamma \in \cPn{n}{T\cM}$, letting $\mu := [\pi]^{(n)}(\gamma)$, we have
    \begin{align}
        \bE^{(n)}_\gamma[\|v\|^2_x] &= \bE^{(n)}_\gamma[d^2((o,0),(x,v)) - d^2(o,x)] = \bE^{(n)}_\gamma[d^2((o,0),(x,v))] - \bE^{(n)}_\mu[d^2(o,x)] \\
        &= \W_2^2(\delta_{(o,0)}, \cE_{T\cM}^{(n)}(\gamma)) - \W_2^2(\delta_o,\cE_{\cM}^{(n)}(\mu)) = \W_2^2(\delta_{(o,0)}^{(n)}, \gamma) - \W_2^2(\delta_o^{(n)}, \mu)
    \end{align}
    where we used Equations \eqref{eq:geodesic_distance_sasaki_2}, \eqref{eq:n_expectancy_of_pushforward} and \eqref{eq:total_collapse_preserve_2nd_moment}. This finishes the proof.
\end{proof}

\begin{remark}
    For any $o \in \cM$, we have $\delta^{(n)}_{(o,0)} = \bm{0}_{\delta^{(n)}_o}$. Indeed, by naturality of $\delta : \Id_{\Pold} \mapsto \cP_2$, we have
    \begin{equation}
        \bm{0}_{\delta^{(n)}_o} = [\iota_0]^{(n)}(\delta^{(n)}_o) =  [\iota_0]^{(n)} \circ \delta_{\cM}^{(n)}(o) = \delta_{T\cM}^{(n)} \circ \iota_0(o) = \delta_{T\cM}^{(n)}(o,0) = \delta_{(o,0)}^{(n)}.
    \end{equation}
\end{remark}

\begin{lemma} \label{lemma:pseudo_norm_bounds_dist_to_zero}
    Let $n \geq 0$. For every $\mu \in \cPn{n}{\cM}$ and $\gamma \in \cPn{n}{T\cM}_\mu$, we have $\W_2(\bm{0}_\mu,\gamma) \leq \|\gamma\|$.
\end{lemma}

\begin{proof}
    We prove this by induction. For $n = 0$, this amounts to showing that for every $(x,v) \in T\cM$, it holds $d^2((x,0),(x,v)) \leq \|v\|^2$. This is the case by \eqref{eq:geodesic_distance_sasaki_2} (and is in fact we even have equality). Let $n > 0$ and assume that the result holds for $n-1$. Let $\bP \in \cPn{n}{\cM}$ and $\bGamma \in \cPn{n}{T\cM}_\bP$. Since $\bm{0}_\bP = [\iota_0]^{(n)}(\bP) = [\iota_0 \circ \pi]^{(n)}(\bGamma)$, $([\iota_0 \circ \pi]^{(n-1)},\id)_\#\bGamma$ is a transport plan between $\bm{0}_\bP$ and $\bGamma$ (which is not necessarily optimal), so that 
    \begin{align}
        \W_2^2(\bm{0}_\bP,\bGamma) &\leq \int \W_2^2([\iota_0 \circ \pi]^{(n-1)}(\gamma),\gamma)\dd\bGamma(\gamma) = \int \W_2^2(\bm{0}_{[\pi]^{(n-1)}(\gamma)},\gamma) \dd\bGamma(\gamma) \\
        &\leq \int \|\gamma\|^2\dd\bGamma(\gamma) = \|\bGamma\|^2
    \end{align}
    where we used the induction hypothesis to obtain the second line. This finishes the proof.
\end{proof}

\begin{proposition} \label{prop:pseudo_norm_larger_than_w2}
    Let $n \geq 0$. For every $\gamma \in \cPn{n}{T\cM}$ with marginals $\mu = [\pi]^{(n)}(\gamma)$ and $\nu = [\exp]^{(n)}(\gamma)$, we have $\|\gamma\| \geq \W_2(\mu, \nu)$.
\end{proposition}

\begin{proof}
    We show this by induction. For $n = 0$, this is simply the fact that for every $(x,v) \in T\cM$, $\|v\|_x \geq d(x,\exp_x(v))$. Let $n > 0$ and assume that the property holds for $n-1$. Then for every $\bGamma \in \cPn{n}{T\cM}$ with $\bP = [\pi]^{(n)}(\bGamma)$ and $\bQ = [\exp]^{(n)}(\bGamma)$, since $([\pi]^{(n-1)},[\exp]^{(n-1)})_\#\bGamma$ is a transport plan between $\bP$ and $\bQ$,
    \begin{equation}
        \W_2^2(\bP,\bQ) \leq \int \W_2^2([\pi]^{(n-1)}(\gamma), [\exp]^{(n-1)}(\gamma)) \dd\bGamma(\gamma) \leq \int \|\gamma\|^2 \dd\bGamma(\gamma) = \|\bGamma\|^2
    \end{equation}
    This finishes the proof.
\end{proof}

\begin{proposition} \label{prop:opt_vel_plan_gives_opt_trans_plan}
    Let $n > 0$ and $\mu, \nu \in \cPn{n}{\cM}$, $\gamma \in \Gamma(\mu,\nu)$. Then the measure $\tilde{\gamma} = ([\pi]^{(n-1)}, [\exp]^{(n-1)})_\#\gamma$ is in $\Pi(\mu,\nu)$. Moreover, $\gamma \in \Gamma_o(\mu,\nu)$ if and only if $\tilde{\gamma} \in \Pi_o(\mu,\nu)$ and $\gamma$ is concentrated on $\Gamma_o^{(n-1)}(\cM)$.
\end{proposition}

\begin{proof}
    Let $\bP, \bQ \in \cPn{n}{\cM}$, and $\bGamma \in \Gamma(\bP,\bQ)$. It is then a direct consequence of the definition of $\Gamma(\bP,\bQ)$ that $\tilde{\bGamma} = ([\pi]^{(n-1)}, [\exp]^{(n-1)})_\#\bGamma$ is in $\Pi(\bP,\bQ)$. Furthermore we have
    \begin{equation}
        \W_2^2(\bP,\bQ) \leq \int \W_2^2(\mu,\nu) \dd\tilde{\bGamma}(\mu,\nu) = \int \W_2^2([\pi]^{(n-1)}(\gamma), [\pi]^{(n-1)}(\gamma)) \dd\bGamma(\gamma) \leq \int \|\gamma\|^2 \dd\bGamma(\gamma) = \|\bGamma\|^2 \label{eq:l1444}
    \end{equation}
    where we used \Cref{prop:pseudo_norm_larger_than_w2} to obtain the second inequality. Now, notice that:
    \begin{itemize}
        \item The inequalities in \eqref{eq:l1444} are all equalities if and only if $\W_2^2(\bP,\bQ) = \|\bGamma\|^2$, that is $\bGamma \in \Gamma_o(\bP,\bQ)$.
        \item The first inequality in \eqref{eq:l1444} holds if and only if $\W_2^2(\bP, \bQ) = \int \W_2^2(\mu,\nu) \dd\tilde{\bGamma}(\mu,\nu)$, that is $\tilde{\bGamma} \in \Pi_o(\bP, \bQ)$.
        \item The second inequality in \eqref{eq:l1444} holds if and only if for $\bGamma$-a.e. $\gamma$, $\|\gamma\|^2 = \W_2^2(\mu,\nu)$ with $\mu = [\pi]^{(n-1)}(\gamma)$ and $\nu = [\exp]^{(n-1)}(\gamma)$, that is, if and only if $\bGamma$ is concentrated on $\Gamma_o^{(n-1)}(\cM)$.
    \end{itemize}
    From these points, we conclude that $\bGamma \in \Gamma_o(\bP,\bQ)$ if and only if $\tilde{\bGamma} \in \Pi_o(\bP,\bQ)$ and $\bGamma$ is concentrated on $\Gamma_o^{(n-1)}(\cM)$. 
\end{proof}

\begin{proposition} \label{prop:opt_vel_plans_are_closed}
    The set $\Gamma_o^{(n)}(\cM)$ is closed in $\cPn{n}{T\cM}$. In particular it is a Polish space.
\end{proposition}

\begin{proof}
    Indeed, $\Gamma_o^{(n)}(\cM)$ is by definition the zero set of the continuous function $f : \cPn{n}{T\cM} \mapsto \R$ defined by $f(\gamma) := \|\gamma\|^2 - \W_2^2([\pi]^{(n)}(\gamma), [\exp]^{(n)}(\gamma))$. It is thus closed.
\end{proof}

\begin{remark} \label{rk:exists_opt_trans_plans_selection}
    Let $(X,d)$ be a Polish space. Denote $\Pi_o(X)$ the set of all optimal transport plans, that is
    \begin{equation}
        \Pi_o(X) := \bigcup_{\mu, \nu \in \cP_2(X)} \Pi_o(\mu,\nu)
    \end{equation}
    It is a direct consequence of \citep[Theorem 5.20]{villani2009optimal} that $\Pi_o(X)$ is closed in $\cP_2(X \times X)$ (both for the topology of weak convergence and of $\W_2$). Furthermore, by \citep[Corollary 5.21, Corollary 5.22]{villani2009optimal}, each $\Pi_o(\mu,\nu)$ is weakly compact, and the map $\Pi_o(X) \mapsto \cP_2(X) \times \cP_2(X)$ projecting each plan to its two marginals admits a measurable right inverse (that is a measurable map $s : \cP_2(X) \times \cP_2(X) \mapsto \Pi_o(X)$ such that for every $\mu, \nu \in \cP_2(X)$, $s(\mu,\nu) \in \Pi_o(\mu,\nu)$). Furthermore, this maps admits linearly growing antecedents. Indeed, if $\mu, \nu \in \cP_2(X)$, $\gamma \in \Pi_o(\mu,\nu)$, one has for any $x_0 \in X$,
    \begin{align}
        \W_2^2(\delta_{(x_0,x_0)},\gamma) &= \int d^2((x_0,x_0),(x,y))\dd\gamma(x,y) = \int d^2(x_0,x) + d^2(x_0,y) \dd\gamma(x,y) \\
        &= \W_2^2(\delta_{x_0},\mu) + \W_2^2(\delta_{x_0},\nu).
    \end{align}
    This last equation, in conjunction with \Cref{th:convergence_in_w2_space}, also implies that $\Pi_o(\mu,\nu)$ is compact in the $\W_2$ topology, as all the elements $\gamma \in \Pi_o(\mu,\nu)$ share the same second order moment with respect to $\delta_{(x_0,x_0)}$.
\end{remark}

Let $F$ be a closed subset of $\cM$, we define $\Gamma_o^{(0)}(F) \subseteq T\cM$ by 
\begin{equation}
    \begin{split}
        \Gamma_o^{(0)}(F) &:= \{(x,v) \in T\cM \setcond x \in F, \exp_x(v) \in F, \|v\|_x = d(x,\exp_x(v)) \} \\
        &= \pi^{-1}(K) \cap \exp^{-1}(K) \cap \Gamma_o^{(0)}(\cM)
    \end{split}
\end{equation}
It is not difficult to see that $\Gamma_o^{(0)}(F)$ is closed, and that if $F$ is compact, it is a subset of $\{(x,v) \in T\cM \setcond \|v\|_x \leq \mathrm{diam}(F)\}$, and thus is compact. We then have the following results:

\begin{lemma} \label{lemma:compact_base_implies_opt_plans_compact_base}
    Let $n \geq 0$ and $F \subseteq \cM$ closed. Then for every $\mu, \nu \in \cPn{n}{F}$, we have $\Gamma_o(\mu,\nu) \subseteq \cPn{n}{\Gamma_o^{(0)}(F)}$.
\end{lemma}

\begin{proof}
    We prove this by induction. For $n = 0$, this is just a consequence of the definition of $\Gamma_o^{(0)}(F)$. For $n > 0$, assuming the proposition shown for $n-1$, letting $\bP, \bQ \in \cPn{n}{F}$ and $\bGamma \in \Gamma_o(\bP,\bQ)$: by \Cref{prop:opt_vel_plan_gives_opt_trans_plan}, $\bGamma$ is concentrated on $\Gamma_o^{(n-1)}(\cM)$. However, any velocity plan $\gamma \in \Gamma_o^{(n-1)}(\cM)$ whose projections $\mu = [\pi]^{(n-1)}(\gamma)$ and $\nu = [\exp]^{(n-1)}(\gamma)$ are in $\cPn{n-1}{F}$ satisfies by induction $\gamma \in \cPn{n-1}{\Gamma_o^{(0)}(F)}$. Thus, since $\bP, \bQ$ are in $\cPn{n}{F}$, $\bGamma$ is concentrated on $\cPn{n-1}{\Gamma_o^{(0)}(F)}$, and this finishes the proof.
\end{proof}

We will note $\Gamma_o^{(n)}(F) := \Gamma_o^{(n)}(\cM) \cap \cPn{n}{\Gamma_o^{(0)}(F)}$. If $F$ is compact, then this is a compact set, as $\cPn{n}{\Gamma_o^{(0)}(F)}$ is compact and $\Gamma_o^{(n)}(\cM)$ is closed. Note that our definition is consistent: if $F = \cM$, then the $\Gamma_o^{(n)}(F)$ that we just defined is equal to the $\Gamma_o^{(n)}(\cM)$ that we previously defined.

\begin{lemma} \label{lemma:exists_meas_log}
    There exists a measurable map $\log : \cM \times \cM \mapsto T\cM$ such that, for every $x,y \in \cM$, $\log_x(y) = (x,v) \in T_x\cM$, with $\exp_x(v) = y$ and $d(x,y) = \|v\|_x$.
\end{lemma}

\begin{proof}
    This is essentially a reformulation of \citep[Lemma C.2]{bonet2025flowing}.
\end{proof}

\begin{proposition} \label{prop:opt_vel_plans_exist}
    For every $n \geq 0$, the projection map $([\pi]^{(n)}, [\exp]^{(n)}) : \Gamma_o^{(n)}(\cM) \mapsto \cPn{n}{\cM} \times \cPn{n}{\cM}$ is surjective (and with linearly growing antecedents). Furthermore, if $n > 0$, then the map $[[\pi]^{(n-1)}, [\exp]^{(n-1)}] : \Gamma_o^{(n)}(\cM) \mapsto \Pi_o(\cPn{n}{\cM})$ is surjective (and with linearly growing antecedents). In particular, for any $\mu, \nu \in \cPn{n}{\cM}$, $\Gamma_o(\mu,\nu)$ is not empty.  
\end{proposition}

\begin{proof}
    We prove this by induction. For $n = 0$, this is simply because the map $(\pi,\exp) : \Gamma_o^{(n)}(\cM) \mapsto \cM \times \cM$ is surjective and with linearly growing antecedents\footnote{Indeed, if $(x,v) \in \Gamma_o^{(0)}(\cM)$, then $d^2((o,0),(x,v)) = d^2(o,x) + \|v\|^2 = d^2(o,x) + d^2(x,\exp_x(v)) \leq C(d^2(o,x) + d^2(o,\exp_x(v)))$.}. Now, let $n > 0$, and assume that the result holds for $n-1$. Applying the functor $[\cdot]$ to the map $([\pi]^{(n-1)}, [\exp]^{(n-1)}) : \Gamma_o^{(n-1)}(\cM) \mapsto \cPn{n-1}{\cM} \times \cPn{n-1}{\cM}$ yields by \Cref{prop:pushforward_by_surjective_is_surjective_in_P2} a surjective map $\cP_2(\Gamma_o^{(n-1)}(\cM)) \mapsto \cP_2(\cPn{n-1}{\cM} \times \cPn{n-1}{\cM})$ with linearly growing antecedents. However, \Cref{prop:opt_vel_plan_gives_opt_trans_plan} implies that the preimage of $\Pi_o(\cPn{n}{\cM})$ by this map is precisely $\Gamma_o^{(n)}(\cM)$. Thus the map $[[\pi]^{(n-1)}, [\exp]^{(n-1)}] : \Gamma_o^{(n)}(\cM) \mapsto \Pi_o(\cPn{n}{\cM})$ is surjective and with linearly growing antecedents. Composing it with the projection map $\Pi_o(\cPn{n}{\cM}) \mapsto \cPn{n}{\cM} \times \cPn{n}{\cM}$ which by \Cref{rk:exists_opt_trans_plans_selection} is also surjective and with linearly growing antecedents, we conclude that the resulting map $([\pi]^{(n)}, [\exp]^{(n)}) : \Gamma_o^{(n)}(\cM) \mapsto \cPn{n}{\cM} \times \cPn{n}{\cM}$ is also surjective and with linearly growing antecedents.
\end{proof}

\begin{remark} \label{rk:opt_vel_plans_exist_restricted}
    In fact, if $F$ is a closed subset of $\Gamma_o^{(0)}(\cM)$ such that $(\pi,\exp)(F) = \cM \times \cM$ (that is for every pair $x,y \in \cM$ there exists $v \in T_x\cM$ such that $(x,v) \in F$, $\exp_x(v) = y$ and $\|v\|_x = d(x,y)$), then a simple variation of the proof of \Cref{prop:opt_vel_plans_exist} shows that the projection maps $([\pi]^{(n)},[\exp]^{(n)}) : \Gamma_o^{(n)}(\cM) \cap \cPn{n}{F} \mapsto \cPn{n}{\cM} \times \cPn{n}{\cM}$ and $[[\pi]^{(n-1)},[\exp]^{(n-1)}] : \Gamma_o^{(n)}(\cM) \cap \cPn{n}{F} \mapsto \Pi_o(\cPn{n}{\cM})$ are both surjective. In other words, given any $\mu,\nu \in \cPn{n}{\cM}$, any optimal transport plan between them is induced by an optimal velocity plan whose base support is in $F$ (i.e. for which the mass is moved exclusively through geodesics described in $F$). \newline
    Similarly and more generally, if $F$ is a closed subset of $\Gamma_o^{(k)}(\cM)$ such that $([\pi]^{(k)},[\exp]^{(k)})(F) = \cPn{k}{\cM} \times \cPn{k}{\cM}$, then the projection maps $([\pi]^{(n)},[\exp]^{(n)}) : \Gamma_o^{(n)}(\cM) \cap \cPn{n-k}{F} \mapsto \cPn{n}{\cM} \times \cPn{n}{\cM}$ and $[[\pi]^{(n-1)},[\exp]^{(n-1)}] : \Gamma_o^{(n)}(\cM) \cap \cPn{n-k}{F} \mapsto \Pi_o(\cPn{n}{\cM})$ are both surjective.
\end{remark}

\begin{proposition} \label{prop:opt_vel_plan_proj_map_is_proper}
    For every $n \geq 0$, the projection map $([\pi]^{(n)}, [\exp]^{(n)}) : \Gamma_o^{(n)}(\cM) \mapsto \cPn{n}{\cM} \times \cPn{n}{\cM}$ is proper. In particular, for every $\mu, \nu \in \cPn{n}{\cM}$, $\Gamma_o(\mu,\nu)$ is compact.
\end{proposition}

\begin{proof}
    We prove this by induction. For the case $n = 0$, we want to prove that the map $(\pi,\exp) : \Gamma_o^{(0)}(\cM) \mapsto \cM \times \cM$ is proper. Let $K_1, K_2 \subseteq \cM$ compact, and let $L = \pi^{-1}(K) \cap \exp^{-1}(L) \cap \Gamma_o^{(0)}(\cM)$. It is closed, and it is contained in the bounded set $\{(x,v) \in T\cM \setcond x \in K, \|v\|_x \leq R\}$ where $R := \max_{(x,y) \in K \times L} d(x,y)$, so it is compact. Thus the map $(\pi,\exp) : \Gamma_o^{(0)}(\cM) \mapsto \cM \times \cM$ is proper. Now, let $n > 0$, and assume that the map $([\pi]^{(n-1)}, [\exp]^{(n-1)}) : \Gamma_o^{(n-1)}(\cM) \mapsto \cPn{n-1}{\cM} \times \cPn{n-1}{\cM}$ is proper. Then, by \Cref{lemma:tightness_criterion}, the continuous\footnote{The pushforward operator by a continuous function is continuous for the topology of weak convergence.} map $([\pi]^{(n)}, [\exp]^{(n)}) : \cP(\Gamma_o^{(n-1)}(\cM)) \mapsto \cP(\cPn{n-1}{\cM}) \times \cP(\cPn{n-1}{\cM})$ is proper for the topology of weak convergence. Now, let $\cK_1, \cK_2 \subseteq \cPn{n}{\cM}$ be two compact sets (for the $\W_2$ topology), and let $\cL := ([\pi]^{(n)})^{-1}(\cK_1) \cap ([\exp]^{(n)})^{-1}(\cK_2) \cap \Gamma_o^{(n)}(\cM)$. We want to show that $\cL$ is compact. By \Cref{prop:rel_compact_sets_in_w2_topology}, $\cK_1$ and $\cK_2$ are tight, so $\cL$ is tight as $([\pi]^{(n)}, [\exp]^{(n)})$ is proper for the topology of weak convergence. Moreover, $\cL$ is closed in the $\W_2$ topology. Now, let $(\gamma_n)_n$ be a sequence of elements in $\cL$, and let $\mu_n := [\pi]^{(n)}(\gamma_n)$ and $\nu_n := [\exp]^{(n)}(\gamma_n)$. Then, since $\cK_1$ and $\cK_2$ are compact and $\cL$ is tight, up to extracting subsequences, there exists $\mu \in \cK_1$, $\nu \in \cK_2$ and $\gamma \in \cP(\Gamma_o^{(n-1)}(\cM))$ such that $\mu_n \to \mu$, $\nu_n \to \nu$ and $\gamma_n \rightharpoonup \gamma$ (in particular, $\mu = [\pi]^{(n)}(\gamma)$ and $\nu = [\pi]^{(n)}(\gamma)$). We have in fact $\gamma_n \in \cP_2(\Gamma_o^{(n-1)}(\cM))$: indeed, we have
    \begin{align}
        \W_2^2(\delta_{(o,0)}^{(n)}, \gamma) &\leq \liminf_{n \to +\infty} \W_2^2(\delta_{(o,0)}^{(n)}, \gamma_n) \\
        &\leq \liminf_{n \to +\infty} \W_2^2(\delta_o^{(n)}, \mu_n) + \|\gamma_n\|^2 \\
        &\leq \liminf_{n \to +\infty} \W_2^2(\delta_o^{(n)}, \mu_n) + \W_2^2(\mu_n, \nu_n) = \W_2^2(\delta_o^{(n)}, \mu) + \W_2^2(\mu, \nu) < +\infty,
    \end{align}
    where $o \in \cM$ is some fixed base point, the first line is a consequence of the lower semicontinuity of $\W_2^2(\delta_{(o,0)}^{(n)},\cdot)$ with respect to the topology of weak convergence (see \Cref{sec:appendix:w2_space}), we used \Cref{lemma:pseudo_norm_is_finite} in the second line, and the optimality of the $\gamma_n$ in the third line. Moreover, by \Cref{prop:opt_vel_plan_gives_opt_trans_plan}, for every $n$ we have $([\pi]^{(n-1)},[\exp]^{(n-1)})_\#\gamma_n \in \Pi_o(\cPn{n-1}{\cM})$, which is a weakly closed set by \Cref{rk:exists_opt_trans_plans_selection}, so that $([\pi]^{(n-1)},[\exp]^{(n-1)})_\#\gamma \in \Pi_o(\cPn{n-1}{\cM})$. Thus, by \Cref{prop:opt_vel_plan_gives_opt_trans_plan}, since $\gamma$ is concentrated on $\Gamma_o^{(n-1)}(\cM)$, this implies $\gamma \in \Gamma_o^{(n)}(\cM)$. Now, by point \ref{enum:w2_convergence:1_2nd_moment_cvg} of \Cref{th:convergence_in_w2_space}, all we need to do to prove that $\gamma_n \to \gamma$ is to show that $\W_2^2(\delta_{(o,0)}^{(n)},\gamma_n) \xrightarrow[n \to +\infty]{} \W_2^2(\delta_{(o,0)}^{(n)},\gamma)$. This is the case since, by optimality of $\gamma$ and the $\gamma_n$ and by \Cref{lemma:pseudo_norm_is_finite}, 
    \begin{align}
        \lim_{n \to +\infty} \W_2^2(\delta_{(o,0)}^{(n)},\gamma_n) &= \lim_{n \to +\infty} \W_2^2(\delta_{(o,0)}^{(n)},\mu_n) + \|\gamma_n\|^2 = \lim_{n \to +\infty} \W_2^2(\delta_o^{(n)}, \mu_n) + \W_2^2(\mu_n, \nu_n) \\
        &= \W_2^2(\delta_o^{(n)}, \mu) + \W_2^2(\mu, \nu) = \W_2^2(\delta_o^{(n)}, \mu) + \|\gamma\|^2 = \W_2^2(\delta_{(o,0)}^{(n)},\gamma)
    \end{align}
    Thus $\gamma_n \to \gamma$, and $\gamma \in \cL$. This proves that $\cL$ is compact. Thus the map $([\pi]^{(n)}, [\exp]^{(n)}) : \Gamma_o^{(n)}(\cM) \mapsto \cPn{n}{\cM} \times \cPn{n}{\cM}$ is proper, and this finishes the proof.
\end{proof}

A direct consequence of \Cref{prop:opt_vel_plan_proj_map_is_proper} is:

\begin{corollary} \label{cor:limit_points_of_seq_of_opt_vel_plans}
    Let $(\gamma_m)_m$ be a sequence in $\Gamma_o^{(n)}(\cM)$. For every $m$, let $\mu_m := [\pi]^{(n)}(\gamma^m)$ and $\nu_m := [\exp]^{(n)}(\gamma_m)$. Assume that there exists $\mu, \nu \in \cPn{n}{\cM}$ such that $\mu_m \to \mu$ and $\nu_m \to \nu$. Then the set $\{\gamma_m\}_m$ is relatively compact, and the limit points of the sequence $(\gamma_m)_m$ are in $\Gamma_o(\mu,\nu)$. Moreover, if there is an unique optimal velocity plan $\gamma \in \Gamma_o(\mu,\nu)$, then $\gamma_m \to \gamma$.
\end{corollary}

Another consequence of \Cref{prop:opt_vel_plan_proj_map_is_proper} is that we can obtain a measurable selection of $\Gamma_o^{(n)}(\cM) \mapsto \cPn{n}{\cM} \times \cPn{n}{\cM}$:

\begin{corollary} \label{cor:selection_map_opt_vel_plans}
    Then for every $n \geq 0$, the projection map $([\pi]^{(n)}, [\exp]^{(n)}) : \Gamma_o^{(n)}(\cM) \mapsto \cPn{n}{\cM} \times \cPn{n}{\cM}$ admits a measurable right inverse $s_n : \cPn{n}{\cM} \times \cPn{n}{\cM} \mapsto \Gamma_o^{(n)}(\cM)$.
\end{corollary}

\begin{proof}
    The map $([\pi]^{(n)}, [\exp]^{(n)}) : \Gamma_o^{(n)}(\cM) \mapsto \cPn{n}{\cM} \times \cPn{n}{\cM}$ is surjective by \Cref{prop:opt_vel_plans_exist} and Borel. For any $\mu, \nu \in \cPn{n}{\cM}$, its fiber is $\Gamma_o(\mu,\nu)$ which is compact by \Cref{prop:opt_vel_plan_proj_map_is_proper}. Therefore this map admits a measurable right inverse by \Cref{th:selection}.
\end{proof}

\subsection{Couplings between velocity plans} 

An important question is whether, given $\mu \in \cPn{n}{\cM}$ and two velocity plans $\gamma_1, \gamma_2 \in \cPn{n}{T\cM}_\mu$, the set of couplings $\Gamma_\mu(\gamma_1, \gamma_2)$ is non-empty. Indeed, the couplings $\alpha \in \Gamma_\mu(\gamma_1, \gamma_2)$ will play in the following an important role for defining a distance and ``inner product" on $\cPn{n}{T\cM}_\mu$, as well as for defining the differential structure of $\cPn{n}{\cM}$. This question can be formulated in terms of fiber products: consider the two maps $[\pi_1]^{(n)}, [\pi_2]^{(n)} : \cPn{n}{T^2\cM} \mapsto \cPn{n}{T\cM}$. By the universal property of the fiber product, they induce a morphism $h_n : \cPn{n}{T^2\cM} \mapsto \cPn{n}{T\cM} \times_{\cPn{n}{\cM}} \cPn{n}{T\cM}$ in the category $\Pol$. Then, for every $\mu \in \cPn{n}{\cM}$ and $\gamma_1, \gamma_2 \in \cPn{n}{T\cM}_\mu$, one has $\Gamma_\mu(\gamma_1,\gamma_2) = h_n^{-1}((\gamma_1,\gamma_2))$, so that the question of the existence of couplings between $\gamma_1$ and $\gamma_2$ reduces to that of the surjectivity of $h_n$. Fortunately, while it is not immediate, it can be shown that $h_n$ is indeed surjective.

\begin{proposition} \label{prop:couplings_exist}
    The map $h_n : \cPn{n}{T^2\cM} \mapsto \cPn{n}{T\cM} \times_{\cPn{n}{\cM}} \cPn{n}{T\cM}$ is surjective.
\end{proposition}

\begin{proof}
    By \Cref{lemma:hierarchical_gluing_lemma}, all we need to do is to prove that $h_0$ is surjective and with linearly growing antecedents. This is trivially the case since $T^2\cM = T\cM \times_\cM T\cM$, so $h_0$ is just the identity map of $T^2\cM$.
\end{proof}

\begin{remark} \label{rk:general_couplings_exist}
    In fact, one can generalize \Cref{prop:couplings_exist} to the following one: if $m > 0$ and $\{i_{1,1},\ldots,i_{1,k_1}\}$, ..., $\{i_{p,1},\ldots,i_{p,k_p}\}$ is a partition of $\{1,\ldots,m\}$, then for any $n > 0$ the projection maps $[\pi^m_{i_{j,1},\ldots,i_{j,k_j}}]^{(n)} : \cPn{n}{T^m\cM} \mapsto \cPn{n}{T^{k_j}\cM}$ induce a surjective morphism
    \begin{equation}
        \cPn{n}{T^m \cM} \longrightarrow \cPn{n}{T^{k_1}\cM} \times_{\cPn{n}{\cM}} \ldots \times_{\cPn{n}{\cM}} \cPn{n}{T^{k_p}\cM}.
    \end{equation}
\end{remark}

We have thus shown that the sets of couplings $\Gamma_\mu(\gamma_1,\gamma_2)$ are always non-empty. In fact, they are even compact, as the following proposition shows:

\begin{proposition} \label{prop:couplings_between_plans_are_compact}
    For every $\mu \in \cPn{n}{\cM}$ and $\gamma_1,\gamma_2 \in \cPn{n}{T\cM}_\mu$, $\Gamma_\mu(\gamma_1,\gamma_2)$ is compact.
\end{proposition}

\begin{proof}
    The set $\Gamma_\mu(\gamma_1,\gamma_2)$ is the preimage of $(\gamma_1,\gamma_2)$ by the map $([\pi_1]^{(n)},[\pi_2]^{(n)}) : \cPn{n}{T^2\cM} \mapsto \cPn{n}{T\cM} \times_{\cPn{n}{\cM}} \cPn{n}{T\cM}$, which is proper by \Cref{lemma:coupling_proj_map_is_proper} below. Therefore, it is compact.
\end{proof}

\begin{lemma} \label{lemma:coupling_proj_map_is_proper}
    For every $n \geq 0$, the map $([\pi_1]^{(n)},[\pi_2]^{(n)}) : \cPn{n}{T^2\cM} \mapsto \cPn{n}{T\cM} \times_{\cPn{n}{\cM}} \cPn{n}{T\cM}$ is proper.
\end{lemma}

\begin{proof}
    This is simply a consequence of \Cref{prop:p_of_fiber_to_fiber_of_p_is_proper}, as $T^2\cM$ identifies as the fiber product $T\cM \times_{\cM} T\cM$.
\end{proof}

\begin{corollary} \label{cor:measurable_selection_couplings}
    For every $n \geq 0$, there exists a measurable map $s : \cPn{n}{T\cM} \times_{\cPn{n}{\cM}} \cPn{n}{T\cM} \mapsto \cPn{n}{T^2\cM}$ which associates to every $\mu \in \cPn{n}{\cM}$ and $\gamma_1, \gamma_2 \in \cPn{n}{T\cM}_\mu$ a coupling $s(\gamma_1,\gamma_2) \in \Gamma_\mu(\gamma_1,\gamma_2)$.
\end{corollary}

\begin{proof}
    The projection map $\cPn{n}{T^2\cM} \mapsto \cPn{n}{T\cM} \times_{\cPn{n}{\cM}} \cPn{n}{T\cM}$ is surjective by \Cref{prop:couplings_exist} and has compact fibers by \Cref{prop:couplings_between_plans_are_compact}. The existence of the map $s$ is then given by \Cref{th:selection}.
\end{proof}

Another direct consequence of \Cref{lemma:coupling_proj_map_is_proper} is:

\begin{corollary} \label{cor:limit_points_of_seq_of_couplings}
    Let $(\alpha^m)_m$ be a sequence in $\cPn{n}{T^2\cM}$. For every $m$, let $\gamma^m_1 := [\pi_1]^{(n)}(\alpha^m)$ and $\gamma^m_2 := [\pi_2]^{(n)}(\alpha^m)$. Assume that there exists $\gamma_1, \gamma_2 \in \cPn{n}{T\cM}$ such that $\gamma_1^m \to \gamma_1$ and $\gamma_2^m \to \gamma_2$. Then the set $\{\alpha^m\}_m$ is relatively compact, and the limit points of the sequence $(\alpha^m)_m$ are in $\Gamma_\mu(\gamma_1,\gamma_2)$ (with $\mu = [\pi]^{(n)}(\gamma_1)$). Moreover, if there is an unique coupling $\alpha \in \Gamma_\mu(\gamma_1,\gamma_2)$, then $\alpha^m \to \alpha$.
\end{corollary}

\subsection{The distance \texorpdfstring{$\W_\mu$}{W\_mu} and its ``inner product"}

Let $n \geq 0$ and $\mu \in \cPn{n}{\cM}$. We define a distance and an ``inner product" on the space $\cPn{n}{T\cM}_\mu$ the following way: let $\gamma_1, \gamma_2 \in \cPn{n}{T\cM}$, the distance $\W_\mu(\gamma_1,\gamma_2)$ is defined as
\begin{equation}
    \W_\mu^2(\gamma_1,\gamma_2) := \inf_{\alpha \in \Gamma_\mu(\gamma_1,\gamma_2)} \bE^{(n)}_\alpha[\|v_1 - v_2\|^2_x]
\end{equation}
This infimum is taken among a non-empty set, as by \Cref{prop:couplings_exist}, $\Gamma_\mu(\gamma_1,\gamma_2)$ is never empty. Since $\|v_1 - v_2\|^2_x \leq 4(\|v_1\|^2_x + \|v_2\|^2_x)$ for $(x,v_1,v_2) \in T^2 \cM$, we have for every $\alpha \in \Gamma_\mu(\gamma_1,\gamma_2)$ that
\begin{equation}
    \bE^{(n)}_\alpha[\|v_1 - v_2\|^2_x] \leq 4(\bE^{(n)}_{\gamma_1}[\|v_1\|^2] + \bE^{(n)}_{\gamma_2}[\|v_2\|^2]) = 4(\|\gamma_1\|^2 + \|\gamma_2\|^2)
\end{equation}
Therefore $\W^2_\mu(\gamma_1,\gamma_2) \leq 4(\|\gamma_1\|^2+\|\gamma_2\|^2) < +\infty$. We also define the ``inner product" $\sca{\gamma_1}{\gamma_2}_\mu$ by
\begin{equation}
    \sca{\gamma_1}{\gamma_2}_\mu := \sup_{\alpha \in \Gamma_\mu(\gamma_1,\gamma_2)} \bE^{(n)}_\alpha[\sca{v_1}{v_2}_x]
\end{equation}
Since $\|v_1 - v_2\|_x^2 = \|v_1\|^2_x - 2\sca{v_1}{v_2}_x + \|v_2\|^2_x$ for $(x,v_1,v_2) \in T^2 \cM$, $\sca{\gamma_1}{\gamma_2}_\mu$ is well-defined and finite, with
\begin{equation} \label{eq:w_mu_norm_inner_prod_rel}
    \W_\mu^2(\gamma_1, \gamma_2) = \|\gamma_1\|^2 - 2\sca{\gamma_1}{\gamma_2}_\mu + \|\gamma_2\|^2
\end{equation}

\begin{proposition} \label{prop:dist_and_inner_prod_are_min_and_max}
    The $\inf$ and the $\sup$ in the definitions of $\W_\mu(\gamma_1,\gamma_1)$ and $\sca{\gamma_1}{\gamma_2}_\mu$ are actually a $\min$ and a $\max$.
\end{proposition}

\begin{proof}
    Indeed, by \Cref{lemma:n_expectancy_regularity}, the maps $\alpha \mapsto \bE^{(n)}_\alpha[\|v_1-v_2\|^2_x]$ and $\alpha \mapsto \bE^{(n)}_\alpha[\sca{v_1}{v_2}_x]$ are continuous on $\cPn{n}{T^2\cM}$. Moreover, by \Cref{prop:couplings_between_plans_are_compact}, for any $\mu \in \cPn{n}{\cM}$ and $\gamma_1, \gamma_2 \in \cPn{n}{T\cM}_\mu$, the set $\Gamma_\mu(\gamma_1, \gamma_2)$ is compact, so any minimizing sequence of $\alpha \mapsto \bE^{(n)}_\alpha[\|v_1-v_2\|^2_x]$ (resp. maximizing sequence of $\alpha \mapsto \bE^{(n)}_\alpha[\sca{v_1}{v_2}_x]$) on $\Gamma_\mu(\gamma_1,\gamma_2)$ converges up to extracting a subsequence to a minimizer (resp. a maximizer).
\end{proof}

By \eqref{eq:w_mu_norm_inner_prod_rel}, the minimizers $\alpha$ for $\W_\mu(\gamma_1,\gamma_2)$ are exactly the maximizers for $\sca{\gamma_1}{\gamma_2}_\mu$. We will call these elements the \emph{optimal couplings between $\gamma_1$ and $\gamma_2$}, and we will denote $\Gamma_{\mu,o}(\gamma_1,\gamma_2)$ the set of such couplings.

\begin{proposition} \label{prop:w_mu_is_distance}
    For every $\mu \in \cPn{n}{\cM}$, $\W_\mu$ is a distance on $\cPn{n}{T\cM}_\mu$.
\end{proposition}

\begin{proof}
    We prove this by induction. For $n = 0$, $\W_x$ is just the distance induced by the norm $\|\cdot\|_x$ on the vector space $T_x \cM$. Let $n > 0$, and assume that the result holds for $n-1$. Fix $\bP \in \cPn{n}{T\cM}$, then:
    \begin{itemize}
        \item For every $\bGamma \in \cPn{n}{T\cM}_\bP$, $\W_\bP(\bGamma,\bGamma) = 0$. Indeed, taking $\bA := [(x,v) \mapsto (x,v,v)]^{(n)}(\bGamma) \in \Gamma_\bP(\bGamma,\bGamma)$, we have
        \begin{equation}
            \W_\bP^2(\bGamma,\bGamma) \leq \bE^{(n)}_\bA[\|v_1 - v_2\|^2] = \bE^{(n)}_\bGamma[0] = 0.
        \end{equation}
        \item For every $\bGamma_1, \bGamma_2 \in \cPn{n}{T\cM}_\bP$, if $\W_\bP(\bGamma_1,\bGamma_2) = 0$, then $\bGamma_1 = \bGamma_2$. Let $\bA \in \Gamma_\bP(\bGamma_1,\bGamma_2)$ be optimal, then
        \begin{equation}
            0 = \bE^{(n)}_\bA[\|v_1 - v_2\|^2_x] = \int \bE^{(n-1)}_\alpha[\|v_1 - v_2\|^2_x] \dd\bA(\alpha)
        \end{equation}
        so that, for $\bA$-a.e. $\alpha$, we have $\bE^{(n-1)}_\alpha[\|v_1 - v_2\|^2_x] = 0$, that is $\W_{[\pi]^{(n-1)}(\alpha)}([\pi_1]^{(n-1)}(\alpha), [\pi_2]^{(n-1)}(\alpha)) = 0$ so that $[\pi_1]^{(n-1)}(\alpha) = [\pi_2]^{(n-1)}(\alpha)$. Therefore $[\pi_1]^{(n-1)}$ and $[\pi_2]^{(n-1)}$ coincide $\bA$-almost everywhere, so that $\bGamma_1 = [\pi_1]^{(n)}(\bA) = [\pi_2]^{(n)}(\bA) = \bGamma_2$.
        \item For every $\bGamma_1, \bGamma_2 \in \cPn{n}{T\cM}_\bP$, it holds $\W_\bP(\bGamma_1, \bGamma_2) = \W_\bP(\bGamma_1,\bGamma_2)$. Indeed, let $\bA \in \Gamma(\bGamma_1,\bGamma_2)$ be optimal, and let $\bA' = [(x,v_1,v_2) \mapsto (x,v_2,v_1)]^{(n)}(\bA)$. Then $\bA' \in \Gamma_\bP(\bGamma_2,\bGamma_1)$, and
        \begin{equation}
            \W^2_\bP(\bGamma_1,\bGamma_2) = \bE^{(n)}_\bA[\|v_1-v_2\|^2_x] = \bE^{(n)}_{\bA'}[\|v_2 - v_1\|^2_x] \geq \W^2_\bP(\bGamma_2,\bGamma_1),
        \end{equation}
        and by symmetry we conclude that $\W_\bP(\bGamma_1, \bGamma_2) = \W_\bP(\bGamma_1,\bGamma_2)$.
        \item For every $\bGamma_1, \bGamma_2, \bGamma_3 \in \cPn{n}{T\cM}_\bP$, we have $\W_\bP(\bGamma_1,\bGamma_3) \leq \W_\bP(\bGamma_1,\bGamma_2) + \W_\bP(\bGamma_2,\bGamma_3)$. Let $\bA_1 \in \Gamma_\bP(\bGamma_1,\bGamma_2)$ and $\bA_1 \in \Gamma_\bP(\bGamma_2,\bGamma_3)$ be optimal, and let $\bB \in \Gamma_\bP(\bGamma_1,\bGamma_2,\bGamma_3)$ be such that $[\pi^3_{1,2}]^{(n)}(\bB) = \bA_1$ and $[\pi^3_{2,3}]^{(n)}(\bB) = \bA_2$ (such a $\bB$ exists: see \Cref{lemma:generalized_couplings_exist_3}). Then $[\pi^3_{1,3}]^{(n)}(\bB) \in \bGamma_\bP(\bGamma_1,\bGamma_3)$, so that
        \begin{align}
            \W_\bP(\bGamma_1,\bGamma_3) &\leq \sqrt{\bE^{(n)}_\bB[\|v_1 - v_3\|^2_x]} \leq \sqrt{\bE^{(n)}_{\bB}[\|v_1 - v_2\|^2_x]} + \sqrt{\bE^{(n)}_{\bB}[\|v_2 - v_3\|^2_x]} \\
            &\leq \sqrt{\bE^{(n)}_{\bA_1}[\|v_1 - v_2\|^2_x]} + \sqrt{\bE^{(n)}_{\bA_2}[\|v_1 - v_2\|^2_x]} = \W_\bP(\bGamma_1,\bGamma_2) + \W_\bP(\bGamma_2,\bGamma_3)
        \end{align}
        where we used the Minkowski inequality (\Cref{lemma:n_minkowsky_ineq}) in the first line, and the optimality of $\bA_1$ and $\bA_2$ in the second line.
    \end{itemize}
    We have thus shown that on $\cPn{n}{T\cM}_\bP$, $\W_\bP$ is reflexive, symmetric, and transitive. It is therefore a distance.
\end{proof}

\begin{proposition} \label{prop:various_results_on_inner_prod}
    Let $n > 0$, $\mu \in \cPn{n}{\cM}$ and $\gamma, \gamma_1, \gamma_2, \gamma_3 \in \cPn{n}{T\cM}_\mu$. The following statements hold:
    \begin{enumerate}
        \item $\lambda \bm{0}_\mu = \bm{0}_\mu$ for any $\lambda \in \R$.
        \item $0 \gamma = \bm{0}_\mu$.
        \item $\sca{\bm{0}_\mu}{\gamma}_\mu = 0$, and in fact $\bE^{(n)}_\alpha[\sca{v_1}{v_2}_x] = 0$ for any $\alpha \in \Gamma_\mu(\bm{0}_\mu, \gamma)$.
        \item \label{enum:tangent_struct:inner_prod_homogeneous} $\sca{\lambda_1 \gamma_1}{\lambda_2 \gamma_2}_\mu = \lambda_1 \lambda_2 \sca{\gamma_1}{\gamma_2}_\mu$ for any $\lambda_1, \lambda_2 \geq 0$.
        \item $(\lambda_1,\lambda_2) \cdot \alpha \in \Gamma_\mu(\lambda_1 \gamma_1, \lambda_2 \gamma_2)$ for any $\alpha \in \Gamma_\mu(\gamma_1,\gamma_2)$, $\lambda_1,\lambda_2 \in \R$.
        \item \label{enum:tangent_struct:inner_prod_subadditive} $\sca{\gamma_1 +_\alpha \gamma_2}{\gamma_3}_\mu \leq \sca{\gamma_1}{\gamma_3}_\mu + \sca{\gamma_2}{\gamma_3}_\mu$ for any $\alpha \in \Gamma_\mu(\gamma_1,\gamma_2)$.
        \item \label{enum:tangent_struct:commutativity} $\gamma_1 +_\alpha \gamma_2 = \gamma_2 +_{\alpha'} \gamma_1$ for any $\alpha \in \Gamma_\mu(\gamma_1,\gamma_2)$, $\alpha' = [\pi_2,\pi_1]^{(n)}(\alpha)$.
        \item \label{enum:tangent_struct:minus_eq_plus_neg} $\gamma_1 -_\alpha \gamma_2 = \gamma_1 +_{\alpha'} (-\gamma_2)$ for any $\alpha \in \Gamma_\mu(\gamma_1,\gamma_2)$, $\alpha' = (1,-1) \cdot \alpha$.
        \item \label{enum:tangent_struct:zero_neutral_add} $\gamma +_\alpha \bm{0}_\mu = \gamma$ for any $\alpha \in \Gamma_\mu(\gamma,\bm{0}_\mu)$.
        \item \label{enum:tangent_struct:scalar_mult_distributive_scalar_add} $(\lambda_1 + \lambda_2)\gamma = \lambda_1 \gamma +_\alpha \lambda_2 \gamma$ for any $\lambda_1, \lambda_2 \geq 0$, $\alpha = [(x,v) \mapsto (x,\lambda_1 v, \lambda_2 v)]^{(n)}(\gamma)$.
        \item \label{enum:tangent_struct:scalar_mult_distributive_vector_add} $\lambda (\gamma_1 +_\alpha \gamma_2) = \lambda \gamma_1 +_{\alpha'} \lambda \gamma_2$ for any $\lambda \in \R$, $\alpha \in \Gamma_\mu(\gamma_1,\gamma_2)$, $\alpha' = (\lambda,\lambda) \cdot \alpha$.
        \item \label{enum:tangent_struct:triangle_ineq} $\|\gamma_1 +_\alpha \gamma_2\| \leq \|\gamma_1\| + \|\gamma_2\|$ for any $\alpha \in \Gamma_\mu(\gamma_1,\gamma_2)$.
        \item \label{enum:tangent_struct:inner_prod_symmetric} $\sca{\gamma_1}{\gamma_2}_\mu = \sca{\gamma_2}{\gamma_1}_\mu$.
        \item \label{enum:tangent_struct:inner_prod_cauchy_schwarz} $|\sca{\gamma_1}{\gamma_2}_\mu| \leq \|\gamma_1\| \|\gamma_2\|$.
        \item \label{enum:tangent_struct:inner_prod_positive_def} $\sca{\gamma}{\gamma}_\mu = \|\gamma\|^2$. 
        \item \label{enum:tangent_struct:norm_is_positive_def} $\|\gamma\| = 0$ if and only if $\gamma = \bm{0}_\mu$.
    \end{enumerate}
\end{proposition}

\begin{proof}
    \begin{enumerate}
        \item Indeed $\lambda \bm{0}_\mu = [m_\lambda \circ \iota_0]^{(n)}(\mu) = [\iota_0]^{(n)}(\mu) = \bm{0}_\mu$.
        \item By induction: for $n = 0$ it's just the fact that $0v = 0$ for every $v \in T_x\cM$. For $n > 0$, we have
        \begin{align}
            0_\bGamma &= [m_0]^{(n)}(\bGamma) = (\gamma \to [m_0]^{(n-1)}(\gamma))_\#\bGamma = (\gamma \to 0\gamma)_\#\bGamma = (\gamma \to \bm{0}_{[\pi]^{(n-1)}(\gamma)})_\#\bGamma \\
            &= ([\iota_0]^{(n-1)} \circ [\pi]^{(n-1)})_\#\bGamma = [\iota_0]^{(n)}([\pi]^{(n)}(\bGamma)) = [\iota_0]^{(n)}(\bP) = \bm{0}_\bP.
        \end{align}
        \item The base support of $\bm{0}_\mu$ is in $\{(x,0), x \in \cM\} \subseteq T\cM$. Therefore the base support of any $\alpha \in \Gamma_\mu(\bm{0}_\mu,\gamma)$ is in $\{(x,0,v), (x,v) \in T\cM\} \subseteq T^2\cM$. Thus for every $(x,v_1,v_2) \in \spt_{T^2\cM}(\alpha)$, $\sca{v_1}{v_2}_x = 0$, so that $\bE^{(n)}_\alpha[\sca{v_1}{v_2}_x] = 0$.
        \item If $\lambda_1 = 0$ or $\lambda_2 = 0$, then $\sca{\lambda_1 \gamma_1}{\lambda_2 \gamma_2}_\mu = 0 = \lambda_1 \lambda_2 \sca{\gamma_1}{\gamma_2}_\mu$ by the preceding points. Assume now $\lambda_1, \lambda_2 > 0$. Then for every $\alpha \in \Gamma_\mu(\gamma_1,\gamma_2)$, we have $(\lambda_1,\lambda_2) \cdot \alpha \in \Gamma_\mu(\lambda_1 \gamma_1, \lambda_2 \gamma_2)$ so that 
        \begin{equation}
            \sca{\lambda_1 \gamma_1}{\lambda_2 \gamma_2}_\mu \geq \bE^{(n)}_{(\lambda_1,\lambda_2)\alpha}[\sca{v_1}{v_2}_x] = \lambda_1 \lambda_2 \bE^{(n)}_\alpha[\sca{v_1}{v_2}_x]
        \end{equation}
        and taking the $\sup$ over $\alpha$ we find $\sca{\lambda_1 \gamma_1}{\lambda_2 \gamma_2}_\mu \geq \lambda_1 \lambda_2 \sca{\gamma_1}{\gamma_2}_\mu$. We obtain the converse inequality by applying this inequality to $\tilde{\gamma}_1 = \lambda_1 \gamma_1$, $\tilde{\gamma}_2 = \lambda_2 \gamma_2$, $\tilde{\lambda}_1 = \lambda_1^{-1}$ and $\tilde{\lambda}_2 = \lambda_2^{-1}$. 
        \item Let $\alpha' = (\lambda_1,\lambda_2)\cdot \alpha$. Then, for $i \in \{1,2\}$, we have $[\pi_i]^{(n)}(\alpha') = [\pi_i \circ m_{\lambda_1,\lambda_2}]^{(n)}(\alpha) = [m_{\lambda_i} \circ \pi_i]^{(n)}(\alpha) = [m_{\lambda_i}]^{(n)}(\gamma_i) = \lambda_i \gamma_i$.
        \item Let $\alpha \in \Gamma_\mu(\gamma_1,\gamma_2)$. We have
        \begin{align}
            \sca{\gamma_1 +_\alpha \gamma_2}{\gamma_3}_\mu &= \sup_{\beta \in \Gamma_\mu(\gamma_1 +_\alpha \gamma_2, \gamma_3)} \bE^{(n)}_\beta[\sca{v_{12}}{v_3}_x] \\
            &= \sup_{\beta \in \Gamma_\mu(\alpha,\gamma_3)} \bE^{(n)}_\beta[\sca{v_1 + v_2}{v_3}_x] \\
            &= \sup_{\beta \in \Gamma_\mu(\alpha,\gamma_3)} \bE^{(n)}_\beta[\sca{v_1}{v_3}_x] + \bE^{(n)}_\beta[\sca{v_2}{v_3}_x] \\
            &\leq \sup_{\beta \in \Gamma_\mu(\alpha,\gamma_3)} \bE^{(n)}_\beta[\sca{v_1}{v_3}_x] + \sup_{\beta \in \Gamma_\mu(\alpha,\gamma_3)} \bE^{(n)}_\beta[\sca{v_2}{v_3}_x] \\
            &\leq \sup_{\beta \in \Gamma_\mu(\gamma_1,\gamma_3)} \bE^{(n)}_\beta[\sca{v_1}{v_2}_x] + \sup_{\beta \in \Gamma_\mu(\gamma_2,\gamma_3)} \bE^{(n)}_\beta[\sca{v_1}{v_2}_x] \\
            &\leq \sca{\gamma_1}{\gamma_3}_\mu + \sca{\gamma_2}{\gamma_3}_\mu
        \end{align}
        Here $\Gamma_\mu(\alpha, \gamma_3)$ is the set of $\beta \in \cPn{n}{T^3\cM}$ such that $[\pi^3_{1,2}]^{(n)}(\beta) = \alpha$ and $[\pi^3_3]^{(n)}(\beta) = \gamma_3$, which is non-empty by \Cref{rk:general_couplings_exist}. The second line is justified by the fact that $[(\pi^3_1+\pi^3_2,\pi^3_3)]^{(n)}(\Gamma_\mu(\alpha,\gamma_3)) = \Gamma_\mu(\gamma_1 +_\alpha \gamma_2, \gamma_3)$ (the inclusion $\subseteq$ is straightforward. The equality, not so much, and it is proved in \Cref{lemma:generalized_couplings_exist_2}), and the fifth line is justified by $[(\pi^3_1,\pi^3_3)]^{(n)}(\Gamma_\mu(\alpha,\gamma_3)) \subseteq \Gamma_\mu(\gamma_1, \gamma_3)$ and $[(\pi^3_2,\pi^3_3)]^{(n)}(\Gamma_\mu(\alpha,\gamma_3)) \subseteq \Gamma_\mu(\gamma_2, \gamma_3)$.
        \item Clearly $\alpha' \in \Gamma_\mu(\gamma_2,\gamma_1)$, and $\gamma_2 +_{\alpha'} \gamma_1 = [\pi_1 + \pi_2]^{(n)}(\alpha') = [(\pi_1 + \pi_2) \circ (\pi_2,\pi_1)]^{(n)}(\alpha) = [\pi_1+\pi_2]^{(n)}(\alpha) = \gamma_1 +_\alpha \gamma_2$.
        \item We have $\gamma_1 -_\alpha \gamma_2 = [\pi_1-\pi_2]^{(n)}(\alpha) = [(\pi_1+\pi_2) \circ m_{1,-1}]^{(n)}(\alpha) = [\pi_1+\pi_2]^{(n)}(\alpha') = \gamma_1 +_{\alpha'} (-\gamma_2)$.
        \item Again, if $\alpha \in \Gamma_\mu(\gamma,\bm{0}_\mu)$, then its base support is in $\{(x,v,0), (x,v) \in T\cM\} \subseteq T^2\cM$. In particular, $\pi_1+\pi_2 = \pi_1$ on $\spt_{T^2\cM}(\alpha)$, so that $\gamma +_\alpha \bm{0}_\mu = [\pi_1+\pi_2]^{(n)}(\alpha) = [\pi_1]^{(n)}(\alpha) = \gamma$. 
        \item We have $\lambda_1 \gamma +_\alpha \lambda_2 \gamma = [\pi_1 + \pi_2]^{(n)}(\alpha) = [(\pi_1 + \pi_2) \circ ((x,v) \mapsto (x, \lambda_1 v, \lambda_2 v))]^{(n)}(\gamma) = [m_{\lambda_1+\lambda_2}]^{(n)}(\gamma) = (\lambda_1+\lambda_2)\gamma$.
        \item We have $\lambda (\gamma_1 +_\alpha \gamma_2) = [m_\lambda \circ (\pi_1 + \pi_2)]^{(n)}(\alpha) = [(\pi_1 + \pi_2) \circ m_{\lambda,\lambda}]^{(n)}(\alpha) = [\pi_1 + \pi_2]^{(n)}(\alpha') = \lambda \gamma_1 +_{\alpha'} \lambda \gamma_2$
        \item By the Minkowski inequality (\Cref{lemma:n_minkowsky_ineq}), we have
        \begin{align}
            \|\gamma_1 +_\alpha \gamma_2\| &= \sqrt{\bE^{(n)}_\alpha[\|v_1 + v_2\|^2_x]} \\
            &\leq \sqrt{\bE^{(n)}_\alpha[\|v_1\|^2_x]} + \sqrt{\bE^{(n)}_\alpha[\|v_2\|^2_x]} \\
            &\leq \sqrt{\bE^{(n)}_{\gamma_1}[\|v_1\|^2_x]} + \sqrt{\bE^{(n)}_{\gamma_2}[\|v_2\|^2_x]} = \|\gamma_1\| + \|\gamma_2\|.
        \end{align}
        \item The inner product $\sca{\cdot}{\cdot}_\mu$ is symmetric as a consequence of \eqref{eq:w_mu_norm_inner_prod_rel} and of the fact that $\W_\mu$ is symmetric as a distance on $\cPn{n}{T\cM}_\mu$ (\Cref{prop:w_mu_is_distance}).
        \item For any $\alpha \in \Gamma_\mu(\gamma_1,\gamma_2)$, we have by the Cauchy-Schwarz inequality (\Cref{lemma:n_holder_ineq})
        \begin{equation}
            |\bE^{(n)}_\alpha[\sca{v_1}{v_2}_x]| \leq \bE^{(n)}_\alpha[|\sca{v_1}{v_2}_x|] \leq \bE^{(n)}_\alpha[\|v_1\|_x \|v_2\|_x] \leq \sqrt{\bE^{(n)}_\alpha[\|v_1\|^2_x]} \sqrt{\bE^{(n)}_\alpha[\|v_2\|^2_x]} = \|\gamma_1\| \|\gamma_2\|.
        \end{equation}
        Taking $\alpha \in \Gamma_{\mu,o}(\gamma_1, \gamma_2)$, we thus find $|\sca{\gamma_1}{\gamma_2}_\mu| \leq \|\gamma_1\|\|\gamma_2\|$.
        \item By the previous point, $\sca{\gamma}{\gamma}_\mu \leq \|\gamma\|^2$. Taking $\alpha = [(x,v) \mapsto (x,v,v)]^{(n)}(\gamma) \in \Gamma_\mu(\gamma,\gamma)$, we then observe that
        \begin{equation}
            \sca{\gamma}{\gamma}_\mu \geq \bE^{(n)}_\alpha[\sca{v_1}{v_2}_x] = \bE^{(n)}_\gamma[\|v\|^2_x] = \|\gamma\|^2,
        \end{equation}
        so that $\sca{\gamma}{\gamma}_\mu = \|\gamma\|^2$.
        \item First, if $\gamma = \bm{0}_\mu = [\iota_0]^{(n)}(\mu)$, then $\|\gamma\|^2 = \bE^{(n)}_\gamma[\|v\|^2] = \bE^{(n)}_\mu[0] = 0$. Conversely, if $\|\gamma\| = 0$, then by \Cref{lemma:pseudo_norm_bounds_dist_to_zero}, we have $\W_2(\bm{0}_\mu,\gamma) = 0$ so that $\gamma = \bm{0}_\mu$.
    \end{enumerate}
\end{proof}

\begin{lemma} \label{lemma:generalized_couplings_exist_3}
    Let $\mu \in \cPn{n}{\cM}$ and $\gamma_1, \gamma_2, \gamma_3 \in \cPn{n}{T\cM}_\mu$, then the map $([\pi^3_{1,2}]^{(n)},[\pi^3_{2,3}]^{(n)}) : \Gamma_\mu(\gamma_1,\gamma_2,\gamma_3) \mapsto \Gamma_\mu(\gamma_1,\gamma_2) \times \Gamma_\mu(\gamma_2,\gamma_3)$ is surjective.
\end{lemma}

\begin{proof}
    To prove the lemma, we reformulate the problem in terms of fiber products. Let $\alpha \in \Gamma_\mu(\gamma_1,\gamma_2)$ and $\alpha' \in \Gamma_\mu(\gamma_2,\gamma_2)$, then they satisfy $[\pi_2]^{(n)}(\alpha) = [\pi_1]^{(n)}(\alpha') = \gamma_2$. To prove surjectivity, we are looking for some $\beta \in \cPn{n}{T^3\cM}$ such that $[\pi_{1,2}^3]^{(n)}(\beta) = \alpha$ and $[\pi_{2,3}^3]^{(n)}(\beta) = \alpha'$. In other words, we want to prove that the map $h_n$ induced by the universal property of the fiber product in the commutative diagram
    \begin{equation}
        \begin{tikzcd}
            & \cPn{n}{T^3\cM} \arrow[ddl, bend right, "{[\pi^3_{1,2}]^{(n)}}"'] \arrow[ddr, bend left, "{[\pi^3_{2,3}]^{(n)}}"] \arrow[d, "h_n"] &  \\
            & \cPn{n}{T^2\cM} \times_{\cPn{n}{T\cM}} \cPn{n}{T^2\cM} \arrow[dl] \arrow[dr] & \\
            \cPn{n}{T^2\cM} \arrow[r, "{[\pi_2]^{(n)}}"'] & \cPn{n}{T\cM} & \cPn{n}{T^2\cM} \arrow[l, "{[\pi_1]^{(n)}}"] 
        \end{tikzcd}
    \end{equation}
    is surjective. By \Cref{lemma:hierarchical_gluing_lemma}, we only need to prove that $h_0$ is surjective and with linearly growing antecedents. But this is straightforward: indeed, we can check that the fiber product 
    \begin{equation}
        T^2\cM \times_{T\cM} T^2\cM = \{((x,v_1,v_2),(y,u_1,u_2)) \in T^2\cM \times T^2\cM \setcond (x,v_2) = (y,u_1) \}
    \end{equation}
    identifies as $T^3\cM$, so that $h_0$ is the identity map.
\end{proof}

\begin{lemma} \label{lemma:generalized_couplings_exist_2}
    Let $\mu \in \cPn{n}{\cM}$, $\gamma_1, \gamma_2, \gamma_3 \in \cPn{n}{T\cM}_\mu$ and $\alpha \in \Gamma_\mu(\gamma_1,\gamma_2)$, then the map $[\pi_1^3+\pi_2^3, \pi_3^3]^{(n)}$ from $\Gamma_\mu(\alpha,\gamma_3)$ to $\Gamma_\mu(\gamma_1 +_\alpha \gamma_2, \gamma_3)$ is surjective.
\end{lemma}

\begin{proof}
    Again, we first reformulate this lemma in terms of fiber products. Let $\alpha' \in \Gamma_\mu(\gamma_1 +_\alpha \gamma_2, \gamma_3)$, then the measures $\alpha$ and $\alpha'$ satisfy $[\pi_1+\pi_2]^{(n)}(\alpha) = [\pi_1]^{(n)}(\alpha') = \gamma_1 +_\alpha \gamma_2$. To prove surjectivity, we are looking for some $\beta \in \cPn{n}{T^3\cM}$ such that $[\pi^3_{1,2}]^{(n)}(\beta) = \alpha$ and $[\pi_1^3+\pi_2^3,\pi_3^3]^{(n)}(\beta) = \alpha'$. In other words, we want to prove that the map $h_n$ induced by the universal property of the fiber product in the commutative diagram
    \begin{equation} \label{diag:l1934}
        \begin{tikzcd}
            & \cPn{n}{T^3\cM} \arrow[ddl, bend right, "{[\pi^3_{1,2}]^{(n)}}"'] \arrow[ddr, bend left, "{[\pi^3_1+\pi^3_2,\pi^3_3]^{(n)}}"] \arrow[d, "h_n"] &  \\
            & \cPn{n}{T^2\cM} \times_{\cPn{n}{T\cM}} \cPn{n}{T^2\cM} \arrow[dl] \arrow[dr] & \\
            \cPn{n}{T^2\cM} \arrow[r, "{[\pi_1+\pi_2]^{(n)}}"'] & \cPn{n}{T\cM} & \cPn{n}{T^2\cM} \arrow[l, "{[\pi_1]^{(n)}}"] 
        \end{tikzcd}
    \end{equation}
    is surjective. By \Cref{lemma:hierarchical_gluing_lemma}, we only need to prove that $h_0$ is surjective and with linearly growing antecedents. Let $\bm{x} = ((x,v_1,v_2),(y,u_1,u_2)) \in T^2\cM \times_{T\cM} T^2\cM$. Then $(\pi_1+\pi_2)(x,v_1,v_2) = (x,v_1+v_2) = (y, u_1) = \pi_1(y,u_1,u_2)$, so that $h_0(x,v_1,v_2,u_2) = \bm{x}$. Thus $h_0$ is surjective, and we can check that the element of the preimage of $\bm{x}$ that we constructed satisfies the assumption of linearly growing antecedents. This finishes the proof.
\end{proof}

We finish this subsection by deriving a characterization of optimal couplings in $\cPn{n}{T^2\cM}$. Let $n > 0$, and fix $\bP \in \cPn{n}{\cM}$, $\bGamma_1, \bGamma_2 \in \cPn{n}{T\cM}_\bP$ and $\bA \in \Gamma_\bP(\bGamma_1,\bGamma_2)$. Applying repeatedly the disintegration theorem (\Cref{th:disintegration}), we can write $\bA$ as
\begin{equation}
    \dd \bA(\alpha) = \dd \bA_{\gamma_1,\gamma_1}(\alpha) \dd \bA_\mu(\gamma_1,\gamma_2) \dd \bP(\mu)
\end{equation}
where for $\bP$-a.e. $\mu$, $\bA_\mu$ is a probability measure on $\cPn{n-1}{T\cM}_\mu \times \cPn{n-1}{T\cM}_\mu$, and for $\bP$-a.e. $\mu$ and $\bA_\mu$-a.e. $\gamma_1, \gamma_2$, $\bA_{\gamma_1,\gamma_2}$ is a probability measure on $\Gamma_\mu(\gamma_1,\gamma_2)$. Similarly, for $i \in \{1,2\}$, we write $\bGamma_i$ as
\begin{equation}
    \dd \bGamma_i(\gamma) = \dd \bGamma_{i,\mu}(\gamma) \dd \bP(\mu)
\end{equation}
where for $\bP$-a.e. $\mu$, $\bGamma_{i,\mu}$ is a probability measure on $\cPn{n-1}{T\cM}_\mu$. For any nonnegative measurable map $f : \cPn{n-1}{T\cM} \mapsto \R$, we have
\begin{align}
    \int f(\gamma) \dd\bGamma_i(\gamma) &= \int f([\pi_i]^{(n-1)}(\alpha)) \dd\bA(\alpha) \\
    &= \iiint f([\pi_i]^{(n-1)}(\alpha)) \dd\bA_{\gamma_1,\gamma_2}(\alpha) \dd \bA_\mu(\gamma_1,\gamma_2) \bP(\mu) \\
    &= \iint f(\gamma_i) \dd \bA_\mu(\gamma_1,\gamma_2) \bP(\mu) \\
    &= \iint f(\gamma) \dd (\pi_{i\#}\bA_\mu)(\gamma) \bP(\mu)
\end{align}
so by unicity of the disintegration, $\pi_{i,\#}\bA_\mu = \bGamma_{i,\mu}$ for $\bP$-a.e. $\mu$. In particular, for $\bP$-a.e. $\mu$, $\bA_\mu$ is a transport plan between $\bGamma_{1,\mu}$ and $\bGamma_{2,\mu}$. The following necessary and sufficient conditions then characterize the optimality of $\bA$:

\begin{proposition} \label{prop:opt_coupling_characterization}
    The coupling $\bA$ is optimal if and only if both the following conditions hold:
    \begin{enumerate}
        \item \label{enum:opt_coupling:1} For $\bP$-a.e. $\mu$, $\bA_\mu$ is an optimal transport plan between $\bGamma_{1,\mu}$ and $\bGamma_{2,\mu}$ for the ground cost $\W^2_\mu$.
        \item \label{enum:opt_coupling:2} For $\bP$-a.e. $\mu$ and $\bA_\mu$-a.e. $\gamma_1,\gamma_2$, $\bA_{\gamma_1,\gamma_2}$ is concentrated on the set of optimal couplings $\Gamma_{\mu,o}(\gamma_1,\gamma_2)$.
    \end{enumerate}
\end{proposition}

\begin{proof}
    First, assume that there exists $\bA_0 \in \Gamma_\bP(\bGamma_1,\bGamma_2)$ satisfying conditions \ref{enum:opt_coupling:1} and \ref{enum:opt_coupling:2}. We will show that 
    \begin{equation} \label{eq:l_1057}
        \bE^{(n)}_\bA[\|v_1 - v_2\|^2_x] \geq \bE^{(n)}_{\bA_0}[\|v_1-v_2\|_x^2]
    \end{equation}
    with equality if and only if $\bA$ itself satisfies conditions \ref{enum:opt_coupling:1} and \ref{enum:opt_coupling:2}. Indeed, rewriting $\bA_0$ using the disintegration theorem as
    \begin{equation}
        \dd \bA_0(\alpha) = \dd \bA_{0,\gamma_1,\gamma_1}(\alpha) \dd \bA_{0,\mu}(\gamma_1,\gamma_2) \dd \bP(\mu)
    \end{equation}
    just like for $\bA$. Then we have
    \begin{align}
        \bE^{(n)}_{\bA}[\|v_1-v_2\|^2_x] &= \iiint \bE_\alpha^{(n-1)}[\|v_1-v_2\|^2_x] \dd \bA_{\gamma_1,\gamma_1}(\alpha) \dd \bA_\mu(\gamma_1,\gamma_2) \dd \bP(\mu) \\
        &\geq \iiint \W_\mu^2(\gamma_1,\gamma_2) \dd \bA_{\gamma_1,\gamma_1}(\alpha) \dd \bA_\mu(\gamma_1,\gamma_2) \dd \bP(\mu) \label{eq:l_1067} \\
        &= \iint \W_\mu^2(\gamma_1,\gamma_2) \dd \bA_\mu(\gamma_1,\gamma_2) \dd \bP(\mu) \\
        &\geq \iint \W_\mu^2(\gamma_1,\gamma_2) \dd \bA_{0,\mu}(\gamma_1,\gamma_2) \dd \bP(\mu) \label{eq:l_1069} \\
        &= \iint \W_\mu^2(\gamma_1,\gamma_2) \dd \bA_{0,\gamma_1,\gamma_2}(\alpha) \dd \bA_{0,\mu}(\gamma_1,\gamma_2) \dd \bP(\mu) \\
        &= \iint \bE^{(n-1)}_\alpha[\|v_1-v_2\|^2_x] \dd \bA_{0,\gamma_1,\gamma_2}(\alpha) \dd \bA_{0,\mu}(\gamma_1,\gamma_2) \dd \bP(\mu) \\
        &= \bE^{(n)}_{\bA_0}[\|v_1-v_2\|^2_x]
    \end{align}
    where we used in the fourth line that $\bA_{0,\mu}$ is for $\bP$-a.e. $\mu$ an optimal transport plan between $\bGamma_{1,\mu}$ and $\bGamma_{2,\mu}$ for the ground cost $\W_\mu^2$, and we used in the sixth line that $\bA_{0,\gamma_1,\gamma_2}$ is concentrated on $\Gamma_{\mu,o}(\gamma_1,\gamma_2)$ for $\bP$-a.e. $\mu$ and $\bA_{0,\mu}$-a.e. $\gamma_1,\gamma_2$. This shows \eqref{eq:l_1057}. Moreover, the inequality \eqref{eq:l_1067} is an equality if and only if $\bA$ satisfies condition \ref{enum:opt_coupling:2}, and the inequality \eqref{eq:l_1069} is an equality if and only if $\bA$ satisfies condition \ref{enum:opt_coupling:1}, so the inequality \eqref{eq:l_1057} is an equality if and only if $\bA$ satisfies the conditions \ref{enum:opt_coupling:1} and \ref{enum:opt_coupling:2}. We have thus shown that the conditions \ref{enum:opt_coupling:1} and \ref{enum:opt_coupling:2} are necessary and sufficient for a coupling between $\bGamma_1$ and $\bGamma_2$ to be optimal, \emph{provided at least one such coupling $\bA_0$ exists}. We thus finish the proof by constructing $\bA_0$.
    \medbreak
    Applying \Cref{prop:ot_problem_in_fiber_product_has_solutions} to the fiber product $\cPn{n}{T\cM} \times_{\cPn{n}{\cM}} \cPn{n}{T\cM}$, the measures $\bGamma_1$ and $\bGamma_2$, and the cost function $\W^2$ (which is lower semicontinuous in the $\W_2$ topology by \Cref{prop:w_mu_is_lsc}), there exists $\tilde{\bA}_0 \in \cP(\cPn{n}{T\cM} \times_{\cPn{n}{\cM}} \cPn{n}{T\cM})$ which writes as
    \begin{equation}
        \dd\tilde{\bA}_0(\gamma_1,\gamma_2) = \dd\tilde{\bA}_{0,\mu}(\gamma_1,\gamma_2) \dd\bP(\mu)
    \end{equation}
    where $\tilde{\bA}_0$ is for $\bP$-a.e. $\mu$ an optimal transport plan between $\bGamma_{1,\mu}$ and $\bGamma_{2,\mu}$ for the ground cost $\W_\mu^2$. Moreover it is easily checked that $\tilde{\bA}_0$ has finite order 2 moment. Now, let $B_0$ be the set of couplings $\alpha \in \cPn{n}{T^2\cM}$ which are optimal between their marginal velocity plans. Then $B_0$ is a Borel subset of $\cPn{n}{T^2\cM}$, as $B_0 = f^{-1}(0)$ where $f : \cPn{n}{T^2\cM} \mapsto \R$ is the measurable function defined by
    \begin{equation}
        f(\alpha) := \bE^{(n)}_\alpha[\|v_1-v_2\|_x^2] - \W_{[\pi]^{(n)}(\alpha)}^2([\pi_1]^{(n)}(\alpha), [\pi_2]^{(n)}(\alpha))
    \end{equation}
    (recall that lower semicontinuous functions are measurable). Thus, by \Cref{prop:dist_and_inner_prod_are_min_and_max} and Equation \eqref{eq:2nd_moment_of_coupling}, the projection map $([\pi_1]^{(n)},[\pi_2]^{(n)}) : \cPn{n}{T^2\cM} \mapsto \cPn{n}{T\cM} \times_{\cPn{n}{\cM}} \cPn{n}{T\cM}$ has linearly growing antecedents in $B_0$. Therefore, by \Cref{prop:pushforward_by_surjective_is_surjective_in_P2}, there exists $\bA_0 \in \cPn{n+1}{T^2\cM}$ which is concentrated on $B_0$ and such that $[[\pi_1]^{(n)},[\pi_2]^{(n)}](\bA_0) = \tilde{\bA}_0$. Using the disintegration theorem, we write
    \begin{equation}
        \dd\bA_0(\alpha) = \dd\bA_{0,\gamma_1,\gamma_2}(\alpha) \dd\tilde{\bA}_0(\gamma_1,\gamma_2) = \dd\bA_{0,\gamma_1,\gamma_2}(\alpha) \dd\tilde{\bA}_{0,\mu}(\gamma_1,\gamma_2) \dd\bP(\mu)
    \end{equation}
    where $\bA_{0,\gamma_1,\gamma_2}$ is concentrated on $\Gamma_{[\pi]^{(n)}(\gamma_1)}(\gamma_1,\gamma_2)$ for $\tilde{\bA}_0$-a.e. $\gamma_1,\gamma_2$ (that is for $\bP$-a.e. $\mu$ and $\tilde{\bA}_{0,\mu}$-a.e. $\gamma_1,\gamma_2$). But since $\bA_0$ is concentrated on $B_0$, this implies that $\bA_{0,\gamma_1,\gamma_2}$ is concentrated on $\Gamma_{\mu,o}(\gamma_1,\gamma_2)$ for $\bP$-a.e. $\mu$ and $\tilde{\bA}_{0,\mu}$-a.e. $\gamma_1,\gamma_2$. Thus $\bA_0$ satisfies conditions \ref{enum:opt_coupling:1} and \ref{enum:opt_coupling:2}, and this finishes the proof.
\end{proof}

\section{Fully deterministic velocity plans} \label{sec:4_fully_det}

\subsection{Fully deterministic measures}

The goal of this subsection is to elaborate a theory of ``fully deterministic measures", which will be in turn used to define a generalization to the spaces $\cPn{n}{\cM}$ of the notion of transport maps. As we strive for full generality, we will fix a category $\cC$ and a functor $P : \cC \mapsto \cC$, which may be either $\Pol$ and $\cP$ or $\Pold$ and $\cP_2$, and we will formulate the results of this subsection in terms of $\cC$ and $P$.
\medbreak
In this subsection, we will fix a morphism $\pi : V \mapsto X$, which we will assume to be surjective and, if $\cC = \Pold$, with linearly growing antecedents. For every $n \geq 0$ and $\mu \in \Pn{n}{X}$, we define $\Pn{n}{V}_\mu$ to be the set of $\gamma \in \Pn{n}{V}$ such that $[\pi]^{(n)}(\gamma) = \mu$, and for every $\gamma_1, \gamma_2 \in \Pn{n}{V}_\mu$, we define $\Gamma_\mu(\gamma_1,\gamma_2)$ to be the set of $\alpha \in \Pn{n}{V \times_X V}$ such that $[\pi_1]^{(n)}(\alpha) = \gamma_1$ and $[\pi_2]^{(n)}(\alpha) = \gamma_2$. Note that these notations are consistent with those of the previous section as they define the same sets when $\pi$ is the projection $T\cM \mapsto \cM$ in $\Pold$. Similarly, we know, using  \Cref{prop:p_of_fiber_to_fiber_of_p_is_proper} and \Cref{lemma:hierarchical_gluing_lemma}, that $\Gamma_\mu(\gamma_1,\gamma_2)$ is not empty and is a compact subset of $\Pn{n}{V \times_X V}$, and using \Cref{th:selection}, that there is a choice $\alpha_{\gamma_1,\gamma_2} \in \Gamma_\mu(\gamma_1,\gamma_2)$ which is measurable. Given a morphism $f : Y \mapsto Z$ in $\cC$ and $\mu \in \cP(Y)$, we will call a \emph{$\mu$-section} of $f$ any Borel measurable map $s : Z \mapsto Y$ such that $f(s(z)) = z$ for $\mu$-almost every $z \in Z$.

\begin{definition}
    The set of \emph{deterministic probability measures} on $V$ relatively to $\pi$, denoted $P_{\det,\pi}(V)$ or $P_{\det}(V)$, is the set of measures $\gamma \in P(V)$ for which there exists a $\mu$-section $s : X \mapsto V$ of $\pi$ such that $\gamma = s_\#\mu$, where $\mu := \pi_\#\gamma \in P(X)$.
\end{definition}

\begin{definition} \label{def:fully_det_measure}
    For every $n \geq 0$, we define inductively the set of \emph{fully deterministic probability measures} on $V$ relatively to $\pi$, denoted $P^{(n)}_{\det,\pi}(V)$ or more simply $\Pndet{n}{V}$, by setting $\Pndet{0}{V} := V$, and for $n > 0$, by defining $\Pndet{n}{V}$ to be the set of measures $\gamma \in P(\Pndet{n-1}{V})$ of the form $\gamma = s_\#\mu$ where $\mu := [\pi]^{(n)}(\gamma)$ and $s : \Pn{n-1}{X} \mapsto \Pn{n-1}{V}$ is a $\mu$-section of $[\pi]^{(n-1)}$ which is valued $\mu$-almost everywhere in $\Pndet{n-1}{V}$.
\end{definition}

Note that we indeed have $\Pndet{1}{V} = P_{\det}(V)$. For every $\mu \in \Pn{n}{X}$, we adopt the natural choice of notation $\Pndet{n}{V}_\mu := \Pndet{n}{V} \cap \Pn{n}{V}_\mu$.

\begin{proposition} \label{prop:coupling_with_fully_det_is_unique_general}
    Let $n \geq 0$, $\mu \in \Pn{n}{X}$ and $\gamma_1, \gamma_2 \in \Pndet{n}{V}_\mu$. If at least one of the $\gamma_1$ and $\gamma_2$ is fully deterministic, then $\Gamma_\mu(\gamma_1,\gamma_2)$ is a singleton.
\end{proposition}

\begin{proof}
    We show this by induction. In the case $n = 0$, this is immediate, since $\Gamma_x(v_1,v_2) = \{(v_1,v_2)\}$ is always a singleton. Let $n > 0$ and assume the proposition holds for $n-1$. Let $\bP \in \Pn{n}{X}$, $\bGamma_1,\bGamma_2 \in \Pn{n}{V}_\bP$, and assume that at least one of them is fully deterministic. Since $[(\pi_2,\pi_1)]^{(n)}$ induces a bijection between $\Gamma_\bP(\bGamma_1,\bGamma_2)$ and $\Gamma_\bP(\bGamma_2,\bGamma_1)$, we may assume without loss of generality that it is $\bGamma_1$ which is fully deterministic, and we write $\bGamma_1 = s_{1\#}\bP$ where $s_1 : \Pn{n-1}{X} \mapsto \Pn{n-1}{V}$ is a $\bP$-section of $[\pi]^{(n-1)}$ which is valued $\bP$-a.e. in $\Pndet{n-1}{V}$. We know that $\Gamma_\mu(\gamma_1,\gamma_2) \neq \emptyset$. Fix then $\bA \in \Gamma_\mu(\gamma_1,\gamma_2)$, and apply the disintegration theorem (\Cref{th:disintegration}) to write it as
    \begin{equation}
        \dd\bA(\alpha) = \dd\bA_{\gamma_1,\gamma_2}(\alpha) \dd\bA_\mu(\gamma_1,\gamma_2) \dd\bP(\mu)
    \end{equation}
    where $\bA_\mu$ is (for $\bP$-a.e. $\mu$) concentrated on $\Pn{n-1}{V}_\mu \times \Pn{n-1}{V}_\mu$ and $\bA_{\gamma_1,\gamma_2}$ is (for $\bP$-a.e. $\mu$, $\bA_\mu$-a.e. $\gamma_1,\gamma_2$) concentrated on $\Gamma_\mu(\gamma_1,\gamma_2)$. We similarly rewrite $\bGamma_1$ and $\bGamma_2$ as
    \begin{equation}
        \dd\bGamma_1(\gamma_1) = \dd\delta_{s_1(\mu)}(\gamma_1) \dd\bP(\mu), \quad \dd\bGamma_2(\gamma_2) = \dd\bGamma_{2,\mu}(\gamma_2) \dd\bP(\mu)
    \end{equation}
    where $\bGamma_{2,\mu}$ is (for $\bP$-a.e. $\mu$) supported on $\Pn{n-1}{V}_\mu$. We can check furthermore that for $\bP$-a.e. $\mu$, $\bA_\mu$ is a transport plan between $\delta_{s_1(\mu)}$ and $\bGamma_{2,\mu}$, so that $\bA_\mu = \delta_{s_1(\mu)} \otimes \bGamma_{2,\mu}$. There is thus a set $A \subseteq \Pn{n-1}{X}$ of full $\bP$-measure such that $s_1(\mu) \in \Pndet{n-1}{V}_\mu$ and $\bA_\mu = \delta_{s_1(\mu)} \otimes \bGamma_{2,\mu}$ for every $\mu \in A$. Moreover, up to restricting $A$, we may also assume that for every $\mu \in A$, there is a set $A_\mu$ of full $\bA_\mu$-measure such that $\bA_{\gamma_1,\gamma_2}$ is concentrated on $\Gamma_\mu(\gamma_1,\gamma_2)$ for every $(\gamma_1,\gamma_2) \in A_\mu$. Fixing $\mu \in A$, since $\bA_\mu = \delta_{s_1(\mu)} \otimes \bGamma_{2,\mu}$, we may assume up to restricting $A_\mu$ that $A_\mu = \{s_1(\mu)\} \times B_\mu$ where $B_\mu$ is of full $\bGamma_{2,\mu}$-measure. But since $s_1(\mu) \in \Pndet{n-1}{V}_\mu$ is fully deterministic, by the induction hypothesis, for every $\gamma_2 \in B_\mu$, there is a unique element in $\Gamma_\mu(s_1(\mu),\gamma_2)$ which is $\alpha_{s_1(\mu),\gamma_2}$ (where $\eta_1,\eta_2 \mapsto \alpha_{\eta_1,\eta_2}$ is the measurable selection defined above), so that $\bA_{s_1(\mu),\gamma_2}$, which is concentrated on $\Gamma_\mu(s_1(\mu),\gamma_2)$, is equal to $\delta_{\alpha_{s_1(\mu),\gamma_2}}$. Therefore, for every measurable $f : \Pn{n-1}{V \times_X V} \mapsto \R_+$, we have
    \begin{align}
        \int f(\alpha) \dd\bA(\alpha) &= \iint f(\alpha) \dd\bA_{\gamma_1,\gamma_2}(\alpha) \dd\bA_\mu(\gamma_1,\gamma_2) \dd\bP(\mu) \\
        &= \int_A \int_{A_\mu} \int f(\alpha) \dd\bA_{\gamma_1,\gamma_2}(\alpha) \dd(\delta_{s_1(\mu)} \otimes \bGamma_{2,\mu})(\gamma_1,\gamma_2) \dd\bP(\mu) \\
        &= \int_A \int_{B_\mu} \int f(\alpha) \dd\bA_{s_1(\mu),\gamma_2} \dd\bGamma_{2,\mu}(\gamma_2) \dd\bP(\mu) \\
        &= \int_A \int_{B_\mu} f(\alpha_{s_1(\mu),\gamma_2}) \dd\bGamma_{2,\mu}(\gamma_2) \dd\bP(\mu) \\
        &= \int f(\alpha_{s_1([\pi]^{(n-1)}(\gamma)), \gamma}) \dd\bGamma_2(\gamma).
    \end{align}
    Therefore, we have $\bA = (\gamma \mapsto \alpha_{[\pi]^{(n-1)}(\gamma),\gamma})_\#\bGamma_2$, and in particular $\bA$ is uniquely defined by $s_1$ and $\bGamma_2$. This finishes the proof.
\end{proof}

\begin{example} \label{ex:morphism_induces_fully_det}
    For any morphism $f : X \mapsto V$ such that $\pi \circ f = \id_X$, for every $\mu \in \Pn{n}{X}$, $\gamma := [f]^{(n)}(\mu)$ is fully deterministic. Moreover, for any other $\eta \in \Pn{n}{V}_\mu$, the unique $\alpha \in \Gamma_\mu(\gamma,\eta)$ is given by $\alpha = [v \mapsto (f(\pi(v)), v)]^{(n)}(\eta)$.  This is straightforward to check by induction.
\end{example}

We now define an ``unrolling" operation on measures, which will prove useful when investigating the links between fully deterministic measures in $\Pndet{n}{V}$ and measurable functions $\Pn{n-1}{X} \times \ldots \times P(X) \times X \mapsto V$.

\begin{proposition}
    Let $X$ and $Y$ be two elements of $\cC$. Then there exists a morphism $\cU_{X,Y} : P(X \times P(Y)) \mapsto P(X \times P(Y) \times Y)$ in $\cC$, such that for every measure $\bP \in P(X \times P(Y))$, $\tilde{\bP} := \cU_{X,Y}(\bP)$ is the unique measure in $P(X \times P(Y) \times Y)$ such that for every Borel measurable $f : X \times P(Y) \times Y \mapsto \R_+$,
    \begin{equation} \label{eq:unrolled_prob_definition}
        \int f(x,\mu,y) \dd\tilde{\bP}(x,\mu,y) = \iint f(x,\mu,y) \dd\mu(y) \dd\bP(x,\mu).
    \end{equation}
    Moreover, the morphisms $\cU_{X,Y}$ are natural in $X,Y$, in the sense that for every pair of morphisms $f : X \mapsto X'$ and $g: Y \mapsto Y'$ in $\cC$, the equality of morphisms $[f,[g],g] \circ \cU_{X,Y} = \cU_{X',Y'} \circ [f,[g]]$ holds\footnote{In other words, the $\cU_{X,Y}$ define a natural transformation $\cU$ between the functors $F_1, F_2 : \cC \times \cC \mapsto \cC$ defined by $F_1(X,Y) := P(X \times P(Y))$ and $F_2(X,Y) := P(X \times P(Y) \times Y)$.}.
\end{proposition}

\begin{proof}
    We first consider the case $\cC = \Pol$ and $P = \cP$. Let $X$ and $Y$ be two Polish spaces, and fix $\bP \in \cP(X \times \cP(Y))$. Then the family $(\mu)_{(x,\mu) \in X \times \cP(Y)}$ is clearly a Borel family of measures, so that there exists a unique measure $\tilde{\bP} \in \cP(X \times \cP(Y) \times Y)$ satisfying \eqref{eq:unrolled_prob_definition}, and $\cU_{X,Y}(\bP) := \tilde{\bP}$ is well-defined. We want to prove that $\cU_{X,Y}$ is continuous for the topology of weak convergence. Let $(\bP_n)_n$ be a sequence in $\cP(X \times \cP(Y))$ such that $\bP_n \rightharpoonup \bP$. Denote $\tilde{\bP}_n := \cU_{X,Y}(\bP_n)$ and $\tilde{\bP} := \cU_{X,Y}(\bP)$, we want to prove that $\tilde{\bP}_n \rightharpoonup \tilde{\bP}$. For this, set $\varphi \in C_b(X \times \cP(X) \times X)$ a continuous and bounded map. Then, we have
    \begin{align}
        \int f(x,\mu,y) \dd\tilde{\bP}_n(x,\mu,y) &= \iint f(x,\mu,y) \dd\mu(y) \dd\bP_n(x,\mu) \\
        &\xrightarrow[n \to +\infty]{} \iint f(x,\mu,y) \dd\mu(y) \dd\bP(x,\mu) \\
        &= \int f(x,\mu,y) \dd\tilde{\bP}(x,\mu,y)
    \end{align}
    where we obtained the second line by using the weak convergence of $\bP_n$ to $\bP$ in conjunction with the continuity of the map $(x,\mu) \mapsto \int f(x,\mu,y)\dd\mu(y)$ (which follows by a standard argument in measure theory, see \eqref{eq:partial_integral_weakly_continuous}). Therefore, $\tilde{\bP}_n \rightharpoonup \tilde{\bP}$, and $\cU_{X,Y}$ is continuous. We now prove the naturality of the $\cU_{X,Y}$: let $f : X \mapsto X'$ and $g : Y \mapsto Y'$ be two continuous maps in $\Pol$ and $\bP \in \cP(X \times \cP(Y))$, and set $\tilde{\bP} := \cU_{X,Y}(\bP)$, $\bP' := [f,[g]](\bP)$ and $\tilde{\bP}' := \cU_{X',Y'}(\bP')$, we want to prove that $\tilde{\bP}' = [f,[g],g](\tilde{\bP})$. We have for every Borel measurable $\varphi : X' \times \cP(Y') \times Y' \mapsto \R_+$ that
    \begin{align}
        \int \varphi \dd([f,[g],g](\tilde{\bP})) &= \int \varphi(f(x),[g](\mu),g(y)) \dd\tilde{\bP}(x,\mu,y) \\
        &= \iint \varphi(f(x),[g](\mu),g(y)) \dd\mu(y) \dd\bP(x,\mu) \\
        &= \iint \varphi(f(x),[g](\mu),y') \dd([g](\mu))(y) \dd\bP(x,\mu)
    \end{align}
    and likewise
    \begin{align}
        \int \varphi \dd\tilde{\bP}' &= \iint \varphi(x',\mu',y') \dd\mu'(y') \dd\bP'(x',\mu') \\
        &= \iint \varphi(x',\mu',y') \dd\mu'(y') \dd([f,[g]](\bP))(x',\mu') \\
        &= \iint \varphi(f(x),[g](\mu),y') \dd([g](\mu))(y') \dd\bP(x,\mu)
    \end{align}
    so these two integrals are equal, and we thus indeed have $\tilde{\bP}' = [f,[g],g](\tilde{\bP})$. This finishes proving the proposition in the case $\cC = \Pol$. \newline
    Now, we prove the case where $\cC = \Pold$ and $P = \cP_2$. Let $(X,d_X)$ and $(Y,d_Y)$ be two Polish spaces. All we need to do is to prove that the map $\cU_{X,Y} : \cP(X \times \cP(Y)) \mapsto \cP(X \times \cP(Y) \times Y)$ restricts to a map $\cP_2(X \times \cP_2(Y)) \mapsto \cP_2(X \times \cP_2(Y) \times Y)$ which is continuous with linear growth in the $\W_2$ topology. Fix $\bP \in \cP_2(X \times \cP_2(Y))$ and set $\tilde{\bP} := \cU_{X,Y}(\bP)$. Since $\bP$ is the marginal of $\tilde{\bP}$ on the first two variables, we have $\tilde{\bP} \in \cP(X \times \cP_2(Y) \times Y)$. Moreover, for every $(x_0,y_0) \in X \times Y$, we have
    \begin{align}
        \W_2^2(\delta_{(x_0,\delta_{y_0},y_0)}, \tilde{\bP}) &= \int d^2((x_0,\delta_{y_0},y_0),(x,\mu,y)) \dd\tilde{\bP}(x,\mu,y) \\
        &= \iint d^2((x_0,\delta_{y_0},y_0),(x,\mu,y)) \dd\mu(y) \dd\bP(x,\mu) \\ 
        &= \iint d^2(x_0,x) + \W_2^2(\delta_{y_0},\mu) + d^2(y_0,y)\dd\mu(y) \dd\bP(x,\mu) \\ 
        &= \int d^2(x_0,x) + 2\W_2^2(\delta_{y_0},\mu) \dd\bP(x,\mu) \label{eq:l_1258} \\ 
        &\leq 2\int d^2((x_0,\delta_{y_0}),(x,\mu))\dd\bP(x,\mu) = 2\W_2^2(\delta_{(x_0,\delta_{y_0})},\bP) < +\infty
    \end{align}
    so from this we see that $\bP$ is indeed in $\cP_2(X \times \cP_2(Y) \times Y)$ and that $\cU_{X,Y}$ restricted to $\cP_2(X \times \cP_2(Y))$ has linear growth (with respect to the $\W_2$ distance). All that is left is to check that it is $\W_2$-continuous. Let $(\bP_n)_n$ be a sequence in $\cP_2(X \times \cP_2(Y))$ such that $\bP_n \to \bP$ in the $\W_2$ topology. Denote $\tilde{\bP}_n := \cU_{X,Y}(\bP_n)$ and $\tilde{\bP} := \cU_{X,Y}(\bP)$, we want to prove that $\tilde{\bP}_n \to \tilde{\bP}$. Since we already now that $\bP_n \rightharpoonup \bP$, the $\W_2$-convergence of $\bP_n$ will follow from point \ref{enum:w2_convergence:1_2nd_moment_cvg} of \Cref{th:convergence_in_w2_space} and from the fact that for every $n$,
    \begin{align}
        \W_2^2(\delta_{(x_0,\delta_{y_0},y_0)}, \tilde{\bP}_n) &= \int d^2(x_0,x) + 2\W_2^2(\delta_{y_0},\mu) \dd\bP_n(x,\mu) \\
        &\xrightarrow[n \to +\infty]{} \int d^2(x_0,x) + 2\W_2^2(\delta_{y_0},\mu) \dd\bP(x,\mu) \\
        &= \W_2^2(\delta_{(x_0,\delta_{y_0},y_0)}, \tilde{\bP}),
    \end{align}
    where we used \eqref{eq:l_1258} in the first and the third line, and we used point \ref{enum:w2_convergence:4_quadratic_growth} of \Cref{th:convergence_in_w2_space} to obtain the second line. This finishes the proof.
\end{proof}

For every $Y$ in $\cC$ and every $n > k \geq 0$, we define a map $\cU_Y^{(n,k)} : \Pn{n}{Y} \mapsto P(\Pn{n-1}{Y} \times \ldots \times \Pn{k}{Y})$ by
\begin{equation}
    \cU_Y^{(n,k)} := \cU_{\Pn{n-1}{Y} \times \ldots \times \Pn{k+2}{Y}, \Pn{k}{Y}} \circ \ldots \circ \cU_{\Pn{n-1}{Y}, \Pn{n-3}{Y}} \circ \cU_{\{*\},\Pn{n-2}{Y}}.
\end{equation}
It associates to every $\bP \in \Pn{n}{Y}$ the only probability measure $\tilde{\bP} := \cU_Y^{(n,k)}(\bP)$ such that for every measurable $f : \Pn{n-1}{X} \times \ldots \times \Pn{k}{Y} \mapsto \R_+$,
\begin{equation}
    \int f \dd\tilde{\bP} = \iint \ldots \iint f(\mu^{n-1},\ldots,\mu^k) \dd\mu^{k+1}(\mu^k) \dd\mu^{k+2}(\mu^{k+1}) \ldots \dd\mu^{n-1}(\mu^{n-2}) \dd\bP(\mu^{n-1}).
\end{equation}
In particular, notice that for every $\bP \in \Pn{n}{Y}$, for every $n > l > k \geq 0$, the marginal on the first $n-l$ variables of $\cU_Y^{(n,k)}(\bP)$ is
\begin{equation} \label{eq:marginal_of_unrolled}
    \pi_{1,\ldots,n-l\#}\cU_Y^{(n,k)}(\bP) = \cU_Y^{(n,l)}(\bP) \in P(\Pn{n-1}{Y} \times \ldots \times \Pn{l}{Y}).
\end{equation}
Moreover, the disintegration of $\cU_Y^{(n,k)}(\bP)$ with respect to the first $n-l$ variables is given for $\cU_Y^{(n,l)}(\bP)$-a.e. $\mu^{n-1},\ldots,\mu^l$ by
\begin{equation} \label{eq:disintegration_of_unrolled}
    \cK((\mu^{n-1},\ldots,\mu^l), \cU_Y^{(n,k)}(\bP)) = \cU_Y^{(l,k)}(\mu^l).
\end{equation}
where $\cK$ is the disintegration map defined in \Cref{th:universal_disintegration} (we omit the subscript $\cK_{X,Y}$ for the sake of brevity). We will also denote $\cU_Y^{(n)} := \cU_Y^{(n,0)}$.
\medbreak
In the following, for any $n \geq 1$ and $\bP \in \Pn{n}{X}$, we will denote by $\cM_{\pi,\bP}$ the set of Borel measurable functions $f : \Pn{n-1}{X} \times \ldots \times P(X) \times X \mapsto V$ such that $\pi(f(\mu^{n-1},\ldots,\mu,x)) = x$ for $\cU_X^{(n)}(\bP)$-almost every $(\mu^{n-1},\ldots,\mu,x)$. In particular, for $n = 1$, for every $\mu \in P(X)$, $\cM_{\pi,\mu}$ is simply the set of $\mu$-sections of $\pi$.

\begin{proposition} \label{prop:maps_induce_fully_det}
    For every $n \geq 1$ and $\bP \in \Pn{n}{X}$, we can construct a map $j_\bP : \cM_{\pi,\bP} \mapsto \cPzndet{n}{V}_\bP$\footnote{Note that $j_n$ is valued in $\cPzn{n}{V}$ and not in $\Pn{n}{V}$ !} by induction on $n$, in the following way:
    \begin{itemize}
        \item For $n = 1$ and $\mu \in P(X)$, we set $j_\mu(f) := f_\#\mu$ for every $f \in \cM_{\pi,\mu}$.
        \item For $n \geq 2$ and $\bP \in \Pn{n}{X}$, we set $j_\bP(f) := (\mu \mapsto j_{\mu}(f(\mu,\cdot)))_\#\bP$ for every $f \in \cM_{\pi,\bP}$ (where $f(\mu,\cdot)$ is the partial evaluation of $f$ at $\mu$).
    \end{itemize}
\end{proposition}

\begin{proof}
    We must prove that the $j_\mu$ are well-defined. We show this by induction on $n$. For $n = 1$, if $\mu \in P(X)$ and $f \in \cM_{\pi,\mu}$ is a $\mu$-section of $\pi$, then clearly $j_\mu(f) := f_\#\mu$ is well-defined, and is an element of $\Pndet{1}{V}_\mu = P_{\det}(V)_\mu$ by definition of deterministic measures. \newline
    Now, let $n > 1$, and assume that the proposition holds for $1,\ldots,n-1$. Fix $f : \Pn{n-1}{X} \times \ldots P(X) \times X \mapsto V$ an element of $\cM_{\pi,\bP}$. Let $0 \leq k \leq n-1$, then, by \eqref{eq:marginal_of_unrolled} and \eqref{eq:disintegration_of_unrolled}, the partial evaluation $f(\mu^{n-1},\ldots,\mu^k,\cdot)$ is an element of $\cM_{\pi,\mu^k}$ for $\cU_X^{(n,k)}(\bP)$-almost every $(\mu^{n-1},\ldots,\mu^k)$, so that $j_{\mu^k}(f(\mu^{n-1},\ldots,\mu^k,\cdot))$ is well-defined by the induction hypothesis (for $k = 0$, we set $j_x(v) = v$ for every $x \in X$, $v \in V$). In particular, the function 
    \begin{equation}
        j_{n,k} : \left\{\begin{array}{ccc}
             \Pn{n-1}{X} \times \ldots \Pn{k}{X} & \rightarrow & \cPzn{k}{V} \\
             (\mu^{n-1}, \ldots, \mu^k) & \rightarrow & j_{\mu^k}(f(\mu^{n-1},\ldots,\mu^k,\cdot)) 
        \end{array}\right.
    \end{equation}
    is well-defined $\cU_X^{(n,k)}(\bP)$-almost everywhere. We want to prove that it is Borel measurable. We again demonstrate this by induction on $k$. For $k = 0$ this is trivial since $j_{n,0}$ is just the evaluation map $j_{n,0}(\mu^{n-1},\ldots,\mu^1,x) = f(x)$. Let $k > 0$, and assume that $j_{n,k-1}$ is measurable. To prove the measurability of $j_{n,k}$, we only need to prove that for every $\varphi \in C_b(\cPzn{k-1}{V})$ continuous and bounded, the function
    \begin{equation} \label{eq:l_1300}
        (\mu^{n-1},\ldots,\mu^k) \in \Pn{n-1}{X} \times \ldots \times \Pn{k}{X} \mapsto \int \varphi \dd j_{n,k}(\mu^{n-1},\ldots,\mu^k)
    \end{equation}
    is measurable (see the discussion in \Cref{sec:appendix:disintegration}). Let $\cB$ be the class of Borel measurable functions $z : \Pn{n-1}{X} \times \ldots \Pn{k-1}{X} \mapsto \R$ for which the map 
    \begin{equation}
        (\mu^n,\ldots,\mu^k) \in \Pn{n-1}{X} \times \ldots \times \Pn{k}{X} \mapsto \int z(\mu^n,\ldots,\mu^k,\mu^{k-1}) \dd\mu^k(\mu^{k-1})
    \end{equation}
    is Borel measurable. Then, by \eqref{eq:partial_integral_weakly_continuous}, $\cB$ contains the continuous and bounded functions. It is not difficult to check that $\cB$ is also closed under uniform and monotone limits, so that, by the functional monotone class theorem \citep[Theorem 2.12.9]{bogachev2007measure}, $\cB$ contains all the bounded Borel functions. In particular, taking $z = \varphi \circ j_{n,k-1}$, which is bounded (as $\varphi$ is bounded) and Borel (as $j_{n,k-1}$ is Borel by the induction hypothesis), and using the fact that for $\cU_X^{(n,k)}(\bP)$-almost every $\mu^n,\ldots,\mu^k \in \Pn{n}{X} \times \ldots \times \Pn{k}{X}$ we have
    \begin{align}
        \int \varphi \dd j_{n,k}(\mu^n,\ldots,\mu^k) &= \int \varphi \dd j_{\mu^k}(f(\mu^n,\ldots,\mu^k,\cdot)) \\
        &= \int \varphi(j_{\mu^{k-1}}(f(\mu^n,\ldots,\mu^k,\mu^{k-1},\cdot))) \dd\mu^k(\mu^{k-1}) \\
        &= \int \varphi(j_{n,k-1}(\mu^n,\ldots,\mu^{k-1})) \dd\mu^k(\mu^{k-1}),
    \end{align}
    we deduce that the map defined in \eqref{eq:l_1300} is Borel measurable. This finishes the induction on $k$. \newline
    Therefore, the map $s : \cPzn{n-1}{X} \mapsto \cPzn{n-1}{V}$ defined by $s(\mu) := j_{n,n-1}(\mu) = j_\mu(f(\mu,\cdot))$ is well-defined $\bP$-almost everywhere (as $\cU_X^{(n,n-1)}(\bP) = \bP$) and Borel measurable. Moreover, by the induction hypothesis on the $j_\mu$ for $\mu \in \Pn{n-1}{X}$, $s$ is a $\bP$-section of $[\pi]^{(n-1)}$ valued $\bP$-a.e. in $\cPzndet{n-1}{V}$. Thus, the measure $j_\bP(f) := s_\#\bP$ is well-defined, and is an element of $\cPzndet{n}{V}_\bP$ by definition of fully deterministic measures. This finishes the proof.
\end{proof}

We will now prove that in the case the maps $j_\mu$ are surjective:

\begin{proposition} \label{prop:fully_det_induced_by_map}
    Let $n \geq 1$, $\gamma \in \Pndet{n}{V}$, and $\mu := [\pi]^{(n)}(\gamma)$. Define the probability measure $\tilde{\gamma} \in P(\Pn{n-1}{X} \times \ldots \times P(X) \times X \times V)$ by $\tilde{\gamma} := [[\pi]^{(n-1)},\ldots,[\pi],\pi,\id_V] \circ \cU_V^{(n)}(\gamma)$. Then, there exists a map $f \in \cM_{\pi,\mu}$ such that $\tilde{\gamma} = (\id,f)_\#\cU_X^{(n)}(\mu)$. Moreover, this $f$ satisfies $j_\mu(f) = \gamma$.
\end{proposition}

\begin{proof}
    We prove this by induction. In the case $n = 1$, if $\gamma \in \Pndet{1}{V}$ with $\mu = \pi_\#\gamma$, then there exists a $\mu$-section $f : X \mapsto V$ of $\pi$ such that $\gamma = f_\#\mu$. We then have $f \in \cM_{\pi,\mu}$ (by definition), and $j_\mu(f) = f_\#\mu = \gamma$. This shows the proposition in the case $n = 1$. \newline
    Now, let $n > 1$, and assume that the proposition holds for $n-1$. Let $\bGamma \in \Pndet{n}{V}$, and $\bP := [\pi]^{(n)}(\bGamma)$. Let $\tilde{\bGamma} := [[\pi]^{(n-1)},\ldots,[\pi],\pi,\id_V](\cU_V^{(n)}(\bGamma))$. Notice that by naturality of $\cU$, we have
    \begin{equation}
        [[\pi]^{(n-1)},\ldots,[\pi],\pi](\cU_V^{(n)}(\bGamma)) = \cU_X^{(n)}([\pi]^{(n)}(\bGamma)) = \cU_X^{(n)}(\bP).
    \end{equation}
    In other words the marginal of $\bGamma$ with respect to the first $n$ variables is $\cU_X^{(n)}(\bP)$. Let $s : \Pn{n-1}{X} \mapsto \Pn{n-1}{V}$ be a $\bP$-section of $[\pi]^{(n-1)}$, valued $\bP$-a.e. in $\Pndet{n-1}{V}$, and such that $\bGamma = s_\#\bP$. Then, for every Borel measurable map $f : \Pn{n-1}{X} \times \ldots \times P(X) \times X \times V \mapsto \R_+$, we have
    \begin{align}
        & \int f(\mu^{n-1},\ldots,\mu,x,v) \dd\tilde{\bGamma}(\mu^{n-1},\ldots,\mu,x,v) \\
        = &\int f([\pi]^{(n-1)}(\gamma^{n-1}),\ldots,[\pi](\gamma),\pi(v),v) \dd\cU_V^{(n)}(\bGamma)(\gamma^{n-1},\ldots,\gamma,v) \\
        = & \iint \ldots \iint f([\pi]^{n-1}(\gamma^{n-1}),\ldots,[\pi](\gamma),\pi(v),v) \dd\gamma(v)\dd\gamma^2(\gamma) \ldots \dd\gamma^{n-1}(\gamma^{n-2}) \dd\bGamma(\gamma^{n-1}) \\
        = & \iint \ldots \iint f(\mu^{n-1},\ldots,[\pi](\gamma),\pi(v),v) \dd\gamma(v)\dd\gamma^2(\gamma) \ldots \dd(s(\mu^{n-1}))(\gamma^{n-2}) \dd\bP(\mu^{n-1}) \\
        = & \iint f(\mu^{n-1},\ldots,\mu,x,v) \dd \tilde{s}(\mu^{n-1})(\mu^{n-2},\ldots,\mu,x,v) \dd \bP(\mu^{n-1})
    \end{align}
    where we note $\tilde{s}(\mu) := [[\pi]^{(n-2)},\ldots,[\pi],\pi,\id_V](\cU_V^{(n-1)}(s(\mu)))$ for every $\mu \in \Pn{n-1}{X}$. Thus, we see that the marginal of $\tilde{\bGamma}$ with respect to the first variable is $\bP$, and that
    \begin{equation}
        \cK(\mu^{n-1}, \tilde{\bGamma}) = \tilde{s}(\mu^{n-1}), \quad \bP-\hbox{a.e.} \mu^{n-1}.
    \end{equation}
    Moreover, since $s(\mu^{n-1}) \in \Pndet{n-1}{V}_\mu$ for $\bP$-almost every $\mu^{n-1}$, there exists by the induction hypothesis $f_{\mu^{n-1}} \in \cM_{\pi,\mu^{n-1}}$ such that $\tilde{s}(\mu^{n-1}) = (\id,f_{\mu^{n-1}})_\#\cU_X^{(n-1)}(\mu^{n-1})$ and $j_{\mu^{n-1}}(f_{\mu^{n-1}}) = s(\mu^{n-1})$. Thus, by \Cref{rk:disintegration_given_by_map_of_two_variables}, there exists a mesurable $f : \Pn{n-1}{X} \times \ldots \times P(X) \times X \mapsto V$ such that $(\id,f)_\#\cU_X^{(n)}(\bP) = \tilde{\bGamma}$. Furthermore, for $\bP$-a.e. $\mu^{n-1}$, $f(\mu^{n-1},\cdot)$ coincides $\cU_X^{(n-1)}(\mu^{n-1})$-a.e. with $f_{\mu^{n-1}}$, and from this we also conclude that $f \in \cM_{\pi,\bP}$. All that is left to do is to prove $j_\bP(f) = \bGamma$. This is the case since
    \begin{equation}
        j_\bP(f) = (\mu^{n-1} \mapsto j_{\mu^{n-1}}(f(\mu^{n-1},\cdot)))_\#\bP = (\mu^{n-1} \mapsto j_{\mu^{n-1}}(f_{\mu^{n-1}}))_\#\bP = s_\#\bP = \bGamma.
    \end{equation}
    This finishes the proof.
\end{proof}

\subsection{The case of velocity plans}

Now, we will apply our theory of fully deterministic measures to the case where $\cC = \Pold$ and where $\pi$ is the projection map $T\cM \mapsto \cM$. We will call the elements of $\cPn{n}{T\cM}$ the \emph{fully deterministic velocity plans}.

\begin{proposition} \label{prop:coupling_with_fully_det_is_unique}
    Let $n \geq 0$, $\mu \in \cPn{n}{\cM}$ and $\gamma_1, \gamma_2 \in \cPn{n}{T\cM}_\mu$. If $\gamma_1$ is fully deterministic, then $\Gamma_\mu(\gamma_1,\gamma_2)$ is a singleton.
\end{proposition}

\begin{proof}
    This is simply a consequence of \Cref{prop:coupling_with_fully_det_is_unique_general} applied to $\pi : T\cM \mapsto \cM$ in $\Pold$.
\end{proof}

\begin{remark} \label{rk:unique_coupling_expression}
    Let $n > 0$, $\mu \in \cPn{n}{\cM}$ and $\gamma_1, \gamma_2 \in \cPn{n}{T\cM}_\bP$, with $\gamma_1$ fully deterministic. Let $s_1 : \cPn{n-1}{\cM} \mapsto \cPndet{n-1}{T\cM}$ be a measurable section of $[\pi]^{(n-1)}$ such that $\gamma_1 = s_{1\#}\mu$. Then the proof of \Cref{prop:coupling_with_fully_det_is_unique_general} gives an expression of the unique coupling $\alpha \in \Gamma_\mu(\gamma_1,\gamma_2)$: it is given by $\alpha = s(s_1 \circ [\pi]^{(n-1)}, \id)_\#\gamma_2$, where $s$ is the measurable coupling selection map constructed in \Cref{cor:measurable_selection_couplings}.
\end{remark}

\begin{example}
    For any continuous section $f : \cM \mapsto T\cM$ with linear growth and any $\mu \in \cPn{n}{\cM}$, the velocity plan $\gamma := [f]^{(n)}(\mu)$ is fully deterministic, and for any other $\eta \in \cPn{n}{T\cM}_\mu$, the unique coupling $\alpha \in \Gamma_\mu(\gamma,\eta)$ is given by $\alpha := [(x,v) \mapsto (x,f(x),v)]^{(n)}(\eta)$ (this is a particular case of \Cref{ex:morphism_induces_fully_det}). In particular the velocity plan $\bm{0}_\mu$ is fully deterministic.
\end{example}

Fully deterministic plans have the particularity that the distance $\W_\mu$, while being only lower semicontinuous in general for the $\W_2$ topology, is continuous at pairs of plans $\gamma_1,\gamma_2$ where at least one of the plans is fully deterministic.

\begin{proposition} \label{prop:w_mu_is_lsc}
    Denote $X = \cPn{n}{T\cM} \times_{\cPn{n}{\cM}} \cPn{n}{T\cM}$. The map $\W : X \mapsto \R_+$ which associates to every $(\gamma_1,\gamma_2) \in X$ the value $\W(\gamma_1,\gamma_2) := \W_\mu(\gamma_1,\gamma_2)$, where $\mu = [\pi]^{(n)}(\gamma_1) = [\pi]^{(n)}(\gamma_2)$, is lower semicontinuous (with respect to the $\W_2$ topology). Moreover, if $(\gamma_1,\gamma_2) \in X$ is such that either $\gamma_1$ or $\gamma_2$ is fully deterministic, then $\W$ is continuous at $(\gamma_1,\gamma_2)$.
\end{proposition}

\begin{proof}
    Let $(\gamma_1^m,\gamma_2^m)_m$ be a sequence in $X$ which converges to $(\gamma_1,\gamma_2) \in X$, and note $\mu^m := [\pi]^{(n)}(\gamma_1^m)$ and $\mu := [\pi]^{(n)}(\gamma_1)$. We want to show that
    \begin{equation} \label{eq:l_2153}
        \W_\mu(\gamma_1,\gamma_2) \leq \liminf_{m \to +\infty} \W_{\mu^m}(\gamma_1^m, \gamma_2^m).
    \end{equation}
    For every $m$, let $\alpha^m \in \Gamma_{\mu^m}(\gamma_1^m,\gamma_2^m)$ be a optimal coupling, so that $\W_{\mu^m}^2(\gamma_1^m,\gamma_2^m) = \bE^{(n)}_{\alpha^m}[\|v_1 - v_2\|^2_x]$. Let $(m_k)_k$ be a subsequence such that
    \begin{equation} \label{eq:l_2157}
        \liminf_{m \to +\infty} \W_{\mu^m}(\gamma_1^m, \gamma_2^m) = \lim_{k \to +\infty} \W_{\mu^{m_k}}(\gamma_1^{m_k}, \gamma_2^{m_k}) = \lim_{k \to +\infty} \sqrt{\bE^{(n)}_{\alpha^{m_k}}[\|v_1-v_2\|^2_x]}.
    \end{equation}
    By \Cref{cor:limit_points_of_seq_of_couplings}, up to re-extracting a subsequence, there exists $\alpha \in \Gamma_\mu(\gamma_1,\gamma_2)$ such that $\alpha^{m_k} \to \alpha$, so that by continuity of the function $\beta \mapsto \bE^{(n)}_\beta[\|v_1-v_2\|^2_x]$, we have
    \begin{equation} \label{eq:l_2161}
        \lim_{k \to +\infty} \sqrt{\bE^{(n)}_{\alpha^{m_k}}[\|v_1-v_2\|^2_x]} = \sqrt{\bE^{(n)}_{\alpha}[\|v_1-v_2\|^2_x]} \geq \W_\mu(\gamma_1,\gamma_2).
    \end{equation}
    Equations \eqref{eq:l_2161} and \eqref{eq:l_2157} together imply \eqref{eq:l_2153}, so that $\W$ is indeed lower semicontinuous.
    \medbreak
    Now, assume that either $\gamma_1$ or $\gamma_2$ is fully deterministic. Then, by \Cref{prop:coupling_with_fully_det_is_unique}, there is an unique coupling $\alpha \in \Gamma_\mu(\gamma_1,\gamma_2)$ (which is then optimal). Therefore, by \Cref{cor:limit_points_of_seq_of_couplings}, we have $\alpha^m \to \alpha$, so that
    \begin{equation}
        \lim_{m \to +\infty} \W_{\mu^m}(\gamma_1^m,\gamma_2^m) = \lim_{m \to +\infty} \sqrt{\bE^{(n)}_{\alpha^m}[\|v_1-v_2\|^2_x]} = \sqrt{\bE^{(n)}_{\alpha}[\|v_1-v_2\|^2_x]} = \W_\mu(\gamma_1,\gamma_2),
    \end{equation}
    and this shows that $\W$ is continuous at $(\gamma_1,\gamma_2)$.
\end{proof}

Note that, if $\mu \in \cPn{n}{\cM}$ and $\gamma_1, \gamma_2 \in \cPn{n}{T\cM}_\mu$ are such that at least one of the $\gamma_i$ is fully deterministic, we can denote $\gamma_1 + \gamma_2 := \gamma_1 +_\alpha \gamma_2$ and $\gamma_1 - \gamma_2 := \gamma_1 -_\alpha \gamma_2$ where $\alpha$ is the unique coupling between $\gamma_1$ and $\gamma_2$. Moreover, by unicity of $\alpha$, we have necessarily
\begin{equation} \label{eq:fully_det_w_mu_is_norm_of_diff}
    \W_\mu^2(\gamma_1,\gamma_2) = \bE^{(n)}_\alpha[\|v_1 - v_2\|^2_x] = \|\gamma_1 - \gamma_2\|^2.
\end{equation}

\begin{lemma} \label{lemma:sum_of_fully_det_is_det}
    Let $\mu \in \cPn{n}{\cM}$ and $\gamma_1, \gamma_2 \in \cPndet{n}{T\cM}_\mu$, and $\lambda \in \R$. Then $\gamma_1 + \gamma_2$ and $\lambda \gamma_1$ are both in $\cPndet{n}{T\cM}_\mu$.
\end{lemma}

\begin{proof}
    We show this by induction. For $n = 0$ this is trivial since $\cPndet{0}{T\cM}_x = T_x \cM$. Let $n > 0$ and assume that the proposition holds for $n-1$. Let $\bP \in \cPn{n}{\cM}$, $\bGamma_1,\bGamma_2 \in \cPndet{n}{T\cM}_\bP$ and $\lambda \in \R$. We write $\bGamma_1 = s_{1\#}\bP$ and $\bGamma_2 = s_{2\#}\bP$ where $s_1,s_2$ are measurable maps $\cPn{n-1}{\cM} \mapsto \cPn{n-1}{T\cM}$ which are sections of $[\pi]^{(n-1)}$ and are valued $\bP$-a.e. in $\cPndet{n-1}{T\cM}$. We then have
    \begin{equation} \label{eq:fully_det_times_scalar_expression}
        \lambda \bGamma_1 = [m_\lambda]^{(n)}(s_{1\#} \bP) = ([m_\lambda]^{(n-1)} \circ s_1)_\#\bP = (\mu \mapsto \lambda s_1(\mu))_\#\bP
    \end{equation}
    and by induction the map $\mu \mapsto \lambda s_1(\mu)$ is still valued $\bP$-a.e. in $\cPndet{n-1}{T\cM}$, so that $\lambda \bGamma_1$ is fully deterministic.  Moreover, by \Cref{rk:unique_coupling_expression}, the unique coupling $\bA \in \Gamma_\bP(\bGamma_1,\bGamma_2)$ is given by $\bA = (\mu \mapsto \alpha_{s_1(\mu),s_2(\mu)})_\#\bP$ where, for $\bP$-a.e. $\mu$, $\alpha_{s_1(\mu),s_2(\mu)}$ is the unique coupling between the fully deterministic plans $s_1(\mu)$ and $s_2(\mu)$. Therefore,
    \begin{align}
        \bGamma_1 + \bGamma_2 &= [\pi_1 + \pi_2]^{(n)}(\bA) \\
        &= ([\pi_1 + \pi_2]^{(n-1)} \circ (\mu \mapsto \alpha_{s_1(\mu),s_2(\mu)}))_\#\bP \\
        &= (\mu \mapsto s_1(\mu) + s_2(\mu))_\#\bP \label{eq:fully_det_addition_expression}
    \end{align}
    and by induction the map $\mu \mapsto s_1(\mu) + s_2(\mu)$ is still valued $\bP$-a.e. in $\cPndet{n-1}{T\cM}$, so that $\bGamma_1 + \bGamma_2$ is fully deterministic.
\end{proof}

\begin{proposition} \label{prop:fully_det_plan_addition_associative_commutative}
    Let $\mu \in \cPn{n}{\cM}$ and $\gamma_1, \gamma_2, \gamma_3 \in \cPn{n}{T\cM}_\mu$. Assume that at least two of the $\gamma_i$ are fully deterministic. Then the following hold:
    \begin{enumerate}
        \item (Associativity) $(\gamma_1 + \gamma_2) + \gamma_3 = \gamma_1 + (\gamma_2 + \gamma_3)$.
        \item (Commutativity) $\gamma_1 + \gamma_2 = \gamma_2 + \gamma_1$.
    \end{enumerate}
    Moreover, if $\gamma_1$ is fully deterministic, then $\gamma_1 - \gamma_1 = \bm{0}_\mu$.
\end{proposition}

\begin{proof}
    The commutativity is simply a consequence of point \ref{enum:tangent_struct:commutativity} of \Cref{prop:various_results_on_inner_prod} and of the unicity of couplings. Now, fix any $\alpha \in \Gamma_\mu(\gamma_1,\gamma_2,\gamma_3)$ (which exists by \Cref{rk:general_couplings_exist}). Then, letting $\alpha_1 := [\pi^3_{1,2}]^{(n)}(\alpha) \in \Gamma_\mu(\gamma_1,\gamma_2)$ and $\alpha_2 := [(x,v_1,v_2,v_3) \mapsto (x,v_1+v_2,v_3)]^{(n)}(\alpha) \in \Gamma_\mu(\gamma_1 +_{\alpha_1} \gamma_2, \gamma_3)$, we have
    \begin{align}
        (\gamma_1 + \gamma_2) + \gamma_3 &= (\gamma_1 +_{\alpha_1} \gamma_2) +_{\alpha_2} \gamma_3 \\
        &= [\pi_1 + \pi_2]^{(n)}(\alpha_2) \\
        &= [\pi_1 + \pi_2 + \pi_3]^{(n)}(\alpha).
    \end{align}
    Similarly, letting this time $\alpha_3 := [\pi^3_{2,3}]^{(n)}(\alpha) \in \Gamma_\mu(\gamma_2,\gamma_3)$ and $\alpha_4 := [(x,v_1,v_2,v_3) \mapsto (x,v_1,v_2+v_3)]^{(n)}(\alpha) \in \Gamma_\mu(\gamma_1, \gamma_2  +_{\alpha_3} \gamma_3)$, we find
    \begin{align}
        \gamma_1 + (\gamma_2 + \gamma_3) &= \gamma_1 +_{\alpha_4} (\gamma_2 +_{\alpha_3} \gamma_3 \\
        &= [\pi_1 + \pi_2]^{(n)}(\alpha_4) \\
        &= [\pi_1 + \pi_2 + \pi_3]^{(n)}(\alpha),
    \end{align}
    so that $(\gamma_1 + \gamma_2) + \gamma_3 = \gamma_1 + (\gamma_2 + \gamma_3)$. Finally, assume that $\gamma_1$ is fully deterministic. Let $\alpha \in \Gamma_\mu(\gamma_1,\gamma_1)$ be defined by $\alpha := [(x,v) \mapsto (x,v,v)]^{(n)}(\gamma_1)$. Then
    \begin{equation}
        \gamma_1 - \gamma_1 = [\pi_1 - \pi_2]^{(n)}(\alpha) = [(x,v) \mapsto (x,v-v)]^{(n)}(\gamma_1) = [\iota_0 \circ \pi]^{(n)}(\gamma_1) = [\iota_0]^{(n)}(\mu) = \bm{0}_\mu. 
    \end{equation}
    This finishes the proof.
\end{proof}

\begin{proposition} \label{prop:fully_det_is_vector_space}
    For every $\mu \in \cPn{n}{T\cM}_\mu$, the space $\cPndet{n}{T\cM}_\mu$ equipped with $+$, $-$ and the scalar multiplication is a vector space. Its zero vector is $\bm{0}_\mu$ and the additive inverse of $\gamma \in \cPndet{n}{T\cM}_\mu$ is $-\gamma$. 
\end{proposition}

\begin{proof}
    This is simply a consequence of \Cref{prop:fully_det_plan_addition_associative_commutative} and of the points \ref{enum:tangent_struct:zero_neutral_add}, \ref{enum:tangent_struct:scalar_mult_distributive_vector_add}, \ref{enum:tangent_struct:scalar_mult_distributive_scalar_add} and \ref{enum:tangent_struct:minus_eq_plus_neg} of \Cref{prop:various_results_on_inner_prod} combined with the unicity of couplings.
\end{proof}

\begin{proposition} \label{prop:fully_det_is_prehilbert}
    For every $\mu \in \cPn{n}{T\cM}_\mu$, the vector space $\cPndet{n}{T\cM}_\mu$ is a prehilbertian space with associated norm $\|\cdot\|$, inner product $\sca{\cdot}{\cdot}_\mu$ and distance $\W_\mu$.
\end{proposition}

\begin{proof}
    The fact that $\|\cdot\|$ is a norm on $\cPndet{n}{T\cM}_\mu$ is a consequence of the unicity of couplings combined with points \ref{enum:tangent_struct:norm_is_positive_def} and \ref{enum:tangent_struct:triangle_ineq} of \Cref{prop:various_results_on_inner_prod} (it is clear that $\|\lambda \gamma\| = |\lambda| \|\gamma\|$ for any $\lambda \in \R$ and velocity plan $\gamma$). Moreover, we know by \Cref{eq:fully_det_w_mu_is_norm_of_diff} that $\W_\mu$ is the distance associated to the norm $\|\cdot\|$ on $\cPndet{n}{T\cM}_\mu$. We want to show that $\sca{\cdot}{\cdot}_\mu$ is an inner product associated with $\|\cdot\|$. Let $\gamma_1,\gamma_2,\gamma_3 \in \cPndet{n}{T\cM}_\mu$, and let $\lambda \in \R$. Fix any $\alpha \in \Gamma_\mu(\gamma_1,\gamma_2,\gamma_3)$. Then we can check that $\alpha' := [(x,v_1,v_2,v_3) \mapsto (x,v_1+\lambda v_2,v_3)]^{(n)}(\alpha)$ is the unique coupling in $\Gamma_\mu(\gamma_1 + \lambda \gamma_2, \gamma_3)$, so that
    \begin{align}
        \sca{\gamma_1 + \lambda \gamma_2}{\gamma_3}_\mu &= \bE^{(n)}_{\alpha'}[\sca{v_1}{v_2}_x] = \bE^{(n)}_\alpha[\sca{v_1 + \lambda v_2}{v_3}_x] \\
        &= \bE^{(n)}_\alpha[\sca{v_1}{v_3}_x] + \lambda \bE^{(n)}_\alpha[\sca{v_2}{v_3}_x] \\
        &= \sca{\gamma_1}{\gamma_3}_\mu + \lambda \sca{\gamma_2}{\gamma_3}_\mu.
    \end{align}
    This, combined with points \ref{enum:tangent_struct:inner_prod_symmetric} and \ref{enum:tangent_struct:inner_prod_positive_def} of \Cref{prop:various_results_on_inner_prod}, shows that $\sca{\cdot}{\cdot}_\mu$ is a symmetric bilinear form on $\cPndet{n}{T\cM}_\mu$ which is positive definite ; and it is associated to the norm $\|\cdot\|$ by \eqref{eq:w_mu_norm_inner_prod_rel}.
\end{proof}

Let now $n > 0$ and $\bP \in \cPndet{n}{\cM}$. It turns out, in fact, that the space $\cPndet{n}{T\cM}_\mu$ has the structure of a Hilbert space. To see this, we set $\tilde{\bP} := \cU_X^{(n)}(\bP)$, and we define $L^2_\pi(\bP,T\cM)$ to be space of measurable functions $f : \cPn{n-1}{\cM} \times \ldots \times \cP_2(\cM) \times \cM \mapsto T\cM$ such that $\pi(f(\mu^{n-1},\ldots,\mu,x)) = x$ $\tilde{\bP}$-almost everywhere, and with finite $L^2$ norm
\begin{equation}
    \|f\|_{L^2(\bP)}^2 := \int \|f(\mu^{n-1},\ldots,\mu,x)\|_x^2 \dd\tilde{\bP}(\mu^{n-1},\ldots,\mu,x) < +\infty.
\end{equation}
In particular, $L^2_\pi(\bP,T\cM)$ can be seen as a subset of $\cM_{\pi,\bP}$, and it can be shown to be a Hilbert space using standard arguments.

\begin{proposition}
    The map $j_\bP$ defined in \Cref{prop:maps_induce_fully_det} restricted to $L^2_\pi(\bP,T\cM)$ is a bijective linear isometry between $L^2_\pi(\bP,T\cM)$ and $\cPndet{n}{T\cM}_\bP$. In particular, $\cPndet{n}{T\cM}_\bP$ is a Hilbert space.
\end{proposition}

\begin{proof} 
    We show this by induction on $n$. Consider first the case $n = 1$. If $\mu \in \cP_2(\cM)$, $f \in \cM_{\pi,\mu}$ and $\gamma := j_\mu(f) = f_\#\mu \in \cP_{\det}(T\cM)$, then we have
    \begin{equation} \label{eq:l_1534}
        \|\gamma\|^2 = \int \|v\|_x^2 \dd\gamma(v) = \int \|f(x)\|_x^2 \dd\mu(x) = \|f\|^2_{L^2(\mu)}.
    \end{equation}
    In particular, by \Cref{lemma:pseudo_norm_is_finite}, if $f \in L^2_\pi(\mu,T\cM)$, then $\gamma \in \cP_2(T\cM)$, with $\|\gamma\| = \|f\|_{L^2(\mu)}$. Moreover, if $\lambda \in \R$, then by \eqref{eq:fully_det_times_scalar_expression}, we have $\lambda j_\mu(f) = \lambda f_\#\mu = (\lambda f)_\#\mu = j_\mu(\lambda f)$. Furthermore, if $f_1, f_2 \in L^2_\pi(\mu,T\cM)$, we have from \eqref{eq:fully_det_addition_expression} that $j_\mu(f_1) + j_\mu(f_2) = f_{1\#}\mu + f_{2\#}\mu = (f_1+f_2)_\#\mu = j_\mu(f_1+f_2)$. Thus $j_\mu$ restricts to a isometric linear map $L^2_\pi(\mu,T\cM) \mapsto \cPndet{n}{T\cM}_\mu$. Finally, this map is surjective. Indeed, for every $\gamma \in \cPndet{n}{T\cM}_\mu$, there exists by \Cref{prop:fully_det_induced_by_map} some $f \in \cM_{\pi,\mu}$ such that $j_\mu(f) = \gamma$, and we have in fact $f \in L^2_\pi(\mu,T\cM)$ by \eqref{eq:l_1534}. This finishes the proof in the case $n = 1$. \newline
    Now, let $n > 1$, assume the proposition is true for $n-1$, and let $\bP \in \cPn{n}{\cM}$. Then, for every $f \in \cM_{\pi,\bP}$, letting $\bGamma := j_\bP(f)$, we find
    \begin{align}
        \|\bGamma\|^2 &= \int \|\gamma\|^2 \dd\bGamma(\gamma) = \int \|j_\mu(f(\mu,\cdot))\|^2 \dd\bP(\mu) \\
        &= \int \|f(\mu,\cdot)\|^2_{L^2(\mu)} \dd\bP(\mu) = \|f\|^2_{L^2(\bP)} \label{eq:l_1541}
    \end{align}
    where we used the induction hypothesis in the second line. Thus, if $f \in L^2_\pi(\bP,T\cM)$, we have again by \Cref{lemma:pseudo_norm_is_finite} that $\bGamma \in \cPn{n}{T\cM}$ with $\|\bGamma\| = \|f\|_{L^2(\bP)}$. Moreover, if $\lambda \in \R$, then by \eqref{eq:fully_det_times_scalar_expression}, we have 
    \begin{align}
        \lambda j_\bP(f) &= \lambda (\mu \to j_\mu(f(\mu,\cdot)))_\#\bP = (\mu \to \lambda j_\mu(f(\mu,\cdot)))_\#\bP \\
        &= (\mu \to j_\mu(\lambda f(\mu,\cdot)))_\#\bP = j_\bP(\lambda f)
    \end{align}
    where we used the induction hypothesis in the second line. Furthermore, if $f_1, f_2 \in L^2_\pi(\bP,T\cM)$, we have from \eqref{eq:fully_det_addition_expression} that 
    \begin{align}
        j_\bP(f_1) + j_\bP(f_2) &= (\mu \to j_\mu(f_1(\mu,\cdot)))_\#\bP + (\mu \to j_\mu(f(\mu,\cdot)))_\#\bP \\
        &= (\mu \to j_\mu(f_1(\mu,\cdot)) + j_\mu(f_2(\mu,\cdot)))_\#\bP \\
        &= (\mu \to j_\mu(f_1(\mu,\cdot) + f_2(\mu,\cdot)))_\#\bP = j_\bP(f_1+f_2)
    \end{align}
    where we used the induction hypothesis in the third line. Thus $j_\bP$ restricts to a isometric linear map $L^2_\pi(\bP,T\cM) \mapsto \cPndet{n}{T\cM}_\bP$. Finally, this map is surjective. Indeed, for every $\bGamma \in \cPndet{n}{T\cM}_\bP$, there exists by \Cref{prop:fully_det_induced_by_map} some $f \in \cM_{\pi,\bP}$ such that $j_\bP(f) = \bGamma$, and we have in fact $f \in L^2_\pi(\bP,T\cM)$ by \eqref{eq:l_1541}. This finishes the proof.
\end{proof}

\section{Geodesics} \label{sec:5_geodesics}

In this section, we will study the properties of the geodesics of the $n$-th hierarchical space $\cPn{n}{\cM}$. 

\subsection{Terminology}

We clarify in this subsection the terminology that we will use regarding the notion of geodesics. Indeed, the concept of ``geodesic" arises both in Riemannian geometry and in the field of analysis in metric spaces. However, the meanings this notion has in these two fields do not completely overlap. Since we will here refer to the notion of geodesics from both fields, this calls for some clarifications in order to avoid any ambiguity. In the following we fix $(X,d)$ to be a metric space. In this subsection only, $\cM$ will denote a generic Riemannian manifold.

\begin{definition}
    Let $I \subseteq \R$ be a bounded interval. A curve $c : I \mapsto X$ is said to be \emph{absolutely continuous} if there exists $m \in L^1(I)$ such that $d(c(s),c(t)) \leq \int_s^t m(r) \dd r$ for every $s,t \in I$ with $s \leq t$. We denote $AC(I;X)$ the class of absolutely continuous curves $c : I \mapsto X$.
\end{definition}

\begin{definition}
    The \emph{length} of a curve $c : [a,b] \mapsto X$ is defined by
    \begin{equation}
        \mathrm{Length}(c) := \sup \left\{ \sum_{k=0}^{n-1} d(c(t_k),c(t_{k+1})) \setcond n \geq 1, a = t_0 < t_1 < \ldots < t_n = b \right\}.
    \end{equation}
\end{definition}

Clearly, for any curve $c : [a,b] \mapsto X$, one has $d(c(a),c(b)) \leq \mathrm{Length}(c)$.

\begin{definition}
    An absolutely continuous curve $c : [0,1] \mapsto X$ is said to be a \emph{geodesic} between $x = c(0)$ and $y = c(1)$ if $\mathrm{Length}(c) = d(x,y)$. It is said to be a \emph{constant speed geodesic} if $d(c(s),c(t)) = |t-s| d(x,y)$ for every $s,t \in [0,1]$.
\end{definition}

It is clear from the definition of the length that a constant speed geodesic is a geodesic. To differentiate this notion of geodesic from the one encountered in Riemannian geometry, we adopt the following terminology:

\begin{definition}
    Let $I \subseteq \R$ be an interval. A $C^1$ curve $c : I \mapsto \cM$ is said to be a \emph{Riemannian geodesic} if it is a geodesic in the sense of Riemannian geometry, that is if it satisfy the geodesic equation $D_{\dot{c}} \dot{c} = 0$.  
\end{definition}

By \citep[Theorem 3.15]{burtscher2015lengthstructure}, for any piecewise $C^1$ curve $c : I \mapsto \cM$, where $I$ is a bounded interval, one has
\begin{equation}
    \mathrm{Length}(c) = \int_I \|\dot{c}(t)\|_{c(t)} \dd t.
\end{equation}

A Riemannian geodesic is in general not a geodesic, as it is only locally length minimizing. On the other hand, a constant speed geodesic on a Riemannian manifold is in fact a Riemannian geodesic, as the following proposition shows:

\begin{proposition} \label{prop:cs_geodesics_in_manifolds}
    Let $c : [0,1] \mapsto \cM$ be a constant speed geodesic between $x = c(0)$ and $y = c(1)$. Then, there exists a unique $v \in T_x \cM$ such that $c(t) = \exp_x(tv)$ for every $t \in [0,1]$, and it satisfies $\|v\|_x = d(x,y)$. Moreover, for every $t,s \in [0,1]$, whenever one of the $t,s$ is not in $\{0,1\}$, there is a unique $w \in T_{c(t)}\cM$ such that $\exp_{c(t)}(w) = c(s)$ and $\|w\|_{c(t)} = d(c(t),c(s))$, and it is given by $w = (s-t)\PT_t(x,v,v)$.
\end{proposition}

\begin{proof}
    We only need to prove that the path is $c$ is piecewise $C^1$, and the rest will follow from the properties of Riemannian geodesics. Since $c([0,1])$ is compact in $\cM$, by \citep[Theorem 2.92]{gallot2004riemannian}, we know that there exists some subdivision $0 = t_0 < t_1 < \ldots < t_n = 1$ such that for every $i \in \{0,\ldots,n-1\}$ and $t_i \leq t \leq s \leq t_{i+1}$, there exists an unique length minimizing Riemannian geodesic $\gamma_{t,s} : [t,s] \mapsto \cM$ between $c(t)$ and $c(s)$. Now, if we let $i \in \{0,\ldots,n-1\}$ and $s \in (t_i,t_{i+1})$, then the curve $\gamma : [t_i, t_{i+1}] \mapsto \cM$ obtained by concatenation of the curves $\gamma_{t_i,s}$ and $\gamma_{s,t_{i+1}}$ is a length minimizing piecewise $C^1$ curve from $c(t_i)$ to $c(t_{i+1})$, since
    \begin{align}
        \mathrm{Length}(\gamma) &\leq \mathrm{Length}(\gamma_{t_i,s}) + \mathrm{Length}(\gamma_{s,t_{i+1}}) = d(c(t_i),c(s)) + d(c(s),c(t_{i+1})) \\ 
        &= (s-t_i) d(x,y) + (t_{i+1}-s)d(x,y) = (t_{i+1}-t_i) d(x,y) = d(c(t_i),c(t_{i+1})). 
    \end{align}
    Therefore, by \citep[Proposition 2.97]{gallot2004riemannian}, $\gamma$ is actually a length minimizing Riemannian geodesic, so that $\gamma = \gamma_{t_i,t_{i+1}}$ by unicity of $\gamma_{t_i,t_{i+1}}$. In particular, we see that $\gamma_{t_i,t_{i+1}}(s) = c(s)$ for any $s \in [t_i, t_{i+1}]$. From this, we conclude that for every $i \in \{0,\ldots,n-1\}$, $c$ coincides with $\gamma_{t_i,t_{i+1}}$ on $[t_i,t_{i+1}]$. Thus $c$ is a length minimizing piecewise $C^1$ curve from $x$ to $y$. This finishes the proof.
\end{proof}

\subsection{Optimal velocity plans induce constant speed geodesics}

We first show that we can use optimal velocity plans $\gamma \in \Gamma_o^{(n)}(\cM)$ to construct constant speed geodesics between any pair of measures $\mu, \nu \in \cPn{n}{\cM}$. In particular, the space $\cPn{n}{\cM}$ will be a geodesic space.

\begin{proposition} \label{prop:opt_vel_plans_give_geodesics}
    Let $n \geq 0$, $\mu, \nu \in \cPn{n}{\cM}$ and $\gamma \in \Gamma_o(\mu,\nu)$. For every $t \in \R$, let $\mu_t := [\exp]^{(n)}(t\gamma)$. Then $t \in [0,1] \mapsto \mu_t$ is a constant speed geodesic between $\mu$ and $\nu$. 
\end{proposition}

\begin{proof}
    We prove this by induction. For $n = 0$, this is simply a reformulation of a well-known property of Riemannian manifolds : if $(x,y) \in \cM^2$ and $v \in T_x\cM$ are such that $\exp_x(v) = y$ and $d(x,y) = \|v\|_x$, then $t \in [0,1] \mapsto \exp_x(tv)$ is a length-minimizing constant speed geodesic from $x$ to $y$. Now, let $n > 0$, and assume that the proposition has been proved for $n-1$. Let $\bP, \bQ \in \cPn{n}{\cM}$ and $\bGamma \in \Gamma_o(\bP,\bQ)$. For every $t \in \R$ let $\bP_t := [\exp]^{(n)}(t\bGamma)$. Then for every $t,s \in [0,1]$,
    \begin{align}
        \W_2^2(\bP_t, \bP_s) &\leq \int \W_2^2([\exp]^{(n-1)}(t\gamma), [\exp]^{(n-1)}(s\gamma)) \dd\bGamma(\gamma) \\
        &\leq \int (t-s)^2 \W_2^2([\pi]^{(n-1)}(\gamma), [\exp]^{(n-1)}(\gamma)) \dd\bGamma(\gamma) \\
        &\leq (t-s)^2 \W_2^2(\bP, \bQ)
    \end{align}
    where we used that $([\exp \circ m_t]^{(n-1)}, [\exp \circ m_s]^{(n-1)})_\#\bGamma$ is a (not necessarily optimal) transport plan between $\bP_t$ and $\bP_s$ in the first line, the induction hypothesis in the second line and the optimality of $\bGamma$ is the third line. Thus $\W_2(\bP_t,\bP_s) \leq |t-s|\W_2(\bP,\bQ)$ for every $t,s \in [0,1]$. In particular, for every $t,s \in [0,1]$ with $t \leq s$, using the triangle inequality,
    \begin{align}
        \W_2(\bP,\bQ) &\leq \W_2(\bP,\bP_t) + \W_2(\bP_t,\bP_s) + \W_2(\bP_s,\bQ)\\
        &\leq \W_2(\bP_0,\bP_t) + \W_2(\bP_t,\bP_s) + \W_2(\bP_s,\bP_1) \\
        &\leq t\W_2(\bP,\bQ) + (s-t)\W_2(\bP,\bQ) + (1-s)\W_2(\bP,\bQ) = \W_2(\bP,\bQ),
    \end{align}
    implying that all these inequalities are equalities, so that $\W_2(\bP_t, \bP_s) = |t-s|\W_2(\bP,\bQ)$ for every $t,s \in [0,1]$. This finishes the proof.
\end{proof}

\begin{corollary}
    For every $n > 0$, $\cPn{n}{\cM}$ is a geodesic space.
\end{corollary}

\begin{proof}
    Indeed, by \Cref{cor:selection_map_opt_vel_plans}, for any pair $\mu,\nu \in \cPn{n}{\cM}$, $\Gamma_o(\mu,\nu)$ is not empty, so by \Cref{prop:opt_vel_plans_give_geodesics} there are geodesics from $\mu$ to $\nu$.
\end{proof}

\subsection{Parallel transport of velocity plans}

The converse implication, that is, that if $(\mu_t)_{t \in [0,1]}$ is a constant speed geodesic between $\mu$ and $\nu$ in $\cPn{n}{\cM}$, then it is induced by some optimal velocity plan $\gamma \in \Gamma_o(\mu,\nu)$, is more difficult to show. For this, we will need to define for every $t \in \R$ an operator $\PT_t^{(n)} : \cPn{n}{T\cM} \mapsto \cPn{n}{T\cM}$ by 
\begin{equation}
    \PT_t^{(n)} := [(x,v) \mapsto (\exp_x(tv), \PT_t(x,v,v))]^{(n)}.
\end{equation}
Given $\gamma \in \cPn{n}{T\cM}$, the map $t \mapsto \PT_t^{(n)}(\gamma)$ can be considered to be the ``parallel transport" of $\gamma$ through the curve $s \to [\exp]^{(n)}(s\gamma)$. In the case $t = 1$, we will also use the shorthand notation $\PT^{(n)} := \PT^{(n)}_1$. Clearly, $\PT_0^{(n)}$ is the identity map of $\cPn{n}{T\cM}$. 

\begin{lemma} \label{lemma:pt_group_and_commutator}
    The operators $\PT_t^{(n)}$ satisfy the following properties:
    \begin{enumerate}
        \item $\PT_t^{(n)} \circ \PT_s^{(n)} = \PT_{t+s}^{(n)}$ for any $t,s \in \R$.
        \item $[\pi]^{(n)} \circ \PT_t^{(n)} = [\exp \circ \ m_t]^{(n)}$ for any $t \in \R$.
        \item $\PT_t^{(n)} \circ [m_s]^{(n)} = [m_s]^{(n)} \circ \PT_{st}^{(n)}$ for any $t,s \in \R$.
    \end{enumerate}
\end{lemma}

\begin{proof}
    Since $\PT_t^{(n)} := [(x,v) \mapsto (\exp_x(tv), \PT_t(x,v,v))]^{(n)}$, by functoriality of $[\cdot]$, we only need to prove these points for $n = 0$. This follows by the basic properties of geodesic curves and parallel transport on Riemannian manifolds.
\end{proof}

The second point of the lemma means that $[\pi]^{(n)}(\PT_t^{(n)}(\gamma)) = [\exp]^{(n)}(t\gamma)$ for any $t \in \R$ and $\gamma \in \cPn{n}{T\cM}$, while the third point means that $\PT_t^{(n)}(s\gamma) = s\PT_{st}^{(n)}(\gamma)$ for any $s,t \in \R$ and $\gamma \in \cPn{n}{T\cM}$.

\begin{lemma} \label{lemma:pt_of_plan_follows_curve}
    Let $\gamma \in \cPn{n}{T\cM}$, and let $\mu_t := [\exp]^{(n)}(t\gamma)$ for every $t \in \R$. Then, for every $t,s \in \R$, one has $\|\PT_t^{(n)}(\gamma)\| = \|\gamma\|$ and $s\PT_t^{(n)}(\gamma) \in \Gamma(\mu_t, \mu_{t+s})$. In particular, the curve $(\mu_t)_{t \in \R}$ is $\|\gamma\|$-Lipschitz continuous. 
\end{lemma}

\begin{proof}
    Indeed, we have
    \begin{align}
        \|\PT_t^{(n)}(\gamma)\|^2 &= \bE^{(n)}_{\PT_t^{(n)}(\gamma)}[\|v\|^2_x] \\
        &= \bE^{(n)}_\gamma[\|\PT_t(x,v,v)\|^2_{\exp_x(tv)}] \\
        &= \bE^{(n)}_\gamma[\|v\|^2] = \|\gamma\|^2
    \end{align}
    where we obtain the third line by using the fact that the parallel transport is an isometry between tangent spaces. Moreover, we have by point 2 of \Cref{lemma:pt_group_and_commutator} that
    \begin{equation}
        [\pi]^{(n)}(s\PT_t^{(n)}(\gamma)) = [\pi]^{(n)}(\PT_t^{(n)}(\gamma)) = [\exp]^{(n)}(t\gamma) = \mu_t
    \end{equation}
    and also that
    \begin{align}
        [\exp]^{(n)}(s\PT_t^{(n)}(\gamma)) &= [\pi]^{(n)}(\PT_s^{(n)}(\PT_t^{(n)}(\gamma))) \\
        &= [\pi]^{(n)}(\PT_{t+s}^{(n)}(\gamma)) \\
        &= [\exp]^{(n)}((t+s)\gamma) = \mu_{t+s}
    \end{align}
    where we obtained the first, second and third line using respectively point 2, 1 and again 2 of \Cref{lemma:pt_group_and_commutator}. This proves $s\PT_t^{(n)}(\gamma) \in \Gamma(\mu_t,\mu_{t+s})$. Now, we show that $(\mu_t)_t$ is $\|\gamma\|$-Lipschitz continuous. Let $t,s \in \R$, then by the preceding, we have $(s-t)\PT_t^{(n)}(\gamma) \in \Gamma(\mu_t,\mu_s)$, so that 
    \begin{equation}
        \W_2(\mu_t,\mu_s) \leq \|(s-t)\PT_t^{(n)}(\gamma)\| = |t-s| \cdot \|\PT_t^{(n)}(\gamma)\| = |t-s| \cdot \|\gamma\|.
    \end{equation}
    This finishes the proof.
\end{proof}

A consequence of this proposition is that if we let $\gamma \in \cPn{n}{T\cM}$ and define the curve $\mu_t := [\exp]^{(n)}(t\gamma)$, $t \in \R$, then, if we fix $t_0, t_1 \in \R$ and $\tilde{\gamma} := t_1 \PT_{t_0}^{(n)}(\gamma)$, then the curve $\tilde{\mu}_s := [\exp]^{(n)}(s\gamma)$, $s \in \R$ is simply a linear reparametrization of the curve $(\mu_t)_t$: that is 
\begin{equation} \label{eq:pt_of_vel_plan_induce_reparametrized_geodesic}
    \tilde{\mu}_s = \mu_{t_0 + st_1}, \quad \forall s \in \R.
\end{equation}

\begin{proposition} \label{prop:pt_of_opt_vel_plan_is_optimal}
    Let $n \geq 0$, $\mu, \nu \in \cPn{n}{\cM}$ and $\gamma \in \Gamma_o(\mu,\nu)$, and let $\mu_t := [\exp]^{(n)}(t\gamma)$ for every $t \in \R$. Then, for every $t \in [0,1]$ and $s \in \R$ such that $t+s \in [0,1]$, $s\PT_t^{(n)}(\gamma) \in \Gamma_o(\mu_t,\mu_{t+s})$.
\end{proposition}

\begin{proof}
    We already know by \Cref{lemma:pt_of_plan_follows_curve} that $s\PT_t^{(n)}(\gamma) \in \Gamma(\mu_t,\mu_{t+s})$. We show that it is optimal: indeed, we have
    \begin{align}
        \|s\PT_t^{(n)}(\gamma)\| &= |s| \ \|\PT_t^{(n)}(\gamma)\| \\
        &= |s| \ \|\gamma\| \\
        &= |s| \ \W_2(\mu,\nu) \\
        &= \W_2(\mu_t, \mu_{t+s})
    \end{align}
    where we used \Cref{lemma:pt_of_plan_follows_curve} to obtain the second line, the optimality of $\gamma$ for the third line, and the fact that $(\mu_t)_t$ is a constant speed geodesic (\Cref{prop:opt_vel_plans_give_geodesics}) for the fourth line. This shows optimality of $s\PT_t^{(n)}(\gamma)$.
\end{proof}

\subsection{Constant speed geodesics arise from optimal velocity plans}

We are now ready to show the converse of \Cref{prop:opt_vel_plans_give_geodesics}:

\begin{proposition} \label{prop:cs_geodesics_come_from_opt_vel_plans}
    Let $n \geq 0$ and $\mu, \nu \in \cPn{n}{\cM}$. If $(\mu_t)_{t \in [0,1]}$ is a constant speed geodesic from $\mu$ to $\nu$, then there exists a unique $\gamma \in \Gamma_o(\mu,\nu)$ such that $\mu_t = [\exp]^{(n)}(t\gamma)$ for every $t \in [0,1]$. Moreover, for every $t,s \in [0,1]$, whenever one of $t,s$ is not in $\{0,1\}$, there is a unique optimal velocity plan $\gamma_{t,s} \in \Gamma_o(\mu_t,\mu_s)$, which is given by $\gamma_{t,s} := (s-t)\PT_t^{(n)}(\gamma)$. Moreover, when $t \in (0,1)$, $\gamma_{t,s}$ is fully deterministic.
\end{proposition}

\begin{proof}
    We show this by induction. For $n = 0$, this is simply \Cref{prop:cs_geodesics_in_manifolds}. Now, let $n > 0$, and assume that the proposition holds for $n-1$. Let $(\bP_t)_{t \in [0,1]}$ be a constant speed geodesic in $\cPn{n}{\cM}$. For every $t,s \in [0,1]$, fix some $\bGamma_{t,s} \in \Gamma_o(\bP_t,\bP_s)$.
    \begin{enumerate}
        \item First, we notice that either $(\bP_t)_{t \in [0,1]}$ is constant or it is injective\footnote{This is a general property of constant speed geodesics on metric spaces. Indeed, if $c : [0,1] \mapsto (X,d)$ is a constant speed geodesic on a metric space, then  for every $t \neq s \in [0,1]$, $d(c(s),c(t)) = |t-s| d(c(0),c(1))$, so that $c(t) = c(s)$ if and only if $c(0) = c(1)$.}. If it is constant, then we can take $\bGamma := \bm{0}_\bP$. It is the only possible choice of $\bGamma \in \Gamma_o(\bP,\bQ)$ by point 11 of \Cref{prop:various_results_on_inner_prod} as we need $\|\bGamma\| = \W_2(\bP_0,\bP_1) = 0$. Thus, in the following, we will assume that $(\bP_t)_t$ is injective.
        \item \label{enum:l_2365:pt_2} Next, we show that for every $t \in (0,1)$, we have $(1-t)\bGamma_{t,0} = -t\bGamma_{t,1}$. Let indeed $t_0 \in (0,1)$ and $\bA \in \Gamma_{\bP_t}(\bGamma_{t,0},\bGamma_{t,1})$. Then $([\exp \circ \pi_1]^{(n-1)}, [\exp \circ \pi_2]^{(n-1)})_\#\bA$ is a (not necessarily optimal) transport plan between $\bP_0$ and $\bP_1$, so that
        \begin{align}
            \W_2^2(\bP_0,\bP_1) &\leq \int \W_2^2([\exp \circ \pi_1]^{(n-1)}(\alpha), [\exp \circ \pi_2]^{(n-1)}(\alpha)) \dd\bA(\alpha) \\
            &\leq \int \W_2^2(\nu_{1,\alpha}, \nu_{2,\alpha}) \dd\bA(\alpha)
        \end{align}
        where we used the shorthand notations
        \begin{align}
            \gamma_{i,\alpha} &:= [\pi_i]^{(n-1)}(\alpha), &i \in \{1,2\} \\
            \nu_{i,\alpha} &:= [\exp]^{(n-1)}(\gamma_{i,\alpha}), &i \in \{1,2\} \\
            \mu_\alpha &:= [\pi]^{(n-1)}(\gamma_{1,\alpha}) = [\pi]^{(n-1)}(\gamma_{2,\alpha}).
        \end{align}
        However, by the triangular inequality, we have 
        \begin{equation}
            \W_2(\nu_{1,\alpha},\nu_{2,\alpha}) \leq \W_2(\nu_{1,\alpha},\mu_\alpha) + \W_2(\mu_\alpha,\nu_{2,\alpha}),
        \end{equation}
        so that
        \begin{align}
            \W_2(\bP_0,\bP_1) &\leq \sqrt{\int \W_2^2(\nu_{1,\alpha}, \nu_{2,\alpha}) \dd\bA(\alpha)} \\
            &\leq \sqrt{\int (\W_2(\nu_{1,\alpha}, \mu_\alpha) + \W_2(\mu_\alpha, \nu_{2,\alpha}))^2 \dd\bA(\alpha)} \\
            &\leq \sqrt{\int \W_2^2(\nu_{1,\alpha},\mu_\alpha) \dd\bA(\alpha)} + \sqrt{\int \W_2^2(\nu_{2,\alpha},\mu_\alpha) \dd\bA(\alpha)} \label{eq:l_2383}\\
            &\leq \sqrt{\int \W_2^2([\exp]^{(n-1)}(\gamma),[\pi]^{(n-1)}(\gamma)) \dd\bGamma_{t_0,0}(\gamma)} \\
            &\quad + \sqrt{ \int \W_2^2([\exp]^{(n-1)}(\gamma),[\pi]^{(n-1)}(\gamma)) \dd\bGamma_{t_0,1}(\gamma)} \\
            &\leq \W_2(\bP_0,\bP_{t_0}) + \W_2(\bP_{t_0},\bP_1) \\
            &\leq \W_2(\bP_0,\bP_1)
        \end{align}
        where we used the Minkowski inequality to obtain the third inequality, the optimality of $\bGamma_{t_0,0}$ and $\bGamma_{t_0,1}$ to obtain the fourth, and the fact that $(\bP_t)_t$ is a constant speed geodesic to obtain the fifth. In particular, all these inequalities are equalities, and the equality case of the Minkowski inequality gives that:
        \begin{itemize}
            \item Either $\W_2(\nu_{2,\alpha},\mu_\alpha) = 0$ for $\bA$-a.e. $\alpha$. 
            \item Or there exists $\lambda \geq 0$ such that $\W_2(\nu_{1,\alpha},\mu_\alpha) = \lambda \W_2(\nu_{2,\alpha},\mu_\alpha)$ for $\bA$-a.e. $\alpha$.
        \end{itemize}
        The first possibility cannot hold as it would imply that
        \begin{equation}
            \W_2^2(\bP_t,\bP_1) = \int \W_2^2(\nu_{2,\alpha},\mu_\alpha) \dd\bA(\alpha) = 0
        \end{equation}
        so that $\bP_{t_0} = \bP_1$, which contradicts our assumption that $(\bP_t)_t$ is injective. Thus the second possibility holds, and there exists $\lambda \geq 0$ such that $\W_2(\nu_{1,\alpha},\mu_\alpha) = \lambda \W_2(\nu_{2,\alpha},\mu_\alpha)$ for $\bA$-a.e. $\alpha$. Therefore we have
        \begin{align}
            \W_2^2(\bP_0,\bP_{t_0}) &= \int \W_2^2(\nu_{1,\alpha},\mu_\alpha) \dd\bA(\alpha) = \lambda^2 \int \W_2^2(\nu_{2,\alpha},\mu_\alpha) \dd\bA(\alpha) = \lambda^2 \W_2^2(\bP_{t_0},\bP_1)
        \end{align}
        and since $\W_2(\bP_0,\bP_{t_0}) = t_0\W_2(\bP_0,\bP_1)$ and $\W_2(\bP_{t_0},\bP_1) = (1-t_0)\W_2(\bP_0,\bP_1)$, we conclude that $\lambda = t_0(1-t_0)^{-1}$. Moreover, since we have determined that \eqref{eq:l_2383} is in fact an equality, we have $\W_2(\nu_{1,\alpha},\nu_{2,\alpha}) = \W_2(\nu_{1,\alpha},\mu_\alpha) + \W_2(\nu_{2,\alpha},\mu_\alpha)$ for $\bA$-a.e. $\alpha$. Furthermore, since $\bGamma_{t_0,0}$ and $\bGamma_{t_0,1}$ are optimal, they are concentrated on $\Gamma_o^{(n-1)}(\cM)$ by \Cref{prop:opt_vel_plan_gives_opt_trans_plan}. From these considerations, we deduce that there exists a set $S \subseteq \cPn{n}{T^2\cM}$ such that $\bA(S) = 1$ and for every $\alpha \in S$,
        \begin{itemize}
            \item $\W_2(\nu_{1,\alpha},\nu_{2,\alpha}) = \W_2(\nu_{1,\alpha},\mu_\alpha) + \W_2(\nu_{2,\alpha},\mu_\alpha)$
            \item $\gamma_{1,\alpha} \in \Gamma_o(\mu_\alpha,\nu_{1,\alpha})$ and $\gamma_{2,\alpha} \in \Gamma_o(\mu_\alpha, \nu_{2,\alpha})$
            \item $\W_2(\nu_{1,\alpha},\mu_\alpha) = \frac{t_0}{1-t_0} \W_2(\nu_{2,\alpha},\mu_\alpha)$
        \end{itemize}
        Fix $\alpha \in S$. By \Cref{prop:opt_vel_plans_give_geodesics}, $c_1 : t \in [0,1] \mapsto [\exp]^{(n-1)}(t\gamma_{1,\alpha})$ is a constant speed geodesic from $\mu_\alpha$ to $\nu_{1,\alpha}$, and $c_2 : t \in [0,1] \mapsto [\exp]^{(n-1)}(t\gamma_{2,\alpha})$ is a constant speed geodesic from $\mu_\alpha$ to $\nu_{2,\alpha}$. Therefore, by \Cref{lemma:gluing_geodesics} below, the curve $c : [0,1] \mapsto \cPn{n}{\cM}$ defined by $c(t) = c_1((t_0-t)/t_0)$ for $t \in [0,t_0]$ and $c(t) = c_2((t-t_0)/(1-t_0))$ for $t \in [t_0,1]$ is a constant speed geodesic from $\nu_{1,\alpha}$ to $\nu_{2,\alpha}$ (indeed, we have $\W_2(\nu_{1,\alpha},\nu_{2,\alpha}) = (1 + (1-t_0)t_0^{-1})\W_2(\nu_{1,\alpha},\mu_\alpha) = t_0^{-1}\W_2(\nu_{1,\alpha},\mu_\alpha)$, so the lemma applies). Then, by the induction hypothesis, there exists a unique $\gamma_\alpha \in \Gamma_o(\nu_{1,\alpha},\nu_{2,\alpha})$ such that $c(t) = [\exp]^{(n-1)}(t\gamma_\alpha)$ for every $t \in [0,1]$ with $\gamma_{1,\alpha} = -t_0 \PT_{t_0}^{(n-1)}(\gamma_\alpha)$ and $\gamma_{2,\alpha} = (1-t_0)\PT_{t_0}^{(n-1)}(\gamma_\alpha)$ fully deterministic. From this, we conclude that $\gamma_{1,\alpha} = -t_0(1-t_0)^{-1}\gamma_{2,\alpha}$ for every $\alpha \in S$. Therefore, the maps $[\pi_1]^{(n-1)}$ and $[m_{-t_0(1-t_0)^{-1}}]^{(n-1)} \circ [\pi_2]^{(n-1)}$ coincide on $S$, on which $\bA$ is concentrated, so that
        \begin{equation}
            \bGamma_{t_0,0} = [\pi_1]^{(n)}(\bA) = [m_{-t_0(1-t_0)^{-1}}]^{(n)}([\pi_2]^{(n)}(\bA)) = -\frac{t_0}{1-t_0} \bGamma_{t_0,1}.
        \end{equation}
        \item \label{enum:l_2365:pt_3} Since the equality $\bGamma_{t_0,0} = -t_0(1-t_0)^{-1} \bGamma_{t_0,1}$ holds for any choice of $\bGamma_{t_0,0} \in \Gamma_o(\bP_{t_0},\bP_0)$ and $\bGamma_{t_0,1} \in \Gamma_o(\bP_{t_0},\bP_1)$, it implies that these plans are unique. Moreover, since, by \Cref{prop:pt_of_opt_vel_plan_is_optimal}, the map $\bGamma \mapsto -\PT^{(n-1)}(\bGamma)$ induces a bijection between $\Gamma_o(\bP,\bQ)$ and $\Gamma_o(\bQ,\bP)$ for any pair $\bP,\bQ \in \cPn{n}{\cM}$, this implies that the plans $\bGamma_{0,t_0}$ and $\bGamma_{1,t_0}$ are also unique.
        \item We have also seen that for any $\bA \in \Gamma_{\bP_{t_0}}(\bGamma_{t_0,0},\bGamma_{t_0,1})$, there is a set $S_\bA$ of couplings such that $\bA(S_\bA) = 1$ and for every $\alpha \in S_\bA$, $[\pi_1]^{(n-1)}(\alpha) = -t_0(1-t_0)^{-1}[\pi_2]^{(n-1)}(\alpha)$ with $[\pi_1]^{(n-1)}(\alpha)$ and  $[\pi_2]^{(n-1)}(\alpha)$ fully deterministic. Consider then the disintegrations $\dd\bGamma_{t_0,0}(\gamma) = \dd\bGamma_{t_0,0,\mu}(\gamma)\dd\bP_{t_0}(\mu)$ and $\dd\bGamma_{t_0,1}(\gamma) = \dd\bGamma_{t_0,1,\mu}(\gamma)\dd\bP_{t_0}(\mu)$, and consider $\bA \in \Gamma_{\bP_{t_0}}(\bGamma_{t_0,0},\bGamma_{t_0,1})$ having disintegration $\dd\bA(\alpha) = \dd\bA_{\gamma_1,\gamma_2}(\alpha) \dd(\bGamma_{t_0,0,\mu} \otimes \bGamma_{t_0,1,\mu})(\gamma_1,\gamma_2)\dd\bP_{t_0}(\mu)$ (the proof of \Cref{prop:opt_coupling_characterization} shows that such a $\bA$ exists). Thus, we have for $\bP_{t_0}$-a.e. $\mu$ and $\bGamma_{t_0,0,\mu} \otimes \bGamma_{t_0,1,\mu}$-a.e. $\gamma_1,\gamma_2$ that $\bA_{\gamma_1,\gamma_2}(S_\bA) = 1$, so that we have for $\bP_{t_0}$-a.e. $\mu$ and $\bGamma_{t_0,0,\mu} \otimes \bGamma_{t_0,1,\mu}$-a.e. $\gamma_1,\gamma_2$ that $\gamma_1 = -t_0(1-t_0)^{-1} \gamma_2$ with $\gamma_1,\gamma_2$ fully deterministic. In particular, this implies that for $\bP$-a.e. $\mu$, both $\bGamma_{t_0,0,\mu}$ and $\bGamma_{t_0,1,\mu}$ are Dirac masses on fully deterministic velocity plans. Thus, $\bGamma_{t_0,0}$ and $\bGamma_{t_0,1}$ are both fully deterministic.
        \item \label{enum:l_2365:pt_4} In fact, for any pair $t_1,t_2 \in [0,1]$, if either $t_1$ or $t_2$ is not in $\{0,1\}$, then $\bGamma_{t_1,t_2}$ is uniquely, with $\bGamma_{t_1,t_2} = -\PT^{(n)}(\bGamma_{t_2,t_1})$, and, when $t_1 \notin \{0,1\}$, it is fully deterministic. Indeed, this follows from the previous points: for instance, in the case $0 < t_1 < t_2$, the unicity of $\bGamma_{t_1,t_2}$ follows from point \ref{enum:l_2365:pt_2} applied to the rescaled constant speed geodesic $(\bP_{tt_2})_{t \in [0,1]}$ with $t_0 := t_1/t_2$. The other cases are proved similarly (and for $t_1 = t_2$, we have $\bGamma_{t_1,t_2} = \bm{0}_{\bP_{t_1}}$ as the only possible choice).
        \item \label{enum:l_2365:pt_5} For any $t_1,t_2,t_3 \in (0,1)$ with $t_1 < t_2 < t_3$, it holds $\bGamma_{t_2,t_1} = \frac{t_1-t_2}{t_3-t_2}\bGamma_{t_2,t_3}$. This again follows by applying point \ref{enum:l_2365:pt_2} to the rescaled geodesic $(\bP_{t_1 + t(t_3-t_1)})_{t \in [0,1]}$ and $t_0 := (t_2 - t_1)/(t_3-t_1)$.
        \item For $t_0 \in (0,1)$, the velocity plan $\bGamma_0 := \frac{1}{t_0}\bGamma_{0,t_0}$ does not depend on $t_0$. Indeed, assume that $t_0 < t_1$. Then:
        \begin{align}
            \frac{1}{t_0}\bGamma_{0,t_0} &= -\frac{1}{t_0} \PT^{(n)}(\bGamma_{t_0,0}) \\
            &= -\frac{1}{t_0} \PT^{(n)}\left(\frac{-t_0}{t_1-t_0} \bGamma_{t_0,t_1} \right) \\
            &= -\frac{1}{t_0} \PT^{(n)}\left(\frac{t_0}{t_1-t_0} \PT^{(n)}(\bGamma_{t_1,t_0}) \right) \\
            &= -\frac{1}{t_0} \PT^{(n)}\left(\frac{t_0}{t_1-t_0} \PT^{(n)}\left(\frac{t_0-t_1}{1-t_1}\bGamma_{t_1,1}\right) \right) \\
            &= -\frac{1}{t_0} \PT^{(n)}\left(\frac{t_0}{t_1-t_0} \PT^{(n)}\left(\frac{t_1-t_0}{t_1}\bGamma_{t_1,0}\right) \right) \\
            &= -\frac{1}{t_0} \PT^{(n)}\left(\frac{t_0}{t_1} \PT^{(n)}_{\frac{t_1-t_0}{t_1}}(\bGamma_{t_1,0}) \right) \\
            &= -\frac{1}{t_1} \PT^{(n)}(\bGamma_{t_1,0}) \\
            &= \frac{1}{t_1} \bGamma_{0,t_1}
        \end{align}
        where we used point \ref{enum:l_2365:pt_4} to obtain the first, third, and eighth lines, we used point \ref{enum:l_2365:pt_5} to obtain the second, fourth, and fifth lines, and we used \Cref{lemma:pt_group_and_commutator} to obtain the sixth and seventh lines.
        \item Now, we have $[\exp]^{(n)}(0 \cdot \bGamma_0) = [\exp]^{(n)}(\bm{0}_{\bP_0}) = \bP_0$, and for every $t_0 \in (0,1)$, we have $[\exp]^{(n)}(t_0\bGamma_0) = [\exp]^{(n)}(\bGamma_{0,t_0}) = \bP_{t_0}$. Moreover, since $(\bP_t)_{t \in [0,1]}$ is continuous and $t \mapsto [\exp]^{(n)}(t\bGamma_0)$ is $\|\bGamma_0\|$-Lipschitz continuous by \Cref{lemma:pt_of_plan_follows_curve}, we also have $[\exp]^{(n)}(\bGamma_0) = \bP_1$. Thus $\bGamma_0 \in \Gamma(\bP_0,\bP_1)$. Finally, $\bGamma_0$ is optimal: indeed, if we fix any $t_0 \in (0,1)$, we have
        \begin{equation}
            \|\bGamma_0\| = \frac{1}{t_0}\|\bGamma_{0,t_0}\| = \frac{1}{t_0}\W_2(\bP_0,\bP_{t_0}) = \W_2(\bP_0,\bP_1)
        \end{equation}
        using the optimality of $\bGamma_{0,t_0}$ and the fact that $(\bP_t)_t$ has constant speed. Thus we have found an optimal velocity plan $\bGamma_0 \in \Gamma_o(\bP_0,\bP_1)$ which induces the constant speed geodesic $(\bP_t)_{t \in [0,1]}$.
        \item Finally, we prove that there is an unique such $\bGamma_0$. Indeed, if $\bGamma \in \Gamma_o(\bP_0,\bP_1)$ is another optimal velocity plan such that $[\exp]^{(n)}(t\bGamma) = \bP_t$ for every $t \in [0,1]$, then, by \Cref{prop:pt_of_opt_vel_plan_is_optimal}, fixing any $t_0 \in (0,1)$, we must have $t_0 \bGamma \in \Gamma_o(\bP_0,\bP_{t_0})$. Thus, we have $t_0\bGamma = \bGamma_{0,t_0}$ by unicity of $\bGamma_{0,t_0}$, so that $\bGamma = t_0^{-1} \bGamma_{0,t_0} = \bGamma_0$.
        \item Clearly, for every $t,s \in [0,1]$, whenever either $t,s$ is not in $\{0,1\}$, by \Cref{prop:pt_of_opt_vel_plan_is_optimal} and by unicity of $\bGamma_{t,s}$, we have $\bGamma_{t,s} = (s-t)\PT_t^{(n)}(\bGamma_0)$.
    \end{enumerate}
    This finishes the proof.
\end{proof}

\begin{lemma} \label{lemma:gluing_geodesics}
    Let $(X,d)$ be a metric space. Let $x,y,z \in X$ be such that $d(x,z) = d(x,y) + d(y,z)$, and let $c_1 : [0,1] \mapsto X$ be a constant speed geodesic from $x$ to $y$ and $c_2 : [0,1] \mapsto X$ be a constant speed geodesic from $y$ to $z$. Let $t_0 \in [0,1]$ be such that $t_0 d(x,z) = d(x,z)$. If $t_0 \notin \{0,1\}$, then the curve $c : [0,1] \mapsto X$ defined by
    \begin{equation}
        c(t) := \begin{cases}
            c_1(t/t_0) & \hbox{ if } t \in [0,t_0] \\
            c_2((t-t_0)/(1-t_0)) & \hbox{ if } t \in [t_0,1]
        \end{cases}
    \end{equation}
    is a constant speed geodesic from $x$ to $z$.
\end{lemma}

\begin{proof}
    It is enough to prove that for every $s,t \in [0,1]$, $d(c(t),c(s)) \leq |t-s| d(x,z)$. Indeed, in this case, by the triangular inequality, we have for every $0 \leq t < s \leq 1$ that
    \begin{equation}
        d(x,z) \leq d(c(0),c(t)) + d(c(t),c(s)) + d(c(s),c(1)) \leq (t + (s-t) + (1-s))d(x,z) = d(x,z),
    \end{equation}
    so all these inequalities are equalities and in particular $d(c(t),c(s)) = |t-s|d(x,z)$. \newline
    Let $t,s \in [0,1]$ with $t \leq s$. The possible cases are the following:
    \begin{itemize}
        \item If $t,s \in [0,t_0]$. Then, since $c_1$ is a constant speed geodesic from $x$ to $y$,
        \begin{equation}
            d(c(t),c(s)) = d(c_1(t/t_0),c_1(s/t_0)) = |t-s|t_0^{-1}d(x,y) = |t-s|d(x,z)
        \end{equation}
        by definition of $t_0$.
        \item Similarly if $t,s \in [t_0,1]$, we have
        \begin{align}
            d(c(t),c(s)) &= d(c_2((t-t_0)/(1-t_0)),c_2((s-t_0)/(1-t_0))) \\
            &= |t-s|(1-t_0)^{-1}d(y,z) = |t-s|d(x,z)
        \end{align}
        as indeed $d(y,z) = d(x,z) - d(x,y) = (1-t_0)d(x,z)$.
        \item Finally, if $t < t_0 < s$, then we have
        \begin{align}
            d(c(t),c(s)) &\leq d(c(t),c(t_0)) + d(c(t_0),c(s)) \\
            &\leq d(c_1(t/t_0),c_1(1)) + d(c_2(0),c_2((s-t_0)/(1-t_0))) \\
            &\leq \left(1 - \frac{t}{t_0}\right)d(x,y) + \frac{s-t_0}{1-t_0} d(y,z) \\
            &\leq (t_0 - t)d(x,z) + (s-t_0) d(x,z) = (s-t)d(x,z).
        \end{align}
    \end{itemize}
    This finishes the proof.
\end{proof}

Finally, we show that distinct constant speed geodesics in $\cPn{n}{\cM}$ with the same endpoints not cross at intermediate times, in the following sense:

\begin{proposition}
    Let $n \geq 0$, and let $(\mu_t)_{t \in [0,1]}$, $(\tilde{\mu}_t)_{t \in [0,1]}$ be two constant speed geodesics in $\cPn{n}{\cM}$ such that $\mu_0 = \tilde{\mu}_0 = \mu$ and $\mu_1 = \tilde{\mu}_1 = \nu$. If there exists $t_0 \in (0,1)$ such that $\mu_{t_0} = \tilde{\mu}_{t_0}$, then we have $\mu_t = \tilde{\mu}_t$ for every $t \in [0,1]$.
\end{proposition}

\begin{proof}
    By \Cref{prop:cs_geodesics_come_from_opt_vel_plans}, there exists unique optimal velocity plans $\gamma, \tilde{\gamma} \in \Gamma_o(\mu,\nu)$ such that $\mu_t = [\exp]^{(n)}(t\gamma)$ and $\tilde{\mu}_t = [\exp]^{(n)}(t\tilde{\gamma})$ for every $t \in [0,1]$. Moreover, for every $t \in (0,1)$, the sets $\Gamma_o(\mu,\mu_t)$ and $\Gamma_o(\mu,\tilde{\mu}_t)$ are singletons and equal to respectively $\{t\gamma\}$ and $\{t\tilde{\gamma}\}$. However, since we have assumed that $\mu_{t_0} = \tilde{\mu}_{t_0}$, this implies that $t_0\gamma = t_0\tilde{\gamma}$, so that (since $t_0 \neq 0$) $\gamma = \tilde{\gamma}$. Therefore, for every $t \in [0,1]$, we have $\mu_t = [\exp]^{(n)}(t\gamma) = [\exp]^{(n)}(t\tilde{\gamma}) = \tilde{\mu}_t$. This finishes the proof.
\end{proof}

\section{Functionals on the hierarchical spaces} \label{sec:6_functionals}

\subsection{Subdifferentiability}

\begin{definition}
    Let $n > 0$, and fix two subsets $A, B \subseteq \cPn{n}{T\cM}$. Consider a proper functional $\cF : \cPn{n}{\cM} \mapsto \R \cup \{+\infty\}$, and $\mu \in \cPn{n}{\cM}$ such that $\cF(\mu) < +\infty$. We say that
    \begin{enumerate}
        \item $\gamma \in \cPn{n}{T\cM}_\mu$ is a \emph{regular $(A,B)$-subgradient of $\cF$ at $\mu$}, written $\gamma \in \hat{\partial}_{A,B}\cF(\mu)$, if $\gamma \in A$ and for every $\nu \in \cPn{n}{\cM}$ and every $\xi \in \Gamma(\mu,\nu) \cap B$, it holds $\cF(\nu) - \cF(\mu) \geq \sca{\xi}{\gamma}_\mu + o(\|\xi\|)$. We say that $\cF$ is regularly $(A,B)$-subdifferentiable at $\mu$.
        \item $\gamma \in \cPn{n}{T\cM}$ is a \emph{(general) $(A,B)$-subgradient of $\cF$ at $\mu$}, written $\gamma \in \partial_{A,B} \cF(\mu)$, if there are sequences $\mu_m \to \mu$ with $\cF(\mu_m) \to \cF(\mu)$ and $\gamma_m \in \hat{\partial}_{(A,B)}\cF(\mu_m)$ with $\gamma_m \to \gamma$. We say that $\cF$ is $(A,B)$-subdifferentiable at $\mu$.
    \end{enumerate}
    Likewise, we define regular $(A,B)$-supergradients and general $(A,B)$-supergradients as elements of respectively $-\hat{\partial}_{A,B}(-\cF)(\mu)$ and $-\partial_{A,B}(-\cF)(\mu)$. The functional $\cF$ is said to be \emph{$(A,B)$-differentiable} at $\mu$ there exists a velocity plan which is both a regular $(A,B)$-subgradient and supergradient, that is $\hat{\partial}_{A,B}\cF(\mu) \cap -\hat{\partial}_{A,B}(-\cF)(\mu) \neq \emptyset$. These elements are called $(A,B)$-gradients of $\cF$ at $\mu$.
\end{definition}

\begin{remark} (Landau notation) 
    Here $o(\|\xi\|)$ denotes a term whose exact value depends on any number of variables, but such that $|o(\|\xi\|)| \leq g(\|\xi\|)$, where $g : \R_+ \mapsto \R_+$ is some upper bound such that $g(\veps)/\veps \xrightarrow[\veps \to 0^+]{} 0$.
\end{remark}

\begin{remark}
    If $A \subseteq A'$ and $B' \subseteq B$, then any regular (resp. general) $(A,B)$-subgradient is a regular (resp. general) $(A',B')$-subgradient, and similarly for the supergradients.
\end{remark}

\begin{remark}
    In the case $n = 1$ and $\cM = \R^d$, and when $A = T\cP_2(\cM)$ and $B = \Gamma_o^{(1)}(\cM)$, regular and general $(A,B)$-subgradients correspond to the notion of regular and general subgradients defined in \citep[Definition 2.2]{lanzetti2024variational}, and likewise for supergradients.
\end{remark}

\begin{proposition}
    Consider a proper functional $\cF : \cPn{n}{\cM} \mapsto \R \cup \{+\infty\}$, and $\mu \in \cPn{n}{\cM}$ such that $\cF(\mu) < +\infty$. Then the following statements hold:
    \begin{enumerate}
        \item If $\cG : \cPn{n}{\cM} \mapsto \R \cup \{+\infty\}$ is such that $\cG(\nu) \leq \cF(\nu)$ for every $\nu \in \cPn{n}{\cM}$, and $\cG(\mu) = \cF(\mu)$, then $\hat{\partial}_{A,B}\cG(\mu) \subseteq \hat{\partial}_{A,B}\cF(\mu)$.
        \item An element $\gamma \in \cPn{n}{T\cM}_\mu \cap A$ is an $(A,B)$-gradient of $\cF$ at $\mu$ if and only if for every $\xi \in \cPn{n}{T\cM}_\mu \cap B$ and $\alpha \in \Gamma_\mu(\gamma,\xi)$,
        \begin{equation}
            \cF([\exp]^{(n)}(\xi)) - \cF(\mu) = \bE^{(n)}_\alpha[\sca{v_1}{v_2}_x] + o(\|\xi\|)
        \end{equation}
    \end{enumerate}
\end{proposition}

\begin{proof}
    \begin{enumerate}
        \item Let $\gamma \in \hat{\partial}_{A,B}\cG(\mu)$, then for every $\xi \in \cPn{n}{T\cM}_\mu \cap B$, we have
        \begin{equation}
            \cF([\exp]^{(n)}(\xi)) - \cF(\mu) \geq \cG([\exp]^{(n)}(\xi)) - \cG(\mu) \geq \sca{\gamma}{\xi}_\mu + o(\|\xi\|).
        \end{equation}
        Therefore $\gamma \in \hat{\partial}_{A,B}\cF(\mu)$.
        \item Assume that $\gamma \in \cPn{n}{T\cM}_\mu \cap A$ is an $(A,B)$-gradient of $\cF$ at $\mu$. Then $\gamma$ is a regular $(A,B)$-subgradient of $\cF$ at $\mu$, so that for every $\xi \in \cPn{n}{T\cM}_\mu \cap B$ and $\alpha \in \Gamma_\mu(\gamma,\xi)$,
        \begin{equation}
            \cF([\exp]^{(n)}(\xi)) - \cF(\mu) \geq \sca{\gamma}{\xi}_\mu + o(\|\xi\|) \geq \bE^{(n)}_\alpha[\sca{v_1}{v_2}_x] + o(\|\xi\|).
        \end{equation}
        But $\gamma$ is also a regular $(A,B)$-supergradient of $\cF$ at $\mu$, that is, $-\gamma$ is a regular $(A,B)$-subgradient of $-\cF$ at $\mu$, so that 
        \begin{equation}
            -\cF([\exp]^{(n)}(\xi)) + \cF(\mu) \geq \sca{-\gamma}{\xi}_\mu + o(\|\xi\|),
        \end{equation}
        that is
        \begin{equation}
            \cF([\exp]^{(n)}(\xi)) - \cF(\mu) \leq -\sca{-\gamma}{\xi}_\mu + o(\|\xi\|) \leq -\bE^{(n)}_{\alpha'}[\sca{v_1}{v_2}_x] + o(\|\xi\|) = \bE^{(n)}_\alpha[\sca{v_1}{v_2}_x] + o(\|\xi\|)
        \end{equation}
        where $\alpha' := (-1,1) \cdot \alpha$. Therefore we have
        \begin{equation}
            \cF([\exp]^{(n)}(\xi)) - \cF(\mu) = \bE^{(n)}_\alpha[\sca{v_1}{v_2}_x] + o(\|\xi\|).
        \end{equation}
        Conversely, if this equation holds for any for every $\xi \in \cPn{n}{T\cM}_\mu \cap B$ and $\alpha \in \Gamma_\mu(\gamma,\xi)$, then taking the $\sup$ and the $\inf$ among $\alpha \in \Gamma_\mu(\gamma,\xi)$ (using the fact that $\inf_{\alpha \in \Gamma_\mu(\gamma,\xi)} \bE^{(n)}_\alpha[\sca{v_1}{v_2}_x] = -\sca{-\gamma}{\xi}_\mu$), we obtain that $\gamma$ is both a regular $(A,B)$-subgradient and a regular $(A,B)$-supergradient of $\cF$ at $\mu$.
    \end{enumerate}
\end{proof}

In the following, we will often consider classes $A$ and $B$ which satisfy the following assumption:

\begin{assumption} \label{assumption:a_b_stable}
    The sets $A, B \subseteq \cPn{n}{T\cM}$ satisfy the following conditions:
    \begin{enumerate}
        \item For every $\gamma \in A$, $t \gamma \in B$ for small enough $t \in \R$.
        \item $A$ is stable for the ``vector space structure" of $\cPn{n}{T\cM}$ : that is for every $\mu \in \cPn{n}{\cM}$, $\gamma_1, \gamma_2 \in \cPn{n}{T\cM}_\mu \cap A$, $\lambda \in \R$ and $\alpha \in \Gamma_\mu(\gamma_1, \gamma_2)$, we have $\lambda \gamma_1 \in B$ and $\gamma_1 +_\alpha \gamma_2 \in A$.
    \end{enumerate}
\end{assumption}

\begin{proposition} \label{prop:gradient_is_unique}
    Consider a proper functional $\cF : \cPn{n}{\cM} \mapsto \R \cup \{+\infty\}$, and $\mu \in \cPn{n}{\cM}$ such that $\cF(\mu) < +\infty$. Assume that $A$ and $B$ satisfy \Cref{assumption:a_b_stable}. Then either one of the two following assertions hold:
    \begin{enumerate}
        \item At least one between $\hat{\partial}_{A,B}\cF(\mu)$ or $\hat{\partial}_{A,B}(-\cF)(\mu)$ is empty.
        \item There is an unique $\gamma \in \cPn{n}{T\cM}_\mu$ such that $\hat{\partial}_{A,B}\cF(\mu) = -\hat{\partial}_{A,B}(-\cF)(\mu) = \{\gamma\}$.
    \end{enumerate}
    This proposition allows in particular to speak of ``the" $(A,B)$-gradient of $\cF$ at $\mu$ when it is $(A,B)$-differentiable.
\end{proposition}

\begin{proof}
    Assume that there exists $\gamma_l \in \hat{\partial}_{A,B}\cF(\mu)$ and $-\gamma_u \in \hat{\partial}_{A,B}(-\cF)(\mu)$. We want to prove that $\gamma_l = \gamma_u$. By property of $(A,B)$-subgradients, for every $\xi \in \cPn{n}{T\cM}_\mu \cap B$,
    \begin{equation}
        \cF([\exp]^{(n)}(\xi)) - \cF(\mu) \geq \sca{\gamma_l}{\xi}_\mu + o(\|\xi\|) \label{eq:l_2167}
    \end{equation}
    and by property of $(A,B)$-subgradients,
    \begin{equation}
        -\cF([\exp]^{(n)}(\xi)) + \cF(\mu) \geq \sca{-\gamma_u}{\xi}_\mu + o(\|\xi\|)
    \end{equation}
    that is,
    \begin{align}
        \cF([\exp]^{(n)}(\xi)) - \cF(\mu) &\leq -\sca{-\gamma_u}{\xi}_\mu + o(\|\xi\|) \\
        &\leq - \max_{\alpha \in \Gamma_\mu(-\gamma_u,\xi)} \bE^{(n)}_\alpha[\sca{v_1}{v_2}_x] + o(\|\xi\|) \\
        &\leq \min_{\alpha \in \Gamma_\mu(-\gamma_u, \xi)} \bE^{(n)}_\alpha[\sca{-v_1}{v_2}_x] + o(\|\xi\|) \\
        &\leq \min_{\alpha \in \Gamma_\mu(\gamma_u, \xi)} \bE^{(n)}_\alpha[\sca{v_1}{v_2}_x] + o(\|\xi\|). \label{eq:l_2178}
    \end{align}
    where we used in the fourth line that the component-wise scalar multiplication by $(-1,1)$ defines a bijection between $\Gamma_\mu(\gamma_u,\xi)$ and $\Gamma_\mu(-\gamma_u,\xi)$. Thus, combining \eqref{eq:l_2167} and \eqref{eq:l_2178}, we find
    \begin{equation} \label{eq:l_2181}
        o(\|\xi\|) \geq \max_{\alpha \in \Gamma_\mu(\gamma_l,\xi)} \bE^{(n)}_\alpha[\sca{v_1}{v_2}_x] - \min_{\alpha \in \Gamma_\mu(\gamma_u,\xi)} \bE^{(n)}_\alpha[\sca{v_1}{v_2}_x].
    \end{equation}
    Now, let $\alpha^* \in \Gamma_\mu(\gamma_l,\gamma_u)$ be a optimal coupling, and for every $\veps > 0$, let $\xi_\veps := [\veps(\pi_1-\pi_2)]^{(n)}(\alpha^*)$. We have $\xi_\veps \in B$ for $\veps > 0$ small enough, indeed $\xi_\veps = \veps (\gamma_l -_{\alpha'} (- \gamma_u))$ with $\alpha' = (1,-1) \cdot \alpha$ and $\gamma_u, \gamma_l \in A$, and $A$ and $B$ are assumed to satisfy \Cref{assumption:a_b_stable}. We then have by construction $\|\xi_\veps\| = \veps \W_\mu(\gamma_l,\gamma_u)$, and we also have $\alpha_u = [(x,v_1,v_2) \mapsto (x,v_2,\veps(v_1-v_2))]^{(n)}(\alpha^*) \in \Gamma_\mu(\gamma_u, \xi_\veps)$ and $\alpha_l = [(x,v_1,v_2) \mapsto (x,v_1,\veps(v_1-v_2))]^{(n)}(\alpha^*) \in \Gamma_\mu(\gamma_l, \xi_\veps)$, so that plugging $\xi_\veps$ in \eqref{eq:l_2181}, we find
    \begin{align}
        o(\veps) &\geq \max_{\alpha \in \Gamma_\mu(\gamma_l,\xi_\veps)} \bE^{(n)}_\alpha[\sca{v_1}{v_2}_x] - \min_{\alpha \in \Gamma_\mu(\gamma_u,\xi_\veps)} \bE^{(n)}_\alpha[\sca{v_1}{v_2}_x] \\
        &\geq \bE^{(n)}_{\alpha_l}[\sca{v_1}{v_2}_x] - \bE^{(n)}_{\alpha_u}[\sca{v_1}{v_2}_x] \\
        &\geq \bE^{(n)}_{\alpha^*}[\sca{v_1}{\veps(v_1-v_2)}_x] - \bE^{(n)}_{\alpha^*}[\sca{v_2}{\veps(v_1-v_2)}_x] \\
        &\geq \veps \bE^{(n)}_{\alpha^*}[\|v_1-v_2\|_x^2] = \veps \W^2_\mu(\gamma_l,\gamma_u).
    \end{align}
    Dividing this inequality by $\veps$, and letting $\veps \to 0^+$, we thus find $\W_\mu(\gamma_l,\gamma_u) = 0$, so that $\gamma_l = \gamma_u$.
\end{proof}

\begin{remark} \label{rk:gradient_constant}
    If $A$ and $B$ satisfy \Cref{assumption:a_b_stable} and $\cF : \cPn{n}{\cM} \mapsto \R$ is constant, then $\hat{\partial}_{A,B}\cF(\mu) = \{\bm{0}_\mu\}$ at every $\mu \in \cPn{n}{\cM}$. \newline
    Indeed, $\gamma \in \cPn{n}{T\cM}_\mu \cap A$ is in $\hat{\partial}_{A,B}\cF(\mu)$ if and only if for every $\xi \in \cPn{n}{T\cM}_\mu \cap B$, $0 \geq \sca{\gamma}{\xi}_\mu + o(\|\xi\|)$. Clearly $\gamma = \bm{0}_\mu$ satisfies this condition. Moreover, if $\gamma$ satisfies this condition, since $A$, $B$ satisfy \Cref{assumption:a_b_stable}, we have $\veps \gamma \in B$ for $\veps > 0$ small enough so that, taking $\xi := \veps \gamma$, we find $0 \geq \sca{\gamma}{\veps\gamma}_\mu + o(\veps) = \veps \|\gamma\|^2 + o(\veps)$ using \Cref{prop:various_results_on_inner_prod}, so that dividing by $\veps$ and letting $\veps \to 0$ we find $\|\gamma\| = 0$ and thus $\gamma = \bm{0}_\mu$. \newline
    In particular (since $\|\cdot\|$ is continuous in the $\W_2$ topology), this implies that $\partial_{A,B}\cF(\mu) = \{\bm{0}_\mu\}$ at every $\mu$.
\end{remark}

\begin{proposition} \label{prop:subgradient_is_local}
    Let $\cF_1, \cF_2 : \cPn{n}{\cM} \mapsto \R \cup \{+\infty\}$ be two proper functionals, and $\mu \in \cPn{n}{\cM}$ such that $\cF_1(\mu), \cF_2(\mu) < +\infty$, and $\cF_1$ and $\cF_2$ coincide on a neighborhood of $\mu$. Then $\hat{\partial}_{A,B}\cF_1(\mu) = \hat{\partial}_{A,B}\cF_2(\mu)$.
\end{proposition}

\begin{proof}
    Let $\veps > 0$ be such that for every $\nu \in \cPn{n}{\cM}$ such that $\W_2(\mu,\nu) \leq \veps$, $\cF_1(\nu) = \cF_2(\nu)$. Let  $\gamma \in \hat{\partial}_{A,B}\cF_1(\mu)$. Then, for every $\xi \in \cPn{n}{T\cM}_\mu \cap B$ and $\nu := [\exp]^{(n)}(\xi)$, whenever $\|\xi\| \leq \veps$, we have $\W_2(\nu,\mu) \leq \|\xi\| \leq \veps$, so that
    \begin{equation}
        \cF_2(\nu) = \cF_1(\nu) \geq \cF_1(\mu) + \sca{\gamma}{\xi}_\mu + o(\|\xi\|) = \cF_2(\mu) + \sca{\gamma}{\xi}_\mu + o(\|\xi\|)
    \end{equation}
    and thus $\gamma \in \hat{\partial}_{A,B}\cF_2(\mu)$. We thus have $\hat{\partial}_{A,B}\cF_1(\mu) \subseteq \hat{\partial}_{A,B}\cF_2(\mu)$, and by symmetry, this is an equality.
\end{proof}

\begin{proposition}
    Let $\cF_1, \cF_2 : \cPn{n}{\cM} \mapsto \R \cup \{+\infty\}$ be two proper functionals, $g : \R \cup \{+\infty\} \mapsto \R \cup \{+\infty\}$ a monotone map with $g(+\infty) = +\infty$, and $\mu \in \cPn{n}{\cM}$ such that $\cF_1(\mu), \cF_2(\mu) < +\infty$, $g(\cF_1(\mu)) < +\infty$ and $g$ is continuously differentiable at $\cF_1(\mu)$. Assume that $A$ and $B$ satisfy \Cref{assumption:a_b_stable}. Then the following rules hold:
    \begin{enumerate}
        \item Sum rule:
            \begin{align}
                \hat{\partial}_{A,B}(\cF_1+\cF_2)(\mu) &\supseteq \hat{\partial}_{A,B}\cF_1(\mu) + \hat{\partial}_{A,B}\cF_2(\mu) \\
                &:= \{\gamma_1 +_\alpha \gamma_2 \setcond \gamma_i \in \hat{\partial}_{A,B}\cF_i(\mu), \alpha \in \Gamma_\mu(\gamma_1,\gamma_2) \}.
            \end{align}
        \item For every $\lambda \geq 0$, 
            \begin{equation}
                \hat{\partial}_{A,B}(\lambda \cF_1)(\mu) = \lambda  \hat{\partial}_{A,B}\cF_1(\mu) := \{\lambda \gamma, \gamma \in \hat{\partial}_{A,B}\cF_1(\mu) \}
            \end{equation}
        \item Chain rule: if $g'(\cF_1(\mu)) > 0$ and $\cF_1$ is continuous in the Wasserstein topology at $\mu$, then
            \begin{equation}
                \hat{\partial}_{A,B}(g \circ \cF_1)(\mu) = g'(\cF_1(\mu)) \hat{\partial}_{A,B}\cF_1(\mu).
            \end{equation}
    \end{enumerate}
\end{proposition}

\begin{proof}
    \begin{enumerate}
        \item Let $\gamma_1 \in \hat{\partial}_{A,B}\cF_1(\mu)$, $\gamma_2 \in \hat{\partial}_{A,B}\cF_2(\mu)$ and $\alpha \in \Gamma_\mu(\gamma_1,\gamma_2)$. Since $A$, $B$ satisfy \Cref{assumption:a_b_stable}, we indeed have $\gamma_1 +_\alpha \gamma_2 \in A$. Then, for every $\xi \in \cPn{n}{\cM}_\mu \cap B$, letting $\nu := [\exp]^{(n)}(\xi)$, it holds
        \begin{align}
            \cF_1(\nu) + \cF_2(\nu) &\geq \cF_1(\mu) + \sca{\gamma_1}{\xi}_\mu + \cF_2(\mu) + \sca{\gamma_2}{\xi}_\mu + o(\|\xi\|) \\
            &\geq \sca{\gamma_1 +_\alpha \gamma_2}{\xi}_\mu + o(\|\xi\|_\mu)
        \end{align}
        where we used the definition of the regular subgradient in the first line and point \ref{enum:tangent_struct:inner_prod_subadditive} of \Cref{prop:various_results_on_inner_prod} in the second line. Thus $\gamma_1 +_\alpha \gamma_2 \in \hat{\partial}_{A,B}(\cF_1+\cF_2)(\mu)$.
        \item If $\lambda = 0$, we have by \Cref{rk:gradient_constant} that 
        \begin{equation}
             \hat{\partial}_{A,B}(0\cF_1)(\mu) = 0 \hat{\partial}_{A,B}\cF_1(\mu) = \{\bm{0}_\mu\}.
        \end{equation}
        Thus, we assume $\lambda > 0$. Let $\gamma \in \hat{\partial}_{A,B}\cF_1(\mu)$. Since $A$, $B$ satisfy \Cref{assumption:a_b_stable}, we indeed have $\lambda \gamma \in A$. Then for every $\xi \in \cPn{n}{T\cM}_\mu \cap B$,
        \begin{align}
            \lambda_1 \cF_1([\exp]^{(n)}(\xi)) &\geq \lambda (\cF_1(\mu) + \sca{\gamma}{\xi}_\mu + o(\|\xi\|)) \\
            &\geq \lambda \cF_1(\mu) + \sca{\lambda \gamma}{\xi}_\mu + o(\|\xi\|)
        \end{align}
        where the second line is obtained using \ref{enum:tangent_struct:inner_prod_homogeneous} of \Cref{prop:various_results_on_inner_prod}. Thus $\lambda \gamma \in \hat{\partial}_{A,B}\cF_1(\mu)$, and $\lambda \hat{\partial}_{A,B}\cF_1(\mu) \subseteq \hat{\partial}_{A,B}(\lambda \cF_1(\mu))$. Similarly, we have $\lambda^{-1} \hat{\partial}_{A,B}(\lambda \cF_1)(\mu) \subseteq \hat{\partial}_{A,B}(\cF_1(\mu))$, so that $\lambda \hat{\partial}_{A,B}\cF_1(\mu) = \hat{\partial}_{A,B}(\lambda \cF_1(\mu))$.
        \item Let $\gamma \in \hat{\partial}_{A,B}\cF_1(\mu)$. Since $A$, $B$ satisfy \Cref{assumption:a_b_stable}, we indeed have $g'(\cF_1(\mu)) \gamma \in A$. Then for every $\xi \in \cPn{n}{T\cM}_\mu \cap B$,
        \begin{align}
            g(\cF_1([\exp]^{(n)}(\xi))) &\geq g(\cF_1(\mu) + \sca{\gamma}{\xi}_\mu + o(\|\xi\|)) \\
            &\geq g(\cF_1(\mu)) + g'(\cF_1(\mu))(\sca{\gamma}{\xi}_\mu + o(\|\xi\|)) + o(\sca{\gamma}{\xi}_\mu + o(\|\xi\|)) \\
            &\geq g(\cF_1(\mu)) + \sca{g'(\cF_1(\mu))\gamma}{\xi}_\mu + o(\|\xi\|)
        \end{align}
        where we used the monotonicity of $g$ in the first line, and we wrote the first order Taylor expansion of $g$ at $\cF_1(\mu)$ in the second line. To obtain the third line, we notice in particular that since by point \ref{enum:tangent_struct:inner_prod_cauchy_schwarz} of \Cref{prop:various_results_on_inner_prod}, $|\sca{\gamma}{\xi}_\mu| \leq \|\gamma\|\|\xi\|$, this implies $o(\sca{\gamma}{\xi}_\mu) = o(\|\xi\|)$. Therefore, $g'(\cF_1(\mu))\gamma \in \hat{\partial}_{A,B}(g \circ \cF_1)(\mu)$, so that 
        \begin{equation} \label{eq:l_1693}
            g'(\cF_1(\mu))\hat{\partial}_{A,B}\cF_1(\mu) \subseteq \hat{\partial}_{A,B}(g \circ \cF_1)(\mu).
        \end{equation}
        Since $g$ is continuously differentiable at $\cF_1(\mu)$ with $g'(\cF_1(\mu)) > 0$, there exists a monotone function $h : \R \mapsto \R$ which is a local inverse of $g$ in a neighborhood of $g(\cF_1(\mu))$. In particular, it is continuously differentiable at $g(\cF_1(\mu))$ and $h'(g(\cF_1(\mu))) = g'(\cF_1(\mu))^{-1}$. Applying \eqref{eq:l_1693} to $g \circ \cF_1$ and $h$, we find that 
        \begin{equation} \label{eq:l_1697}
            h'(g(\cF_1(\mu)))\hat{\partial}_{A,B}(g \circ \cF_1)(\mu) \subseteq \hat{\partial}_{A,B}(h \circ g \circ \cF_1)(\mu).
        \end{equation}
        However, since $\cF_1$ is assumed to be continuous at $\mu$, this implies that $h \circ g \circ \cF_1$ coincides with $\cF_1$ in a neighborhood of $\mu$, so that by \Cref{prop:subgradient_is_local}, they have the same $(A,B)$-regular subgradients at $\mu$. Therefore, \eqref{eq:l_1697} becomes
        \begin{equation}
            h'(g(\cF_1(\mu))) \hat{\partial}_{A,B}(g \circ \cF_1)(\mu) \subseteq \hat{\partial}_{A,B}\cF_1(\mu).
        \end{equation}
        Using the fact that $h'(g(\cF_1(\mu))) = g'(\cF_1(\mu))^{-1}$, and combining this with \eqref{eq:l_1693}, we obtain
        \begin{equation}
            g'(\cF_1(\mu))\hat{\partial}_{A,B}\cF_1(\mu) = \hat{\partial}_{A,B}(g \circ \cF_1)(\mu).
        \end{equation}
    \end{enumerate}
\end{proof}

Using our framework, we can establish the existence of gradients of some basic functionals defined on $\cPn{n}{\cM}$. In the following, we will \textbf{always} take $A = B = \cPn{n}{T\cM}$, which clearly satisfy \Cref{assumption:a_b_stable}.

\subsection{Potential functions} 
Consider first the case of potentials. Let $V : \cM \mapsto \R$ be a $C^2$ function whose Hessian is uniformly bounded by $L \geq 0$ in operator norm, and, for every $n > 0$, define the functional $\cV^{(n)} : \cPn{n}{\cM} \mapsto \R$ by $\cV^{(n)}(\mu) := \bE^{(n)}_\mu[V]$. We show that this functional is well-defined and continuous:

\begin{lemma} \label{lemma:taylor_ineq_bounded_hessian}
    Let $f : \cM \mapsto \R$ be a $C^2$ function with Hessian uniformly bounded by $L$ in operator norm. Then for every $(x,v) \in T\cM$,
    \begin{equation}
        |f(\exp_x(v)) - f(x) - \sca{\nabla_\cM f(x)}{v}_x| \leq \frac 12 L \|v\|^2_x
    \end{equation}
\end{lemma}

\begin{proof}
    Fix $(x,v) \in T\cM$. Applying \citep[Exercise 5.40]{boumal2023introduction} to the geodesic $c : t \in [0,1] \mapsto \exp_x(tv)$, we have that there exists $t \in (0,1)$ such that
    \begin{equation}
        f(\exp_x(v)) = f(x) + \sca{\nabla_\cM f(x)}{v}_x + \frac 12 \sca{\Hess f[c(t)](c'(t))}{c'(t)}_{c(t)} + \frac 12 \sca{\nabla_\cM f(x)}{c"(t)}_{c(t)} 
    \end{equation}
    Since $c$ is a geodesic, $c"(t) = 0$ and $\|c'(t)\|_{c(t)} = \|v\|_x$, so that 
    \begin{equation}
        |f(\exp_x(v)) - f(x) - \sca{\nabla_\cM f(x)}{v}_x| \leq  \frac 12 \|\Hess f[c(t)](c'(t))\|_{c(t)}\|c'(t)|_{c(t)} \leq \frac 12 L \|c'(t)\|^2_{c(t)} = \frac 12 L \|v\|^2_x
    \end{equation}
\end{proof}

Fix some base point $x_0 \in \cM$. Then, for every $x \in \cM$, applying \Cref{lemma:taylor_ineq_bounded_hessian} to $V$ and $(x_0,v) = \log_{x_0}(x)$, we find that 
\begin{equation}
    V(x) \leq V(x_0) + \sca{\nabla_\cM V(x_0}{v}_{x_0} + \frac 12 L \|v\|^2_x \leq C_1 + C_2 \|v\|^2_{x_0} = C_1 + C_2 d^2(x,x_0)
\end{equation}
where $C_1, C_2$ are constants depending only on $V(x_0)$, $\nabla_\cM V(x_0)$, and $L$. We can derive a similar lower bound on $V$, so that by \Cref{lemma:n_expectancy_regularity}, $\cV^{(n)}$ is well-defined and continuous on $\cPn{n}{\cM}$ for every $n > 0$. We now show that $\cV^{(n)}$ has an $(A,B)$-gradient at every point:

\begin{proposition}
    For every $\mu \in \cPn{n}{\cM}$, the velocity plan $\gamma_0 := [\nabla_\cM V]^{(n)}(\mu)$ is the $(A,B)$-gradient of $\cV^{(n)}$ at $\mu$. In fact, for every $\gamma \in \cPn{n}{\cM}_\mu$, $\alpha \in \Gamma_\mu(\gamma,\gamma_0)$ and $\nu = [\exp]^{(n)}(\gamma)$, it holds
    \begin{equation}
        |\cV^{(n)}(\nu) - \cV^{(n)}(\mu) - \bE^{(n)}_\alpha[\sca{v_1}{v_2}_x]| \leq \frac 12 L \|\gamma\|^2
    \end{equation}
\end{proposition}

\begin{proof}
    We prove this by induction. For $n = 0$, this is just \Cref{lemma:taylor_ineq_bounded_hessian}. Let $n > 0$, and assume the proposition holds for $n-1$. Let $\bP \in \cPn{n}{\cM}$, $\bGamma_0 := [\nabla_\cM V]^{(n)}(\bP)$, and let $\bGamma \in \cPn{n}{T\cM}_\bP$, $\bQ := [\exp]^{(n)}(\bGamma)$, and $\bA \in \Gamma_\bP(\bGamma, \bGamma_0)$. We have
    \begin{equation}
        \cV^{(n)}(\bP) = \int \bE^{(n-1)}_\mu[V] \dd\bP(\mu) = \int \cV^{(n-1)}(\mu)\dd\bP(\mu) = \int \cV^{(n-1)}([\pi \circ \pi_1]^{(n-1)}(\alpha)) \dd\bA(\alpha)
    \end{equation}
    and similarly
    \begin{equation}
        \cV^{(n)}(\bQ) = \int \cV^{(n-1)}([\exp \circ \pi_1]^{(n-1)}(\alpha)) \dd\bA(\alpha)
    \end{equation}
    \begin{equation}
        \bE_\bA^{(n)}[\sca{v_1}{v_2}_x] = \int \bE^{(n-1)}_\alpha[\sca{v_1}{v_2}_x] \dd\bA(\alpha)
    \end{equation}
    so that 
    \begin{equation}
        \cV^{(n)}(\bQ) - \cV^{(n)}(\bP) - \bE^{(n)}_\bA[\sca{v_1}{v_2}_x] = \int \cV^{(n-1)}(\nu_\alpha) - \cV^{(n-1)}(\mu_\alpha) - \bE^{(n-1)}_\alpha[\sca{v_1}{v_2}_x] \dd\bA(\alpha)
    \end{equation}
    with the notations $\gamma_\alpha := [\pi_1]^{(n-1)}(\alpha)$, $\nu_\alpha := [\exp]^{(n-1)}(\gamma_\alpha)$, $\mu_\alpha := [\pi]^{(n-1)}(\gamma_\alpha)$, $\eta_\alpha := [\pi_2]^{(n-1)}(\alpha)$. Since $\bGamma_0 = [\pi_2]^{(n)}(\bA) = [\nabla_\cM V]^{(n)}(\bP)$, this implies that for $\bA$-a.e. $\alpha$, $\eta_\alpha = [\nabla_\cM V]^{(n-1)}(\mu_\alpha)$. Thus, by the induction hypothesis, 
    \begin{align}
        |\cV^{(n)}(\bQ) - \cV^{(n)}(\bP) - \bE^{(n)}_\bA[\sca{v_1}{v_2}_x]| &\leq \int |\cV^{(n-1)}(\nu_\alpha) - \cV^{(n-1)}(\mu_\alpha) - \bE^{(n-1)}_\alpha[\sca{v_1}{v_2}_x]| \dd\bA(\alpha) \\
        &\leq \frac 12 L \int \|\gamma_\alpha\|^2 \dd\bA(\alpha) = \frac 12 L \int \|\gamma\|^2 \dd\bGamma(\gamma) = \frac 12 L \|\bGamma\|^2
    \end{align}
    This finishes the proof.
\end{proof}

Notice that the gradient of $\cV^{(n)}$ is actually \emph{Lipschitz continuous}, as it is given by $\mu \mapsto [\nabla_\cM V]^{(n)}(\mu)$.

\begin{remark}
    We can use a similar argument to show that if, for some fixed $k > 0$, $\cV : \cPn{k}{\cM} \mapsto \R$ is a continuous functional which admits a $L$-Lipschitz continuous $(A,B)$-gradient $\mu \mapsto \gW \cV(\mu)$, such that for every $\mu, \nu \in \cPn{k}{\cM}$, $\gamma \in \Gamma(\mu,\nu)$ and $\alpha \in \Gamma_\mu(\gamma, \gW \cV(\mu))$, the inequality
    \begin{equation}
        |\cV(\nu) - \cV(\mu) - \bE^{(k)}_\alpha[\sca{v_1}{v_2}_x]| \leq \frac 12 L \|\gamma\|^2
    \end{equation}
    holds, then for every $n > k$, the functional $\cV^{(n)}$ defined by $\cV^{(n)}(\mu) := \bE^{(n-k)}_\mu[\cV]$ is well-defined and continuous, has a $L$-Lipschitz continuous gradient given by $\mu \mapsto [\gW \cV]^{(n-k)}(\mu)$, and is such that for every $\mu, \nu \in \cPn{n}{\cM}$, $\gamma \in \Gamma(\mu,\nu)$ and $\alpha \in \Gamma_\mu(\gamma, [\gW \cV]^{(n-k)}(\mu))$, the inequality
    \begin{equation}
        |\cV^{(n)}(\nu) - \cV^{(n)}(\mu) - \bE^{(n)}_\alpha[\sca{v_1}{v_2}_x]| \leq \frac 12 L \|\gamma\|^2
    \end{equation}
    holds.
\end{remark}

Finally, we prove some results relating to the convexity of potential functionals.

\begin{definition}
    Let $(X,d)$ be a metric space, $\lambda \in \R$, and consider a function $\phi : X \mapsto \R \cap \{+\infty\}$ with proper effective domain $D(\phi) = \{x \in X \setcond \phi(x) < +\infty\} \neq \emptyset$. If $c : [0,1] \mapsto X$ is a curve in $X$, $\phi$ is said to be \emph{$\lambda$-convex} (or simply \emph{convex} when $\lambda = 0$) on $c$ if 
    \begin{equation} \label{eq:lambda_convexity}
        \phi(c(t)) \leq (1-t) \phi(x) + t \phi(y) - \frac 12 \lambda t(1-t) d^2(x,y), \quad \forall t \in [0,1]. 
    \end{equation}
    The function $\phi$ is said to be \emph{$\lambda$-geodesically convex} (or simply \emph{geodesically convex} when $\lambda = 0$) if for any $x,y \in D(\phi)$, there exists a constant speed geodesic $c : [0,1] \mapsto X$ from $x$ to $y$ on which $\phi$ is $\lambda$-convex ; and it is said to be \emph{totally $\lambda$-geodesically convex} (or \emph{totally geodesically convex} when $\lambda = 0$) if $\phi$ is $\lambda$-convex on \emph{every} constant speed geodesic from $x$ to $y$ for every $x,y \in D(\phi)$. 
\end{definition}

\begin{proposition}
    Let $n \geq 0$ and $\cV : \cPn{n}{\cM} \mapsto \R$. Assume that $\cV$ is lower semicontinuous, and that there exists $A,B > 0$ and $\mu_0 \in \cPn{n}{\cM}$ such that $\phi(\mu) \geq -A - B\W_2^2(\mu_0,\mu)$ for every $\mu \in \cPn{n}{\cM}$. Let $\cF : \cPn{n+1}{\cM} \mapsto \R \cup \{+\infty\}$ be the functional defined by $\cF(\mu) := \bE_\mu[\cV]$ for every $\mu \in \cPn{n+1}{\cM}$, then: 
    \begin{itemize}
        \item If $\cV$ is continuous and $\lambda$-geodesically convex, then $\cF$ is $\lambda$-geodesically convex.
        \item If $\cV$ is totally $\lambda$-geodesically convex, then $\cF$ is totally $\lambda$-geodesically convex.
    \end{itemize}
\end{proposition}

\begin{proof}
    By \Cref{lemma:n_expectancy_regularity}, we know that $\cF$ is lower semicontinuous, and that $\cF(\bP) \geq -A - B\W_2^2(\bP,\delta_{\mu_0})$ for every $\bP \in \cPn{n+1}{\cM}$. Let $F$ be the set of the $\gamma \in \Gamma_o^{(n)}(\cM)$ such that $\cV$ is $\lambda$-convex on the constant speed geodesic $t \mapsto [\exp]^{(n)}(t\gamma)$. We will show that for any $\bGamma \in \Gamma_o^{(n+1)}(\cM)$ concentrated on $F$, $\cF$ is $\lambda$-convex on the constant speed geodesic $t \mapsto [\exp]^{(n+1)}(t\bGamma)$. Indeed, letting $\bP = [\pi]^{(n+1)}(\bGamma)$ and $\bQ = [\exp]^{(n+1)}(\bGamma)$, we have
    \begin{align}
        \cF([\exp]^{(n+1)}(t\bGamma)) &= \int \cV([\exp]^{(n+1)}(t\gamma)) \dd\bGamma(\gamma) \\
        &\leq \int (1-t) \cV([\pi]^{(n)}(\gamma)) + t\cV([\exp]^{(n)}(\gamma)) \\
        &\quad - \frac 12 \lambda t(1-t) \W_2^2([\pi]^{(n)}(\gamma), [\exp]^{(n)}(\gamma)) \dd\bGamma(\gamma) \\
        &\leq (1-t) \int \cV(\mu) \dd\bP(\mu) + t \int \cV(\mu) \dd\bQ(\mu) \\
        &\quad - \frac 12 \lambda t(1-t)  \int \W_2^2([\pi]^{(n)}(\gamma), [\exp]^{(n)}(\gamma)) \dd\bGamma(\gamma) \\
        &\leq (1-t) \cF(\bP) + t \cF(\bQ) - \frac 12 \lambda t(1-t) \W_2^2(\bP, \bQ),
    \end{align}
    where the second inequality comes from the fact that $\bGamma$ is concentrated on the set $F$, and the fourth inequality comes from the optimality of $\bGamma$ and \Cref{prop:opt_vel_plan_gives_opt_trans_plan}.
    \medbreak
    Now, assume that $\cV$ is totally $\lambda$-geodesically convex. Then $F = \Gamma_o^{(n)}(\cM)$. Let $\bP, \bQ \in \cPn{n+1}{\cM}$ and let $c : [0,1] \mapsto \cPn{n+1}{\cM}$ be a constant speed geodesic from $\bP$ to $\bQ$. By \Cref{prop:cs_geodesics_come_from_opt_vel_plans}, there exists $\bGamma \in \Gamma_o(\bP,\bQ)$ such that $c(t) = [\exp]^{(n+1)}(t\bGamma)$ for every $t \in [0,1]$, and by \Cref{prop:opt_vel_plan_gives_opt_trans_plan}, $\bGamma$ is supported on $F = \Gamma_o^{(n+1)}(\cM)$. Therefore, by what we have just proved, $\cF$ is $\lambda$-convex on $c$. Thus, $\cF$ is totally $\lambda$-geodesically convex.
    \medbreak
    Finally, assume that $\cV$ is continuous and $\lambda$-geodesically convex. Then:
    \begin{itemize}
        \item It holds $([\pi]^{(n)},[\exp]^{(n)})(F) = \cPn{n}{\cM} \times \cPn{n}{\cM}$. Indeed, if $\mu, \nu \in \cPn{n}{\cM}$, there exists by $\lambda$-geodesical convexity of $\cV$ a constant speed geodesic $c: [0,1] \mapsto \cPn{n}{\cM}$ from $\mu$ to $\nu$ on which $\cV$ is $\lambda$-convex, and by \Cref{prop:cs_geodesics_come_from_opt_vel_plans} there exists $\gamma \in \Gamma_o(\mu,\nu)$ such that $c(t) = [\exp]^{(n)}(t\gamma)$ for every $t \in [0,1]$. Thus $\gamma \in F$ with $[\pi]^{(n)}(\gamma) = \mu$, $[\exp]^{(n)}(\gamma) = \nu$.
        \item The set $F$ is closed. Indeed, $F = \Gamma_o^{(n)}(\cM) \cap \bigcap_{t \in [0,1]} f_t^{-1}((-\infty,0])$ where $f_t : \cPn{n}{T\cM} \mapsto \R$ is the continuous function defined by $f_t(\gamma) = \cV([\exp]^{(n)}(t\gamma)) - (1-t)\cV([\pi]^{(n)}(\gamma)) - t \cV([\exp]^{(n)}(\gamma)) + \frac 12 \lambda t(1-t) \W_2^2([\pi]^{(n)}(\gamma), [\exp]^{(n)}(\gamma))$. 
    \end{itemize}
    Let then $\bP, \bQ \in \cPn{n+1}{\cM}$ be fixed. By \Cref{rk:opt_vel_plans_exist_restricted}, there exists $\bGamma \in \Gamma_o(\bP, \bQ)$ which is supported on $F$. Then, by what precedes, $\cF$ is $\lambda$-convex on the constant speed geodesic $t \mapsto [\exp]^{(n+1)}(t\bGamma)$. Thus $\cF$ is $\lambda$-geodesically convex.
\end{proof}

\subsection{The higher-order Wasserstein distances} 

Now, consider the case of the square of the Wasserstein distance, with respect to a fixed reference measure. We fix $n \geq 0$ and $\bar{\mu} \in \cPn{n}{\cM}$, and we consider the functional $\cF : \cPn{n}{\cM} \mapsto \R$ defined by $\cF(\mu) := \frac 12 \W_2^2(\mu,\bar{\mu})$. A first result is that, when $\cM$ has nonnegative curvature, supergradients of $\cF$ are given by optimal velocity plans.

\begin{lemma} \label{lemma:triangle_ineq_nonnegative_curvature}
    Let $\cM$ be a Riemannian manifold with nonnegative sectional curvature. Let $x \in \cM$ and $u, v \in T_x\cM$, then we have $d\big(\exp_x(u),\exp_x(v)\big) \leq \|u - v\|_x$.
\end{lemma}

\begin{proof}
    Let $x\in\cM$, $\xi,\eta\in T_x\cM$. Let also $y = \exp_x(\xi)$ and $z = \exp_x(\eta)$. Consider also vectors $\xi', \eta' \in \R^d$ such that $\|\xi\|_x = \|\xi'\|$, $\|\eta\|_x = \|\eta'\|$ and $\sca{\xi}{\eta}_x = \sca{\xi'}{\eta'}$, and let $x',y',z' \in \R^d$ be points such that $y' = x' + \xi'$ and $z' = x' + \eta'$. Then, applying Toponogov's theorem \citep[Theorem 79 (Hinge version)]{petersen2006riemannian}, considering the hinge in $\cM$ formed by the geodesics $t \mapsto \exp_x((1-t)\xi)$ and $t \mapsto \exp_x(t\eta)$, with vertices $y,x,z$, and the comparison hinge in $\R^d$ (the manifold of dimension $d$ with constant sectional curvature 0) formed by the geodesics $t \mapsto x' + (1-t)\xi'$, $t \mapsto x' + t\eta'$, with vertices $y',x',z'$, we find that $d(y,z) \leq \|y'-z'\|$, that is, $d\big(\exp_x(\xi),\exp_x(\eta)\big) \leq \|\xi'-\eta'\| = \|\xi-\eta\|_x$ (note that although the statement of the theorem requires the geodesics of the hinges to be minimal, its proof only requires that the geodesics of the comparison hinge are minimal, which is always the case in $\R^d$ - we refer to the discussion in \citep[Chapter 11.2]{petersen2006riemannian} just after the proof of the theorem).
\end{proof}

\begin{proposition} \label{prop:opt_vel_plans_give_w2_supergradient}
    Assume that $\cM$ has nonnegative sectional curvature. Let $\mu \in \cPn{n}{\cM}$, then for any $\bar{\gamma} \in \Gamma_o(\mu,\bar{\mu})$, $-\bar{\gamma}$ is a regular $(A,B)$-supergradient of $\cF$ at $\mu$. In fact, for every $\nu \in \cPn{n}{\cM}$, $\gamma \in \Gamma(\mu,\nu)$ and $\alpha \in \Gamma_\mu(\gamma,\bar{\gamma})$, it holds
    \begin{equation} \label{eq:hierarchical_w2_supergradient}
        \cF(\nu) \leq \cF(\mu) - \bE^{(n)}_\alpha[\sca{v_1}{v_2}_x] + \frac 12 \|\gamma\|^2.
    \end{equation}
\end{proposition}

\begin{proof}
    Taking $\alpha \in \Gamma_\mu(\gamma,\bar{\gamma})$ to be optimal, \eqref{eq:hierarchical_w2_supergradient} rewrites as
    \begin{equation}
        -\cF(\nu) \geq -\cF(\mu) + \sca{\gamma}{\bar{\gamma}}_\mu - \frac 12 \|\gamma\|^2
    \end{equation}
    for every $\nu \in \cPn{n}{\cM}$ and $\gamma \in \Gamma(\mu,\nu)$. In particular, this means that $\bar{\gamma}$ is a regular $(A,B)$-subgradient of $-\cF$ at $\mu$, that is $-\bar{\gamma}$ is a regular $(A,B)$-supergradient of $\cF$ at $\mu$. Thus the first part of the proposition is a consequence of \eqref{eq:hierarchical_w2_supergradient}, and we only need to prove that this inequality holds.
    \medbreak
    We prove this by induction. We first examine the case $n = 0$, in which the problem reformulates as such: let $x_0 \in \cM$ fixed, for every $x,y \in \cM$, and $v,v_0 \in T_x \cM$ such that $\exp_x(v) = y$, $\exp_x(v_0) = x_0$ and $d(x,x_0) = \|v_0\|_x$, we want to show that
    \begin{equation}
        \frac 12 d^2(y,x_0) \leq \frac 12 d^2(x,x_0) - \sca{v}{v_0}_x + \frac 12 \|v\|^2_x.
    \end{equation}
    This is the case since
    \begin{align}
        \frac 12 d^2(y,x_0) &= \frac 12 d^2(\exp_x(v), \exp_x(v_0)) \\
        &\leq \frac 12 \|v - v_0\|^2_x = \frac 12 \|v_0\|^2_x - \sca{v}{v_0}_x + \frac 12 \|v\|^2_x \\
        &\leq \frac 12 d^2(x,x_0) - \sca{v}{v_0}_x + \frac 12 \|v\|^2_x
    \end{align}
    where we used that $\cM$ has nonnegative sectional curvature and \Cref{lemma:triangle_ineq_nonnegative_curvature} to obtain the second line. This shows the Proposition for $n = 0$.
    \medbreak
    Now, let $n > 1$, and assume that the proposition holds for $n - 1$. Let $\bar{\bP} \in \cPn{n}{\cM}$, and define the functional $\cF : \cPn{n}{\cM} \mapsto \R$ by $\cF(\bP) := \frac 12 \W_2^2(\bP,\bar{\bP})$. Let $\bP, \bQ \in \cPn{n}{\cM}$, $\bar{\bGamma} \in \Gamma_o(\bP, \bar{\bP})$, $\bGamma \in \Gamma(\bP, \bQ)$ and $\bA \in \Gamma_\bP(\bGamma, \bar{\bGamma})$. Then $([\exp \circ \pi_1]^{(n-1)}, [\exp \circ \pi_2]^{(n-1)})_\#\bA$ is a (not necessarily optimal) transport plan between $\bQ$ and $\bar{\bP}$. Therefore,
    \begin{align}
        \cF(\bQ) &= \frac 12 \W_2^2(\bQ,\bar{\bP}) \leq \frac 12 \int \W^2_2([\exp \circ \pi_1]^{(n-1)}(\alpha), [\exp \circ \pi_2]^{(n-1)}(\alpha)) \dd\bA(\alpha) \\
        &\leq \frac 12 \int \W_2^2(\nu_\alpha, \bar{\mu}_\alpha) \dd\bA(\alpha)
    \end{align}
    with the notations $\gamma_\alpha = [\pi_1]^{(n-1)}(\alpha)$, $\bar{\gamma}_\alpha = [\pi_2]^{(n-1)}(\alpha)$, $\nu_\alpha = [\exp]^{(n-1)}(\gamma_\alpha)$, $\bar{\mu}_\alpha = [\exp]^{(n-1)}(\bar{\gamma}_\alpha)$ and $\mu_\alpha = [\pi]^{(n-1)}(\gamma_\alpha)$. Since $\bar{\bGamma}$ is optimal, it is concentrated on $\Gamma_o^{(n-1)}(\cM)$, so that $\bar{\gamma}_\alpha$ is optimal for $\bA$-a.e. $\alpha$. Thus, by the induction hypothesis, for $\bA$-a.e. $\alpha$, it holds
    \begin{equation}
        \frac 12 \W_2^2(\nu_\alpha, \bar{\mu}_\alpha) \leq \frac 12 \W_2^2(\mu_\alpha, \bar{\mu}_\alpha) - \bE^{(n-1)}_\alpha[\sca{v_1}{v_2}_x] + \frac 12 \|\gamma_\alpha\|^2
    \end{equation}
    so that
    \begin{align}
        \cF(\bQ) &\leq \int \frac 12 \W_2^2(\mu_\alpha, \bar{\mu}_\alpha) - \bE^{(n-1)}_\alpha[\sca{v_1}{v_2}_x] + \frac 12 \|\gamma_\alpha\|^2 \dd\bA(\alpha) \\
        &\leq \frac 12 \int \W_2^2([\pi]^{(n-1)}(\gamma), [\exp]^{(n-1)}) \dd\bar{\bGamma}(\gamma) - \bE^{(n)}_\bA[\sca{v_1}{v_2}_x] + \frac 12 \int \|\gamma\|^2 \dd\bGamma(\gamma) \\
        &\leq \frac 12 \W_2^2(\bP,\bar{\bP}) - \bE^{(n)}_\bA[\sca{v_1}{v_2}_x] + \frac 12 \|\bGamma\|^2 = \cF(\bP) - \bE^{(n)}_\bA[\sca{v_1}{v_2}_x] + \frac 12 \|\bGamma\|^2
    \end{align}
    where we used the optimality of $\bar{\bGamma}$ in the last line. This finishes the proof.
\end{proof}

A more difficult question concerns the existence of regular $(A,B)$-subgradients of $\cF$. Again assuming that $\cM$ has nonnegative curvature, since by \Cref{prop:opt_vel_plans_give_w2_supergradient} each optimal velocity plan $\bar{\gamma} \in \Gamma_o(\mu,\bar{\mu})$ gives a regular supergradient, by \Cref{prop:gradient_is_unique}, a necessary condition for the existence of regular subgradients at $\mu$ is the unicity of the optimal velocity plan $\bar{\gamma} \in \Gamma_o(\mu,\bar{\mu})$. On the other hand, a sufficient condition is given by the following proposition:

\begin{proposition}
    Assume that $\cM$ has nonnegative sectional curvature. Let $\mu \in \cPn{n}{\cM}$ such that there exists an unique optimal velocity plan $\bar{\gamma} \in \Gamma_o(\mu,\bar{\mu})$, and assume that $\bar{\gamma}$ is fully deterministic. Then $-\bar{\gamma}$ is a regular $(A,B)$-subgradient of $\cF$ at $\mu$.
\end{proposition}

\begin{proof}
    It is sufficient to prove that, for every sequence $(\xi_m)_m \in \cPn{n}{T\cM}_\mu$ of velocity plans such that $\|\xi_m\| \to 0$, we have the inequality
    \begin{equation}
        \liminf_{m \to +\infty} \frac{\cF([\exp]^{(n)}(\xi_m)) - \cF(\mu) - \sca{\xi_m}{-\bar{\gamma}}_\mu}{\|\xi_m\|} \geq 0.
    \end{equation}
    (We may assume $\|\xi_m\| \neq 0$ because otherwise, $\xi_m = \bm{0}_\mu$ and $\cF([\exp]^{(n)}(\xi_m)) - \cF(\mu) - \sca{\xi_m}{-\bar{\gamma}}_\mu = \cF(\mu) - \cF(\mu) - 0 = 0$.) 
    \medbreak
    For every $m$, let $\mu_m := [\exp]^{(n)}(\xi_m)$, and fix $\bar{\gamma}_m \in \Gamma_o(\mu_m,\bar{\mu})$. Let also $\alpha_m \in \Gamma_\mu(\xi_m,\bar{\gamma})$. Since $\bar{\gamma}$ is fully deterministic, by \Cref{prop:coupling_with_fully_det_is_unique}, $\alpha_m$ is unique. In particular, $(1,-1)\alpha_m$ is also the unique element of $\Gamma_\mu(\xi_m,-\bar{\gamma})$ and is thus optimal. Define furthermore $\tilde{\alpha}_m := (-1,1) \cdot [\PT^2_{1,1}]^{(n)}(\alpha_m)$, that is, $\tilde{\alpha}_m$ is given by
    \begin{equation}
        \tilde{\alpha}_m = [(x,v_1,v_2) \mapsto (\exp_x(v_1), -\PT_1(x,v_1,v_1),\PT_1(x,v_1,v_2))]^{(n)}(\alpha_m).
    \end{equation}
    Letting $\tilde{\xi}_m := [\pi_1]^{(n)}(\tilde{\alpha}_m)$ and $\tilde{\gamma}_m := [\pi_2]^{(n)}(\tilde{\alpha}_m)$, we then have $\tilde{\xi}_m, \tilde{\gamma}_m \in \cPn{n}{T\cM}_{\mu_m}$, $\tilde{\xi}_m \in \Gamma(\mu_m,\mu)$, and, since the parallel transport operator is an isometry, $\|\tilde{\xi}_m\| = \|\xi_m\|$. Let then $\beta_m \in \Gamma_{\mu_m}(\tilde{\xi}_m,\bar{\gamma}_m,\tilde{\gamma}_m)$ be such that $[\pi^3_{1,3}]^{(n)}(\beta_m) = \tilde{\alpha}_m$ and such that $[\pi^3_{2,3}]^{(n)}(\beta_m)$ is a optimal coupling between $\bar{\gamma}_m$ and $\tilde{\gamma}_m$ (such a $\beta_m$ exists by \Cref{lemma:generalized_couplings_exist_3}). Since $\bar{\gamma}_m$ is optimal from $\mu_m$ to $\bar{\mu}$, applying \eqref{eq:hierarchical_w2_supergradient} in \Cref{prop:opt_vel_plans_give_w2_supergradient} to $\nu \leftarrow \mu$, $\mu \leftarrow \mu_n$, $\alpha \leftarrow [\pi_{1,2}^3]^{(n)}(\beta_m)$ and $\gamma \leftarrow \tilde{\xi}_m$, we find 
    \begin{equation} \label{eq:l_2505}
        \cF(\mu) \leq \cF(\mu_m) - \bE^{(n)}_{\beta_m}[\sca{v_1}{v_2}_x] + \frac 12 \|\tilde{\xi}_m\|^2 = \cF(\mu_m) - \bE^{(n)}_{\beta_m}[\sca{v_1}{v_2}_x] + \frac 12 \|\xi_m\|^2.
    \end{equation}
    Furthermore, we have 
    \begin{equation} \label{eq:l_2509}
        \bE^{(n)}_{\beta_m}[\sca{v_1}{v_2}_x] = \bE^{(n)}_{\beta_m}[\sca{v_1}{v_2-v_3}] + \bE^{(n)}_{\beta_m}[\sca{v_1}{v_3}_x],
    \end{equation}
    with
    \begin{align}
        \bE^{(n)}_{\beta_m}[\sca{v_1}{v_3}_x] &= \bE^{(n)}_{\tilde{\alpha}_m}[\sca{v_1}{v_2}_x] \\
        &= \bE^{(n)}_{\alpha_m}[\sca{-\PT_1(x,v_1,v_2)}{\PT_1(x,v_1,v_2)}_{\exp_x(v_1)}] \\
        &= \bE^{(n)}_{\alpha_m}[-\sca{v_1}{v_2}_x] \\
        &= \bE^{(n)}_{(1,-1)\alpha_m}[\sca{v_1}{v_2}_x] = \sca{\xi_m}{-\bar{\gamma}}_\mu \label{eq:l_2517},
    \end{align}
    where we obtained the third line by using the fact that the parallel transport operator is an isometry, and the fourth line by using the optimality of $(1,-1)\alpha_m$ ; and with
    \begin{align}
        |\bE^{(n)}_{\beta_m}[\sca{v_1}{v_2-v_2}_x]| &\leq  \bE^{(n)}_{\beta_m}[|\sca{v_1}{v_2-v_3}_x|] \\
        &\leq \bE^{(n)}_{\beta_m}[\|v_1\|_x\|v_2-v_3\|_x] \\
        &\leq \sqrt{\bE^{(n)}_{\beta_m}[\|v_1\|^2_x]} \sqrt{\bE^{(n)}_{\beta_m}[\|v_2-v_3\|_x|^2]} \\
        &\leq \|\tilde{\xi}_m\| \W_{\mu_m}(\bar{\gamma}_m,\tilde{\gamma}_m) \\
        &\leq \|\xi_m\| \W_{\mu_m}(\bar{\gamma}_m,\tilde{\gamma}_m) \label{eq:l_2525},
    \end{align}
    where we used the Cauchy-Schwarz inequality (\Cref{lemma:n_holder_ineq}) to obtain the third line, and the optimality of $[\pi^3_{2,3}]^{(n)}(\beta_m)$ to obtain the fourth line. Combining Equations \eqref{eq:l_2505}, \eqref{eq:l_2509}, \eqref{eq:l_2517} and \eqref{eq:l_2525}, we then obtain
    \begin{equation} \label{eq:l_2540}
        \cF(\mu_m) \geq \cF(\mu) + \sca{\xi_m}{-\bar{\gamma}}_\mu - \frac 12 \|\xi_m\|^2 - \|\xi_m\| \W_{\mu_m}(\bar{\gamma}_m,\tilde{\gamma}_m).
    \end{equation}
    Moreover, we have $\W_{\mu_m}(\bar{\gamma}_m, \tilde{\gamma}_m) \xrightarrow[m \to +\infty]{} 0$. Indeed:
    \begin{itemize}
        \item We have $\mu_m \to \mu$. Indeed, $\W_2(\mu_m,\mu) \leq \|\xi_m\|$ with $\|\xi_m\| \to 0$ by assumption.
        \item We have $\xi_m \to \bm{0}_\mu$. Indeed, since $\mu_m \to \mu$, we have $\bm{0}_{\mu_m} \to \bm{0}_\mu$, and since $\W_2(\bm{0}_{\mu_m},\xi_m) \leq \|\xi_m\|$ by \Cref{lemma:pseudo_norm_bounds_dist_to_zero} with $\|\xi_m\| \to 0$, we conclude that $\xi_m \to \bm{0}_\mu$.
        \item For every $m$, $\alpha_m \in \Gamma_\mu(\xi_m,\bar{\gamma})$, with $\xi_m \to \bm{0}_\mu$. Moreover there is an unique $\alpha \in \Gamma_\mu(\bm{0}_\mu,\bar{\gamma})$, as both plans are fully deterministic, and $\alpha$ is given by $\alpha = [(x,v) \mapsto (x,0,v)]^{(n)}(\bar{\gamma})$. Thus, by \Cref{cor:limit_points_of_seq_of_couplings}, $\alpha_m \to \alpha$.
        \item We have $\tilde{\gamma}_m = [\pi_2 \circ \PT_{1,1}^2]^{(n)}(\alpha_m)$ and $\alpha_m \to \alpha$, so that $\tilde{\gamma}_m$ converges to
        \begin{align}
            [\pi_2 \circ \PT_{1,1}^2]^{(n)}(\alpha) &= [\pi_2 \circ \PT_{1,1}^2]^{(n)} \circ [(x,v) \mapsto (x,0,v)]^{(n)}(\bar{\gamma}) \\
            &= [(x,v) \mapsto (\exp_x(0),\PT_{1,1}(x,0,v))]^{(n)}(\bar{\gamma}) \\
            &= [(x,v) \mapsto (x,v)]^{(n)}(\bar{\gamma}) = \bar{\gamma}.
        \end{align}
        Thus $\tilde{\gamma}_m \to \bar{\gamma}$.
        \item For every $m$, $\bar{\gamma}_m \in \Gamma_o(\mu_m,\bar{\mu})$, and $\mu_m \to \mu$. Since $\bar{\gamma}$ is by assumption the unique element of $\Gamma_o(\mu,\bar{\mu})$, by \Cref{cor:limit_points_of_seq_of_opt_vel_plans}, $\bar{\gamma}_m \to \bar{\gamma}$.
        \item We have thus proven that $\tilde{\gamma}_m$ and $\bar{\gamma}_m$ both converge to $\bar{\gamma}$. Since $\bar{\gamma}$ is by assumption fully deterministic, \Cref{prop:w_mu_is_lsc} implies that $\W_{\mu^m}(\bar{\gamma}_m,\tilde{\gamma}_m) \to \W_\mu(\bar{\gamma},\bar{\gamma}) = 0$.
    \end{itemize}
    Therefore, using \eqref{eq:l_2540}, we obtain
    \begin{equation}
        \liminf_{m \to +\infty} \frac{\cF(\mu_m) - \cF(\mu) - \sca{\xi_m}{-\bar{\gamma}}_\mu}{\|\xi_m\|} \geq \liminf_{m \to +\infty} -\frac 12 \|\xi_m\| - \W_{\mu^m}(\bar{\gamma}_m,\tilde{\gamma}_m) = 0.
    \end{equation}
    This finishes the proof.
\end{proof}

Another question relates to the convexity of the $\cF$ along constant speed geodesics. Unfortunately, we know that $\cF$ is in general not $\lambda$-geodesically convex for any $\lambda \in \R$: indeed, \citep[Example 9.1.5]{ambrosio2005gradient} gives an example of $\bar{\mu} \in \cP_2(\R^d)$ and of a constant speed geodesic $(\mu_t)_{t \in [0,1]}$ in $\cP_2(\R^d)$ for which $\W_2^2(\cdot,\bar{\mu})$ if not $\lambda$-convex along $(\mu_t)_t$ for any $\lambda \in \R$. By extension, for any $n > 0$, since $\W_2(\delta^{(n-1)}_\mu,\delta^{(n-1)}_\nu) = \W_2(\mu,\nu)$ for any $\mu, \nu \in \cP_2(\R^d)$, the curve $(\delta^{(n-1)}_{\mu_t})_t$ is a constant speed geodesic in $\cPn{n}{\R^d}$ on which $\W_2^2(\cdot,\delta^{(n-1)}_{\bar{\mu}})$ is not $\lambda$-convex for any $\lambda \in \R$. However, we can find another class of curve on which $\cF$ will be $1$-convex, which we will call ``generalized geodesics":

\begin{definition}
    Let $n \geq 0$ and $\mu_0,\mu_1,\bar{\mu} \in \cPn{n}{\cM}$. A \emph{``generalized geodesic" from $\mu$ to $\nu$ (with base $\bar{\mu}$)} is a curve $(\mu_t)_{t \in [0,1]}$ joining $\mu_0$ to $\mu_1$ in $\cPn{n}{\cM}$ of the form
    \begin{equation}
        \mu_t := [\exp \circ ((1-t)\pi_1 + t \pi_2)]^{(n)}(\alpha), \quad t \in [0,1]
    \end{equation}
    where $\gamma_0 \in \Gamma_o(\bar{\mu},\mu_0)$, $\gamma_1 \in \Gamma_o(\bar{\mu},\mu_1)$ and $\alpha \in \Gamma_{\bar{\mu}}(\gamma_0,\gamma_1)$.
\end{definition}

\begin{proposition}
    For every generalized geodesic $(\mu_t)_{t \in [0,1]}$ with base $\bar{\mu}$ induced by $\gamma_0 \in \Gamma_o(\bar{\mu},\mu_0)$, $\gamma_1 \in \Gamma_o(\bar{\mu},\mu_1)$ and $\alpha \in \Gamma_{\bar{\mu}}(\gamma_0,\gamma_1)$, it holds
    \begin{equation}
        \cF(\mu_t) \leq (1-t) \cF(\mu_0) + t \cF(\mu_1) - \frac 12 t(1-t) \bE^{(n)}_\alpha[\|v_0-v_1\|_x^2], \quad t \in [0,1].
    \end{equation}
    We say that $\cF$ is \emph{$1$-convex along generalized geodesics}.
\end{proposition}

\begin{proof}
    We prove this by induction. We first prove the case $n = 0$: let $y_0, y_1, x \in \cM$, $v_1, v_2 \in T_x\cM$ such that $\exp_x(v_i) = y_i$ and $\|v_i\|_x = d(x,y_i)$ for $i \in \{0,1\}$. For every $t \in (0,1)$, let $y_t := \exp_x((1-t)v_0 + tv_2)$. Then, we have for every $t \in [0,1]$,
    \begin{align}
        d^2(y_t,x) &= d^2(\exp_x((1-t)v_0 + tv_1), x) \\
        &\leq \|(1-t)v_0 + tv_1\|^2_x = (1-t)^2 \|v_0\|^2_x + t^2 \|v_1\|^2_x + 2t(1-t)\sca{v_0}{v_1}_x \\
        &\leq (1-t)^2 \|v_0\|^2_x + t^2 \|v_1\|^2_x + 2t(1-t)\frac{\|v_0\|^2_x + \|v_1\|^2_x - \|v_0-v_1\|^2_x}{2} \\
        &\leq (1-t) \|v_0\|^2_x + t \|v_1\|^2_x - t(1-t) \|v_0 - v_1\|^2_x
    \end{align}
    and this finishes the proof for $n = 0$. Now, let $n > 0$ and assume that the proposition holds for $n-1$. Let $\bar{\bP} \in \cPn{n}{\cM}$, and define the functional $\cF : \cPn{n}{\cM} \mapsto \R$ by $\cF(\bP) := \frac 12 \W_2^2(\bP,\bar{\bP})$. Let $(\bP_t)_{t \in [0,1]}$ be a generalized geodesic with base $\bar{\bP}$, induced by $\bGamma_0 \in \Gamma_o(\bar{\bP},\bP_0)$, $\bGamma_1 \in \Gamma_o(\bar{\bP},\bP_1)$ and $\bA \in \Gamma_{\bar{\bP}}(\bGamma_0,\bGamma_1)$. Then, for every $t \in [0,1]$, it holds
    \begin{align}
        \cF(\bP_t) &= \frac 12 \W_2^2(\bP_t, \bar{\bP}) \\
        &\leq \int \frac 12 \W_2^2([\exp \circ ((1-t)\pi_1+t\pi_2)]^{(n-1)}(\alpha), [\pi]^{(n-1)}(\alpha)) \dd\bA(\alpha) \\
        &\leq \int \frac 12 (1-t) \W_2^2([\exp \circ \pi_1]^{(n-1)}(\alpha), [\pi]^{(n-1)}(\alpha)) + \frac 12 t \W_2^2([\exp \circ \pi_2]^{(n-1)}(\alpha), [\pi]^{(n-1)}(\alpha)) \\
        &\quad - \frac 12 t(1-t) \bE^{(n-1)}_\alpha[\|v_0-v_1\|^2_x] \dd\bA(\alpha).
    \end{align}
    where the last inequality is obtained by applying the induction hypothesis to the integrand. Indeed, the optimality of $\bGamma_0$ and $\bGamma_1$ implies by \Cref{prop:opt_vel_plan_gives_opt_trans_plan} that for $\bA$-a.e. $\alpha$, the velocity plans $[\pi_1]^{(n-1)}(\alpha)$ and $[\pi_2]^{(n-1)}(\alpha)$ are optimal, so that $s \mapsto [\exp \circ ((1-s)\pi_1+s\pi_2)]^{(n-1)}(\alpha)$ defines a generalized geodesic in $\cPn{n-1}{\cM}$. Then, simplifying this inequality, we have
    \begin{align}
        \cF(\bP_t) &\leq \frac 12 (1-t) \int \W_2^2([\exp]^{(n-1)}(\gamma), [\pi]^{(n-1)}(\gamma)) \dd\bGamma_0(\gamma) \\
        &\quad + \frac 12 t \int \W_2^2([\exp]^{(n-1)}(\gamma), [\pi]^{(n-1)}(\gamma)) \dd\bGamma_1(\gamma) \\
        &\quad - \frac 12 t(1-t)\int \bE^{(n-1)}_\alpha[\|v_0-v_1\|^2_x] \dd\bA(\alpha) \\
        &\leq (1-t) \cF(\bP_0) + t\cF(\bP_1) - \frac 12 t(1-t) \bE^{(n)}_\bA[\|v_0-v_1\|^2_x]
    \end{align}
    where we used the optimality of $\bGamma_0$ and $\bGamma_1$. This finishes the proof.
\end{proof}

\bibliography{references}

\appendix
\newpage

\section{Category theory} \label{sec:appendix:category_theory}

We recall here some basic definitions from category theory.

\begin{definition}
    A \emph{category} $\cC$ is the given of:
    \begin{itemize}
        \item A set of \emph{objects} $\ob(\cC)$.
        \item For every $x, y \in \ob(\cC)$, a set of \emph{morphisms} between $x$ and $y$, denoted $\Hom_\cC(x,y)$. We will often write $f : x \mapsto y$ to denote an element of $\Hom_\cC(x,y)$.
        \item For every $x, y, z \in \ob(\cC)$, a \emph{composition} operation 
        \begin{equation}
            \begin{array}{ccc}
                 \Hom_\cC(x,y) \times \Hom_\cC(y,z) & \rightarrow & \Hom_\cC(x,z)  \\
                 (f, g) & \rightarrow & g \circ f
            \end{array}
        \end{equation}
    \end{itemize}
    satisfying the following axioms:
    \begin{itemize}
        \item Associativity: for every triple of morphisms $f : x \mapsto y$, $g : y \mapsto z$ and $h : z \mapsto w$ in $\cC$, it holds $(h \circ g) \circ f = h \circ (g \circ f)$.
        \item Identity: for every $x \in \ob(\cC)$, there exists an \emph{identity morphism} $\id_x : x \mapsto x$ such that, for every $f : x \mapsto y$ and $g : z \mapsto x$, it holds $f \circ \id_x = f$ and $\id_x \circ g = g$.
    \end{itemize}
\end{definition}

Given two objects $x,y$ of a category $\cC$, we say that they are \emph{isomorphic} if there exists a pair of morphisms $f : x \mapsto y$ and $g : y \mapsto x$ such that $f \circ g = \id_y$ and $g \circ f = \id_x$. These morphisms are called \emph{isomorphisms}.

\begin{example}
    Examples of categories include:
    \begin{enumerate}
        \item The category $\mathrm{Set}$ whose objects are sets and whose morphisms are maps between sets. Its isomorphisms are the bijective maps.
        \item The category $\mathrm{Vec}_\R$ whose objects are $\R$-vector spaces and whose morphisms are linear maps. Its isomorphisms are the linear isomorphisms.
        \item The category $\mathrm{Top}$ whose objects are topological spaces and whose morphisms are continuous maps. Its isomorphisms are the homeomorphisms.
        \item The category $\mathrm{Diff}$ whose objects are the differential manifolds and whose morphisms are smooth maps. Its isomorphisms are the smooth diffeomorphisms.
        \item For every partially ordered set $(X,\leq)$, the category $\cC = \mathrm{Pos}_X$ whose objects are the elements of $X$, and where for every $x,y \in X$, $\Hom_\cC(x,y) = \{*\}$ if $x \leq y$ and $\Hom_\cC(x,y) = \emptyset$ otherwise. 
    \end{enumerate}
\end{example}

Given a category $\cC$, we will often consider \emph{diagrams} in $\cC$. A diagram in $\cC$ is a directed graph whose nodes are objects in $\cC$ and where each directed edge corresponds to a morphism in $\cC$. For example, the following is a diagram in $\cC$:
\begin{equation}
    \begin{tikzcd}
        & A \arrow[dl, "f"'] \arrow[dr, "g"] & \\
        B \arrow[rr, "h"] & & C
    \end{tikzcd}
\end{equation}
where $A$, $B$ and $C$ are objects in $\cC$ and $f : A \mapsto B$, $g : A \mapsto C$ and $h : B \mapsto C$ are morphisms in $\cC$. We say that a diagram is \emph{commutative}, or that it \emph{commutes}, if, when we go from a node $x$ to a node $y$, if we compose the morphisms along the path we follow, we obtain the same morphism $x \mapsto y$ regardless of the path we choose. For example, consider the following diagram:
\begin{equation}
    \begin{tikzcd}
        A \arrow[r, "f_1"] \arrow[d, "f_2"] & B \arrow[d, "f_3"] \arrow[dr, bend left, "f_5"] & \\
        C \arrow[r, "f_4"] \arrow[dr, bend right, "f_8"] & D \arrow[r, "f_6"] \arrow[d, "f_7"] & E \arrow[d, "f_9"] \\
        & F \arrow[r, "f_{10}"] & G
    \end{tikzcd}
\end{equation}
Then, when we say that it is commutative, we mean that $f_3 \circ f_1 = f_4 \circ f_2$, $f_{10} \circ f_7 = f_9 \circ f_6$, $f_5 = f_6 \circ f_3$, $f_8 = f_7 \circ f_4$, and so on.
\medbreak
Given two categories $\cC$ and $\cD$, we can also define their \emph{product category} as the category $\cC \times \cD$ whose objects are pairs $(x,y)$ with $x \in \ob(\cC)$ and $y \in \ob(\cD)$, and whose morphisms $(x,x') \mapsto (y,y')$ are pairs $(f,g)$ with $f : x \mapsto x'$ and $g : y \mapsto y'$ (the composition is then the term-wise composition, and the identity maps the pairs $(\id_x,\id_y)$).

\begin{definition}
    Let $\cC$, $\cD$ be two categories. A \emph{functor} $F$ from $\cC$ to $\cD$, denoted $F : \cC \mapsto \cD$, is the given of:
    \begin{itemize}
        \item a map $\ob(\cC) \mapsto \ob(\cD)$ which to every object $x$ in $\cC$ associates an object $F(x)$ in $\cD$,
        \item for every pair $x,y \in \ob(\cC)$, a map $\Hom_\cC(x,y) \mapsto \Hom_\cD(F(x),F(y))$ which associates to every $f : x \mapsto y$ a morphism $F(f) : F(x) \mapsto F(y)$,
    \end{itemize}
    satisfying the following conditions:
    \begin{itemize}
        \item For every $x \in \ob(\cC)$, $F(\id_x) = \id_{F(x)}$.
        \item For every pair of morphisms $f : x \mapsto y$ and $g : y \mapsto z$ in $\cC$, it holds $F(g \circ f) = F(g) \circ F(f)$.
    \end{itemize}
\end{definition}

\begin{example}
    Examples of functors include:
    \begin{enumerate}
        \item Every category $\cC$ has an \emph{identity functor} $\Id_\cC$, such that $\Id_\cC(x) = x$ for every $x \in \ob(\cC)$ and $\Id_\cC(f) = f$ for every morphism $f : x \mapsto y$.
        \item We have a \emph{forgetful functor} $U : \mathrm{Vec}_\R \mapsto \mathrm{Set}$, such that, for every vector space $E$, $U(E)$ is its underlying set, and for every linear map $f : E \mapsto F$, $U(f)$ is the map $f$ seen as a map between sets.
        \item We similarly have forgetful functors $U : \mathrm{Top} \mapsto \mathrm{Set}$ and $U : \mathrm{Diff} \mapsto \mathrm{Set}$.
    \end{enumerate}
\end{example}

Given two functors $F : \cC_1 \mapsto \cC_2$ and $G : \cC_2 \mapsto \cC_3$, we may define their composition $G \circ F$ by $G \circ F(x) := G(F(x))$ and $G \circ F(f) = G(F(f))$ for every object $x$ and morphism $f : x \mapsto y$ in $\cC$. It is not difficult to check that $(H \circ G) \circ F = H \circ (G \circ F)$ for any triple of functors $F : \cC_1 \mapsto \cC_2$, $G : \cC_2 \mapsto \cC_3$, $H : \cC_3 \mapsto \cC_4$. Given a functor $F : \cC \mapsto \cC$, for every $n \geq 0$, we denote $F^{(n)} := F \circ \ldots \circ F$ the functor $\cC \mapsto \cC$ obtained by composing $F$ $n$ times with itself (with in particular $F^{(0)} = \Id_\cC$).

\begin{definition}
    Let $F, G : \cC \mapsto \cD$ be two functors between the same categories. A \emph{natural transformation} $\eta$ from $F$ to $G$, denoted $\eta : F \mapsto G$, is the given for every $x \in \ob(\cC)$ of a morphism $\eta_x : F(x) \mapsto G(x)$ in $\cD$, such that for every morphism $f : x \mapsto y$ in $\cD$, the following diagram commutes:
    \begin{equation}
        \begin{tikzcd}
            F(x) \arrow[r, "\eta_x"] \arrow[d, "F(f)"] & G(x) \arrow[d, "G(f)"] \\
            F(y) \arrow[r, "\eta_y"] & G(y)
        \end{tikzcd}
    \end{equation}
\end{definition}

\begin{definition}
    A \emph{monad} on a category $\cC$ is a triple $(T,\mu,\eta)$ where $T$ is a functor $\cC \mapsto \cC$, $\mu$ is a natural transformation $T \circ T \mapsto T$ and $\eta$ is a natural transformation $\Id_\cC \mapsto \cC$, such that for every object $x$ in $\cC$, the following diagrams commute:
    \begin{equation}
        \begin{array}{cc}
             \begin{tikzcd}
                 T(T(T(x))) \arrow[d, "\mu_{T(x)}"] \arrow[r, "T(\mu_x)"] & T(T(x)) \arrow[d, "\mu_x"] \\
                 T(T(x)) \arrow[r, "\mu_x"] & T(x)
             \end{tikzcd}
             & 
             \begin{tikzcd}
                 T(x) \arrow[r, "T(\eta_x)"] \arrow[d, "\eta_{T(x)}"] \arrow[dr, "\id_x"] & T(T(x)) \arrow[d, "\mu_x"] \\
                 T(T(x)) \arrow[r, "\mu_x"] & T(x)
             \end{tikzcd}
        \end{array}
    \end{equation}
\end{definition}

\section{Measure theory} \label{sec:appendix:measure_theory}

\subsection{The topology of weak convergence} \label{sec:appendix:weak_topology}

Let $X$ be a Polish space. Then we can endow the space of Borel probability measures $\cP(X)$ with the so-called \emph{weak topology}, which is defined as the weakest topology such that for every bounded continuous real valued function $\varphi \in C_b(X)$, the map
\begin{equation}
    \mu \mapsto \int \varphi(x) \dd\mu(x) 
\end{equation}
is continuous. We say that a sequence $(\mu_n)_n$ in $\cP(X)$ \emph{converges weakly} to $\mu \in \cP(X)$, and we note $\mu_n \rightharpoonup \mu$, if it converges to $\mu$ in the weak topology ; that is, for every bounded continuous function $\varphi \in C_b(X)$, 
\begin{equation}
    \int \varphi \dd\mu_n \xrightarrow[n \to +\infty]{} \int \varphi \dd\mu.
\end{equation}
If $f : X \mapsto Y$ is a continuous map between Polish spaces, its pushforward map $[f] : \cP(X) \mapsto \cP(Y)$ which associates to every $\mu \in \cP(X)$ its pushforward $[f](\mu) := f_\#\mu$ by $f$, is continuous with respect to the weak topologies of $\cP(X)$ and $\cP(Y)$.
\medbreak
The space $\cP(X)$ endowed with the weak topology is in fact a Polish space. Examples of distances metrizing the weak topology include the Lévy-Prokhorov distance, or the distance constructed in \citep[Remark 5.1.1]{ambrosio2005gradient}. For every Borel set $B \subseteq X$, the map $\mu \mapsto \mu(B)$ is Borel measurable. Furthermore, it is lower semicontinuous if $B$ is an open set, and it is upper semicontinuous if $B$ is a closed set. Moreover, if $f : X \mapsto \R$ is lower bounded and lower semicontinuous, the the map
\begin{equation}
    \mu \mapsto \int f(x) \dd\mu(x)
\end{equation}
is also lower semicontinuous, and it is instead upper semicontinuous if $f$ is upper bounded and upper semicontinuous (see \citep[Lemma 5.1.7]{ambrosio2005gradient}). Compactness in the weak topology is characterized by \emph{Prokhorov's theorem}:

\begin{theorem}{\citep[Theorem 5.1.3 (Prokhorov)]{ambrosio2005gradient}} \label{th:prokhorov}
    A set $\cK \subseteq \cP(X)$ is relatively compact for the topology of weak convergence if and only if it is \emph{tight}, that is,
    \begin{equation}
        \forall \veps > 0, \quad \exists K_\veps \subseteq X \hbox{ compact}, \quad \forall \mu \in \cK, \quad \mu(X \setminus K_\veps) \leq \veps.
    \end{equation}
\end{theorem}

The following tightness criterion is also useful:
\begin{lemma}{\citep[Lemma 5.2.2 (Tightness criterion)]{ambrosio2005gradient}} \label{lemma:tightness_criterion}
    Let $X,X_1,\ldots,X_n$ be separable metric spaces and let $r^i : X \mapsto X_i$ be continuous maps such that the product map $(r^1,\ldots,r^n) : X \mapsto \prod_{i=1}^n X_i$ is proper (i.e. the preimage of a compact set is compact). Let $\cK \subseteq \cP(X)$ be such that for every $i = 1,\ldots,n$, $\cK_i := r^i_\#(\cK)$ is tight in $\cP(X_i)$. Then $\cK$ is also tight in $\cP(X)$. In particular the map $(r^1_\#,\ldots,r^n_\#) : \cP(X) \mapsto \prod_{i=1}^n \cP(X_i)$ is proper for the weak topology.
\end{lemma}

If $X$ and $Y$ are Polish spaces, and $f \in C_b(X \times Y)$ is a continuous and bounded function, then the map
\begin{equation} \label{eq:partial_integral_weakly_continuous}
    (x,\mu) \in X \times \cP(Y) \mapsto \int f(x,y) \dd\mu(y)
\end{equation}
is continuous. Indeed, let $M$ be a bound on $|f|$, and let $x_n \to x$ in $X$ and $\mu_n \rightharpoonup \mu$ in $\cP(Y)$. Then, for every $n$, we have
\begin{equation}
    \int f(x_n,y) \dd\mu_n(y) = \int f(x,y) \dd\mu_n(y) + \int (f(x_n,y) -f(x,y)) \dd\mu_n(y) =: A_n + B_n
\end{equation}
and $\mu_n \rightharpoonup \mu$ implies that $A_n \mapsto \int f(x,y) \dd\mu(y)$. Moreover, since $(\mu_n)_n$ weakly converges, by Prokhorov's theorem, it is tight, so that for every $\veps > 0$, there exists a compact set $K_\veps \subseteq Y$ such that $\mu_n(Y \setminus K_\veps) \leq \veps$ for every $n$. Furthermore, by compactness of $K_\veps$ and continuity of $f$, we have, for $n$ large enough, $|f(x_n,y) - f(x,y)| \leq \veps$ for every $y \in K_\veps$, so that for $n$ large enough,
\begin{align}
    |B_n| &\leq \int_{K_\veps} |f(x_n,y)-f(x,y)|\dd\mu_n(y) + \int_{Y\setminus K_\veps} |f(x_n,y)-f(x,y)|\dd\mu_n(y) \\
    &\leq \veps + 2M \mu_n(Y \setminus K_\veps) \leq \veps + 2M\veps.
\end{align}
Since $\veps$ is arbitrary, we thus have $B_n \to 0$, and we indeed find
\begin{equation}
    \int f(x_n,y) \dd\mu_n(y) \xrightarrow[n \to +\infty]{} \int f(x,y) \dd\mu(y).
\end{equation}

\subsection{Disintegration of measures} \label{sec:appendix:disintegration}

Let $\Omega$ be a measurable space, and $X$ be a Polish space. A family $(\mu_\omega)_{\omega \in \Omega}$ of measures in $\cP(X)$ is said to be \emph{Borel} if either of the following equivalent conditions hold:
\begin{enumerate}
    \item The map $\omega \in \Omega \mapsto \mu_\omega \in \cP(X)$ is Borel measurable (with the topology of weak convergence on $\cP(X)$).
    \item The map $\omega \in \Omega \mapsto \mu_\omega(B)$ is Borel measurable for every Borel set $B \subseteq X$, 
    \item The map $\omega \in \Omega \mapsto \mu_\omega(B)$ is Borel measurable for every Borel set $B$ in a class of subsets of $X$ generating the Borel $\sigma$-algebra $\cB(X)$.
    \item The map $\omega \in \Omega \mapsto \int \varphi(x)\dd\mu_\omega(x)$ is Borel measurable for every bounded continuous function $\varphi \in C_b(X)$.
\end{enumerate}
In particular, it is readily verified that given a map $f : \Omega \mapsto X$, the family $(\delta_{f(\omega)})_{\omega \in \Omega}$ is Borel if and only if $f$ is Borel measurable.
\medbreak
Now, let $X$, $Y$ be two Polish spaces, and consider a Borel family $(\mu_x)_{x\in X}$ of measures in $\cP(Y)$. Then, there exists corresponding to every $\nu \in \cP(X)$ a unique measure $\mu \in \cP(X \times Y)$ such that for every nonnegative Borel measurable function $f : X \times Y \mapsto \R$, 
\begin{equation}
    \int f(x,y) \dd\mu(x,y) = \iint f(x,y) \dd\mu_x(y) \dd\nu(x).
\end{equation}

Conversely, we have the so-called \emph{disintegration theorem}, whose statement is the following:
\begin{theorem}{\citep[Theorem 5.3.1 (Disintegration)]{ambrosio2005gradient}} \label{th:disintegration}
    Let $Y$, $X$ be Polish spaces and $\pi : Y \mapsto X$ a Borel measurable map. Let $\mu \in \cP(Y)$ and define $\nu := \pi_\#\mu \in \cP(X)$. Then there is a $\nu$-a.e. uniquely defined Borel family $(\mu_x)_{x\in X}$ of probability measures in $\cP(Y)$ such that $\mu_x$ is concentrated on $\pi^{-1}(x)$ for $\nu$-a.e. $x \in X$, and $\mu$ is the unique probability measure such that for every nonnegative Borel measurable map $f : Y \mapsto \R$,
    \begin{equation}
        \int f(y) \dd\mu(y) = \int \int_{\pi^{-1}(x)} f(y) \dd\mu_x(y) \dd\nu(x).
    \end{equation}
    The family $(\mu_x)_x$ is called the \emph{disintegration of $\mu$ with respect to $\nu$}. When $Y = X \times Z$ is a product space, the $\mu_x$ can be seen as a family of measures in $\cP(Z)$, such that for every nonnegative Borel measurable $f : X \times Z \mapsto \R$,
    \begin{equation} \label{eq:disintegration_prod_space_definition}
        \int f(x,z) \dd\mu(x,z) = \iint f(x,z) \dd\mu_x(z) \dd\nu(x).
    \end{equation}
\end{theorem}
Given a measure $\mu \in \cP(Y)$ with marginal $\nu = f_\#\mu \in \cP(X)$, we will often write $\dd\mu(y) = \dd\mu_x(y) \dd\nu(x)$ to denote that $(\mu_x)_x$ is the family of probability measures associated to $\mu$ by the disintegration theorem. It is again readily verified that $\mu$ can be written in the form $\mu = g_\#\nu$, where $g : X \mapsto Y$ is a measurable function such that $f \circ g = \id_X$, if and only if its disintegration $\mu_x$ is a Dirac measure for $\nu$-almost every $x \in X$. We refer to \citep[Chapter 5.3]{ambrosio2005gradient} or \citep[Section 2]{bogachev2020kantorovich} for a more detailed discussion. Of note is also the following result:

\begin{theorem}{\citep[Corollary 1.6 (Universal disintegration)]{kallenberg2017random}} \label{th:universal_disintegration}
    If $X$, $Y$ are Polish spaces, there exists a Borel measurable map $\cK_{X,Y} : X \times \cP(X \times Y) \mapsto \cP(Y)$ such that for every $\mu \in \cP(X \times Y)$ with first marginal $\nu \in \cP(X)$,
    \begin{equation}
        \dd\mu(x,y) = \dd\cK_{X,Y}(x,\mu)(y) \dd \nu(x).
    \end{equation}
\end{theorem}

It is then clear that given a measure $\mu \in \cP(X \times Y \times Z)$, the map $(x,y) \in X \times Y \mapsto \cK_{Y,Z}(y, \cK_{X,Y \times Z}(x, \mu)) \in \cP(Z)$ is Borel, and is also a disintegration of $\mu$ with respect to its marginal $\nu = (\pi_1,\pi_2)_\#\mu$ over the first two variables. Therefore, by unicity, it holds
\begin{equation} \label{eq:disintegration_composition}
    \cK_{Y,Z}(y, \cK_{X,Y \times Z}(x, \mu)) = \cK_{X \times Y, Z}((x,y),\mu), \quad \nu-\hbox{a.e.} (x,y) \in X \times Y.
\end{equation}
The universal disintegration theorem thus implies that we can disintegrate a measure over several variables by disintegrating it with respect to one, and then by successively ``disintegrating the disintegrations".

\begin{remark}
    If $\mu \in \cP(X \times Y)$ has disintegration $\dd\mu(x,y) = \dd\mu_x(y) \dd\nu(y)$, then, for every Borel set $A \subseteq X \times Y$, applying \eqref{eq:disintegration_prod_space_definition} to the indicator function $\bOne_A$, we find
    \begin{equation}
        \mu(A) = \int \mu_x(A_x) \dd\nu(x)
    \end{equation}
    where $A_x := \{y \in Y \setcond (x,y) \in A\}$ for every $x \in X$. In particular, if $P$ is some property on $X \times Y$ such that the set $\{(x,y) \in X \times Y \setcond P(x,y) \hbox{ is true}\}$ is Borel measurable, then the assertions ``$P$ is true for $\mu$-almost every $(x,y)$" and ``$P$ is true for $\nu$-almost every $x$, $\mu_x$-almost every $y$" are equivalent.
\end{remark}

\begin{remark} \label{rk:disintegration_given_by_map_of_two_variables}
    If $\mu \in \cP(X \times Y \times Z)$ has disintegration $\dd\mu(x) = \dd\mu_x(y,z) \dd\nu(x)$, then, since the set $\{\delta_z, z \in Z \}$ is weakly closed in $\cP(Z)$ (this is a consequence of \citep[Proposition 5.1.8]{ambrosio2005gradient}), using the previous remark along with \eqref{eq:disintegration_composition}, the assertions
    \begin{itemize}
        \item There exists a Borel measurable $f : X \times Y \mapsto Z$ such that $\mu = f_\#\pi_{1,2\#}\mu$
        \item For $\nu$-almost every $x$, there exists a Borel measurable $f_x : Y \mapsto Z$ such that $\mu_x = f_{x\#}\pi_{1\#}\mu_x$
    \end{itemize}
    are equivalent.
\end{remark}

\subsection{Measurable selection results}

The following measurable selection theorem was stated in \citep[Chapter 5, Bibliographical notes]{villani2009optimal} :

\begin{theorem} \label{th:selection}
    Let $f : X \mapsto Y$ be a Borel measurable map between Polish spaces which is surjective and such that for every $y \in Y$, the fiber $f^{-1}(y)$ is compact. Then $f$ admits a Borel right-inverse, that is, a Borel measurable function $g : Y \mapsto X$ such that $f \circ g = \id_Y$.
\end{theorem}

\section{Optimal transport theory} \label{sec:appendix:ot_theory}

\subsection{The Kantorovich problem} \label{sec:appendix:ot_solutions}

Let $X$ and $Y$ be two Polish spaces, and $c : X \times Y \mapsto [0,+\infty]$ be a nonnegative lower semicontinuous cost function. Given two probability measures $\mu \in \cP(X)$ and $\nu \in \cP(Y)$, we may define their \emph{optimal transport cost} $\W_c(\mu,\nu)$ associated to the cost $c$, which is defined by
\begin{equation} \label{eq:kantorovich_problem_c}
    \W_c(\mu, \nu) := \inf_{\gamma \in \Pi(\mu,\nu)} \int c(x,y) \dd\gamma(x,y),
\end{equation}
where $\Pi(\mu,\nu)$ is the set of transport plans between $\mu$ and $\nu$, that is, of measures $\gamma \in \cP(X \times Y)$ with respective first and second marginals $\mu$ and $\nu$. It is a general result of optimal transport theory (see for example \citep[Theorem 4.1]{villani2009optimal}) that there exists $\gamma \in \Pi(\mu,\nu)$ that minimize \eqref{eq:kantorovich_problem_c}. We denote $\Pi_c(\mu,\nu)$ the set of such minimizers, which we call \emph{optimal transport plans} between $\mu$ and $\nu$ for the cost $c$. An important result of optimal transport theory is the so-called \emph{Kantorovich's duality}:

\begin{theorem}{\citep[Theorem 5.10 (partial version)]{villani2009optimal}} \label{th:kantorovich_duality_c}
    Assume that there exists two real valued upper semicontinuous functions $a \in L^1(\mu)$ and $b \in L^1(\nu)$ such that $c(x,y) \geq a(x) + b(y)$ for every $(x,y) \in X \times Y$. Then there is duality:
    \begin{align}
        \min_{\gamma \in \Pi(\mu,\nu)} \int c(x,y) \dd\gamma(x,y) &= \sup_{\substack{(\phi,\psi) \in C_b(X) \times C_b(Y) \\ \phi + \psi \leq c}} \left(\int \phi(x) \dd\mu(x) + \int \psi(y) \dd\nu(y) \right) \label{eq:kantorovich_duality_c_continuous} \\
        &= \sup_{\substack{(\phi,\psi) \in L^1(\mu) \times L^1(\nu) \\ \phi + \psi \leq c}} \left(\int \phi(x) \dd\mu(x) + \int \psi(y) \dd\nu(y) \right).
    \end{align}
    Moreover, if $c$ is real valued, $\W_c(\mu,\nu) < +\infty$, and there exists $c_X \in L^1(\mu)$ and $c_Y \in L^1(\nu)$ such that $c(x,y) \leq c_X(x) + c_Y(y)$ for every $(x,y) \in X \times Y$, then the dual problem also has solutions:
    \begin{align}
        \min_{\gamma \in \Pi(\mu,\nu)} \int c(x,y) \dd\gamma(x,y) = \max_{\substack{(\phi,\psi) \in L^1(\mu) \times L^1(\nu) \\ \phi + \psi \leq c}} \left(\int \phi(x) \dd\mu(x) + \int \psi(y) \dd\nu(y) \right).
    \end{align}
\end{theorem}

A direct consequence of \Cref{th:kantorovich_duality_c} is that the optimal transport cost $\W_c : \cP(X) \times \cP(X) \mapsto [0, +\infty]$ is lower semicontinuous (with respect to the topology of weak convergence), as, by \eqref{eq:kantorovich_duality_c_continuous}, $\W_c$ writes as a supremum of continuous functions. We also have the following measurable selection result of optimal transport plans:

\begin{theorem}{\citep[Theorem 4.2]{bogachev2020kantorovich}} \label{th:fibered_opt_plan_selection}
    Let $X$, $Y$ and be Polish spaces, $Z$ a Borel subset of some larger Polish space $Z'$, and $h : Z \times X \times Y \mapsto [0,+\infty]$ a measurable function such that $h_z = h(z,\cdot,\cdot)$ is nonnegative lower semicontinuous for every $z \in Z$. Let $(\mu_z)_{z \in Z}$ and $(\nu_z)_{z \in Z}$ be two Borel families of measures, respectively in $\cP(X)$ and in $\cP(Y)$, such that for every $z \in Z$, $\W_{h_z}(\mu_z,\nu_z) < +\infty$. Then the map $z \mapsto \W_{h_z}(\mu_z,\nu_z)$ is measurable, and there exists a Borel family of measures $(\gamma_z)_{z \in Z}$ in $\cP(X \times Y)$, such that for every $z \in Z$, $\gamma_z \in \Pi_{h_z}(\mu_z,\nu_z)$.
\end{theorem}

\begin{remark}
    Theorem 4.2 in \citep{bogachev2020kantorovich} technically requires $h$ to be valued in $[0,+\infty)$, but one can check that their proof in fact works even when $h$ can take infinite values, as long as $\W_{h_z}(\mu_z,\nu_z) < +\infty$ for every $z \in Z$. 
\end{remark}

\subsection{The Wasserstein space} \label{sec:appendix:w2_space}

Let $(X,d)$ be a Polish space, with associated metric $d$. We may then consider the cost function $c : X \times X \mapsto \R_+$ defined by $c(x,y) := d^2(x,y)$. For every pair $\mu, \nu \in \cP(X)$, we define $\W_2(\mu,\nu) := \sqrt{\W_c(\mu,\nu)}$, which may possibly be infinite. Note that for every $x_0 \in X$, the second order moment of $\mu$ with respect to $x_0$ is given by 
\begin{equation}
    \W_2^2(\delta_{x_0}, \mu) = \int d^2(x_0, x) \dd\mu(x).
\end{equation}
Moreover, for every pair of probability measures with finite second order moment $\mu, \nu \in \cP_2(X)$, it is known that $\W_2(\mu,\nu) < +\infty$. It can then be verified that $\cP_2(X)$ equipped with $\W_2$ is in fact a metric space (see \citep[Definition 6.1]{villani2009optimal}), and we call $\W_2(\mu,\nu)$ the \emph{second order Wasserstein distance} between $\mu$ and $\nu$.

\begin{theorem}{\citep[Theorem 6.18, Remark 6.19]{villani2009optimal}} \label{th:prob_space_is_polish}
    The space $(\cP_2(X), \W_2)$ is a Polish space, which is compact if and only if $X$ is compact. Moreover any probability measure can be approximated by a sequence of probability measures with finite support.
\end{theorem}

We have several equivalent characterizations of convergence in the $\W_2$ topology:
\begin{theorem}{\citep[Theorem 6.9]{villani2009optimal}} \label{th:convergence_in_w2_space}
    Let $(\mu_n)_{n \in \N}$ be a sequence in $\cP_2(X)$ and let $\mu \in \cP(X)$. Then $(\mu_n)_n$ converges to  $\mu$ in the $\W_2$ topology if and only if one of the following equivalent conditions hold for some (and then any) $x_0 \in X$:
    \begin{enumerate}
        \item \label{enum:w2_convergence:1_2nd_moment_cvg} $\mu_n \rightharpoonup \mu$ and $\int d^2(x_0,x) \dd\mu_n(x) \xrightarrow[n \to +\infty]{} \int d^2(x_0,x) \dd\mu(x)$
        \item \label{enum:w2_convergence:2_2nd_moment_lsc} $\mu_n \rightharpoonup \mu$ and $\limsup_{n \to +\infty} \int d^2(x_0,x) \dd\mu_n(x) \leq \int d^2(x_0,x) \dd\mu(x)$
        \item \label{enum:w2_convergence:3_2nd_moment_unif_int} $\mu_n \rightharpoonup \mu$ and $\lim_{R \to +\infty} \limsup_{n \to +\infty} \int_{d(x_0,x) \geq R} d^2(x_0,x) \dd\mu_n(x) = 0$
        \item \label{enum:w2_convergence:4_quadratic_growth} For every continuous function $\varphi : X \mapsto \R$ with $\forall x \in X, |\varphi(x)| \leq C(1 + d^2(x_0,x))$ for some $C \geq 0$, $\int \varphi \dd\mu_n \xrightarrow[n \to +\infty]{} \int \varphi \dd\mu$ 
    \end{enumerate}
\end{theorem}

From this, the following relations hold between the respective topologies and Borel $\sigma$-algebras of $\cP(X)$ equipped with the topology of weak convergence and $\cP_2(X)$ equipped with the $\W_2$ topology: first, the $\W_2$ topology is stronger than the topology of weak convergence. In particular, the inclusion $\cP_2(X) \subseteq \cP(X)$ is continuous. Moreover, the Borel $\sigma$-algebra on $\cP_2(X)$ induced by the $\W_2$ topology is equal to the restriction to $\cP_2(X)$ of the Borel $\sigma$-algebra on $\cP(X)$ induced by the topology of weak convergence, that is, $\cB(\cP_2(X)) = \cB(\cP(X))_{|\cP_2(X)}$. In particular, a map $E \mapsto \cP_2(X)$ is $\cB(\cP_2(X))$-measurable if and only if it is $\cB(\cP(X))$-measurable when seen as a map valued in $\cP(X)$. The argument is the following: first, as the $\W_2$ topology is stronger than the topology of weak convergence, we have $\cB(\cP(X))_{|\cP_2(X)} \subseteq \cB(\cP_2(X))$. Conversely, since $\W_2 : \cP(X) \times \cP(X) \mapsto [0,+\infty]$ is lower semicontinuous for the topology of weak convergence (see the previous subsection), the open balls for $\W_2$ are measurable in $\cB(\cP(X))_{|\cP_2(X)}$, so that $\cB(\cP_2(X)) \subseteq \cB(\cP(X))_{|\cP_2(X)}$. 
\medbreak
Moreover, compactness in the Wasserstein space is characterized by the following result:
\begin{proposition}{\citep[Proposition 7.1.5]{ambrosio2005gradient}} \label{prop:rel_compact_sets_in_w2_topology}
    A set $\cK \subseteq \cP_2(X)$ is relatively compact for the $\W_2$ topology if and only if it is tight and uniformly $2$-integrable - that is, for some (any) $x_0 \in X$,
    \begin{equation}
        \lim_{R \to +\infty} \sup_{\mu \in \cK} \int_{d(x_0,x) \geq R} d^2(x_0,x) \dd\mu(x) = 0
    \end{equation}
\end{proposition}

\section{The Sasaki metric} \label{sec:appendix:sasaki_metric} 

If $(\cM,g)$ is a ($n$-dimensional) Riemannian manifold, its tangent bundle $T\cM$ can be equipped with a natural metric $g_T$ called the \emph{Sasaki metric}, for which the projection $p : (T\cM,g_T) \mapsto (\cM,g)$ is a Riemannian submersion. It is defined the following way: let $D$ be the Levi-Civita connection on $\cM$, and let $(m,x_m) \in T\cM$. Then:
\begin{itemize}
    \item The subspace $\cV_{m,x_m}\cM \subseteq T_{m,x_m}T\cM$ of \emph{vertical vectors} is the kernel of the tangent map at $(m,x_m)$ of the projection $p$. Taking the tangent map at $x_m$ of the inclusion $T_m\cM \subseteq T\cM$ yields a natural isomorphism between $T_m \cM$ and $\cV_{m,x_m}\cM$. It is a $n$-dimensional subspace of $T_{m,x_m}T\cM$.
    \item The \emph{horizontal vectors} of $T_{m,x_m}\cM$ are the velocities of the curves $(c,X) : (-\veps,\veps) \mapsto T\cM$ with $c(0) = m$, $X(0) = x_m$, and such that the vector field $X$ is parallel along $c$ (that is $D_{c'}X = 0$). If, in a set of local coordinates, $(u_1,\ldots,u_n,X^1,\ldots,X^n)$ are the coordinates of $(m,x_m)$, and $(v^1,\ldots,v^n,\Xi^1,\ldots,\Xi^n)$ are the coordinates of $(c'(0),X'(0))$, then $(c'(0),X'(0))$ is horizontal if and only if 
    \begin{equation}
        \Xi_i + \sum_{j,k=1}^n \Gamma_{j,k}^i(u^1,\ldots,u^n)v^j X^k = 0, \quad \forall i \in \{1,\ldots,n\}
    \end{equation}
    (where the $\Gamma^i_{j,k}$ are the Christoffel symbols). The subspace of horizontal vectors is of dimension $n$, and is a complement of the subspace of vertical vectors.
\end{itemize}
The Sasaki metric $g_T$ is then defined the following way:
\begin{itemize}
    \item On $\cV_{m,x_m}\cM$, we identify $g_T$ and $g$ through the isomorphism $\cV_{m,x_m} \cong T_m \cM$. 
    \item If $\Xi, \Xi' \in T_{m,x_m}T\cM$ are horizontal, $g_T(\Xi,\Xi') = g(v,v')$ where $v,v'$ are the projections of $\Xi,\Xi'$ on $T_m\cM$. 
    \item If $\Xi, \Xi' \in T_{m,x_m}T\cM$ with $\Xi$ vertical and $\Xi'$ horizontal, $g_T(\Xi,\Xi') = 0$.
\end{itemize}

The geodesic distance induced by the Sasaki metric has the following expression: for every $(x,u), (y,v) \in T\cM$, one has
\begin{equation} \label{eq:geodesic_distance_sasaki}
    d^2((x,u),(y,v)) = \inf_\gamma \left\{ \mathrm{Length}^2(\gamma) + \|\PT_{x \to y}^\gamma(u) - v\|^2_y \right \}
\end{equation}
where the infimum is taken along piecewise smooth paths $\gamma : [0,1] \mapsto \cM$ from $x$ to $y$ whose derivative is never zero, $\mathrm{Length}(\gamma) := \int_0^1 \|\dot{\gamma}\|_{\gamma(t)} \dd t$ is the length of $\gamma$, and $\PT^\gamma_{x \to y}$ is the parallel transport operator along $\gamma$ from $x$ to $y$. In particular, for every $o \in \cM$ and $(x,v) \in T\cM$, it holds
\begin{equation} \label{eq:geodesic_distance_sasaki_2}
    d^2((o,0), (x,v)) = d^2(o,x) + \|v\|^2_x
\end{equation}
as the parallel transport operator is an isometry.
\medbreak
We refer to \citep[Chapter 2.B.6]{gallot2004riemannian} and to \citep[Appendix II.A.2.1]{Canary_Marden_Epstein_2006} for more details on the Sasaki metric.

\end{document}